\documentclass[a4paper,10pt,reqno]{amsart}
\usepackage{amsmath,amsfonts,amsthm,amssymb,xcolor}
\usepackage[T1]{fontenc}
\usepackage{pdfsync}
\usepackage{csquotes}
\usepackage{graphicx}
\usepackage{pstricks}
\usepackage{lmodern}
\usepackage{calc}
\usepackage{mathabx}
\usepackage{mathrsfs}
\usepackage{ulem}
\usepackage{hyperref}
\usepackage{stmaryrd}
\usepackage{array}

\newcommand{\W}{\mathcal{W}}

\usepackage{caption} 
\usepackage{scalerel}

\newcommand{\pe}{\mathbin{\scaleobj{0.7}{\tikz \draw (0,0) node[shape=circle,draw,inner sep=0pt,minimum size=8.5pt] {\footnotesize $=$};}}}

\newcommand{\pl}{\mathbin{\scaleobj{0.7}{\tikz \draw (0,0) node[shape=circle,draw,inner sep=0pt,minimum size=8.5pt] {\footnotesize $<$};}}}
\newcommand{\pg}{\mathbin{\scaleobj{0.7}{\tikz \draw (0,0) node[shape=circle,draw,inner sep=0pt,minimum size=8.5pt] {\footnotesize $>$};}}}
\newcommand{\ple}{\mathbin{\scaleobj{0.7}{\tikz \draw (0,0) node[shape=circle,draw,inner sep=0pt,minimum size=8.5pt] {\footnotesize $\leqslant$};}}}
\newcommand{\pge}{\mathbin{\scaleobj{0.7}{\tikz \draw (0,0) node[shape=circle,draw,inner sep=0pt,minimum size=8.5pt] {\footnotesize $\geqslant$};}}}

\usepackage{accents}

\usepackage{enumerate}

\numberwithin{equation}{section}
\usepackage{tikz}

\newcommand{\id}{\mbox{Id}}

\newcommand{\1}{\mathbf{1}}

\newcommand{\eps}{\varepsilon}

\newcommand{\dis}{\displaystyle}

\newcommand{\ov}{\overline}

\newcommand{\B}{\mathcal{B}}

\newcommand{\cq}{\mathcal{Q}}

\newcommand{\R}{\mathbb R}
\newcommand{\N}{\mathbb N}
\newcommand{\Z}{\mathbb Z}
\newcommand{\T}{\mathbb T}

\newcommand{\ca}{\mathcal A}
\newcommand{\cb}{\mathcal B}
\newcommand{\cac}{\mathcal C}
\newcommand{\cd}{\mathcal D}
\newcommand{\ce}{\mathcal E}
\newcommand{\cf}{\mathcal F}

\newcommand{\ch}{\mathcal H}
\newcommand{\ci}{\mathcal I}

\newcommand{\cm}{\mathcal M}

\newcommand{\cp}{\mathcal P}

\newcommand{\cw}{\mathcal W}

\newcommand{\calt}{\mathcal T}
\newcommand{\calti}{\widetilde{\calt}}

\newcommand{\frakt}{\mathfrak{T}}
\newcommand{\frakti}{\widetilde{\mathfrak{T}}}

\newcommand{\frakc}{\mathfrak{c}}

\newcommand{\scret}{\mathscr{T}}
\newcommand{\screti}{\widetilde{\mathscr{T}}}

\newcommand{\scrq}{\mathscr{Q}}
\newcommand{\scrf}{\mathscr{F}}
\newcommand{\scra}{\mathscr{A}}
\newcommand{\scrk}{\mathscr{K}}
\newcommand{\scri}{\mathscr{I}}

\newcommand{\al}{\alpha}

\newcommand{\ga}{\gamma}
\newcommand{\gga}{\Gamma}
\newcommand{\ka}{\kappa}
\newcommand{\la}{\lambda}

\newcommand{\si}{\sigma}

\newcommand{\vp}{\varphi}

\newtheorem{theorem}{Theorem}[section]
\newtheorem*{acknowledgements}{Acknowledgements}

\newtheorem{corollary}[theorem]{Corollary}

\newtheorem{definition}[theorem]{Definition}

\newtheorem{lemma}[theorem]{Lemma}

\newtheorem{proposition}[theorem]{Proposition}

\theoremstyle{remark}
\newtheorem{remark}[theorem]{Remark}

\pgfdeclareshape{crosscircle}
{
  \inheritsavedanchors[from=circle] 
  \inheritanchorborder[from=circle]
  \inheritanchor[from=circle]{north}
  \inheritanchor[from=circle]{north west}
  \inheritanchor[from=circle]{north east}
  \inheritanchor[from=circle]{center}
  \inheritanchor[from=circle]{west}
  \inheritanchor[from=circle]{east}
  \inheritanchor[from=circle]{mid}
  \inheritanchor[from=circle]{mid west}
  \inheritanchor[from=circle]{mid east}
  \inheritanchor[from=circle]{base}
  \inheritanchor[from=circle]{base west}
  \inheritanchor[from=circle]{base east}
  \inheritanchor[from=circle]{south}
  \inheritanchor[from=circle]{south west}
  \inheritanchor[from=circle]{south east}
  \inheritbackgroundpath[from=circle]
  \foregroundpath{
    \centerpoint%
    \pgf@xc=\pgf@x%
    \pgf@yc=\pgf@y%
    \pgfutil@tempdima=\radius%
    \pgfmathsetlength{\pgf@xb}{\pgfkeysvalueof{/pgf/outer xsep}}%
    \pgfmathsetlength{\pgf@yb}{\pgfkeysvalueof{/pgf/outer ysep}}%
    \ifdim\pgf@xb<\pgf@yb%
      \advance\pgfutil@tempdima by-\pgf@yb%
    \else%
      \advance\pgfutil@tempdima by-\pgf@xb%
    \fi%
    \pgfpathmoveto{\pgfpointadd{\pgfqpoint{\pgf@xc}{\pgf@yc}}{\pgfqpoint{-0.707107\pgfutil@tempdima}{-0.707107\pgfutil@tempdima}}}
    \pgfpathlineto{\pgfpointadd{\pgfqpoint{\pgf@xc}{\pgf@yc}}{\pgfqpoint{0.707107\pgfutil@tempdima}{0.707107\pgfutil@tempdima}}}
    \pgfpathmoveto{\pgfpointadd{\pgfqpoint{\pgf@xc}{\pgf@yc}}{\pgfqpoint{-0.707107\pgfutil@tempdima}{0.707107\pgfutil@tempdima}}}
    \pgfpathlineto{\pgfpointadd{\pgfqpoint{\pgf@xc}{\pgf@yc}}{\pgfqpoint{0.707107\pgfutil@tempdima}{-0.707107\pgfutil@tempdima}}}
  }
}
\makeatother

\definecolor{gr}{rgb}   {0.,   0.69,   0.23 }
\definecolor{bl}{rgb}   {0.,   0.5,   1. }
\definecolor{mg}{rgb}   {0.85,  0.,    0.85}
\definecolor{yl}{rgb}   {0.8,  0.7,   0.}
\definecolor{or}{rgb}  {0.7,0.2,0.2}

\colorlet{symbols}{black!90!black}
\colorlet{symbolsb}{black!90!black}
\colorlet{testcolor}{green!60!black}

\usetikzlibrary{shapes.misc}
\usetikzlibrary{shapes.symbols}
\usetikzlibrary{decorations}
\usetikzlibrary{decorations.markings}

\def\drawx{\draw[-,solid] (-3pt,-3pt) -- (3pt,3pt);\draw[-,solid] (-3pt,3pt) -- (3pt,-3pt);}
\tikzset{
	root/.style={circle,fill=testcolor,inner sep=0pt, minimum size=2mm},
	dot/.style={circle,fill=black,inner sep=0pt, minimum size=1mm},
	var/.style={circle,fill=black!10,draw=black,inner sep=0pt, minimum size=2mm},
	dotred/.style={circle,fill=black!50,inner sep=0pt, minimum size=2mm},
	generic/.style={semithick,shorten >=1pt,shorten <=1pt},
	dist/.style={ultra thick,draw=testcolor,shorten >=1pt,shorten <=1pt},
	testfcn/.style={ultra thick,testcolor,shorten >=1pt,shorten <=1pt,<-},
	testfcnx/.style={ultra thick,testcolor,shorten >=1pt,shorten <=1pt,<-,
		postaction={decorate,decoration={markings,mark=at position 0.6 with {\drawx}}}},
	kprime/.style={semithick,shorten >=1pt,shorten <=1pt,densely dashed,->},
	kprimex/.style={semithick,shorten >=1pt,shorten <=1pt,densely dashed,->,
		postaction={decorate,decoration={markings,mark=at position 0.4 with {\drawx}}}},
	kernel/.style={semithick,shorten >=1pt,shorten <=1pt,->},
	multx/.style={shorten >=1pt,shorten <=1pt,
		postaction={decorate,decoration={markings,mark=at position 0.5 with {\drawx}}}},
	kernelx/.style={semithick,shorten >=1pt,shorten <=1pt,->,
		postaction={decorate,decoration={markings,mark=at position 0.4 with {\drawx}}}},
	kernel1/.style={->,semithick,shorten >=1pt,shorten <=1pt,postaction={decorate,decoration={markings,mark=at position 0.45 with {\draw[-] (0,-0.1) -- (0,0.1);}}}},
	kernel2/.style={->,semithick,shorten >=1pt,shorten <=1pt,postaction={decorate,decoration={markings,mark=at position 0.45 with {\draw[-] (0.05,-0.1) -- (0.05,0.1);\draw[-] (-0.05,-0.1) -- (-0.05,0.1);}}}},
	kernelBig/.style={semithick,shorten >=1pt,shorten <=1pt,decorate, decoration={zigzag,amplitude=1.5pt,segment length = 3pt,pre length=2pt,post length=2pt}},
	rho/.style={dotted,semithick,shorten >=1pt,shorten <=1pt},
	renorm/.style={shape=circle,fill=white,inner sep=1pt},
	res/.style={circle,draw=symbols,inner sep=0pt,minimum size=1.2mm},
	labl/.style={shape=rectangle,fill=white,inner sep=1pt},
	xi/.style={circle,fill=symbols!10,draw=symbols,inner sep=0pt,minimum size=1.2mm},
	xiblack/.style={circle,fill=symbolsb,draw=symbolsb,inner sep=0pt,minimum size=1.2mm},
	xix/.style={crosscircle,fill=symbols!10,draw=symbols,inner sep=0pt,minimum size=1.2mm},
	xib/.style={circle,fill=symbols!10,draw=symbols,inner sep=0pt,minimum size=1.6mm},
	xibx/.style={crosscircle,fill=symbols!10,draw=symbols,inner sep=0pt,minimum size=1.6mm},
	not/.style={circle,fill=symbols,draw=symbols,inner sep=0pt,minimum size=0.5mm},
	notblack/.style={circle,fill=symbolsb,draw=symbolsb,inner sep=0pt,minimum size=0.5mm},
	>=stealth,
	}
\makeatletter
\def\DeclareSymbol#1#2#3{\expandafter\gdef\csname MH@symb@#1\endcsname{\tikz[baseline=#2,scale=0.15,draw=symbols]{#3}}\expandafter\gdef\csname MH@symb@#1s\endcsname{\scalebox{0.7}{\tikz[baseline=#2,scale=0.15,draw=symbols]{#3}}}}
\def\<#1>{\csname MH@symb@#1\endcsname}
\makeatother

\DeclareSymbol{circle}{0.5}{\draw (0,0.7) node[xi] {};}
\DeclareSymbol{line}{0.5}{\draw (0,0.2) node[not] {} -- (0,1.4) node[not] {};}

\DeclareSymbol{Psi}{0.5}{\draw (0,0) node[not] {} -- (0,1.5) node[xi] {};}
\DeclareSymbol{Psiblack}{0.5}{\draw (0,0) node[notblack] {} -- (0,1.5) node[xiblack] {};}
\DeclareSymbol{Psi2}{0.5}{\draw (-0.6,1.5) node[xi] {} -- (0,0) node[not] {} -- (0.6,1.5) node[xi] {};}
\DeclareSymbol{Psi3}{0.5}{\draw (-1,1.5) node[xi] {} -- (0,0) node[not] {}; \draw (0,1.5) node[xi] {} -- (0,0) node[not] {};\draw (1,1.5) node[xi] {} -- (0,0) node[not] {}}
\DeclareSymbol{IPsi2}{0}{\draw (0,1) -- (0.8,2.2) node[xi] {};\draw (0,-0.25) node[not] {} -- (0,1) node[not] {} -- (-0.8,2.2) node[xi] {};}
\DeclareSymbol{IPsi3}{0}{\draw (0,1) -- (1,2.2) node[xi] {};\draw (0,1) -- (0,2.2) node[xi] {};\draw (0,-0.25) node[not] {} -- (0,1) node[not] {} -- (-1,2.2) node[xi] {};}
\DeclareSymbol{PsiIPsi2}{0}{\draw (0,1) -- (0.8,2.2) node[xi] {};\draw (0,-0.25) node[not] {} -- (0,1) node[not] {} -- (-0.8,2.2) node[xi] {};\draw (0,-0.25) node[not] {} -- (0.8,1) node[xi] {};}

\DeclareSymbol{Psi2IPsi3}{0}{\draw (0,1) -- (1,2.2) node[xi] {};\draw (0,1) -- (0,2.2) node[xi] {};\draw (0,-0.25) node[not] {} -- (0,1) node[not] {} -- (-1,2.2) node[xi] {};\draw (0,-0.25) node[not] {} -- (1,0.75) node[xi] {};\draw (0,-0.25) node[res] {} -- (-1,0.75) node[xi] {};}

\DeclareSymbol{PsiIPsi3}{0}{\draw (0,1) -- (1,2.2) node[xi] {};\draw (0,1) -- (0,2.2) node[xi] {};\draw (0,-0.25) node[not] {} -- (0,1) node[not] {} -- (-1,2.2) node[xi] {};\draw (0,-0.25) node[not] {} -- (-1,0.75) node[xi] {};\draw (0,-0.25) node[res] {};}

\DeclareSymbol{Psi2IPsi2}{0}{\draw (0,1) -- (0.8,2) node[xi] {};\draw (0,-0.25) node[not] {} -- (0,1) node[not] {} -- (-0.8,2) node[xi] {};\draw (0,-0.25) node[not] {} -- (-0.9,0.75) node[xi] {};\draw (0,-0.25) node[not] {} -- (0.9,0.75) node[xi] {};\draw (0,-0.25) node[res] {};}

\date{\today}

\title[On the  parabolic  $\Phi_3^4$ model for the harmonic oscillator]{On the  parabolic  $\Phi_3^4$ model for the harmonic oscillator: diagrams and local existence}

\author{Aur{\'e}lien Deya}

\address[Aur{\'e}lien Deya]{Universit\'e de Lorraine, CNRS, IECL, F-54000 Nancy, France}
\email{aurelien.deya@univ-lorraine.fr}

\author{Reika Fukuizumi}

\address[Reika Fukuizumi]{Department of Mathematics,
School of fundamental science and engineering, Waseda University, 3-4-1, Okubo, Shinjuku-ku, Tokyo, Japan}
\email{fukuizumi@waseda.jp}

\author{Laurent Thomann}

\address[Laurent Thomann]{Universit\'e de Lorraine, CNRS, IECL, F-54000 Nancy, France}
\email{laurent.thomann@univ-lorraine.fr}

\begin{document}

\maketitle

\begin{abstract}
We prove the local wellposedness of the (renormalized)  parabolic   $\Phi^4_3$ model associated with the harmonic oscillator on $\R^3$, that is, the equation formally written as
\begin{equation*}
 \partial_t X + HX= -X^3+\infty\cdot X + \xi, \quad t>0, \quad x \in \R^3,
\end{equation*}
where $H:=-\Delta_{\R^3} +|x|^2$ and $\xi$ denotes a space-time white noise. This model is closely related to the Gross-Pitaevskii equation which is used in the description of Bose-Einstein condensation.
 
\medskip

Our overall formulation of the problem, based on the so-called paracontrolled calculus, follows the strategy outlined by Mourrat and Weber for the $\Phi^4_3$ model on the three-dimensional torus. Significant effort is then required to adapt, within the framework imposed by the harmonic oscillator, the key tools that contribute to the success of this method—particularly the construction of stochastic diagrams at the core of the dynamics.

\bigskip


\medskip

\end{abstract}

\tableofcontents

 \section{Introduction}

 \subsection{Motivation and main result}

In this article, we consider the local existence and uniqueness of the solution to the equation  
\begin{equation} \label{eq:RGL0}
\left\{
\begin{aligned}
&  \partial_t X + HX= -X^3 + \xi, \quad t>0, \quad x \in \R^3, \\
& X(0) = X_0,
\end{aligned}
\right.
\end{equation}
where 
$$H:=-\Delta_{\R^3} +|x|^2$$ is the so-called harmonic oscillator on $\R^3$, and $\xi$ denotes a space-time white noise defined 
on a complete filtered probability space $(\Omega,\cf,\mathbb{P})$, which formally satisfies 
$$\mathbb{E}\big[ \xi(t,x) \xi(s,y) \big] =\delta_{t-s} \delta_{x-y}.$$

Beyond the mathematical challenges posed by this equation—on which we shall elaborate throughout this article—
our study is primarily motivated by the aim of achieving a deeper understanding of the three-dimensional 
Stochastic Projected Gross-Pitaevskii Equation (SPGPE). 
The fundamental dynamics described by the SPGPE emerge as a model for the low-energy wave function $\phi (t,x)$
representing Bose-Einstein condensates (BEC) at finite temperatures. 
This model accounts for the chaotic \textit{growth} phenomenon, 
which arises from changes in the particle population due to collisions in the high-energy region. 
In this paper, we focus on the mathematical analysis of the model for $\phi(t,x)$, 
removing the projection onto the low-energy states and setting the chemical potential to be zero.
Namely,  
\begin{equation} \label{eq:phys}
\partial_t \phi=(i+\gamma) \big(\Delta \phi -V(x) \phi -g|\phi|^2 \phi \big) + \sqrt{2\gamma}\,\xi,  \quad t>0, \quad x \in \R^3,
\end{equation}
where $V(x)$ is a confining potential, often taken as $V(x)=|x|^2$, $\ga>0$ is a growth rate, and the constant $g>0$ characterizes the strength of atomic interactions (see \cite{BBDBG,GD,RBB-1} and the references therein for further details). Numerical studies, such as those conducted in \cite{weiler}, have demonstrated the model fitness 
in capturing the evaporative effect and the vortex formation within BEC.

\smallskip

When $\gamma=0$ in \eqref{eq:phys}, the equation naturally reduces to the standard (deterministic) Gross-Pitaevskii equation, which has been extensively studied both in the physical (see {\it e.g.} \cite{Anderson, Bradley, Davisetal, PS}) and mathematical (see {\it e.g.} \cite{BPT,BT, BTT, C, F, PRT2,T}) literature. In this situation, the Hamiltonian quantity 
$$\mathcal{J}(\phi):=\int_{\R^3}dx\, \bigg[ \frac12 |\nabla \phi(x)|^2+\frac12V(x) |\phi(x)|^2+\frac{g}4 \phi(x)^4 \bigg]$$
is known to be conserved. In the more complex  case $\gamma>0$, the existence of a global solution and an invariant measure for the dynamics \eqref{eq:phys} has been recently established in \cite{dBDF0, dBDF1, dBDF2} for the space dimensions one and two. The Gibbs measure, formally written by 
$$ \rho (d\phi) = \Gamma e^{-\mathcal{J}(\phi)} d\phi, $$
determines the statistical equilibrium of the system \eqref{eq:phys}; it has been rigorously constructed in~\cite{BTT} for $\gamma=0$ 
in the one-dimensional case,  in \cite{Deng} in the  two-dimensional radial case,   and in \cite{dBDF1} for $\gamma \ge 0$ in the two-dimensional case. 

\smallskip

 Although our (three-dimensional) model \eqref{eq:RGL0} captures only the parabolic component of the dynamics described in \eqref{eq:phys}—thereby ensuring a substantial gain in clarity in the presentation of our arguments—we consider that it retains the core analytical challenges. We therefore have no doubt that the methods developed hereafter can be extended to the general equation \eqref{eq:phys}.

\

While keeping this physical motivation in mind, let us briefly recall that the main difficulty in the mathematical analysis of the model \eqref{eq:RGL0} lies in the simultaneous presence of both white noise and a nonlinearity: the irregularity of the noise trajectories compels us to study the equation within a generalized space of distributions, where multiplication is not canonically defined.

\smallskip

To overcome this fundamental challenge, we have chosen to follow the ideas recently developed for the celebrated $\Phi^4_3$ model. Indeed, in some sense, the problem under consideration can be regarded as a variant of the $\Phi^4_3$ equation, in which the Laplacian operator is replaced by the harmonic oscillator in Euclidean space. To avoid any ambiguity, we shall denote the  $\Phi^4_3$ model with Laplace operator by $\Phi^4_3(\Delta)$ in the sequel. Recall that this equation has been a cornerstone of quantum field theory since the late 1960s, following the works of Glimm and Jaffe \cite{glimm,glimm-jaffe}. However, it was not until the early 2010s that a rigorous mathematical framework was established.

\smallskip

The  $\Phi^4_3(\Delta)$ model has now become one of the most extensively studied dynamics in the field of stochastic partial differential equations. The considerable body of literature devoted to this subject is directly linked to the groundbreaking developments of the theory of regularity structures on the one hand \cite{hairer}, and the theory of paracontrolled distributions on the other \cite{GIP}. Alongside the KPZ equation, the~$\Phi^4_3(\Delta)$ equation constitutes one of the most significant applications of these two fundamental new approaches (we refer to \cite{catellier-chouk,MW} for a pedagogical introduction to the subject). 

\smallskip

Notable contributions in this area include (note that we only focus on the three-dimensional case and do not evoke the many papers about the one and two-dimensional situations;   in this latter case we refer to \cite{dPD} and references therein): 

\noindent
$\bullet$	the initial results concerning the interpretation and local existence of solutions in \cite{hairer, petit-hairer}, as well as in \cite{catellier-chouk}, for the $\Phi^4_3(\Delta)$ equation on the torus   (see also \cite{Kupiainen} for an alternative approach based on Wilson's renormalization group analysis);

\noindent
$\bullet$	the global extension of the solution (still on the torus) in \cite{MW,JagaPerko}, with a reinterpretation of the construction of underlying diagrams in \cite{mourrat-weber-xu};

\noindent
$\bullet$	a first lattice approximation of the model in \cite{zhu-zhu};

\noindent
$\bullet$	the transposition of these results to a compact Riemannian manifold in \cite{BDFT1,BDFT2};

\noindent
$\bullet$	the derivation of a global solution on $\R^3$ in \cite{GHof1,GHof2};

\noindent
$\bullet$	the fractional generalization of the model (where $(-\Delta)$ is extended to $(-\Delta)^s$ for $s\in (0,1]$) in~\cite{DGR};

\noindent
$\bullet$	the analysis of certain stochastic properties of the solution in \cite{gassiat-labbe,hairer-steele};

\noindent
$\bullet$  the derivation of fundamental a priori bounds on any compact space-time set of $\R \times \R^3$ in \cite{Moinat-Weber};

\noindent
$\bullet$  the study of the equation on  $\R \times \T^3$ with Hartree nonlinearity in \cite{Oh-Okamoto-Tolomeo};

\noindent
$\bullet$  the direct construction and invariance of  measures for the $\Phi^3_3(\Delta) $-model   on the torus  in~\cite{A-K};

\noindent
$\bullet$   the extension of the problem to the four-dimensional case in \cite{bruned-chandra-chevyrev-hairer,chandra-moinat-weber}, with a perturbation $\xi$ slightly more regular than white noise.

\noindent
$\bullet$   the local well-posedness for the $\Phi^3_3(\Delta) $-model   on the torus  in~\cite{EJS};  

\

For our harmonic-oscillator model \eqref{eq:RGL0}, we have chosen to rely on an adaptation of paracontrolled calculus, owing to the quite flexible nature of this approach—in comparison with the formalism of regularity structures. As a result of this adaptation, we will be able to show \textit{the local wellposedness of (a suitably renormalized version of) the equation}. 

To state our main theorem, let us introduce a basis $(\vp_k)$ of eigenvectors of $H$ (with eigenvalues $(\la_k)$), and consider the white-noise regularization $\xi^{(n)}$ given by
\begin{equation}\label{regu-noise}
 \xi^{(n)}_t:= \frac{dW^{(n)}_t}{dt}, \quad W^{(n)}_t(x):=\sum_{k \geq 0} e^{-\eps_n \la_k} \beta^{(k)}_t  \vp_k(x),
\end{equation} 
with $\eps_n:=2^{-n}$ and where $(\beta^{(k)})_{k\geq 0}$ is a family of independent Brownian motions. Our main result can now be summed up as follows (see Section~\ref{sect-Besov} for a precise definition of the Besov space~$\cb^{\sigma}_{\infty,\infty}(\R^3)$ based on the harmonic oscillator).

\begin{theorem}\label{theo:main-intro0}
There exists a sequence $(\frakc^{(n)})=(\frakc^{(n)}_t(x))$ of \textit{deterministic} functions on $\R_+\times \R^3$ such that the following assertions hold:

\smallskip

\noindent
$(i)$ For all fixed $t>0$ and $x\in \R^d$, there exist $c_1(t,x),c_2(t,x)>0$ such that for $n \geq 1$ large enough
$$ c_1(t,x) 2^{\frac{n}{2}} \leq \frakc^{(n)}_t(x) \leq  c_2(t,x) 2^{\frac{n}{2}}.$$

\smallskip

\noindent
$(ii)$ For a certain test function $\phi\in \cd(\R_+\times \R^3)$, 
$$\langle \frakc^{(n)},\phi\rangle \stackrel{n\to\infty}{\longrightarrow} \infty.$$

\smallskip

\noindent
$(iii)$  Let $\eps>0$ and $X_0 \in \cb^{\frac12+\eps}_{\infty,\infty}(\R^3)$.   The sequence $(X^{(n)})$ of solutions to the renormalized stochastic equation
\begin{equation} \label{appro-equ-introd}
\left\{
\begin{aligned}
&(\partial_t +H)  X^{(n)}  = -(X^{(n)})^3+ \frakc^{(n)}X^{(n)} +\xi^{(n)}, \quad t>0, \quad x \in \R^3, \\
& X^{(n)}(0)=X_0,
\end{aligned}
\right.
\end{equation} 
converges almost surely to a limit solution $X$ in the space $\bigcap_{\eta>0}\cac\big( [0,T) ;\cb^{-\frac12-\eta}_{\infty,\infty}(\R^3)\big)$, up to a possible explosion time $T=T(\omega)>0$.

 Moreover, one has 
$$X - \<Psi> + \<IPsi3>   \in \mathcal{C}\big([0,T); \mathcal{B}_{\infty, \infty}^{ {\frac12 +\eps}}(\R^3)\big), $$
  where $\<Psi>$ and $\<IPsi3> $ are explicit and only depend on the noise $\xi$.
\end{theorem}
 
\

The above statement directly echoes the interpretation and local well-posedness results previously shown for the $\Phi^4_3(\Delta)$ model. In particular, one may compare this statement to the results of~\cite[Theorem 1.15]{hairer} or \cite[Corollary 1.5]{catellier-chouk}.  

\smallskip

\begin{remark}
The \enquote{renormalizing} sequence $\frakc^{(n)}$ appearing in the theorem, and which fundamentally explodes as $n\to \infty$ (as described in items $(i)$ and $(ii)$) still depends on both time and space, which sharply contrasts with the sequence of constants typically used for $\Phi^4_3(\Delta)$. This distinction is in fact naturally linked to the consideration of the harmonic oscillator $H$: unlike the Laplacian case, the kernel of $e^{-tH}$ does not take a convolutional form, depriving us of the simplifications offered by the Gaussian field stationarity (see Section \ref{section:renormalization} for a more detailed expression of $\frakc^{(n)}$). 
\end{remark}

\begin{remark}
As we mentioned earlier, and as can be observed from the statement of the theorem above, our focus here is solely on the {\it local} resolution of the problem. In this context, it is easy to verify that the sign preceding the nonlinearity $X^3$ is, in fact, of no consequence. However, when extending this local solution to a global one—a subject we intend to investigate in a future work—there is no doubt that this very sign will play a fundamental role in establishing a priori bounds on the solution.
\end{remark}

\begin{remark}
For the same \enquote{criticality} reasons as for the $\Phi^4_3(\Delta)$ model (see \cite[Assumption 8.3]{hairer}), we do not expect a treatment of the four-dimensional version of \eqref{eq:RGL0} to be possible with the existing pathwise-type approaches. On the other hand, following the ideas of \cite{bruned-chandra-chevyrev-hairer,chandra-moinat-weber}, it might be possible to handle the case of any noise $\xi\in \R_+\times \R^4$ more regular than white noise, although such an objective goes far beyond the present study.
\end{remark}

\smallskip

 \subsection{The proof base}

Let us informally present  some of the key ideas that initiate the proof of Theorem~\ref{theo:main-intro0}. To this end, we follow the approach developed by Mourrat and Weber in \cite{MW} for the~$\Phi^4_3(\Delta)$ model, and we refer to the introduction of their work for further details. \medskip  

We begin with an approximation of \eqref{eq:RGL0} given by  
\begin{equation} \label{eq-00}
\left\{
\begin{aligned}
&(\partial_t +H)  X^{(n)}  = -(X^{(n)})^3+ \frakc^{(n)}X^{(n)} +\xi^{(n)}, \quad t>0, \quad x \in \R^3, \\
& X^{(n)}(0)=X_0,
\end{aligned}
\right.
\end{equation}  
where $\xi^{(n)}$ is a spatially regularized version of $\xi$ that remains white in time. The function $(t,x) \mapsto  \frakc^{(n)}(t,x)$ is a deterministic term depending only on $\xi^{(n)}$, which must be chosen appropriately to ensure the convergence of the sequence $X^{(n)}$ in a suitable distribution space. \medskip  

The mild formulation of \eqref{eq-00} is given by  
\[
X_t^{(n)} = e^{-t H}X_0 -\int_{0}^t ds \, e^{-(t-s) H}   (X_s^{(n)})^3+ \int_{0}^t ds \, e^{-(t-s) H}   \frakc_s^{(n)}X_s^{(n)} +\<Psi>^{(n)}_t,
\]
where $\<Psi>^{(n)}$ is the unique solution to  
\begin{equation*}
\left\{
\begin{aligned}
&(\partial_t +H) \<Psi>^{(n)}={\xi^{(n)}},  \quad t>0, \quad x \in \R^3, \\
&\<Psi>^{(n)}_0=0.
\end{aligned}
\right.
\end{equation*}  

In Proposition \ref{Prop-luxo}, we will show that for every $T>0$, the sequence $(\<Psi>^{(n)})$ converges almost surely to a limit  
\[
\<Psi> \in \cac{\big( [0,T]; \cb_x^{-\frac12-\eps}\big)}
\]
for every $\eps>0$, where $\cb_x^{\sigma}=\cb_{\infty, \infty}^{\sigma}(\R^3)$ denotes the Besov space associated with the operator~$H$ (see Section~\ref{sect-Besov} for the precise definition). Consequently, we expect that the limit $X$ of $(X^{(n)})$, if it exists, also belongs to $\cac{\big( [0,T]; \cb_x^{-\frac12-\eps}\big)}$. \medskip


Next, we establish in Proposition \ref{prop:cherry} and Proposition \ref{prop-ord3} that, for a suitable choice of the (diverging) function \((t,x) \mapsto \frakc^{\mathbf{1},(n)}(t,x)\), the stochastic objects  
\[
\<Psi2>^{(n)} :=(\<Psi>^{(n)})^2 - \frakc^{\mathbf{1},(n)}, \quad \<Psi3>^{(n)}:=(\<Psi>^{(n)})^3-3\frakc^{\mathbf{1},(n)} \, \<Psi>^{(n)}
\]
\[
t \mapsto  \<IPsi3>^{(n)}_t =\int_0^t ds \, e^{-(t-s)H}\<Psi3>^{(n)}_s
\]
converge almost surely to elements  
\begin{equation}\label{cherry-reg}
\<Psi2> \in \cac{\big( [0,T]; \cb_x^{-1-\eps}\big)}, \quad \<IPsi3> \in \cac{\big( [0,T]; \cb_x^{\frac12-\eps}\big)}.
\end{equation}
Then, following the classical Da Prato-Debussche method, we introduce the decomposition  
\begin{equation}\label{dpd}
X^{(n)}=\<Psi>^{(n)}-\<IPsi3>^{(n)}+\mathcal{U}^{(n)},
\end{equation}
where \(\mathcal{U}^{(n)}\) satisfies \(\mathcal{U}^{(n)}(0)=X_0\) and obeys the equation  
\begin{multline}\label{eq-U}
 (\partial_t  + H)\mathcal{U}^{(n)}  =-\big(\mathcal{U}^{(n)}-\<IPsi3>^{(n)} \big) ^3 - 3 \<Psi>^{(n)} \,(\mathcal{U}^{(n)}-\<IPsi3>^{(n)})^2 \\
 -3 ( \<Psi2>^{(n)}  +3 \frakc^{\mathbf{2},(n)} )(\mathcal{U}^{(n)}-\<IPsi3>^{(n)})  -9   \frakc^{\mathbf{2},(n)} \, \<Psi>^{(n)}.
\end{multline}
Here, we set  
\[
 \frakc^{(n)}=3 \frakc^{\mathbf{1},(n)}  -9   \frakc^{\mathbf{2},(n)}
\]
with \((t,x) \mapsto \frakc^{\mathbf{2},(n)}(t,x)\) to be determined later. The transformation in \eqref{dpd} helps eliminate terms of the lowest regularity from the right-hand side of \eqref{eq-00}.

\medskip

Using \eqref{cherry-reg} and the product rules in Besov spaces (see Proposition~\ref{Prop-est-para}~$(iv)$), we expect the product \(\<Psi2>^{(n)}_t \mathcal{U}^{(n)}_t\) to take values in \(\cb_x^{-1-\eps}\) for all $t\geq 0$. Thus, regardless of the choice of~\(\frakc^{\mathbf{2},(n)}\), solving \eqref{eq-U} shows that the spatial regularity of \(\mathcal{U}^{(n)}\) cannot be better than \(\cb_x^{1-\eps}\), since convolution with the semigroup $e^{-.H}$ increases regularity by $2$, just as in the heat case (see Lemma~\ref{L-Schauder}). This regularity suffices to handle all terms on the right-hand side of \eqref{eq-U}, except for the product \( \<Psi2>^{(n)}  \,\mathcal{U}^{(n)} \) (again using Proposition~\ref{Prop-est-para}~$(iv)$). In other words, we are unable to solve~\eqref{eq-U} directly via a fixed-point argument.

\smallskip

Using a paraproduct decomposition (see Section \ref{parap} for a definition), we isolate the most singular component of this problematic interaction, namely
\[
(\mathcal{U}^{(n)}-\<IPsi3>^{(n)}) \pl \<Psi2>^{(n)}.
\]
We then decompose \(\mathcal{U}^{(n)}\) as  
\[
\mathcal{U}^{(n)}=v^{(n)}+w^{(n)}
\]  
and consider the system  
\begin{eqnarray}
 (\partial_t  + H)v^{(n)} &= &-3 (v^{(n)}+w^{(n)}-\<IPsi3>^{(n)} ) \pl\<Psi2>^{(n)}, \label{eq-v1}\\
 (\partial_t  + H)w^{(n)} &= & G\big(v^{(n)},w^{(n)}\big), \label{eq-w1}
\end{eqnarray}
where $G$ is chosen such that \eqref{eq-v1}-\eqref{eq-w1} is equivalent to \eqref{eq-U}. 
In essence, we first solve \eqref{eq-v1} for~\(v^{(n)}\) and then substitute this solution into \eqref{eq-w1}. The procedure somehow cancels out singular interactions and allows us to renormalize equation \eqref{eq-w1}, which explains the presence of the term~\(\frakc^{\mathbf{2},(n)}\) (we refer to the introduction of \cite{MW} for further details on this argument; see also Section~\ref{para-refor} of the present article). Once renormalized, equation \eqref{eq-w1} can be solved via a fixed-point method. In a sense, we solve equation \eqref{eq-U} using a second-order Picard iteration, but only for a specific part (namely, \(v^{(n)}\)) of the solution \(\mathcal{U}^{(n)}=v^{(n)}+w^{(n)}\).

 \subsection{Main challenges induced by the procedure} 

The problem formulation described above lays the groundwork for the fixed-point argument. Two fundamental ingredients then come into play at the core of the paracontrolled approach: on the one hand, the (deterministic) commutator estimates for the paraproduct; on the other hand, the construction—through stochastic arguments—of products arising from the reformulation \eqref{eq-v1}-\eqref{eq-w1}, and encoded by diagrams.

\smallskip

The transposition of these two ingredients into the present harmonic-oscillator setting proves to be the source of numerous challenges. Generally speaking, it is well established that the study of the standard $\Phi^4_3(\Delta)$ model crucially relies on Fourier analysis, whether in the proof of the commutator estimates or in the construction of the diagrams. When dealing with the operator $H$, most of the simplifications provided by the trigonometric basis are no longer available, making a precise study of frequency interactions necessary.

\smallskip

\noindent
$(i)$ First, regarding the derivation of commutator estimates, we naturally turned to microlocal analysis methods (see {\it e.g.} \cite{Robert, Martinez, Zworski}) and adapted them to the framework of paracontrolled calculus induced by the harmonic oscillator. This leads us to two main results, namely Proposition~\ref{prop:commutor} and Lemma~\ref{lem1.2}, whose formulation parallels that of the corresponding results for the Laplacian on the torus (see \cite[Proposition 2.4 and Lemma 2.5]{catellier-chouk}).
  
\smallskip

\noindent
$(ii)$ As for the construction of the diagrams at the heart of the dynamics, we once again had to compensate for the inefficacy of standard Fourier calculus (in this specific context) by implementing several novel technical tools. Among these, we can highlight:

$\bullet$ the central use of the intermediate operators $\mathcal{R}$ and $\mathcal{P}_j^{(\al)}$, introduced in Section \ref{section:setting} and fundamentally linked to paracontrolled calculus;

$\bullet$ the adaptation of the topologies involved (see Proposition \ref{prop:conv-arbre} and the subsequent remarks), and the resulting interpretation of the equation via Young integrals, in the spirit of the approach developed in \cite{GLT} for the treatment of fractional noises.

\

Finally, let us note that beyond the model \eqref{eq:RGL0}, we hope that this study and its intermediate results will also contribute to a deeper understanding of the fundamental properties of the harmonic oscillator.

\
  \subsection{Outline of the paper} 
The article is structured as follows:

\smallskip

\noindent
$\bullet$	In Section \ref{section:setting}, we begin by introducing the functional framework associated with the harmonic oscillator, and which will serve as a reference throughout the article. We also highlight certain spectral properties and key estimates related to $H$. Together with the results from Appendix \ref{appendix:microlocal} (based on microlocal analysis), these will provide us with the necessary deterministic technical tools for implementing the paracontrolled approach.

\smallskip

\noindent
$\bullet$	Section \ref{section:local-wellposed} is devoted to the reformulation of the problem following the scheme described in~\cite{MW} for $\Phi^4_3(\Delta)$, while temporarily assuming the existence of the underlying diagrams. With minor technical variations, the fixed-point argument will also largely follow the methodology used in the~$\Phi^4_3(\Delta)$ case. In particular, the procedure will rely exclusively on deterministic arguments.

\smallskip

\noindent
$\bullet$	Sections \ref{sec:first-order-diagram} to \ref{section:fifth-order-diagram} will focus on the stochastic analysis of the problem, which, within our trajectory-based approach, essentially reduces to the construction of the diagrams (or trees) at the core of the dynamics.

\smallskip

\noindent
$\bullet$	In the (concise) Section \ref{section:renormalization}, we conduct a detailed examination of the asymptotic behavior of the renormalization factors arising from the diagram construction, ultimately leading to the conclusions of items $(i)$ and $(ii)$ in Theorem~\ref{theo:main-intro0}.

\smallskip

\noindent
$\bullet$	As mentioned earlier, Appendix \ref{appendix:microlocal} revisits the paracontrolled calculus associated with the harmonic oscillator, with a particular emphasis on the so-called paracontrolled commutator estimates. Finally, Appendices \ref{appendix:young} and \ref{appendix:technical} gather several auxiliary technical lemmas used throughout the study, including details on mild (time) integration in the Young sense.

\


\section{Setting and technical tools}\label{section:setting}

\subsection{Basic spectral properties of the harmonic oscillator}

Recall that the harmonic oscillator on $\R^3$ is defined as  
\[ H = -\Delta_{\mathbb{R}^3} + |x|^2. \]  
Let \(\{\varphi_k \}_{k \geq 0}\) be an orthonormal basis of \(L^2(\mathbb{R}^3)\) consisting of eigenfunctions of \(H\). The corresponding eigenvalues of \(H\) are given by  
\[ \big\{2(\ell_1 + \ell_2 + \ell_3) + 3 \;\big|\; \ell_1,\ell_2,\ell_3  \in \mathbb{N} \big\}. \]  
We order these eigenvalues in a non-decreasing sequence \(\{\lambda_k\}_{k \geq 0}\), repeated according to their multiplicities, such that  
\[ H \varphi_k = \lambda_k \varphi_k. \]  
It follows that there exists a constant $c>0$ such that  
\[ \lambda_k \sim c k^{\frac{1}{3}}, \quad \text{as } k \to +\infty. \]

\medskip

It is well known that the asymptotic behavior and \(L^p\)-estimates of the functions \(\varphi_k\) depend on the choice of the Hilbertian basis (see {\it e.g.} \cite{Koch-Tataru, Ko-Ta-Zw, PRT1}). In this context, it is natural to introduce the spectral function associated with \(H\), defined for \(j \geq 0\) as  
\begin{equation}\label{defPsi}
\Psi_j(x) := \sum_{\substack{k \geq 0 \\ 2^{2j} \leq \lambda_k \leq 2^{2j+2}}} |\varphi_k(x)|^2.
\end{equation}  
The function \(\Psi_j\) is independent of the choice of the basis \((\varphi_k)_{k \geq 0}\). It has been shown (see Thangavelu~\cite[Lemma 3.2.1, p. 69]{Thangavelu}) that  
\begin{equation}\label{PsiLinf}
\|\Psi_j\|_{L^\infty(\mathbb{R}^3)} \lesssim 2^{3j}.
\end{equation}  
By integrating \eqref{defPsi}, we obtain  
\[ \|\Psi_j\|_{L^1(\mathbb{R}^3)} \lesssim 2^{6j}, \]  
which, by H\"older's inequality, implies that for all \(1 \leq p \leq \infty\),  
\begin{equation}\label{PsiLp}
\|\Psi_j\|_{L^p(\mathbb{R}^3)} \lesssim 2^{3j(1 + \frac{1}{p})}.
\end{equation}

\medskip

For \(\gamma \in \mathbb{R}\), we define the operator \(H^\gamma\), whose integral kernel is given by  
\begin{equation}\label{series}
h_{\gamma}(x, y) = \sum_{k \geq 0} \lambda_k^\gamma \varphi_k(x) \varphi_k(y).
\end{equation}  
Notably, the function \(h_\gamma\) is independent of the choice of the Hilbertian basis \((\varphi_k)_{k \geq 0}\). Since \((\varphi_k)_{k\geq 0}\) forms an orthonormal basis, we obtain for all \(x \in \mathbb{R}^3\):  
\begin{equation}\label{normeL2}
\big\|h_{\gamma}(x, \cdot)\big\|^2_{L^2_y(\mathbb{R}^3)} = h_{2\gamma}(x, x).
\end{equation}  
The following result ensures that the series \eqref{series} converges in suitable Lebesgue spaces.

\begin{lemma}\label{Lem-normeLp}
Let \(\gamma < -\frac{3}{2}\). Then, for any \(p > \big(-\frac{2}{3} \gamma -1\big)^{-1}\), we have  
\begin{equation}\label{normeLp}
x \mapsto h_{\gamma}(x, x) \in L^p(\mathbb{R}^3).
\end{equation}
\end{lemma}

\begin{proof}
For all \(x \in \mathbb{R}^3\), we observe that  
\[
h_{\gamma}(x, x) = \sum_{k \geq 0} \lambda_k^\gamma |\varphi_k(x)|^2 \lesssim \sum_{j \geq 0} 2^{2j\gamma} \Psi_j(x).
\]  
Applying \eqref{PsiLp}, we obtain  
\begin{equation*}
\big\|h_{\gamma}(\cdot, \cdot)\big\|_{L^p(\mathbb{R}^3)} \lesssim \sum_{j \geq 0} 2^{2j\gamma} \|\Psi_j\|_{L^p(\mathbb{R}^3)} \lesssim \sum_{j \geq 0} 2^{2j[\gamma + \frac{3}{2}(1 + \frac{1}{p})]}.
\end{equation*}  
If \(\gamma < -\frac{3}{2}\), then the series converges for \(p > \big(-\frac{2}{3} \gamma -1\big)^{-1}\), concluding the proof.
\end{proof}

We insist on the fact that the results in this paper do not specifically depend on the choice of a particular Hilbertian basis \((\varphi_k)_{k \geq 0}\): they only rely on the property of $h_\gamma$ stated in Lemma \ref{Lem-normeLp}, and valid for any Hilbertian basis. 

\

\subsection{Bounds on the Green kernel}

The kernel of the exponential operator $e^{-tH}$ is given by
\begin{equation*}  
K_t(x,y):=\sum_{k\geq 0} e^{-t\lambda_k}\varphi_k(x)\varphi_k(y),
\end{equation*}
and satisfies the Mehler formula (see {\it e.g.} \cite[page 109]{Taylor}):
\begin{equation}  \label{mehler2}
K_t(x,y)  = (2\pi\sinh 2t)^{-\frac{3}{2}} \exp\left(- \frac{ \vert x-y\vert^2}{4\tanh t}-\frac{\tanh t}{4}\vert x+y\vert^2\right).
\end{equation}
For any $1 \le p \le \infty$ and $x,y  \in \mathbb{R}^3$, for $0< t\lesssim 1$, we have the bounds
\begin{equation}\label{normeLpp}
 \| K_t (x,\cdot) \|_{L^{p}_y(\mathbb{R}^3)} \lesssim t^{\frac{3}{2p} -\frac{3}{2}}, \qquad \| K_t (\cdot,y) \|_{L^{p}_x(\mathbb{R}^3)} \lesssim t^{\frac{3}{2p} -\frac{3}{2}}.
 \end{equation}

\begin{lemma}\label{lem-K}
The following bound holds for $i=1,2$ and all $0 < \sigma \leq 1$:
\begin{equation}\label{borne-F22}
\Big| \big( H_{y_i}K_{\sigma}\big)(y_1,y_2)\Big| \lesssim \sigma^{-\frac{5}{2}} \exp\Big( - \frac{|y_1-y_2|^2}{8 \tanh(\sigma)}  \Big).
\end{equation}
More generally, for any $n\geq 1$ and all $0 <\sigma \leq 1$,
\begin{equation} \label{borne-F23}
\Big| \big( H^n_{y_i}K_{\sigma}\big)(y_1,y_2)\Big| \lesssim \sigma^{-\frac{3}{2}-n} \exp\Big( - \frac{|y_1-y_2|^2}{8 \tanh(\sigma)}  \Big).
\end{equation}
\end{lemma}

\begin{proof}
Set $\rho=\tanh(\sigma)$. Define
\[
P(y_1,y_2)=  \frac{ \vert y_1-y_2\vert^2}{4\rho}+\frac{\rho}{4}\vert y_1+y_2\vert^2.
\]
Then, the kernel can be expressed as
\[
K_{\sigma}(y_1,y_2)=(2\pi\sinh 2\sigma)^{-\frac{3}{2}} \exp\big( -P(y_1,y_2)\big),
\]
and a direct computation gives
\begin{equation*} 
 H_{y_1}K_{\sigma}=(2\pi\sinh 2\sigma)^{-\frac{3}{2}}   \big( \Delta_{y_1}P- |\nabla_{y_1}P|^2+|y_1|^2\big)   e^{-P} .
\end{equation*}
We estimate the terms separately:
\begin{equation}\label{esti11}
|\Delta_{y_1}P| \lesssim \rho^{-1} + \rho \lesssim \sigma^{-1}.
\end{equation}
Since $|\nabla_{y_1}P| \lesssim  \rho^{-1}|y_1-y_2|+  \rho|y_1+y_2|$, we obtain
\begin{eqnarray}  \label{esti12}
            |\nabla_{y_1}P|^2  e^{ -\frac{|y_1-y_2|^2}{4\rho}-\frac{\rho}{4}|y_1+y_2|^2  }& \lesssim  & \frac{|y_1-y_2|^2}{\rho^2}    e^{ -\frac{|y_1-y_2|^2}{4\rho}   } +  \rho^2|y_1+y_2|^2e^{  -\frac{\rho}{4}|y_1+y_2|^2  }  e^{ -\frac{|y_1-y_2|^2}{4\rho}} \nonumber\\
         & \lesssim  &   (\rho^{-1}  +\rho) e^{ -\frac{|y_1-y_2|^2}{8\rho}   }.
\end{eqnarray} 
Similarly,
\begin{eqnarray}  \label{esti13}
       |y_1|^2 e^{ -\frac{|y_1-y_2|^2}{4\rho}-\frac{\rho}{4}|y_1+y_2|^2  } &\lesssim &     |y_1-y_2|^2 e^{ -\frac{|y_1-y_2|^2}{4\rho}-\frac{\rho}{4}|y_1+y_2|^2  }+  |y_1+y_2|^2 e^{ -\frac{|y_1-y_2|^2}{4\rho}-\frac{\rho}{4}|y_1+y_2|^2  }  \nonumber \\
          & \lesssim  &   (\rho^{-1} +\rho)  e^{ -\frac{|y_1-y_2|^2}{8\rho}   }.
\end{eqnarray}   
Using $ \rho^{-1}+\rho \lesssim \sigma^{-1}$ and $(2\pi\sinh 2\sigma)^{-\frac{3}{2}}  \lesssim \sigma^{-\frac{3}{2}}$, we combine \eqref{esti11},  \eqref{esti12}, and~\eqref{esti13} to obtain~\eqref{borne-F22}.

\

The general case \eqref{borne-F23} follows by induction.
\end{proof}

\subsection{Sobolev spaces based on the harmonic oscillator}

We define the Sobolev spaces associated with $H$ as follows: for $\si \in \R$ and $1 \leq p\leq \infty$, 
 \begin{equation*} 
         \W_x^{\si, p} =  \W^{\si, p}(\R^3) = \Big\{ u\in L^p(\R^3), \; {H}^{\frac{\si}2}u\in L^p(\R^3)\Big\}.
       \end{equation*}
In the particular case $p=2$, we set 
       \begin{equation*}
   {\mathcal H}_x^{\si}=   {\mathcal H}^{\si}(\R^3) = \W^{\si, 2}(\R^3).
       \end{equation*}
         
These spaces are endowed with the  natural norm 
$$\|u\|_{{\mathcal W}^{\sigma,p}(\R^3)}=\|H^{\frac{\si}2}u\|_{L^p(\R^3)}.$$
For ${1<p<+\infty}$ and $\sigma\geq 0$, an equivalent norm is given by (see \cite[Lemma~2.4]{YajimaZhang2} or \cite{DG})
\begin{equation} \label{eq-nor}
\Vert u\Vert_{\W^{\si,p}(\R^3)}  \equiv \Vert {\langle D_x \rangle^{\si} u\Vert_{L^{p}(\R^3)} }+  \Vert\langle x\rangle^{\si}u\Vert_{L^{p}(\R^3)}.
\end{equation}
Additionally, we recall the following alternative expression for the ${\mathcal H}_x^{\sigma}$ norm: if $u=\sum_{n\geq 0}c_n \varphi_n$, then  
$$  \|u\|^2_{{\mathcal H}^\sigma(\R^3)}=\sum_{n\geq 0}\lambda_n^\sigma |c_n|^2. $$

\subsection{Besov spaces associated with the harmonic oscillator}\label{sect-Besov}

To define the so-called harmonic Besov spaces, we introduce a dyadic partition of unity (see {\it e.g. \cite[Chapter 2]{BCD}}). Consider  the annulus 
$$\mathcal{A}:=\Big\{\xi \in \R_+: \; \frac34 \leq \xi \leq \frac8{3}\,\Big\}.$$
There exist  $\chi_{-1} \in \mathcal{C}_0^{\infty}\big([-\frac43, \frac43]\big)$ satisfying $\chi_{-1} \equiv 1$ in a neighborhood of $0$  and ${\chi}\in \mathcal{C}_0^{\infty}(\mathcal{A})$, both taking values in $[0, 1]$, such that for all $\xi \in \R_+$
\begin{equation}\label{partition}
 \sum_{j=-1}^{+\infty} \chi_j( \xi)=1, \quad \text{with }\quad    \chi_j(\xi ) := \chi(2^{-j} \xi), \quad  \forall j \geq 0.
\end{equation}

We define the Hermite multipliers $(\delta_j)_{j\geq -1}$ by $\delta_{-1}u=\chi_{-1}(\sqrt{H})u$, and for all $j\geq 0$,
\begin{equation}\label{def-delta}
{\delta}_j u = \chi_j(\sqrt{H}) u=  \chi\big(\frac{ \sqrt{H}}{2^j}\big)u. 
\end{equation}
                    
Define $\theta$ and   $\theta_{-1}$ by $\theta(x)=\chi(\sqrt{|x|})$ and $\theta_{-1}(x)=\chi_{-1}(\sqrt{|x|})$. Then $\theta \in \mathcal{C}_0^{\infty}(\R)$, $\theta_{-1} \in \mathcal{C}_0^{\infty}(\R)$  and 
     \begin{equation*} 
\text{Supp} \,\theta  \subset \big\{\xi \in \R_+:\; \;  \big(\frac34\big)^2 \leq \xi \leq \big(\frac83\big)^2  \big\}, \qquad \text{Supp} \,\theta_{-1}  \subset \big\{\xi \in \R_+ :\; \;  0 \leq \xi \leq \big(\frac43\big)^2  \big\}.
     \end{equation*}
In what follows, we will sometimes use the equivalent notation (for $j \geq 0$)
\begin{equation}\label{delta-theta}
 \delta_j=\theta(\frac{H}{2^{2j}}),     
 \end{equation}
as this expression is occasionally more suited to our purposes.       

\medskip               

The Besov spaces based on the harmonic oscillator (or harmonic Besov spaces) are then defined for $1 \leq p, q \leq \infty$ and $\sigma \in \R$ by
        \begin{equation} \label{def-besov}
       \mathcal{B}^{\sigma} _{p,q}(\R^3) = \big\{ u\in \mathscr{S}'(\R^3), \; \big\|2^{j\sigma}  {\delta}_j u \big\| _{L^p(\R^3)}  \in \ell ^q_{j\geq-1}\big\}.
       \end{equation}
These spaces are equipped with the natural norm: for $1\leq q <\infty$
$$ \| u\| _{\mathcal{B}^{\sigma} _{p,q}(\R^3)} = \Bigl( \sum_{j \geq -1} \|2^{j\sigma}  {\delta}_j u \| ^q_{L^p(\R^3)} \Bigr)^{\frac 1 q},
$$    
while for $q=\infty$
$$ \| u\| _{\mathcal{B}^{\sigma} _{p,\infty}(\R^3)} = \sup_{j \geq -1}   \|2^{j\sigma}  {\delta}_j u \|_{L^p(\R^3)} .
$$   
In particular, one can check that $\mathcal{B}^{\sigma} _{2,2}(\R^3)={\mathcal H}^{\sigma}(\R^3)$. In the sequel, we will also use the notation 
$$\mathcal{B}_x^{\sigma}:= \mathcal{B}^{\sigma} _{\infty,\infty}(\R^3).$$

 \subsection{H{\"o}lder spaces (in time)}

 Consider a normed space $\big(E, \| \cdot \|\big)$  and let $f \in \mathcal{C}\big([T_1,T_2] ; E\big)$. For $\eta>0$, we define the space ${\mathcal{C}^{\eta}\big([T_1,T_2]; E\big)}$ through the norm
\begin{equation} \label{def-Ceta}
\big\| f\big\|_{\mathcal{C}^{\eta}([T_1,T_2]; E)} =\big\| f(T_1)\big\|+ \sup_{\substack{     u,v \in[T_1,T_2] \\  u \neq v  }} \frac{ \big\| f(v)-f(u)\big\|}{|v-u|^{ \eta}}  .
\end{equation}  
Similarly, for $\eta>0$, we define the space ${\ov \cac}^{\eta}\big([T_1,T_2]; E\big)$ through the semi-norm
\begin{equation} \label{def-CetaBar}
\big\| f\big\|_{ {\ov \cac}^{\eta}([T_1,T_2]; E)} =  \sup_{\substack{     u,v \in[T_1,T_2] \\  u \neq v  }} \frac{ \big\| f(v)-f(u)\big\|}{|v-u|^{ \eta}}  .
\end{equation}

The following technical definition will also prove useful in the sequel.
\begin{definition}\label{defi:conv-c--la}
For every $\la\in (0,1)$, we will say that a sequence $(f^{(n)})$ of (regular) $E$-valued functions converges in $\cac^{-\la}([0,T];E)$ if the auxiliary sequence $(\widetilde{f}^{(n)})$ defined by
$$\widetilde{f}^{(n)}_t:=\int_0^t f^{(n)}_s \, ds$$
converges in the space $\cac^{1-\la}\big([0,T];E\big)$. 
\end{definition}

  Note that for the sake of clarity, we will occasionally use the notations $\cac_T^\ga \cb^\al_x:=\cac^\ga\big([0,T];\cb^\al_x\big)$ as well as ${\ov \cac}_T^\ga \cb^\al_x:={\ov \cac}^\ga\big([0,T];\cb^\al_x\big)$.

\

The rest of this section is devoted to the introduction and the analysis of two deterministic operators—related to $H$—at the core of our subsequent diagrams constructions.

\subsection{Analysis of a resonance-type operator}

Let $j \geq 0$. To clarify references to the underlying variables, we introduce the notation
\begin{equation}\label{defi:delta-i}
\delta_{j,y \to x}F :=\Big(\theta\Big(\frac{H_y}{2^{2j}}\Big)F\Big)(x).
\end{equation}
The kernel of this operator is given by
$$ \delta_j(y,x): = \sum_{n\geq 0} \theta \Big(\frac{\lambda_n}{2^{2j}} \Big) \varphi_n(y) \varphi_n(x).$$ 
More generally, for $\alpha \in \R$, we define
\begin{equation}\label{defi:delta-alpha-i}
\delta^{(\alpha)}_{j,y \to x}F :=\big(2^{\alpha j}H^{-\frac{\alpha}{2}}_y\theta(\frac{H_y}{2^{2j}})F\big)(x).
\end{equation}

According to Proposition \ref{prop.conti}, the following uniform estimates hold true:
\begin{equation}\label{unif-delta}
\sup_{j\geq 0}\, \| \delta_{j,y \to x}(f)\|_{L_x^\infty} \lesssim \| f\|_{L_y^\infty} \quad \text{and} \quad  \sup_{j\geq 0}\, \| \delta^{(\alpha)}_{j,y \to x}(f)\|_{L_x^\infty} \lesssim \| f\|_{L_y^\infty},
\end{equation}
for all $\al\in \R$.  

\smallskip

We now introduce the \textit{resonance} operator $\mathcal{R} : F  \mapsto \mathcal{R}(F)$ defined for every function \linebreak $F:(y_1,y_2,z_1,z_2) \mapsto F(y_1,y_2,z_1,z_2)$ by
\begin{equation} \label{def-op-T}
\mathcal{R}F(x):= \sum_{i \sim i'}  \sum_{j \sim j'}  \int dz_1 dy_1dz_2 dy_2 \, \delta_i(x,y_1)\,  \delta_{i'}(x,z_1) \delta_j(x,y_2)\,  \delta_{j'}(x,z_2)  F(y_1,y_2,z_1,z_2)  ,
\end{equation}
where we define the index relation as
$$\{j \sim j'\}=\big\{j,j' \geq -1: \; |j-j'| \leq 3\big\}.$$ 
Using the self-adjoint property of $\theta(\frac{H}{2^{2i}})$ in $L^2$, the operator~$\mathcal{R}$ can also be rewritten as
\begin{eqnarray*}
\mathcal{R}F(x) = \sum_{i \sim i'}  \sum_{j \sim j'} \delta_{i, y_1 \to x}\,  \delta_{i', z_1 \to x}\delta_{j, y_2 \to x}\,  \delta_{j', z_2 \to x} F.
\end{eqnarray*}
Furthermore, we define the auxiliary operator
\begin{equation}\label{defi-oper-s}
\mathcal{L}F(x_1,x_2,x_3,x_4):= \sum_{i \sim i'}  \sum_{j \sim j'} \delta_{i, y_1 \to x_1}\,  \delta_{i', z_1 \to x_2}\delta_{j, y_2 \to x_3}\,  \delta_{j', z_2 \to x_4} F, 
\end{equation}
and observe that 
\begin{equation} \label{RL}
(\mathcal{R}F)(x)=(\mathcal{L}F)(x,x,x,x).
\end{equation}

\medskip

\begin{lemma} 
For every $\eps>0$, it holds that
  \begin{eqnarray} 
\|\mathcal{L}F\|_{L^\infty(\R^{12})} &\lesssim& \|H^{\eps}_{y_1}H^{\eps}_{y_2}F\|_{L^{\infty}(\R^{12})}, \label{estiS-0} \\
\|\mathcal{L}F\|_{L^\infty(\R^{12})} &\lesssim &\|H^{1+\eps}_{y_1}H^{-1}_{z_1}H^{\eps}_{y_2}F\|_{L^{\infty}(\R^{12})}, \label{estiS-1}\\
\|\mathcal{L}F\|_{L^\infty(\R^{12})}& \lesssim& \|H^{1+\eps}_{y_2}H^{-1}_{z_2}H^{\eps}_{y_1}F\|_{L^{\infty}(\R^{12})}. \nonumber
\end{eqnarray}
 \end{lemma}

Notice that in \eqref{estiS-1}, we can switch derivatives between $y_1$ and $z_1$ because both variables appear at the same frequency in $\mathcal{L}$, a consequence of the operator's resonant structure.

\begin{proof}
We prove \eqref{estiS-1}, as the other estimates follow by analogous reasoning.
\smallskip

It is readily checked from the definitions in \eqref{defi:delta-i} and \eqref{defi:delta-alpha-i} that for all $\alpha \in \R$,  
$$2^{\alpha i}\delta_{i,y \to x}F =\delta^{(\alpha)}_{i,y \to x} (H^{\frac{\alpha}{2}}_yF).$$
Using this identity, we can rewrite \eqref{defi-oper-s} as
\begin{multline*} 
\mathcal{L}F(x_1,x_2,x_3,x_4)=  \\
\begin{aligned}
&= \sum_{i \sim i'} 2^{-2\eps i} 2^{2(i'-i)}\sum_{j \sim j'} 2^{-2\eps j}\big(2^{2(1+\eps) i}\delta_{i, y_1 \to x_1}\big)\, \big( 2^{-2i'}\delta_{i', z_1 \to x_2}\big)\big( 2^{2\eps j}\delta_{j, y_2 \to x_3}\big)\,  \delta_{j', z_2 \to x_4} F\\
&=\sum_{i \sim i'} 2^{-2\eps i} 2^{2(i'-i)}\sum_{j \sim j'} 2^{-2\eps j} \Big(\delta^{(2(1+\eps))}_{i, y_1 \to x_1}\,    \delta^{(-2)}_{i', z_1 \to x_2}\delta^{(2\eps)}_{j, y_2 \to x_3}\,  \delta_{j', z_2 \to x_4} H^{1+\eps}_{y_1}H^{-1}_{z_1}H^{\eps}_{y_2} F\Big).
\end{aligned}
\end{multline*} 
Thanks to the uniform bounds in \eqref{unif-delta}, we immediately obtain that for every $\eps>0$,
\begin{multline*} 
\| \mathcal{L}F\|_{L^\infty_{x_1,x_2,x_3,x_4}}\lesssim \\
\begin{aligned} 
&\lesssim \sum_{i \sim i'} 2^{-2\eps i} 2^{2(i'-i)}\sum_{j \sim j'} 2^{-2\eps j}  \big\|  \delta^{(2(1+\eps))}_{i, y_1 \to x_1}\,    \delta^{(-2)}_{i, z_1 \to x_2}\delta^{(2\eps)}_{j, y_2 \to x_3}\,  \delta_{j, z_2 \to x_4} H^{1+\eps}_{y_1}H^{-1}_{z_1}H^{\eps}_{y_2} F\big\|_{L^\infty(\R^{12})}\\
& \lesssim \| H^{1+\eps}_{y_1}H^{-1}_{z_1}H^{\eps}_{y_2} F\|_{L^{\infty}(\R^{12})}\sum_{i \geq 0} 2^{-2\eps i} \sum_{j \geq 0} 2^{-2\eps j}  \\
& \lesssim \| H^{1+\eps}_{y_1}H^{-1}_{z_1}H^{\eps}_{y_2}F\|_{L^{\infty}(\R^{12})},
\end{aligned}
\end{multline*} 
which precisely corresponds to \eqref{estiS-1}.
\end{proof}

By \eqref{RL}, it is clear that $\|\mathcal{R}F\|_{L^\infty(\R^3)} \leq \|\mathcal{L}F\|_{L^\infty(\R^{12})}$, and thus any $L^{\infty}$-estimate for $\mathcal{L}$ translates into an $L^{\infty}$-estimate for $\mathcal{R}$. More precisely, we will  rely on the following results.

\begin{lemma}\label{conti}
The following estimates hold.
\begin{enumerate}[$(i)$]
\item For all $\eps>0$,   
 \begin{equation} \label{esti-0}
\|\mathcal{R}F\|_{L^\infty(\R^{3})} \lesssim  \|F\|^{1-\eps}_{L^{\infty}(\R^{12})}  \|H_{y_1}  H_{y_2}F\|^{\eps}_{L^{\infty}(\R^{12})}  .
\end{equation}
 \item  For all $\eps>0$ and $q>\frac32$,
\begin{equation} \label{esti-3}
\|\mathcal{R}F\|_{L^\infty(\R^{3})} \lesssim  \Big(  \sup_{y_1,y_2,z_2 \in \R^3} \big(\int dz_1 \big| H_{y_1}  F\big|^q\big)^{\frac1q} \Big)^{1-\eps}  \|H^{2}_{y_1}H^{-1}_{z_1} H_{y_2}F\|^{\eps}_{L^{\infty}(\R^{12})}  .
\end{equation}
 \item  For all $\eps>0$ and $q>\frac32$,
\begin{equation} \label{esti-3.1}
\|\mathcal{R}F\|_{L^\infty(\R^{3})} \lesssim  \Big(  \sup_{y_1,y_2,z_1 \in \R^3} \big(\int dz_2 \big| H_{y_2}  F\big|^q\big)^{\frac1q} \Big)^{1-\eps}  \|H^{2}_{y_2}H^{-1}_{z_2} H_{y_1}F\|^{\eps}_{L^{\infty}(\R^{12})}  .
\end{equation}
\end{enumerate} 
 
\end{lemma}

\begin{proof} $(i)$ The bound \eqref{esti-0} follows from the combination of \eqref{estiS-0} with a basic interpolation procedure. 

\smallskip

\noindent
$(ii)$ Using \eqref{estiS-1}, we obtain by interpolation that
 $$
 \|\mathcal{R}F\|_{L^\infty(\R^{3})} \lesssim \|\mathcal{L}S\|_{L^\infty(\R^{12})} \lesssim   \|H_{y_1}H^{-1}_{z_1} F\|^{1-\eps}_{L^{\infty}(\R^{12})}  \|H^{2}_{y_1}H^{-1}_{z_1} H_{y_2}F\|^{\eps}_{L^{\infty}(\R^{12})} .
 $$
 Then, since $q>\frac32$, the Sobolev embedding 
$$\mathcal{W}_{z_1}^{2,q}(\R^3) \subset L_{z_1}^{\infty}(\R^3)$$
implies that for all $y_1,y_2,z_1,z_2$,
 \begin{align*} 
 \big| H_{y_1}H^{-1}_{z_1}F\big|(y_1,y_2,z_1,z_2)=\big| H^{-1}_{z_1} \big(H_{y_1}F\big)\big|(y_1,y_2,z_1,z_2)  &\lesssim  \big(\int dz_1 \big| H_{y_1}   F\big|^q\big)^{\frac1q}(y_1,y_2,z_2)\\
&\lesssim \sup_{y_1,y_2,z_2 \in \R^3} \big(\int dz_1 \big| H_{y_1}  F\big|^q\big)^{\frac1q}, 
  \end{align*}
  and the claim follows.

\smallskip	

\noindent
$(iii)$ The bound \eqref{esti-3.1} follows by applying the same argument as in $(ii)$, exchanging the roles of the variables.
\end{proof}

\begin{lemma}\label{conti-T}
For all $\eps>0$ and $1 \leq p \leq \infty$, it holds that
\begin{equation}\label{norme-LP} 
\|\mathcal{R}F\|_{L^p(\R^3)} \lesssim \|\mathcal{R}F\|^{1-\frac1p}_{L^{\infty}(\R^{3})}   \|F\|^{\frac1p}_{\mathcal{H}^{16}(\R^{12})}.
\end{equation}
\end{lemma}

\begin{proof}
Let us decompose $F$ as
$$ F(x_1,x_2,x_3,x_4)=\sum_{k_1,k_2,k_3,k_4} \alpha_{k_1,k_2,k_3,k_4} \vp_{k_1}(x_1)\vp_{k_2}(x_2)\vp_{k_3}(x_3)\vp_{k_4}(x_4),$$
which leads to the expression  
$$\mathcal{R}F(x)=\sum_{i \sim i'}\sum_{j \sim j'}\sum_{k_1,k_2,k_3,k_4} \theta(\frac{\lambda_{k_1}}{2^{2i}})\theta(\frac{\lambda_{k_2}}{2^{2i'}})\theta(\frac{\lambda_{k_4}}{2^{2j}})\theta(\frac{\lambda_{k_4}}{2^{2j'}})\alpha_{k_1,k_2,k_3,k_4} \vp_{k_1}(x)\vp_{k_2}(x)\vp_{k_3}(x)\vp_{k_4}(x).$$
 At this point, we use the estimate \eqref{PsiLinf} to derive an $L^\infty$ bound for $\varphi_k$. If $\lambda_k \sim 2^{2j}$, then 
$$|\varphi_k(x)|^2 \leq \Psi_j(x) \lesssim 2^{3j} \lesssim \lambda_k^{\frac32},$$
which implies that $\|\vp_k\|_{L^\infty(\R^3)} \lesssim  \lambda_k^{\frac34}$. Since we also have $\|\vp_k\|_{L^2(\R^3)}  =1 $,  applying H\"older's inequality yields $\|\vp_k\|_{L^4(\R^3)} \lesssim  \lambda_k^{\frac38}$ (note that while these estimates are sufficient for our purpose, they are far from optimal; for refined bounds,  we refer the reader to \cite{Koch-Tataru, PRT1}). 

\smallskip

Using these estimates, we obtain
$$     \|\vp_{k_1}\vp_{k_2}\vp_{k_3}\vp_{k_4} \|_{L^{1}(\R^3)} \leq  \|\vp_{k_1}      \|_{L^{4}(\R^3)}   \|\vp_{k_2}      \|_{L^{4}(\R^3)}   \|\vp_{k_3}      \|_{L^{4}(\R^3)}   \|\vp_{k_4}      \|_{L^{4}(\R^3)}  \lesssim  (\lambda_{k_1}\lambda_{k_2} \lambda_{k_3} \lambda_{k_4} )^{\frac38}.$$
Then
\begin{align*} 
\|\mathcal{R}F\|_{L^1(\R^3)}&\lesssim \sum_{i \geq 0}\sum_{j \geq 0}\sum_{\lambda_{k_1}, \lambda_{k_2} \sim 2^{2i}} \sum_{\lambda_{k_3}, \lambda_{k_4} \sim 2^{2j}}  |\alpha_{k_1,k_2,k_3,k_4} |  (\lambda_{k_1}\lambda_{k_2} \lambda_{k_3} \lambda_{k_4} )^{\frac38} \\
&\lesssim \sum_{k_1,k_2,k_3,k_4 \geq 1 }  |\alpha_{k_1,k_2,k_3,k_4} | (\lambda_{k_1}\lambda_{k_2} \lambda_{k_3} \lambda_{k_4} )^{\frac38} \\
&\lesssim \big(\sum_{k_1,k_2,k_3,k_4 \geq 1}  |\alpha_{k_1,k_2,k_3,k_4} |^2 {(\lambda_{k_1}\lambda_{k_2} \lambda_{k_3} \lambda_{k_4} )^{4}} \big)^{\frac12}  \big(\sum_{k_1,k_2,k_3,k_4 \geq 1}   \frac{1}{(\lambda_{k_1}\lambda_{k_2} \lambda_{k_3} \lambda_{k_4} )^{\frac{13}{4})}} \big)^{\frac12}.
\end{align*} 
Next, we observe that 
\begin{align*}  \sum_{k_1,k_2,k_3,k_4 \geq 1}  |\alpha_{k_1,k_2,k_3,k_4} |^2 {(\lambda_{k_1}\lambda_{k_2} \lambda_{k_3} \lambda_{k_4} )^{4}}  &\lesssim   \sum_{k_1,k_2,k_3,k_4 \geq 1}  |\alpha_{k_1,k_2,k_3,k_4} |^2 {(\lambda_{k_1}+\lambda_{k_2}+ \lambda_{k_3}+ \lambda_{k_4} )^{16}} ,  \\
&\lesssim \|F\|_{\mathcal{H}^{16}(\R^{12})}.
\end{align*} 
Since $\lambda_k \sim c k^{\frac13}$, it follows that 
\begin{align*} 
\|\mathcal{R}F\|_{L^1(\R^3)}& \lesssim    \big(\sum_{k_1,k_2,k_3,k_4 \geq 1}   \frac{1}{(\lambda_{k_1}\lambda_{k_2} \lambda_{k_3} \lambda_{k_4} )^{\frac{13}{4})}} \big)^{\frac12} \|F\|_{\mathcal{H}^{16}(\R^{12})} \\
&\lesssim   \big(\sum_{k_1,k_2,k_3,k_4 \geq 1}   \frac{1}{({k_1} {k_2}  {k_3} {k_4} )^{ \frac{13}{12})}} \big)^{\frac12}   \|F\|_{\mathcal{H}^{16}(\R^{12})} \lesssim \|F\|_{\mathcal{H}^{16}(\R^{12})}.
\end{align*}

Finally, using the interpolation inequality $\|\mathcal{R}F\|_{L^p(\R^3)} \lesssim \|\mathcal{R}F\|^{1-\frac1p}_{L^{\infty}(\R^{3})}   \|\mathcal{R} F\|^{\frac1p}_{L^1(\R^3)}$, we immediately obtain the desired estimate \eqref{norme-LP}.
\end{proof}

\

For all $\la_1,\la_2,\la_3\in [0,1]$ such that $\la_1+\la_2+\la_3=1$, we can naturally rephrase \eqref{norme-LP} as
\begin{equation*} 
\|\mathcal{R}F\|_{L^p(\R^3)} \lesssim \|\mathcal{R}F\|_{L^{\infty}(\R^{3})}^{\la_1(1-\frac1p)}\|\mathcal{R}F\|_{L^{\infty}(\R^{3})}^{\la_2(1-\frac1p)}\|\mathcal{R}F\|_{L^{\infty}(\R^{3})}^{\la_3 (1-\frac1p)}   \|F\|^{\frac1p}_{\mathcal{H}^{16}(\R^{12})},
\end{equation*}
which, combined with the estimates of Lemma \ref{conti}, provides us with the following general interpolation result.

\begin{corollary}\label{coro:interpol-T}
Let $\la_1,\la_2,\la_3\in [0,1]$ be such that $\la_1+\la_2+\la_3=1$. Then for all $\eps,\eps_1,\eps_2,\eps_3>0$, $q>\frac32$ and $p\geq \frac{1}{\eps}$, it holds that
\small
\begin{align*}
&\|\mathcal{R}F\|_{L^p(\R^3)} \lesssim \Big(1 \vee \|F\|_{L^{\infty}(\R^{12})} \Big)^{\la_1}\Big(1 \vee \sup_{y_1,y_2,z_2 \in \R^3}\Big(\int dz_1 \big| H_{y_1}  F\big|^q\Big)^{\frac1q}  \Big)^{\la_2}\Big(1 \vee \sup_{y_1,y_2,z_1 \in \R^3}\Big(\int dz_2 \big| H_{y_2}  F\big|^q\Big)^{\frac1q}  \Big)^{\la_3}\\
&\hspace{1cm} \Big(1\vee \|F\|_{\mathcal{H}^{16}(\R^{12})}\Big)^{\eps} \Big(\|H_{y_1}  H_{y_2}F\|_{L^{\infty}(\R^{12})}\Big)^{\la_1\eps_1} \Big(\|H^{2}_{y_1}H^{-1}_{z_1} H_{y_2}F\|_{L^{\infty}(\R^{12})}\Big)^{\la_2\eps_2} \Big(\|H^{2}_{y_2}H^{-1}_{z_2} H_{y_1}F\|_{L^{\infty}(\R^{12})}\Big)^{\la_3\eps_3} .
\end{align*}
\normalsize

\end{corollary}

\smallskip            
 
\subsection{Analysis of a high-low-frequency operator}

We now turn to a second operator involving high-low-frequency interactions, which will also arise in our subsequent computations. 

\smallskip

Namely, for $\al\in (0,1)$, we define the operator $\cp^{(\al)}$ acting on $F: (z_1,z'_1,z_2,z'_2) \mapsto F(z_1,z'_1,z_2,z'_2)$ by the formula
\begin{align*}
&\big(\cp^{(\al)}_jF\big)(x_1,x_2)\\
&:=\int dz_1 dz_1' dz_2 dz_2' \, F(z_1,z_1',z_2,z_2')\bigg[2^{-2j\al}\int dy_1 \, \delta_{j}(x_1,y_1)\Big(\sum_{i_1\leq i_1'-4}  \delta_{i_1}(y_1,z_1) H^\al_{y_1}\big(\delta_{i_1'}(y_1,z_1')\big) \Big)\bigg]\\
&\hspace{4cm}\bigg[2^{-2j\al} \int dy_2\, \delta_j(x_2,y_2)\Big(\sum_{i_2\leq i_2'-4}  \delta_{i_2}(y_2,z_1) H^\al_{y_2}\big(\delta_{i_2'}(y_2,z_2')\big) \Big)\bigg].
\end{align*}

With the notations $\delta_{k, z \to y}$ and  $\delta^{(-2\alpha)}_{\ell, z' \to y}$ introduced in \eqref{defi:delta-i}-\eqref{defi:delta-alpha-i}, we can rephrase this definition as
  \begin{equation}\label{defP}
 (\mathcal{P}_j^{(\alpha)}F)(x_1,x_2)=\mathcal{M}_{j, (z_1,z'_1) \to x_1}^{(\alpha)} \mathcal{M}_{j, (z_2,z'_2) \to x_2}^{(\alpha)}F,
   \end{equation}
 where the operator  $ \mathcal{M}_{j, (z,z') \to x}^{(\alpha)}$ is defined for functions $G: (z,z')  \mapsto G (z,z') $ by 
\begin{equation}\label{defMal}
\mathcal{M}_{j, (z,z') \to x}^{(\alpha)} := 2^{-2j\alpha}\delta_{j,y \to x} \sum_{\ell \geq -1} \sum_{k \leq \ell -4} 2^{2\alpha \ell} \delta_{k, z \to y}  \delta^{(-2\alpha)}_{\ell, z' \to y}.
\end{equation}

The following result generalizes the paraproduct estimate given in Proposition~\ref{Prop-est-para} $(ii)$.
  \begin{lemma}\label{Lem-Comp}
Uniformly in $j \geq -1$, we have 
 \begin{equation*}
\| \mathcal{M}_{j, (z,z') \to x}^{(\alpha)} G \|_{L^\infty_{x}} \lesssim \| G \|_{L^\infty_{z,z'}} .
\end{equation*}
\end{lemma}
 
Assuming this result holds true, we can establish the following:
   \begin{lemma}\label{lem:cpal}
 Uniformly in $j \geq -1$, we have 
 \begin{equation*}
\| \mathcal{P}_j^{(\alpha)}F  \|_{L^\infty_{x_1,x_2}} \lesssim \| F \|_{L^\infty_{z_1,z'_1,z_2,z'_2}} .
\end{equation*}
\end{lemma}

 \begin{proof}
 By \eqref{defP} and Lemma \ref{Lem-Comp} applied to $(z_1,z_1',x_1)$, we obtain that for all $x_2 \in \R^3$
 \begin{equation*}
 \| (\mathcal{P}_j^{(\alpha)}F)(\cdot , x_2)  \|_{L^\infty_{x_1}} \lesssim   \| \big( \mathcal{M}_{j, (z_2,z'_2) \to x_2}^{(\alpha)} F\big)(\cdot , x_2)   \|_{L^\infty_{z_1,z'_1}} .
 \end{equation*}
 Applying  Lemma \ref{Lem-Comp} again, this time with respect to $(z_2,z_2',x_2)$, we conclude that for all $z_1,z'_1$ and $x_2$
 $$\big| \big(\mathcal{M}_{j, (z_2,z'_2) \to x_2}^{(\alpha)} F\big)( z_1,z'_1 , x_2) \big| \lesssim  \| F \|_{L^\infty_{z_1,z'_1,z_2,z'_2}} , $$
 which completes the proof.
 \end{proof}
 
  \begin{proof}[Proof of Lemma \ref{Lem-Comp}]
Let us write $ \mathcal{M}_{j, (z,z') \to x}^{(\alpha)} =  \mathcal{M}_{j, (z,z') \to x}^{(\alpha),1} + \mathcal{M}_{j, (z,z') \to x}^{(\alpha),2} + \mathcal{M}_{j, (z,z') \to x}^{(\alpha),3} $, with 
  
  $$ \mathcal{M}_{j, (z,z') \to x}^{(\alpha),1} := 2^{-2j\alpha} \delta_{j,y \to x} \sum_{\substack{  \ell \geq -1 \\  |\ell-j| \leq 3   }} \sum_{k \leq \ell -4} 2^{2\alpha \ell} \delta_{k, z \to y}  \delta^{(-2\alpha)}_{\ell, z' \to y}$$
    $$ \mathcal{M}_{j, (z,z') \to x}^{(\alpha),2} := 2^{-2j\alpha} \delta_{j,y \to x}  \sum_{\ell=-1}^{j-4}  \sum_{k \leq \ell -4} 2^{2\alpha \ell} \delta_{k, z \to y}  \delta^{(-2\alpha)}_{\ell, z' \to y}$$
    $$ \mathcal{M}_{j, (z,z') \to x}^{(\alpha),3} := 2^{-2j\alpha} \delta_{j,y \to x}   \sum_{\ell=j+4}^{+\infty} \sum_{k \leq \ell -4} 2^{2\alpha \ell} \delta_{k, z \to y}  \delta^{(-2\alpha)}_{\ell, z' \to y}.$$
  
  We now study the contribution of each term separately. \medskip
  
  $\bullet$ Study of $ \mathcal{M}_{j, (z,z') \to x}^{(\alpha),1} $. By continuity of $ \delta_{j,y \to x}$,
\begin{equation}\label{est-prelim}
\big\|  \mathcal{M}_{j, (z,z') \to x}^{(\alpha),1}  G \big\|_{L^\infty_x} \lesssim 2^{-2j\alpha}  \big\|   \sum_{\substack{  \ell \geq -1 \\  |\ell-j| \leq 3   }} \sum_{k \leq \ell -4} 2^{2\alpha \ell} \delta_{k, z \to y}  \delta^{(-2\alpha)}_{\ell, z' \to y}G   \big\|_{L^\infty_y}.
\end{equation}
Then for all $y \in \R^3$, using the continuity of $ \delta^{(-2\alpha)}_{\ell, z' \to y} $ and the fact that $|\ell -j| \leq 3$,  
\begin{eqnarray*}
 2^{-2j\alpha}  \Big| \big(\sum_{\substack{  \ell \geq -1 \\  |\ell-j| \leq 3   }} \sum_{k \leq \ell -4} 2^{2\alpha \ell} \delta_{k, z \to y}  \delta^{(-2\alpha)}_{\ell, z' \to y}G \big)(y) \Big| & \lesssim  &\sum_{\substack{  \ell \geq -1 \\  |\ell-j| \leq 3   }}  2^{2\alpha (\ell-j)}  \big\| \sum_{k \leq \ell -4}  \delta_{k, z \to y} G (z, z')  \big\|_{L^\infty_{z'}} \\
 & \lesssim  &  \sup_{\ell\geq -1} \big\| \sum_{k \leq \ell -4}  \delta_{k, z \to y} G (z, z')  \big\|_{L^\infty_{z'}}.
\end{eqnarray*}
For all $z' \in \R^3$, by Proposition \ref{prop.conti}
$$ \sup_{\ell\geq -1} \Big\| \sum_{k \leq \ell -4}  \delta_{k, z \to y} G (z, z')  \Big\|_{L^\infty_y}  \lesssim   \|  G (\cdot, z')  \|_{L^\infty_z}  \lesssim   \|  G   \|_{L^\infty_{z,z'}} ,   $$
which leads to
\begin{equation*}
 2^{-2j\alpha}  \Big| \big(\sum_{\substack{  \ell \geq -1 \\  |\ell-j| \leq 3   }} \sum_{k \leq \ell -4} 2^{2\alpha \ell} \delta_{k, z \to y}  \delta^{(-2\alpha)}_{\ell, z' \to y}G \big)(y) \Big| \lesssim  \|  G   \|_{L^\infty_{z,z'}}.
\end{equation*}
Thus, going back to \eqref{est-prelim}, we have proven that
\begin{equation}\label{est-M1}
\big\|   \mathcal{M}_{j, (z,z') \to x}^{(\alpha),1}  G \big\|_{L^\infty_x} \lesssim \|  G   \|_{L^\infty_{z,z'}}.
\end{equation}

  $\bullet$ Study of $ \mathcal{M}_{j, (z,z') \to x}^{(\alpha),2} $. Since $\ell \leq j-4$, $k \leq j-4$, we will be in a position to apply Lemma~\ref{LemC14}. Using the expansion 
  \begin{equation}\label{expan}
  G(z,z')= \sum_{n,n' \geq 0}c_{n,n'} \vp_{n}(z)\vp_{n'}(z') ,  
    \end{equation}
we obtain
\begin{equation}\label{B-B-1}
\big( \mathcal{M}_{j, (z,z') \to x}^{(\alpha),2}    G\big)(x)
= 2^{-2j\alpha} \sum_{n,n' \geq 0}c_{n,n'} \sum_{\ell=-1}^{j-4}  \sum_{k \leq \ell -4} 2^{2\alpha \ell}\delta_{j,y \to x} \Big( \theta\big(\frac{\lambda_n}{2^{2k}} \big)\widetilde{\theta}\big(\frac{\lambda_{n'}}{2^{2\ell}} \big)  \vp_{n}(y)\vp_{n'}(y)\Big),
\end{equation}
with $ \widetilde{\theta}(t):=t^{\alpha}{\theta}(t) $. Next, thank to the support properties of $\theta$ and $\widetilde{\theta}$, we get $\lambda_{n} \sim 2^{2k}$ and $\lambda_{n'} \sim 2^{2\ell}$. Hence
\begin{equation}\label{est-n}
 n \lesssim 2^{6k} \lesssim 2^{6j} \;\text{ and } \;n' \lesssim 2^{6\ell} \lesssim 2^{6j} .
\end{equation}
For fixed $n, n' \geq 0$, we can apply  Lemma \ref{LemC14} with $f=\vp_{n'}$ and $g=\vp_{n}$, which yields for all $N \geq 1$
\begin{eqnarray*}
\Big\|  \delta_{j,y \to x} \Big( \theta\big(\frac{\lambda_n}{2^{2k}} \big)\widetilde{\theta}\big(\frac{\lambda_{n'}}{2^{2\ell}} \big)   \vp_{n}(y)\vp_{n'}(y)\Big)\Big\|_{L^{\infty}_x}   &\leq &C_N 2^{- jN } \| \delta'_{\ell}  \vp_{n'}\|_{L^{\infty}(\R^d)}\| \delta'_k  \vp_{n}\|_{L^{\infty}(\R^d)}\\
&\leq & C_N 2^{- j(N-3) } ,
\end{eqnarray*}
where in the last line we used the estimate \eqref{PsiLinf} together with \eqref{est-n}. Going back to \eqref{B-B-1}, we deduce that for all $M\geq 1$, 
\begin{equation}\label{est-M}
\big\|   \mathcal{M}_{j, (z,z') \to x}^{(\alpha),2}    G  \big\|_{L^{\infty}_x} \lesssim 2^{-M j} \sum_{n,n' \lesssim 2^{6j}}|c_{n,n'} | .
\end{equation}
Now observe that
$$c_{n,n'}= \int dz dz' G(z,z')  \vp_{n}(z)\vp_{n'}(z'),$$
and so
\begin{equation}\label{est-cn}
|c_{n,n'}| \leq \| G\|_{L^\infty_{z,z'}}\|\vp_{n} \|_{L^1(\R^d)} \|\vp_{n'} \|_{L^1(\R^d)}.
\end{equation}
At this point, note that there exists $K_0>0$ such that 
$$\|\vp_{n} \|_{L^1(\R^d)},\;  \|\vp_{n'} \|_{L^1(\R^d)} \lesssim  2^{K_0 j}. $$
Indeed, by the Cauchy-Schwarz inequality and with \eqref{eq-nor} in mind, one has for $K_0>0$ large enough 
\begin{eqnarray*}
\|\vp_{n} \|_{L^1(\R^d)}  &\leq & \| \langle x \rangle^{-K_0} \|_{L^2(\R^3)} \| \langle x \rangle^{K_0} \vp_n\|_{L^2(\R^3)} \\
& \lesssim & \| \vp_n \|_{\mathcal{H}^{K_0}(\R^3)} \\
& \lesssim & \lambda_n^{\frac{K_0}2} \lesssim 2^{K_0 j}.
\end{eqnarray*}
As a result, from \eqref{est-cn} we get that
$$|c_{n,n'}| \leq 2^{2K_0 j} \| G\|_{L^\infty_{z,z'}}.$$
Combining the latter estimate with \eqref{est-M}, we can conclude that
\begin{equation} \label{est-M2}
\|   \mathcal{M}_{j, (z,z') \to x}^{(\alpha),2}    G  \|_{L^{\infty}_x} \lesssim  \| G\|_{L^\infty_{z,z'}}.
\end{equation}

 $\bullet$ Study of $ \mathcal{M}_{j, (z,z') \to x}^{(\alpha),3} $. In this term, $\ell$ is the largest index and $j \leq \ell-4$, $k \leq \ell-4$. We will thus be able to apply Lemma \ref{lem.SC}. Using again the expansion \eqref{expan} we get 
\begin{equation}
\big( \mathcal{M}_{j, (z,z') \to x}^{(\alpha),3}    G\big)(x)= 2^{-2j\alpha} \sum_{n,n' \geq 0}c_{n,n'}  \sum_{\ell=j+4}^{+\infty}   \sum_{k \leq \ell -4} 2^{2\alpha \ell}\delta_{j,y \to x} \Big( \theta\big(\frac{\lambda_n}{2^{2k}} \big)\widetilde{\theta}\big(\frac{\lambda_{n'}}{2^{2\ell}} \big)  \vp_{n}(y)\vp_{n'}(y)\Big).\label{B-B-2}
\end{equation}
Similarly to the previous case, on the supports of $\theta$ and $\widetilde{\theta}$, one has
\begin{equation*} 
 n \lesssim 2^{6k} \lesssim 2^{6\ell} \;\text{ and } \;n' \lesssim 2^{6\ell}  ,
\end{equation*}
and therefore for all $N \geq 1$
\begin{equation*}
\Big\|  \delta_{j,y \to x} \Big( \theta\big(\frac{\lambda_n}{2^{2k}} \big)\widetilde{\theta}\big(\frac{\lambda_{n'}}{2^{2\ell}} \big)   \vp_{n}(y)\vp_{n'}(y)\Big)\Big\|_{L^{\infty}_x}   \leq  C_N 2^{- \ell N },
\end{equation*}
as well as 
$$|c_{n,n'}| \lesssim  2^{2K_0\ell} \| G\|_{L^\infty_{z,z'}}.$$
This in turn entails that for all $M \geq 1$
\begin{equation}\label{est-M3}
\|   \mathcal{M}_{j, (z,z') \to x}^{(\alpha),3}    G  \|_{L^{\infty}_x} \lesssim  2^{-M \ell} \sum_{n,n' \lesssim 2^{6\ell}}|c_{n,n'} |
 \lesssim   \| G\|_{L^\infty_{z,z'}}. 
\end{equation}

Collecting the estimates \eqref{est-M1}, \eqref{est-M2} and \eqref{est-M3} completes the proof.
 \end{proof}

\

\section{Reformulation of the problem and local wellposedness}\label{section:local-wellposed}

As announced in the introduction, our implementation of the fixed-point argument will rely on a reformulation of the problem that   parallels the approach used in \cite{MW} for the~$\Phi^4_3(\Delta)$ model. Recall that this latter method is essentially based on two ingredients:

\smallskip

\noindent
$\bullet$ the construction of specific stochastic processes, encoded in the form of diagrams, around which the entire dynamics of the equation is structured in a deterministic manner.

\smallskip

\noindent
$\bullet$ the consideration of an auxiliary system of two equations resulting from the expansion of the (approximated) model.

\smallskip

We will naturally transpose these two elements into the framework dictated by the harmonic oscillator. The equivalence with the original model \eqref{eq:RGL0}—at the level of the approximated and renormalized equation—will then be made clear through the statement of Proposition \ref{prop:link}.

\subsection{The $\Phi^4_3$ diagrams}

For all $n\geq 1$, let $(\xi^{(n)})$ be the white-noise approximation introduced in~\eqref{regu-noise}, and denote by $\<Psi>^{(n)}$ the solution of the equation  
\begin{equation*}
\left\{
\begin{aligned}
&(\partial_t +H) \<Psi>^{(n)}={\xi^{(n)}}, \quad   (t,x)\in \R_+\times \R^3,\\
&\<Psi>^{(n)}_0=0.
\end{aligned}
\right.
\end{equation*}

In other words, with representation \eqref{regu-noise} of $\xi^{(n)}$ in mind,
\begin{equation}\label{def-L}
\<Psi>^{(n)}_t(x)=\int_0^t \big(e^{-(t-s)H}dW^{(n)}_s\big)(x)=\int_0^t \int_{\R^3} K_{t-s}(x,y)\, W^{(n)}(ds,dy).
\end{equation}

Then we define the successive diagrams associated to the problem along the formulas (we refer to Section~\ref{parap} in the appendix for the definition of the resonant product $f\pe g$ in the context of paraproducts):
\begin{equation}\label{ord2}
\<Psi2>^{(n)}:=\big(\<Psi>^{(n)}\big)^2-\frakc^{\mathbf{1},(n)}, \quad  \<IPsi2>^{(n)}_t:=\int_0^t ds \, e^{-(t-s)H} \<Psi2>^{(n)}_s, 
\end{equation}
\begin{equation}\label{ord3}
\<Psi3>^{(n)}:=\big(\<Psi>^{(n)}\big)^3-3\frakc^{\mathbf{1},(n)} \<Psi>^{(n)}, \quad \<IPsi3>^{(n)}_t:=\int_0^t ds \, e^{-(t-s)H}\<Psi3>^{(n)}_s,
\end{equation}
\begin{equation}\label{ord4}
\<PsiIPsi3>^{(n)}:=\<Psi>^{(n)}\pe \<IPsi3>^{(n)}, \quad \quad \<Psi2IPsi2>^{(n)}:=\<Psi2>^{(n)} \pe \<IPsi2>^{(n)}-\frakc^{\mathbf{2},(n)},
\end{equation}
\begin{equation}\label{ord5}
\<Psi2IPsi3>^{(n)}:=\<Psi2>^{(n)} \pe \<IPsi3>^{(n)} -3\, \frakc^{\mathbf{2},(n)}\, \<Psi>^{(n)} ,
\end{equation}
where the deterministic sequences $(\frakc^{\mathbf{1},(n)}),(\frakc^{\mathbf{2},(n)})$ are respectively given by
\begin{equation}\label{renorma-cstts}
\frakc^{\mathbf{1},(n)}_t(x):=\mathbb{E}\Big[ \big| \<Psi>^{(n)}_t(x)\big|^2\Big] \quad \text{and} \quad \frakc^{\mathbf{2},(n)}_t(x):=\mathbb{E}\Big[ \<Psi2>^{(n)}_t(x) \<IPsi2>^{(n)}_t(x)\Big].
\end{equation}

\

Our main construction result can now be stated as follows (recall Definition \ref{defi:conv-c--la} for the convergence in the spaces $\cac_T^{-\la}E$, $\la\in (0,1)$).

\begin{proposition}\label{prop:conv-arbre}
For all $T>0$ and $0<\eps<\frac14$, the sequence of diagrams
\begin{equation}\label{regular-drivers}
Z^{(n)}:=\Big(\<Psi>^{(n)}, \<Psi2>^{(n)}, \<IPsi3>^{(n)},\<PsiIPsi3>^{(n)},\<Psi2IPsi2>^{(n)},\<Psi2IPsi3>^{(n)}\Big)
\end{equation}
converges almost surely in the space
\begin{equation}\label{space-for-z}
\mathcal Z_{\eps,T}:=\cac_T \cb_x^{-\frac12-\eps} \times \cac_T\cb_x^{-1-\eps} \times \big(\cac_T\cb_x^{\frac12-\eps}\cap \cac^{\frac14-\eps}_T \cb^\eps_x\big)\times \cac^{-\frac{\varepsilon}{2}}_T\cb^{-\frac{\eps}{2}} _x\times \cac^{-\varepsilon}_T\cb^{-\frac{\eps}{2}}_x \times \cac^{-\frac14-\eps}_T\cb_x^{-\frac14-2\eps} .
\end{equation}
We naturally denote its limit by
$$Z:=\Big(\<Psi>, \<Psi2>, \<IPsi3>,\<PsiIPsi3>,\<Psi2IPsi2>,\<Psi2IPsi3>\Big).$$
\end{proposition}

\

The proof of this fundamental convergence result will be the topic of Sections \ref{sec:prelim-mat} to \ref{section:fifth-order-diagram} below. For better readability, we have summed up the respective regularities of the limit processes in Table \ref{table-reg}.

\begin{table}
{\renewcommand{\arraystretch}{2}
\begin{tabular}{| c |c | c | c |c | c | c |c | c |  }
\hline 
\; $\tau $\; &  \ $\<Psi> $& $\<Psi2>$ &\;  $\<IPsi3>$\;  & \; $\<PsiIPsi3>$ \; & \; $\<Psi2IPsi2>$ \; & $\<Psi2IPsi3>$ \    \\
\hline 
 $\mathcal{E}_{\tau}$ &  $\cac_T \cb_x^{-\frac12-\eps}$  &  $\cac_T\cb_x^{-1-\eps} $  & $\cac_T\cb_x^{\frac12-\eps}\cap \cac^{\frac14-\eps}_T \cb^\eps_x$    &   $ \cac^{-\frac{\varepsilon}{2}}_T\cb^{-\frac{\eps}{2}} _x$   &     $\cac^{-\varepsilon}_T\cb^{-\frac{\eps}{2}}_x $    &  $ \cac^{-\frac14-\eps}_T\cb_x^{-\frac14-2\eps} $   \\[5pt]
 \hline 
\end{tabular}}
 \caption{The diagrams and their regularity : $\tau \in \mathcal{E}_{\tau}$.}\label{table-reg}
 \end{table}

\begin{remark}
If we compare the regularities given above with those of the corresponding trees in the   $\Phi^4_3(\Delta)$ model  (see, for instance, \cite[Table 1]{MW}), we immediately see that the first three elements $\<Psi>$, $\<Psi2>$, $\<IPsi3>$ share the same regularity—while keeping in mind, of course, that the spaces $\cb_x^{\al}$ here refer to the Besov spaces associated with the harmonic oscillator.

However, for purely technical reasons, we have been compelled to slightly modify the topologies used for the convergence of $\<PsiIPsi3>$, $\<Psi2IPsi2>$, $\<Psi2IPsi3>$, and introduce the spaces $ \mathcal{C}_T^{-\lambda} \mathcal{B}_x^{\alpha}$ ($\la\in (0,1),\al\in \R$), which provide greater flexibility in the treatment of temporal regularity (see the related constructions in Sections \ref{Sect-fourth1} to \ref{section:fifth-order-diagram}).

Note also that we do not expect the regularity $\cac^{-\frac14-\eps}_T\cb_x^{-\frac14-2\eps}$ of $\<Psi2IPsi3>$ to be optimal, but it turns out to be sufficient for our purpose.
\end{remark}

\begin{remark}\label{rk:young-interpretation}
According to Definition \ref{defi:conv-c--la}, the convergence $f^{(n)} \to f $ in $ \mathcal{C}_T^{-\lambda} \mathcal{B}_x^{\alpha}$ ($\lambda \in (0,1),\al\in \R $) should be understood as $\widetilde{f}^{(n)} \to \widetilde{f}$ in $\mathcal{C}_T^{1-\lambda} \mathcal{B}_x^\alpha$,for some (time) function $\widetilde{f}$, where we have set $\dis \widetilde{f}^{(n)}_t := \int_0^t f^{(n)}_s \, ds$ (in other words, $f:=\partial_t \widetilde{f}$). Thus, to be perfectly clear, the norm of $Z\in \mathcal{Z}_{\eps,T}$ must be read as
\begin{align*}
&\big\|Z\big\|_{\mathcal{Z}_{\eps,T}}:=\big\|\<Psi>\big\|_{\cac_T \cb_x^{-\frac12-\eps}} + \big\|\<Psi2>\big\|_{\cac_T\cb_x^{-1-\eps}} +\Big( \big\|\<IPsi3>\big\|_{\cac_T\cb_x^{\frac12-\eps}}+\big\|\<IPsi3>\big\|_{\cac^{\frac14-\eps}_T \cb^\eps_x}\Big)\\
&\hspace{6.5cm}+\big\|\widetilde{\<PsiIPsi3>}\big\|_{\cac^{1-\frac{\varepsilon}{2}}_T\cb^{-\frac{\eps}{2}} _x}+ \big\|\widetilde{\<Psi2IPsi2>}\big\|_{\cac^{1-\eps}_T\cb^{-\frac{\eps}{2}}_x} + \big\|\widetilde{\<Psi2IPsi3>}\big\|_{{ \cac}^{\frac34-\eps}_T\cb^{-\frac14-2\eps}_x}.
\end{align*}
In this framework, we will subsequently need to handle mild integrals of the form  
\begin{equation}\label{mild-young-interpr}  
\int_0^t e^{-(t-s)H}\big(u_s \, f_s\big) ds,  \quad \text{with }  \ f_t := \partial_t \widetilde{f}_t \in \cac^{-\la}_T\cb^\al_x
\end{equation}  
and $u$ in a suitable class of functions. \textit{These integrals will systematically be treated in the Young mild sense}, that is, as  
\begin{equation*}
\int_0^t e^{-(t-s)H}\big(u_s \, f_s\big) ds:=\int_0^t e^{-(t-s)H}\big(u_s \, d\widetilde{f}_s\big).  
\end{equation*}  

Further details on the interpretation and control of such a Young mild integral are provided in Appendix \ref{appendix:young}.

\end{remark}

With the above comments in mind, we now propose to review how the diagrams constructed through Proposition \ref{prop:conv-arbre} can be exploited toward a full understanding of the dynamics in \eqref{eq:RGL0}.

\subsection{Reformulation of the problem}\label{para-refor}

\smallskip

Before we introduce the auxiliary system at the core of the analysis, let us point out that the commutator-type operator $\big[ \pl, \pe\big]$ is defined (and studied) in Proposition~\ref{prop:commutor}. We follow here the strategy developed  by Mourrat-Weber in \cite{MW} and we refer to~\cite{MW} for more details. 

\smallskip

\noindent
\textbf{Auxiliary system.} Given a set of diagrams $Z\in \mathcal Z_{\eps,T}$ (as defined in \eqref{space-for-z}) and 
an initial condition $(v_0, w_0)$, we consider the dynamics provided by the two equations:  
\begin{eqnarray} \label{mild:v}
v_t &=& e^{-tH} v_0 + \int_0^t e^{-(t-s)H} F(v_s +w_s ) ds   \\ \label{mild:w}
w_t &=& e^{-tH} w_0 + \int_0^t e^{-(t-s)H} G(v_s,w_s)ds,
\end{eqnarray}
where 
\begin{equation}\label{def-F}
F(v+w) = -3(v+w -\<IPsi3>) \pl  \<Psi2>
\end{equation}
\begin{equation}\label{def-G}
G(v,w) = -(v+w)^3 -3\,  \mathrm{com}^Z (v,w) -3\, w \pe \<Psi2> -3(v+w-\<IPsi3>) \pg \<Psi2>+ P(v+w)
\end{equation}
with
\begin{align}
\mathrm{com}^Z(v,w)&: = \mathrm{com}^Z_1(v,w) \pe  \<Psi2> +\mathrm{com}^Z_2 (v+w) \nonumber \\
\mathrm{com}^Z_1(v,w)_t &:= e^{-tH} v_0 -3\bigg[ \int_0^t e^{-(t-s)H} [(v+w-\<IPsi3>)_s \pl  \<Psi2>_s] \, ds 
-(v+w-\<IPsi3>)_t \pl  \<IPsi2>_t\bigg], \label{eq3-14}\\ 
\mathrm{com}^Z_2 \big(v+w\big)_t &:= \big[ \pl, \pe\big] \Big(-3\big(v_t+w_t-\<IPsi3>_t\big), \<IPsi2>_t, \<Psi2>_t\Big), \nonumber
\end{align}
and 
$$P(v+w)=\tau^{(0)} +\tau^{(1)}\cdot(v+w) +\tau^{(2)}\cdot (v+w)^2$$
\begin{align*}
\tau^{(0)} &:= (\<IPsi3>)^3 -3\Big[ \<Psi>  \pg (\<IPsi3>)^2  + \<Psi>  \pl (\<IPsi3>)^2 +  \<Psi> \pe [\<IPsi3> \pe \<IPsi3>] 
+2 \<IPsi3>\,\<PsiIPsi3>+2[\pl, \pe](\<IPsi3>, \<IPsi3>, \<Psi>)\Big]
 -9\, \<IPsi3>\, \<Psi2IPsi2> +3\ \<Psi2IPsi3>,\\
\tau^{(1)} &:=  6 \Big[\<IPsi3> \pg \<Psi> + \<IPsi3> \pl \<Psi>+\<PsiIPsi3> \Big] - 3 \,(\<IPsi3>)^2+9\, \<Psi2IPsi2>,\\
\tau^{(2)} &:= -3\, \<Psi> +3\<IPsi3> .
\end{align*}

\

\begin{proposition}\label{prop:link}
Let $T>0$ and $0<\eps<\frac14$. Consider the sequence of diagrams \eqref{ord2} -- \eqref{ord5}, 
with $\frakc^{\mathbf{1},(n)},\frakc^{\mathbf{2},(n)}$ introduced in \eqref{renorma-cstts}. 
Fix $n\geq 1$. Let 
$(v^{(n)},w^{(n)})$ be a solution of the system \eqref{mild:v}-\eqref{mild:w} described above, 
for $v_0\in \mathcal{B}_{x}^{\frac12+\eps}$, for $w_0=0$, and 
for $Z^{(n)}$ given by \eqref{regular-drivers}. Then,  
\begin{equation}\label{decomposition-sol}
X^{(n)}=\<Psi>^{(n)}-\<IPsi3>^{(n)}+v^{(n)}+w^{(n)},
\end{equation}
solves the regular (renormalized) equation 
\begin{equation*}  
\left\{
\begin{aligned}
&\partial_t X^{(n)}  +HX^{(n)}= -(X^{(n)})^3+(3\frakc^{\mathbf{1},(n)}-9\frakc^{\mathbf{2},(n)}) X^{(n)} +\xi^{(n)}, \quad t>0, \quad x \in \R^3, \\
&X^{(n)}(0)=v_0.
\end{aligned}
\right.
\end{equation*}
\end{proposition}

\begin{proof}

Let $(v^{(n)},w^{(n)})$ be a solution of the above auxiliary system with initial data $(v_0,w_0)$ and driver $Z^{(n)}$. 
For the sake of simplicity, we drop $(n)$ in what follows. 
We then set $\mathcal{U} :=v+w$, and  
\begin{eqnarray*}
F(v+w) &=& -3(\mathcal{U} -\<IPsi3>) \pl  \<Psi2>,\\
G(v,w) &=& -\mathcal{U}^3 + \tau^{(0)} +\tau^{(1)}\cdot \mathcal{U} +\tau^{(2)}\cdot {\mathcal{U}}^2 + G^{(1)}(v,w) +G^{(2)}(v,w), \\ 
G^{(1)}(v,w) &:=& -3 \mathrm{com}^Z_2 \big(\mathcal{U}\big) = -3 \big[ \pl, \pe\big] \big(-3\big(\mathcal{U}-\<IPsi3>\big), \<IPsi2>, \<Psi2>\big),\\
G^{(2)} (v,w) &:=& -3 \mathrm{com}^Z_1(v,w) \pe  \<Psi2> - 3 w \pe \<Psi2>  -3(\mathcal{U}-\<IPsi3>) \pg \<Psi2>. 
\end{eqnarray*}
Using that for a solution of \eqref{mild:v}-\eqref{mild:w}, by \eqref{eq3-14} we have
$$ \mathrm{com}^Z_1(v,w) = v  + 3 \big[\big(\mathcal{U}-\<IPsi3>\big) \pl \<IPsi2>\big]  , $$
which in turn implies 
$$ \mathrm{com}^Z_1(v,w) \pe  \<Psi2> = v \pe  \<Psi2> + 3 \big[\big(\mathcal{U}-\<IPsi3>\big) \pl \<IPsi2>\big] \pe \<Psi2> , $$
we remark that 
\begin{eqnarray*}
G^{(2)} (v,w) &=& -3 \mathcal{U} \pge \<Psi2> -9 \big[\big(\mathcal{U}-\<IPsi3>\big) \pl \<IPsi2>\big] \pe \<Psi2>  + 3\<IPsi3> \pg \<Psi2>  \\
\tau^{(1)} \mathcal{U}+G^{(1)}(v,w) &=& \mathcal{U} \Big \{ 6 \<IPsi3> \<Psi> -3 (\<IPsi3>)^2  -9\, \frakc^\mathbf{2} \Big\}
+9 \big((\mathcal{U} - \<IPsi3>) \pl \<IPsi2>\big) \pe  \<Psi2> +9 \<IPsi3> (\<IPsi2> \pe   \<Psi2>).
\end{eqnarray*}

 Adding $\tau^{(0)}$ to $\tau^{(1)} \mathcal{U}+G^{(1)}(v,w)$ implies 
 \begin{multline*}
\tau^{(0)} + \tau^{(1)} \mathcal{U}+G^{(1)}(v,w) =  \mathcal{U} \Big \{ 6 \<IPsi3> \<Psi> -3 (\<IPsi3>)^2 \Big\} 
+9 \big((\mathcal{U} - \<IPsi3>) \pl \<IPsi2>\big) \pe  \<Psi2>\\
   +(\<IPsi3>)^3 - 3 \<Psi> (\<IPsi3>)^2 +3 \<IPsi3> \pe \<Psi2> 
  -9 \frakc^\mathbf{2} ( \mathcal{U} - \<IPsi3> +\<Psi> ) .
 \end{multline*}
In summary, 
\begin{multline*}
  \tau^{(0)} +\tau^{(1)}\cdot \mathcal{U} + G^{(1)}(v,w) +G^{(2)}(v,w)+F(\mathcal{U})  =\\
=  \mathcal{U} \Big \{ 6 \<IPsi3> \<Psi> -3 (\<IPsi3>)^2 \Big\} 
  +(\<IPsi3>)^3 - 3 \<Psi> (\<IPsi3>)^2 +3 \<IPsi3>  \<Psi2>   -9 \frakc^\mathbf{2} ( \mathcal{U} - \<IPsi3> +\<Psi> ) -3 \mathcal{U} \<Psi2>.
\end{multline*}
Therefore we obtain finally 
\begin{multline*}
(\partial_t +H) X =  (\partial_t +H) (\mathcal{U} + \<Psi>- \<IPsi3>) =\\
\begin{aligned}
&= -\mathcal{U}^3-3\, (\<Psi> -\<IPsi3>) \mathcal{U}^2 +
  \tau^{(0)} +\tau^{(1)}\cdot \mathcal{U}  + G^{(1)}(v,w) +G^{(2)}(v,w)+F(\mathcal{U})+\xi - \big( (\<Psi>)^3 -3\frakc^\mathbf{1}  \<Psi>\big) \\
&=  -X^3+{(3\frakc^\mathbf{1}-9\frakc^\mathbf{2})} X +\xi
\end{aligned}
\end{multline*}
where we have used the fact that $\<Psi2> = (\<Psi>)^2 - \frakc^\mathbf{1}.$
\end{proof}

\subsection{Local wellposedness}

\subsubsection{Main results}

Let $T>0$. For every small $\eps>0$, we consider the space 
\begin{align*}
\mathcal{X}_{\eps,T} &:= \Big\{ (v,w): \ v \in \mathcal{C}\big([0,T]; \mathcal{B}_x^{ {\frac12 +2\eps}}\big) \cap \mathcal{C}^{\frac14 }\big([0,T]; \cb_x^\eps\big) , 
\ w\in \mathcal{C}\big([0,T]; \mathcal{B}_{x}^{{ 1+2\eps}}\big) \cap \mathcal{C}^{{\frac12}}\big([0,T]; \cb_x^\eps\big) \Big\}
\end{align*}
equipped with its natural norm 
\begin{align*}
\big\|(v,w)\big\|_{\mathcal X_{\eps,T}} &:=  \max \Big\{ \sup_{0\le t \le T} \big\| v_t\big\|_{\mathcal{B}_{x}^{\frac12 +2\eps}}, 
\sup_{0\le s<t \le T} \frac{ \big\|v_t-v_s\big\|_{\cb_x^\eps}} { |t-s|^{\frac14} }, \\
&\hspace{1.5cm} \sup_{0\le t \le T} \big\|w_t\big\|_{\mathcal{B}_{x}^{1+2\eps}}, 
\sup_{0\le s<t \le T} \frac{ \big\|w_t-w_s\big\|_{\cb_x^\eps} } { |t-s|^{{\frac12}} }\Big\}.
\end{align*}

\

Given a set of diagrams $Z\in \mathcal Z_{\eps,T}$ (see \eqref{space-for-z}), we denote by $\gga:=(\gga^{\mathbf{v}},\gga^{\mathbf{w}})$
the map derived from the system \eqref{mild:v}-\eqref{mild:w}, that is
\begin{eqnarray*}
\gga^{\mathbf{v}}[v,w]_t&:=&  e^{-tH} v_0 +\int_0^t e^{-(t-s)H} F(v_s+w_s)\, ds \\
\gga^{\mathbf{w}}[v,w]_t &:=& e^{-tH} w_0 +\int_0^t e^{-(t-s)H} G(v_s,w_s)\, ds.
\end{eqnarray*}

We are now in a position to state our three main results regarding local wellposedness of the auxiliary system (the proofs of these properties will be detailed in Sections \ref{proo-1}, \ref{proo-2} and \ref{proo-3}). 

\begin{proposition}\label{prop:fixed-point}
Let $0<T \le 1$. For every $\eps>0$ small enough and every initial condition $(v_0,w_0)\in \mathcal{B}_x^{\frac12 +2\eps} \times \mathcal{B}_x^{1+2\eps}$, there exists $\nu>0$ such that the following assertions hold.

\smallskip

\noindent
$(i)$ For every  $Z\in \mathcal Z_{\eps,T}$, the map $\Gamma$ is well defined from $\mathcal{X}_{\eps,T}$ to $\mathcal{X}_{\eps,T}$ 
and for every $(v,w)\in \mathcal{X}_{\eps,T}$, one has
\begin{align*}
&\big\| \gga[v,w]\big\|_{\mathcal{X}_{\eps,T}}\leq 
P_1 \big(\big\| Z\big\|_{\mathcal Z_{\eps,T}} \big)\Big\{1+\big\|(v_0,w_0)\big\|_{\mathcal{B}^{\frac12 +2\eps}_x \times \mathcal{B}^{1+2\eps}_x}
+T^\nu \big\| (v,w)\big\|_{\mathcal{X}_{\eps,T}}^{3} \Big\},
\end{align*} 
for some positive polynomial $P_1$ of third degree.
\smallskip

\noindent
$(ii)$ Given two sets of diagrams $Z,Z'\in \mathcal Z_{\eps,T}$ with associated maps $\gga,\gga'$, and two elements $(v,w),(v',w')\in \mathcal{X}_{\eps,T}$, one has
\begin{align*}
&\big\| \gga[v,w]-\gga'[v',w']\big\|_{\mathcal{X}_{\eps,T}}\\
&\leq 
T^\nu P_2 \Big(\big\| Z\big\|_{\mathcal Z_{\eps,T}},\big\| Z'\big\|_{\mathcal Z_{\eps,T}} \Big) 
\Big( 1+  \big\| (v,w)\big\|^2_{\mathcal{X}_{\eps,T}} + \big\| (v',w')\big\|^2_{\mathcal{X}_{\eps,T}} \Big)
\Big\{\big\| (v,w)-(v',w')\big\|_{\mathcal{X}_{\eps,T}} +\big\| Z-Z'\big\|_{\mathcal Z_{\eps,T}}\Big\},
\end{align*}
for some positive polynomial expression $P_2$ of second degree.
\end{proposition}

\

\begin{corollary}\label{cor:fixed-point}
For all $\eps>0$ small enough, all initial condition $(v_0,w_0)\in \mathcal{B}_x^{\frac12 +2\eps} \times \mathcal{B}_x^{1+2\eps}$ and all set of diagrams $Z\in \mathcal Z_{\eps,1}$, there exists a random time $T_*=T_* \big(\omega, \eps, (v_0,w_0),Z\big) \in (0,1]$, which is a continuous 
function of $(v_0,w_0)$, $Z$,  
such that $\gga$ admits a unique fixed point $(v,w)$ in $\mathcal{X}_{\eps,T_*}$.
\end{corollary}

\

\begin{theorem} \label{th:local}
Fix $\eps>0$ small enough and let $(v_0, w_0) \in \mathcal{B}_x^{\frac12+2\varepsilon} \times \mathcal{B}_x^{{ 1+2\eps} }$ and 
$Z \in \mathcal{Z}_{\eps,1}$. Let $(v,w)$ be a unique solution of the system  \eqref{mild:v}-\eqref{mild:w} on $[0,T_*]$ with $T_* \in (0,1]$ obtained in Corollary~\ref{cor:fixed-point}. 

\smallskip

\noindent
$(i)$ Let $\{Z^{(n)}\}_n$ converge to $Z$ in $\mathcal Z_{\eps,1}$ a.s.   
Let $T_*^{(n)} =T_* (v_0^{(n)}, w_0^{(n)}, Z^{(n)})$ and a unique solution $(v^{(n)}, w^{(n)})$ of 
 \eqref{mild:v}-\eqref{mild:w} on $[0, T_*^{(n)}]$ with the initial condition $(v_0, w_0)$, and driven by $Z^{(n)}$. 
 Then, for every $0 <t < T_*$, we have 
 $$ \lim_{n\to \infty} \| (v^{(n)}, w^{(n)}) - (v,w) \|_{\mathcal X_{\eps, t}} =0.$$  

\smallskip

\noindent
$(ii)$ Moreover, for any $T>0$, and for any $Z \in \mathcal Z _{\eps, T}$, there exists a maximal time 
$T_{\max}>0$ such that  the system \eqref{mild:v}-\eqref{mild:w} has a unique solution 
 defined on $[0,T_{\max})$ with values in $\mathcal X_{\eps, t}$ for every $t<T_{\max}$, and   
the blowup alternative holds: 
 $T_{\max}=T$ or  $\lim_{t \uparrow T_{\max}}  \max \big\{
 \| v(t)\|_{\mathcal {B}_{x}^{\frac12+2\eps}}, \| w(t)\|_{\mathcal {B}_x^{1+2\eps}} \big\} = +\infty .$ 
For a sequence $\{Z^{(n)}\}_n$ converging $Z$, we have 
$T_{\max} \le \liminf_{n\to \infty} T_{\max}^{(n)} $ and 
for every $0 <t < T_{\max}$, we have 
 $$ \lim_{n\to \infty} \| (v^{(n)}, w^{(n)}) - (v,w) \|_{\mathcal X_{\eps, t}} =0.$$  

\end{theorem}

\

By combining the result of Theorem \ref{th:local} with the identification property in Proposition \ref{prop:conv-arbre}, we deduce our main (local) convergence result for the model under consideration.

\begin{theorem}\label{theo:main}
Let $X_0\in \cb_x^{\frac12+2\eps}(\R^3)$, for some $\eps>0$. Almost surely, there exists a random time $T=T(\omega)>0$ 
such that the sequence $(X^{(n)})$ of solutions to the renormalized stochastic equation
\begin{equation*} 
\left\{
\begin{aligned}
& \partial_t X^{(n)}  +HX^{(n)}= -(X^{(n)})^3+(3\frakc^{\mathbf{1},(n)}-9\frakc^{\mathbf{2},(n)}) X^{(n)} +\xi^{(n)}, \quad t>0, \quad x \in \R^3, \\
& X^{(n)}(0) = X_0,
\end{aligned}
\right.
\end{equation*}
converges a.s. to a limit solution $X$ in the space $\cac\big([0,T];\cb^{-\frac12-\eta}_x(\R^3)\big)$, for any $\eta>0$. 

Moreover, one has 
$$X - \<Psi> + \<IPsi3>   \in \mathcal{C}\big([0,T]; \mathcal{B}_x^{ {\frac12 +2\eps}}(\R^3)\big). $$

\end{theorem}

\subsubsection{Proof of Proposition \ref{prop:fixed-point}}\label{proo-1}

\

\smallskip

We will only focus on the proof of item $(i)$, and we leave it to the reader to verify that the assertion in item $(ii)$ can be established using entirely similar arguments.

\smallskip

The proof consists in fact in combining  the paracontrolled estimates associated with $H$ (that is, the properties highlighted in Appendix \ref{appendix:microlocal}) with the regularity assumptions on $(v_0,w_0)$, $(v,w)$ and~$Z$. The procedure is broadly the same as in the proof of \cite[Theorem 2.1]{MW}, but given the presence of slight technical variations (in the commutation estimates or in the interpretation of certain time integrals), we have chosen to provide a few details.

\smallskip

Let $\eps>0$ (small enough), $T>0$, $Z\in \mathcal Z_{\eps,T}$, $(v_0, w_0) \in \mathcal{B}_x^{\frac12+2\varepsilon} \times \mathcal{B}_x^{{ 1+2\eps} }$ and $(v,w)\in \mathcal{X}_{\eps,T}$. We will bound the two components $\gga^{\mathbf{v}}[v,w]$ and $\gga^{\mathbf{w}}[v,w]$ of $\gga[v,w]$ separately.

\smallskip

For more clarity, we set in what follows
$$\cb^\al_x =\cb^\al_{\infty,\infty}(\R^3), \quad \big\|Z\big\|_{\eps,T}:=\big\|Z\big\|_{\mathcal Z_{\eps,T}} \quad \text{and} \quad \big\|(v,w)\big\|_{\eps,T}:=\big\|(v,w)\big\|_{\mathcal{X}_{\eps,T}}.$$

\

\noindent
\textit{Bound on $\gga^{\mathbf{v}}[v,w]$.} By Lemma \ref{lem:actisemi}, we obtain first that
\begin{align*}
\big\|\gga^{\mathbf{v}} [v,w]_t \big\| _{\mathcal{B}_x^{\frac12 +2\varepsilon}} 
&\le \| e^{-tH} v_0\|_{\mathcal{B}_x^{\frac12+2\varepsilon}} +\int_0^t  ds\, \Big\|e^{-(t-s)H} F(v_s +w_s )\Big\|_{\mathcal{B}_x^{\frac12 +2\varepsilon}} \\ 
& \lesssim \| v_0\|_{ \mathcal{B}_x^{\frac12 +2 \varepsilon} } + \int_0^t \frac{ds}{(t-s)^{\frac34+\frac{3\eps}{2}}}  \| {F}(v_s +w_s)\|_{ \mathcal{B}_x^{-1-\varepsilon} } .
\end{align*}
Then, by \eqref{def-F} and Proposition \ref{Prop-est-para} $(ii)$, one has for all $s\in [0,T]$
\begin{align}  
\big\| {F}(v_s+w_s) \big\|_{\mathcal{B}_x^{-1-\eps}} &\lesssim \big\| (v_s+w_s-\<IPsi3>_s) \pl  \<Psi2>_s \big\|_{\mathcal{B}_x^{-1-\eps}}\nonumber \\
&  \lesssim   \big\|v_s+w_s-\<IPsi3>_s\big\|_{L^{\infty}(\R^3)} \big\|  \<Psi2>_s\big\|_{\mathcal{B}_x^{-1-\eps}}\nonumber  \\
&\lesssim \big\|v_s+w_s-\<IPsi3>_s\big\|_{\cb_x^{\frac12-\eps}} \big\|  \<Psi2>_s\big\|_{\mathcal{B}_x^{-1-\eps}} \lesssim \Big(\big\|(v,w)\big\|_{\eps,T} +\big\|Z\big\|_{\eps,T}\Big) \big\|Z\big\|_{\eps,T}\label{gav}
\end{align}
and as a result
\begin{align*}
\big\|\gga^{\mathbf{v}} [v,w]_t \big\| _{\mathcal{B}_x^{\frac12 +2\varepsilon}} 
& \lesssim \| v_0\|_{ \mathcal{B}_x^{\frac12 +2 \varepsilon} }  + T^{\frac14 -\frac{3\eps}{2}} \Big(\big\|(v,w)\big\|_{\eps,T} +\big\|Z\big\|_{\eps,T}\Big) \big\|Z\big\|_{\eps,T}.
\end{align*}

\smallskip

On the other hand, for all $0\leq s<t\leq T$, one has by Lemma \ref{lem:actisemi}
\begin{align*}
&\big\|\gga^{\mathbf{v}}[v,w]_t-\gga^{\mathbf{v}}[v,w]_s \big\|_{\cb_x^\eps}\\
&\le  \big\| (e^{-(t-s)H} -\id) e^{-sH} v_0\big\|_{\cb_x^\eps}   +\Big\| (e^{-(t-s)H} -\id) \int_0^s e^{-(s-r)H} {F}(v_r +w_r)\, dr\Big\|_{\cb_x^\eps}  \\
&\hspace{9cm}+  \Big\| \int_s^t e^{-(t-r) H} {F}(v_r +w_r)\, dr\Big\| _{\cb_x^\eps} \\
& \lesssim  (t-s)^{\frac{1}{4} +\frac{\varepsilon}{2}}  \|v_0\|_{\mathcal{B}_x^{\frac12+2\eps}}  +(t-s)^{\frac 14+\frac{\varepsilon}{2}} \int_0^s  \frac{dr}{(s-r)^{\frac34 +\frac{3\eps}{2}}}  \big\| F(v_r +w_r)\big\|_{\mathcal{B}_x^{-1-\varepsilon}}+\int_s^t \frac{dr}{(t-r)^{\frac{1+2\eps}{2}}}  \big\|  F(v_r +w_r)\big\|_{\mathcal{B}_x^{-1-\varepsilon}},
\end{align*}
and we can inject the bound \eqref{gav} to deduce that
\begin{align*}
&\big\|\gga^{\mathbf{v}}[v,w]_t-\gga^{\mathbf{v}}[v,w]_s \big\|_{\cb_x^\eps}\\
& \lesssim (t-s)^{\frac14 +\frac{\varepsilon}{2}} \|v_0\|_{\mathcal{B}_x^{\frac12+2\eps}} +\Big[(t-s)^{\frac14 +\frac{\varepsilon}{2}} s^{\frac14 -\frac{3\eps}{2}}    
+ (t-s)^{\frac12-\varepsilon}\Big] \Big(\big\|(v,w)\big\|_{\eps,T} +\big\|Z\big\|_{\eps,T}\Big) \big\|Z\big\|_{\eps,T}\\
& \lesssim  (t-s)^{\frac14 +\frac{\varepsilon}{2}} \Big[\|v_0\|_{\mathcal{B}_x^{\frac12+2\eps}} +T^{\frac14-\frac{3\eps}{2}}\Big(\big\|(v,w)\big\|_{\eps,T} +\big\|Z\big\|_{\eps,T}\Big) \big\|Z\big\|_{\eps,T}\Big].
\end{align*}

\

\noindent
\textit{Bound on $\gga^{\mathbf{w}}[v,w]$.} Recall \eqref{def-G}. We decompose $G(v,w)$ as
\begin{equation}\label{decompogpro}
 G(v,w) = -6\, \<IPsi3> \,\<PsiIPsi3> -9\, \<IPsi3>\, \<Psi2IPsi2> +3\, \<Psi2IPsi3>  
+\big(6\, \<PsiIPsi3>+ 9\, \<Psi2IPsi2>\big)\, \big(v+w\big)+\tilde{G} (v,w) ,
\end{equation}
with
\begin{align}
\tilde{G}(v,w) :=& -\big(v+w\big)^3-3 \, w \pe \<Psi2> -3\, \big(v+w-\<IPsi3>\big) \pg \<Psi2> -6\big[\pl, \pe\big](\<IPsi3>, \<IPsi3>, \<Psi>)\nonumber\\
&+ \tilde{\tau}^{(0)} + \tilde{\tau}^{(1)}\cdot\big(v+w\big)+\tau^{(2)}\cdot\big(v+w\big)^2-3\, \mathrm{com}^Z (v,w) ,\label{defigtilde}
\end{align}
where
$$\tilde{\tau}^{(0)}:=(\<IPsi3>)^3 -3\Big[ \<Psi>  \pg (\<IPsi3>)^2  + \<Psi>  \pl (\<IPsi3>)^2 +  \<Psi> \pe \big(\<IPsi3> \pe \<IPsi3>\big) 
\Big],$$
$$\tilde{\tau}^{(1)}:=6 \Big[\<IPsi3> \pg \<Psi> + \<IPsi3> \pl \<Psi> \Big] - 3 \,(\<IPsi3>)^2.$$
In this way, one has
\small
\begin{align}
&\big\|\gga^{\mathbf{w}} [v,w]_t \big\| _{\cb_x^{{1+2\varepsilon}}} 
\lesssim \big\|e^{-tH} w_0 \big\|_{\cb_x^{1+2\eps}} + \Big\| \int_0^t e^{-(t-r)H} \big((v_r+w_r) \, \<PsiIPsi3>_r\big)\, dr \Big\|_{\cb_x^{1+2\eps}}\nonumber \\
&+ \Big\| \int_0^t e^{-(t-r)H} \big((v_r+w_r) \, \<Psi2IPsi2>_r\big) \, dr \Big\|_{\cb_x^{1+2\eps}} 
+\Big\| \int_0^t e^{-(t-r)H} (\<IPsi3>_r \, \<PsiIPsi3>_r)\,  dr \Big\|_{\cb_x^{1+2\eps}}+\Big\| \int_0^t e^{-(t-r)H} (\<IPsi3>_r \,  \<Psi2IPsi2>_r)\,  dr \Big\|_{\cb_x^{1+2\eps}}\nonumber   \\
&\hspace{1.5cm}
 +\Big\| \int_0^t e^{-(t-r)H} ( \<Psi2IPsi3>_r)\, dr \Big\|_{\cb_x^{1+2\eps}}+\int_0^t \frac{dr}{(t-r)^{\frac{3}{4}+2\eps}} \big\| \tilde{G}(v_r, w_r) \big\|_{\cb_x^{-\frac12-2\eps}} ,\label{bouggaw}
\end{align} \normalsize
where we have used Lemma \ref{lem:actisemi} $(ii)$ to derive the last term.

\smallskip

In order to control the terms from the second to the sixth in the above bound, we can rely on the (mild) Young estimate of Proposition \ref{prop:young}. Recall indeed that following Remark \ref{rk:young-interpretation}, 
one has for instance
\begin{equation}\label{interpretationyoung}
\int_0^t e^{-(t-r)H} \big((v_r+w_r) \, \<PsiIPsi3>_r\big)\, dr =\int_0^t e^{-(t-r)H} \big((v_r+w_r) \, d_r\widetilde{\<PsiIPsi3>}\big),
\end{equation}
where the latter integral is understood in the Young sense. Based on this interpretation and using Proposition \ref{prop:young}, we obtain successively, for every $\eps>0$ small enough, 
\begin{align*}
\Big\| \int_0^t e^{-(t-r)H} \big((v_r+w_r) \, \<PsiIPsi3>_r\big)\, dr \Big\|_{\cb_x^{1+2\eps}}
&\lesssim |t|^{\frac{1}{2}-3\eps} \Big(\|v\|_{\mathcal{C}_T^{\frac{1}{4} } \cb_x^{\eps}} + \|w\|_{\mathcal{C}_T^{ \frac12} \cb_x^{\eps}} \Big)
\|\widetilde{\<PsiIPsi3>}\|_{\mathcal{C}_T^{1-\frac{\eps}{2}} \cb_x^{-\frac{\eps}{2}}}\\
& \lesssim  T^{\frac{1}{2}-3\eps} \big\|(v,w)\big\|_{\eps,T} \big\|Z\big\|_{\eps,T},
\end{align*}
\begin{align*}
\Big\| \int_0^t e^{-(t-r)H} \big((v_r+w_r) \, \<Psi2IPsi2>_r\big) \, dr \Big\|_{\cb_x^{1+2\eps}} 
&\lesssim 
|t|^{\frac{1}{2}-3\eps} \Big(\|v\|_{\mathcal{C}_T^{\frac{1}{4}} \cb_x^{\eps}} + \|w\|_{\mathcal{C}_T^{\frac12} \cb_x^{\eps}} \Big)
\|\widetilde{\<Psi2IPsi2>}\|_{\mathcal{C}_T^{1-\eps} \cb_x^{-\frac{\eps}{2}}} \\
&\lesssim T^{\frac{1}{2}-3\eps} \big\|(v,w)\big\|_{\eps,T} \big\|Z\big\|_{\eps,T},
\end{align*}
\begin{align*}
\Big\| \int_0^t e^{-(t-r)H} (\<IPsi3>_r \, \<PsiIPsi3>_r)\,  dr \Big\|_{\cb_x^{1+2\eps}} 
&\lesssim  \|\<IPsi3>\|_{\mathcal{C}_T^{\frac{1}{4} -\eps} \cb_x^{\eps}}
\|\widetilde{\<PsiIPsi3> }\|_{\mathcal{C}_T^{1-\frac{\eps}{2}} \cb_x^{-\frac{\eps}{2}}} \lesssim  \|Z\|^2_{\eps,T},
\end{align*}
\begin{align*}
&
\Big\| \int_0^t e^{-(t-r)H} (\<IPsi3>_r \,  \<Psi2IPsi2>_r)\,  dr \Big\|_{\cb_x^{1+2\eps}} 
\lesssim  \|\<IPsi3>\|_{\mathcal{C}_T^{\frac{1}{4} -\eps} \cb_x^{\eps}}
\|\widetilde{ \<Psi2IPsi2>}\|_{\mathcal{C}_T^{1-\eps} \cb_x^{-\frac{\eps}{2}}} \lesssim  \|Z\|^2_{\eps,T},
\end{align*}
\begin{align*}
&
 \Big\| \int_0^t e^{-(t-r)H} ( \<Psi2IPsi3>_r)\, dr \Big\|_{\cb_x^{1+2\eps}} \lesssim 
 \|\widetilde{\<Psi2IPsi3>} \|_{\mathcal{C}_T^{\frac34-\eps} \mathcal{B}_x^{-\frac14-2\eps}} \lesssim  \big\|Z\big\|_{\eps,T}.
\end{align*}

\

As far as the last term in \eqref{bouggaw} is concerned, we propose to show that for all $r\in (0,T]$,
\begin{equation}\label{boutilg}
\big\| \tilde{G}(v_r,w_r)\big\|_{\cb_x^{{ -\frac12 -2\varepsilon}}} \lesssim \big(1+\|v_0\|_{{\mathcal B}^{\frac12+2\eps}_x}+\big\|(v,w)\big\|_{\eps,T}^3\big) \big(1+\big\|Z\big\|_{\eps,T}^3\big) .
\end{equation}
To this end, let us bound each term of the decomposition \eqref{defigtilde} separately. 
We will use the elementary inclusions and estimates of Hermite Besov norms in Lemma \ref{lem:inclusion-besov} and Proposition~\ref{Prop-est-para}.

\smallskip

First, one has clearly
$$\big\| \big(v_r+w_r\big) ^3 \big\| _{\cb_x^{-\frac12 -2\eps}} \lesssim \big\|\big(v_r+w_r\big)^3\big\|_{L^{\infty}(\R^3)} \lesssim \big\|v_r+w_r\big\|_{L^{\infty}(\R^3)}^3\lesssim \big\|v_r+w_r\big\|_{\cb_x^{\frac12+2\eps}}^3\lesssim \big\|(v,w)\big\|_{\eps,T}^3.$$
Then, applying Proposition \ref{Prop-est-para}, we get that
$$\big\| w_r \pe \<Psi2>_r\big\| _{\cb_x^{-\frac12 -2\eps}} \lesssim \big\| w_r \pe \<Psi2>_r\big\| _{\cb_x^{\eps}} \lesssim\big\| w_r\big\|_{\cb_x^{1 +2\eps}} \big\| \<Psi2>_r \big\|_{\cb_x^{-1 -\eps}}\lesssim  \big\|(v,w)\big\|_{\eps,T}\big\|Z\big\|_{\eps,T},$$
and in a similar way
$$\big\| \big(v_r+w_r-\<IPsi3>_r\big) \pg \<Psi2>_r \big\|_{\cb_x^{-\frac12 -2\eps}} \lesssim \big\| v_r+w_r-\<IPsi3>_r\big\|_{\cb_x^{\frac12-\eps}} \big\| \<Psi2>_r \big\|_{\cb_x^{-1 -\eps}}
\lesssim \big\|(v,w)\big\|_{\eps,T}\big\|Z\big\|_{\eps,T} +\big\|Z\big\|_{\eps,T}^2 .$$
For the control of $\big[\pl, \pe\big](\<IPsi3>, \<IPsi3>, \<Psi>)$, we can appeal to Proposition \ref{prop:commutor}, which immediately gives
$$\big\| \big[\pl, \pe\big](\<IPsi3>_r, \<IPsi3>_r, \<Psi>_r)\big\|_{\cb_x^{-\frac12 -2\eps}}\lesssim \big\| \<IPsi3>_r\big\|_{\cb_x^{\frac12 -\eps}}^2 \big\| \<Psi>_r\big\|_{\cb_x^{-\frac12 -\eps}}\lesssim \big\|Z\big\|_{\eps,T}^3.$$
Now, to estimate the next three terms in \eqref{defigtilde}, observe that as an easy consequence of the properties contained Proposition \ref{Prop-est-para}, we get
$$\big\|\tilde{\tau}^{(0)}\big\|_{\cac_T\cb_x^{-\frac12 -2\eps}} \lesssim \big\|Z\big\|_{\eps,T}^3, \quad \quad \big\|\tilde{\tau}^{(1)}\big\|_{\cac_T\cb_x^{-\frac12 -\eps}} \lesssim  \big\|Z\big\|_{\eps,T}^2 \quad \ \text{and} \quad \ \big\|\tau^{(2)}\big\|_{\cac_T\cb_x^{-\frac12 -\eps}} \lesssim \big\|Z\big\|_{\eps,T}.$$
We deduce in particular that
$$\big\| \tilde{\tau}^{(1)}_r\cdot \big(v_r+w_r\big) \big\|_{\cb_x^{-\frac12 -2\eps}} \lesssim \big\| \tilde{\tau}^{(1)}_r\cdot \big(v_r+w_r\big) \big\|_{\cb_x^{-\frac12 -\eps}} \lesssim \big\|\tilde{\tau}^{(1)}_r\big\|_{\cb_x^{-\frac12 -\eps}} \big\|v_r+w_r\big\|_{\mathcal{B}_x^{\frac12+2\eps}}\lesssim \big\|Z\big\|_{\eps,T}^2 \big\|(v,w)\big\|_{\eps,T}$$
as well as
\begin{align*}
\big\| \tau^{(2)}_r\cdot \big(v_r+w_r\big)^2 \big\|_{\cb_x^{-\frac12 -2\eps}} &\lesssim \big\|\tau^{(2)}_r\big\|_{\cb_x^{-\frac12-\eps}}  \big\|\big(v_r+w_r\big)^2\big\|_{{\mathcal B}_x^{\frac12 +2\eps}}\\
& \lesssim \big\|\tau^{(2)}_r\big\|_{\cb_x^{-\frac12-\eps}}  \big\|v_r+w_r\big\|_{{\mathcal B}_x^{\frac12 +2\eps}}^2\lesssim \big\|Z\big\|_{\eps,T} \big\|(v,w)\big\|_{\eps,T}^2 .
\end{align*}
\

To achieve \eqref{boutilg}, it remains us to control the term $\big\| \mathrm{com}^Z (v+w)_r\big\|_{\cb_x^{-\frac12-2\eps}}$. To do so, we combine the technical results of Lemma \ref{lem:com1} and Proposition \ref{prop:commutor}, which yields for every $\eps>0$ small enough,
\begin{align*}
\big\| \mathrm{com}^Z(v,w)_r\big\|_{\cb_x^{-\frac12-2\eps}}& \lesssim  \big\|\mathrm{com}^Z_1(v,w)_r \pe \<Psi2>_r \big\| _{\cb_x^{-\frac12-2\eps}} +  \big\|\mathrm{com}^Z_2 (v+w)_r\big\| _{\cb_x^{-\frac12-2\eps}}\\
&\lesssim \big\|\mathrm{com}^Z_1(v,w)_r \big\| _{{\mathcal{B}}_x ^{1+2\eps}}  \big\| \<Psi2>_r\big\|_{\cb_x^{-1 -\eps}} + \big\|  \mathrm{com}_2^Z (v+w)_r \big\|_{\cb_x^\eps}\\
&\lesssim \Big[\|v_0\|_{{\mathcal B}_x^{\frac12+2\eps}}+ \big( 1+ \|Z\|_{{\eps,T}}^2 \big) \big( 1+ \big\|(v,w)\big\|_{\mathcal{X}_{\eps,T}}\big)\Big] \big\| \<Psi2>_s\big\|_{\cb_x^{-1 -\eps}}\\
&\hspace{1cm}+ \big\| v_r+w_r-\<IPsi3>_r\big\|_{\cb_x^{\frac12-\eps}} \big\| \<IPsi2>_r\big\|_{\cb_x^{1-2\eps}} \big\| \<Psi2>_r\big\|_{\cb_x^{-1-\eps}}\\
&\lesssim \Big[\|v_0\|_{{\mathcal B}_x^{\frac12+2\eps}}+ \big( 1+ \|Z\|_{{\eps,T}}^2 \big) \big( 1+ \big\|(v,w)\big\|_{\mathcal{X}_{\eps,T}}\big)\Big]\|Z\|_{{\eps,T}}\\
&\hspace{1cm}+\Big[ \big\| v_r\big\|_{\cb_x^{\frac12-\eps}}+\big\| w_r\big\|_{\cb_x^{\frac12-\eps}}+\big\| \<IPsi3>_r\big\|_{\cb_x^{\frac12-\eps}} \Big] \big\| \<Psi2>_r\big\|_{\cb_x^{-1-\eps}}^2\\
&\lesssim \big(1+\|v_0\|_{{\mathcal B}_x^{\frac12+2\eps}}+\big\|(v,w)\big\|_{\eps,T}\big) \big(1+\big\|Z\big\|_{\eps,T}^3\big) .
\end{align*}
This completes the proof of \eqref{boutilg}.

\smallskip

Injecting the above estimates into \eqref{bouggaw}, we obtain that for all $t\in [0,T]$ and $\eps>0$ small enough, 
\begin{align*}
\big\|\gga^{\mathbf{w}} [v,w]_t \big\| _{\cb_x^{{1+2\eps}} } 
&\lesssim \| w_0\|_{\cb_x^{1+2\eps}} + \big(1+T^{\frac{1}{2}-3\eps} \big\|(v,w)\big\|_{\eps,T}\big)\big(1+\big\|Z\big\|_{\eps,T}^2\big) \\
&\hspace{1cm}+\big(1+\|v_0\|_{{\mathcal B}_x^{\frac12+2\eps}}+\big\|(v,w)\big\|_{\eps,T}^3\big) \big(1+\big\|Z\big\|_{\eps,T}^3\big) \int_0^t \frac{ds}{(t-s)^{\frac{3}{4}+2\eps}}\\
&\lesssim \big(1+\|v_0\|_{{\mathcal B}_x^{\frac12+2\eps}}+\| w_0\|_{\cb_x^{1+2\eps}}+T^{\frac14-3\eps}\big\|(v,w)\big\|_{\eps,T}^3\big) \big(1+\big\|Z\big\|_{\eps,T}^3\big).
\end{align*}

\

As far as the H{\"o}lder norm of $\gga^{\mathbf{w}}[v,w]$ is concerned, let us first set, with decomposition \eqref{decompogpro} and interpretation \eqref{interpretationyoung} in mind,
\begin{align*}
\gga^{\mathbf{w},\ast}[v,w]_{s,t}&:=\int_s^t e^{-(t-r)H} \big((v_r+w_r) \, d_r\widetilde{\<PsiIPsi3>}\big)+\int_s^t e^{-(t-r)H} \big((v_r+w_r) \, d_r\widetilde{\<Psi2IPsi2>}\big) \\
&\hspace{1cm}+\int_s^t e^{-(t-r)H} \big(\<IPsi3>_r \, d_r\widetilde{\<PsiIPsi3>}\big)+ \int_s^t e^{-(t-r)H} \big(\<IPsi3>_r \,  d_r\widetilde{\<Psi2IPsi2>}\big) \nonumber   \\
&\hspace{1cm}+\int_s^t e^{-(t-r)H} \big( d_r\widetilde{\<Psi2IPsi3>}\big)+\int_s^t e^{-(t-r)H}\tilde{G}(v_r, w_r)\, dr.
\end{align*}
Using this notation, one has
\begin{align*}
& \big\| \gga^{\mathbf{w}}[v,w]_t-\gga^{\mathbf{w}}[v,w]_s \big\| _{\cb_x^\eps}\\ 
&\le  \big\| (e^{-(t-s)H} -\id) e^{-sH} w_0\big\|_{\cb_x^\eps}  +\big\| \gga^{\mathbf{w},\ast}[v,w]_{s,t}\big\| _{\cb_x^\eps}  +\big\| (e^{-(t-s)H} -\id) \gga^{\mathbf{w},\ast}[v,w]_{0,s}\big\| _{\cb_x^\eps}\\
&\lesssim  |t-s|^{\frac12}\big\| w_0\big\|_{\cb_x^{1+2\eps}}  +\big\| \gga^{\mathbf{w},\ast}[v,w]_{s,t}\big\| _{\cb_x^\eps}  +|t-s|^{\frac12}\big\| \gga^{\mathbf{w},\ast}[v,w]_{0,s}\big\| _{\cb_x^{1+2\eps}}.
\end{align*}

To estimate $\big\| \gga^{\mathbf{w},\ast}[v,w]_{0,s}\big\| _{\cb_x^{1+2\eps}}$, we can of course use the same arguments as for the treatment of \eqref{bouggaw}, which immediately gives
\begin{align*}
\sup_{s\in [0,T]}\big\| \gga^{\mathbf{w},\ast}[v,w]_{0,s}\big\| _{\cb_x^{1+2\eps}}
&\lesssim \big(1+\|v_0\|_{{\mathcal B}_x^{\frac12+2\eps}}+\| w_0\|_{\cb_x^{1+2\eps}}+T^{\frac14-3\eps}\big\|(v,w)\big\|_{\eps,T}^3\big) \big(1+\big\|Z\big\|_{\eps,T}^3\big).
\end{align*}

Now, for the control of $\big\| \gga^{\mathbf{w},\ast}[v,w]_{s,t}\big\| _{\cb_x^\eps}$, small adjustments are required. In fact, thanks to Proposition \ref{prop:young}, 
we obtain successively
\begin{align*}
\Big\| \int_s^t e^{-(t-r)H} \big((v_r+w_r) \, d_r\widetilde{\<PsiIPsi3>}\big)\Big\|_{\cb_x^{\eps}}
&\lesssim |t-s|^{1-\frac{9}{4}\eps} \Big(\|v\|_{\mathcal{C}_T^{\frac{1}{4} } \cb_x^{\eps}} + \|w\|_{\mathcal{C}_T^{ \frac12} \cb_x^{\eps}} \Big)
\|\widetilde{\<PsiIPsi3>}\|_{\mathcal{C}_T^{1-\frac{\eps}{2}} \cb_x^{-\frac{\eps}{2}}}\\
& \lesssim |t-s|^{\frac12} T^{ \frac14 -3\eps} \big\|(v,w)\big\|_{\eps,T} \big\|Z\big\|_{\eps,T},
\end{align*}
\begin{align*}
\Big\|\int_s^t e^{-(t-r)H} \big((v_r+w_r) \, d_r\widetilde{\<Psi2IPsi2>}\big) \Big\|_{\cb_x^{\eps}} 
&\lesssim |t-s|^{1-\frac{11}{4} \eps} \Big(\|v\|_{\mathcal{C}_T^{\frac{1}{4}} \cb_x^{\eps}} + \|w\|_{\mathcal{C}_T^{\frac12} \cb_x^{\eps}} \Big)
\|\widetilde{\<Psi2IPsi2>}\|_{C^{1-{\eps}} \cb_x^{-\frac{\eps}{2}}} \\
&\lesssim |t-s|^{\frac12} T^{\frac14-3\eps} \big\|(v,w)\big\|_{\eps,T} \big\|Z\big\|_{\eps,T},
\end{align*}
\begin{align*}
\Big\| \int_s^t e^{-(t-r)H} \big(\<IPsi3>_r \, d_r\widetilde{\<PsiIPsi3>}\big) \Big\|_{\cb_x^{\eps}} 
&\lesssim  |t-s|^{\frac12} \|\<IPsi3>\|_{\mathcal{C}_T^{\frac{1}{4} -\eps} \cb_x^{\eps}}
\|\widetilde{\<PsiIPsi3> }\|_{\mathcal{C}_T^{1-\frac{\eps}{2}} \cb_x^{-\frac{\eps}{2}}} \lesssim \|Z\|^2_{\eps,T},
\end{align*}
\begin{align*}
&
\Big\| \int_s^t e^{-(t-r)H} \big(\<IPsi3>_r \,  d_r\widetilde{\<Psi2IPsi2>}\big) \Big\|_{\cb_x^{\eps}} 
\lesssim |t-s|^{\frac12} \|\<IPsi3>\|_{\mathcal{C}_T^{\frac{1}{4} -\eps} \cb_x^{\eps}}
\|\widetilde{ \<Psi2IPsi2>}\|_{\mathcal{C}_T^{1-\frac{\eps}{2}} \cb_x^{-\frac{\eps}{2}}} \lesssim \|Z\|^2_{\eps,T},
\end{align*}
\begin{align*}
&
 \Big\| \int_s^t e^{-(t-r)H} \big( d_r\widetilde{\<Psi2IPsi3>}\big) \Big\|_{\cb_x^{\eps}} \lesssim 
 |t-s|^{\frac12}
 \|\widetilde{\<Psi2IPsi3>} \|_{\mathcal{C}_T^{\frac34-\eps} \mathcal{B}_x^{-\frac14-2\eps}} \lesssim 
 \big\|Z\big\|_{\eps,T}.
\end{align*}

The last term with $\tilde{G}$ can be estimated using \eqref{boutilg}: 
\begin{eqnarray*}
\Big\| \int_s^t e^{-(t-r)H}\tilde{G}(v_r, w_r)\, dr \Big \|_{\cb_x^{\eps}} 
&\le& 
\int_s^t (t-r)^{-\frac14 -\frac{3}{2} \eps}  \| \tilde{G}(v_r, w_r)\|_{\cb^{-\frac12 -2\eps}} dr \\
& \le & |t-s|^{\frac12} \big(1+\|v_0\|_{{\mathcal B}^{\frac12+2\eps}_x}+T^{\frac14-3\eps} \big\|(v,w)\big\|_{\eps,T}^3\big) \big(1+\big\|Z\big\|_{\eps,T}^3\big).
\end{eqnarray*}

Combining all these estimates, we obtain that for all $s<t \in [0,T]$ and $\eps>0$ small enough, 
\begin{multline*}
 \big\| \gga^{\mathbf{w}}[v,w]_t-\gga^{\mathbf{w}}[v,w]_s \big\| _{\cb_x^\eps} \lesssim \\ 
\lesssim  |t-s|^{\frac12}
\big(1+\|v_0\|_{{\mathcal B}_x^{\frac12+2\eps}}+\| w_0\|_{\cb_x^{1+2\eps}}+T^{\frac14-3\eps}\big\|(v,w)\big\|_{\eps,T}^3\big) \big(1+\big\|Z\big\|_{\eps,T}^3\big).
\end{multline*}

\subsubsection{A technical lemma}

\begin{lemma}\label{lem:com1}
Let $T>0$. For every $\varepsilon>0$ small enough, $Z\in  \mathcal Z_{\eps,T}$ and $(v,w) \in \mathcal{X}_{\eps,T}$, one has
\begin{align*}
\sup_{t\in [0,T]}\big\|\mathrm{com}^Z_1 (v,w)_t- e^{-tH}v_0 \big\|_{\mathcal{B}_x^{1+2\varepsilon}} & \lesssim T^{\frac14-3\eps} \big( 1+ \|Z\|_{\mathcal Z_{\eps,T}}^2 \big) \big( 1+ \big\|(v,w)\big\|_{\mathcal{X}_{\eps,T}}\big).
\end{align*}

\end{lemma}

\begin{proof} 

For convenience, let us set $f_t:=v_t+w_t-\<IPsi3>_t$ and $g_t:=\<Psi2>_t$. 

\smallskip

Using the notation $\big[ e^{-(t-s)H},\pl \big]$ introduced in Lemma \ref{lem1.2}, we can rewrite the definition of $\mathrm{com}^Z_1 (v,w)_t$ as
\begin{align}
\mathrm{com}^Z_1 (v,w)_t&=e^{-tH} v_0-3\bigg[ \int_0^t ds \, e^{-(t-s)H}\big(f_s\pl g_s\big)-\int_0^t ds \, f_t \pl \big(e^{-(t-s)H}g_s\big)\bigg]\nonumber\\
&=e^{-tH} v_0-3\bigg[\int_0^t ds \,\big[ e^{-(t-s)H},\pl \big]\big(f_s,g_s\big)-\int_0^t ds \, (f_t-f_s) \pl \big(e^{-(t-s)H}g_s\big)\bigg].\label{decomp-ci}
\end{align}

Then, using the estimate \eqref{com-1}, we obtain on the one hand that for $\eps>0$ small enough,
\begin{align*}
\Big\|\int_0^t ds \,\big[ e^{-(t-s)H},\pl \big]\big(f_s,g_s\big) \Big\|_{ \B_x^{1+2\eps}}&\lesssim \int_0^t \frac{ds}{|t-s|^{\frac34+3\eps}}\| f_s \|_{\B_x^{\frac12-\eps}}\| g_s \|_{\B_x^{-1-\eps}}\\
&\lesssim T^{\frac14-3\eps}\| f \|_{\mathcal{C}_T\B_x^{\frac12-\eps}}\| g \|_{\mathcal{C}_T\B_x^{-1-\eps}}.
\end{align*}
On the other hand, by Proposition \ref{Prop-est-para} and Lemma \ref{lem:actisemi},
\begin{align*}
\Big\|\int_0^t ds \, (f_t-f_s) \pl \big(e^{-(t-s)H}g_s\big)\Big\|_{ \B_x^{1+2\eps}}&\lesssim \int_0^t ds \, \big\|f_t-f_s\big\|_{L^{\infty}(\R^3)} \big\|e^{-(t-s)H}g_s\big\|_{ \B_x^{1+2\eps}}\\
&\lesssim \| f \|_{\mathcal{C}^{\frac14-\eps}([0,T];L^{\infty}(\R^3))}\| g \|_{\mathcal{C}_T \B_x^{-1-\eps}} \int_0^t \frac{ds}{|t-s|^{\frac34+3\eps}} \\
&\lesssim T^{\frac14-3\eps}\| f \|_{\mathcal{C}^{\frac14-\eps}_T\cb_x^\eps}\| g \|_{\mathcal{C}_T \B_x^{-1-\eps}} .
\end{align*}
Injecting the above bounds into \eqref{decomp-ci}, we deduce that 
$$\Big\|\mathrm{com}^Z_1 (v,w)_t-e^{-tH} v_0\Big\|_{ \B_x^{1+2\eps}}\lesssim T^{\frac14-3\eps} \Big\{\| f \|_{\mathcal{C}_T\B_x^{\frac12-\eps}}+\| f \|_{\mathcal{C}_T^{\frac14-\eps}\cb_x^\eps}\Big\}\big\| \<Psi2> \big\|_{\mathcal{C}_T\B_x^{-1-\eps}},$$
and it only remains us to observe that
$$\| f \|_{\mathcal{C}_T \B_x^{\frac12-\eps}}+\| f \|_{\mathcal{C}^{\frac14-\eps}_T\cb_x^\eps} \lesssim \big\|(v,w)\big\|_{\mathcal{X}_{\eps,T}}+\|Z\|_{\mathcal Z_{\eps,T}}.$$
\end{proof}

\subsubsection{Proof of Corollary \ref{cor:fixed-point}}\label{proo-2}
For every $0<T \le 1$ and $M>0$, we define 
$$B_{M,T} := \big\{ (v,w) \in {\mathcal{X}_{\eps,T}}, \; \;\|(v,w)\|_{\eps,T} \le M \big\}. $$ 
We show that $\gga$ is a contraction map from $B_{M,T}$ to $B_{M,T}$ for small $T>0$ and suitable $M>0$. Let $M\ge 1$. 
By Proposition \ref{prop:fixed-point}, for $(v,w) \in B_{M,T}$, we have   
\begin{align*}
&\big\| \gga[v,w]\big\|_{{\eps,T}}\leq P_1 \big(\big\| Z\big\|_{{\eps,1}} \big)\Big\{1+\big\|(v_0,w_0)
\big\|_{\mathcal{B}^{\frac12 +2\eps}_x \times \mathcal{B}^{1+2\eps}_x}+T^\nu (1+ M^{3}) \Big\}.
\end{align*} 
Moreover, for $(v',w') \in B_{M,T}$,   
\noindent
\begin{align*}
\big\| \gga[v,w]-\gga'[v',w']\big\|_{{\eps,T}}
\leq T^\nu P_2 \Big( \big\| Z\big\|_{{\eps,1}} \Big) 
(1+M^2)\big\| (v,w)-(v',w')\big\|_{{\eps,T}},
\end{align*}
for some $\nu>0$.
Choose $M_* = \max \big\{2P_1\big(\big\| Z\big\|_{{\eps,1}} \big) \Big(1+\big\|(v_0,w_0) \big\|_{\mathcal{B}^{\frac12 +2\eps}_x \times \mathcal{B}^{1+2\eps}_x} \Big), 1 \big\}$, and take 
$T_*$ so small as 
$$T_*^{\nu} = \min \Big\{ \frac{M_*}{2 P_1 (\big\| Z\big\|_{{\eps,1}} )( 1+ M_*^3) }, \frac{1}{2P_2 (\big\| Z\big\|_{{\eps,1}} ) (1+M_*^2) }, 1\Big\}, $$
 then $\gga$ is a contraction on $B_{M_*,T_*}$.
We see that $T_*$ is continuous with respect to $(v_0,w_0)$ and $Z$ from this explicit form. 
The solution is moreover unique in $\mathcal{X}_{\eps, T_*}$. Indeed, consider two solutions $(v, w), (v',w') \in \mathcal{X}_{\eps, T_*}$ with same 
initial datum $(u_0,w_0)$, and diagrams $Z$. By $(ii)$ of Proposition~\ref{prop:fixed-point}, for $0<T \le T_*$, 
\begin{eqnarray*}
\| (v, w) -(v', w')\|_{\eps,T} &=  
&\big\| \gga[v, w]-\gga[v', w']\big\|_{{\eps,T}}\\
&\leq& T^\nu P_2 \Big( \big\| Z\big\|_{{\eps,1}}\Big) 
(1+R^2) \big\| (v, w)-(v', w')\big\|_{{\eps,T}} ,
\end{eqnarray*}
with here $R:=R_{\eps, T_*}=\max \big(\|(v ,w)\|_{\eps,T_*}, \|(v', w')\|_{\eps,T_*}\big)$. 
We deduce that $(v, w)=(v', w')$ on $[0,T_{**}]$, where 
$T_{**}$ is such that $T_{**}^{\nu} P_2 (\|Z\|_{\eps,1}) (1+R^2) = \frac12,$ and 
by iterating the argument on $[T_*, 2T_*]$, $[2T_*, 3T_*]$, \dots and so on, we finally obtain 
$(v,w) =(v', w')$ on $[0,T_*]$.

\subsubsection{Proof of Theorem \ref{th:local}}\label{proo-3}
$(i)$ First of all, Corollary \ref{cor:fixed-point} ensures the existence of 
the unique solution $(v^{(n)}, w^{(n)})$ and $T^{(n)}_*$, and by continuity of $T_*$ with respect to $Z$ 
we have $T^{(n)}_* \to T_*$ a.s. From the similar argument as above,  for $t <T_*^{(n)}$, 
$$ \|(v^{(n)}, w^{(n)}) \|_{\eps,t} \le \big\| \gga[v^{(n)},w^{(n)}]\big\|_{{\eps,T_*^{(n)}}} 
\le M_*^{(n)}.
$$
Since $Z^{(n)} \to Z$, and $M_*$ is continuous in $Z$, we remark that  $ \|(v^{(n)}, w^{(n)}) \|_{\eps,t} \le 2M_*$ for large $n$.
Now we fix  $t< T^{(n)}_* \wedge T_*$ for each $n$ which is large enough.  We have,  
\begin{eqnarray*}
&&\big\| (v,w)-(v^{(n)},w^{(n)}) \|_{\eps,t}\\
&& \leq t^\nu P_2 \Big( \big\| Z\big\|_{{\eps,1}},\big\| Z^{(n)} \big\|_{{\eps,1}} \Big)  
\Big(1+ \big\| (v,w)\big\|^2_{{\eps,t}}+ \big\| (v^{(n)},w^{(n)})\big\|^2_{{\eps,t}} \Big)  \\
&& \hspace{15mm} \times \Big\{\big\| (v,w)-(v^{(n)},w^{(n)})\big\|_{{\eps,t}} +\big\| Z-Z^{(n)} \big\|_{{\eps,1}}\Big\}\\
&& \lesssim t^{\nu} P_2 \Big( \big\| Z\big\|_{{\eps,1}}\Big) 
\Big\{\big\| (v,w)-(v^{(n)},w^{(n)})\big\|_{{\eps,t}} +\big\| Z-Z^{(n)} \big\|_{\eps,1}\Big\}.
\end{eqnarray*}
Therefore, there exists $t_*(\nu, (v_0,w_0), Z)<t$ such that $\lim_{n\to \infty} \| (v,w) -(v^{(n)},w^{(n)})\|_{\eps,s}=0$ for each $s \le t_*$.  
Repeating this argument on $[0,t_*]$, $[t_*, 2t_*]$,\dots, we have the convergence on $[0,t].$ 

Finally we verify $(ii)$. Let $(v,w) \in \mathcal{X}_{\eps, T_*} $ be the unique solution. Define 
$$T_{\max} := \sup \big\{t >0; \ \mbox{there exists a unique solution of \eqref{mild:v}-\eqref{mild:w} such that}\  \| (v,w)\|_{\eps,t} <\infty\big\}.  $$ 
We know that $T_{\max} >0$ by Corollary \ref{cor:fixed-point}. 
We repeat the same arguments as in Corollary \ref{cor:fixed-point} on $[T_*,2T_*]$, $[2T_*, 3T_{*}]$, \dots until $[0,T_{\max})$. 
If $T_{\max} =\infty$, the solution exists globally.
Next consider the case $T_{\max} <\infty$, and in this case we assume 
$\lim_{t \uparrow T_{\max}} \max \big(\|v(t)\|_{\mathcal{B}_x^{\frac12+2\eps}}, \|w(t)\|_{\mathcal{B}_x^{1+2\eps}}\big)<\infty.$ 
Then there is a small $\delta>0$ such that $\| \big(v(T_{\max}-\delta),w(T_{\max}-\delta)\big)\|_{ \mathcal{B}^{\frac12+2\eps} \times \mathcal{B}^{1+2\eps} } <\infty$. 
If we then start to solve the problem \eqref{mild:v}-\eqref{mild:w} from 
$\big(v(T_{\max}-\delta),w(T_{\max}-\delta)\big)$,  there is the time $T_*>0$ that we have found 
in Corollary \ref{cor:fixed-point}, such that there exists a unique solution $(v, w)$ on $[0,T_*]$,  
namely there exists a unique solution on $[T_{\max}-\delta, T_{\max}-\delta+T_*]$.  
Choosing $\delta$ so small one finds $T_{\max}-\delta+T_*>T_{\max}$ which is a contradiction. 

Finally we let $t <T_{\max}.$ Then, there is a solution $(v, w) \in \mathcal{X}_{\eps,t}.$ On the other hand 
if $Z^{(n)} \to Z$ as $n \to \infty$, as we have seen as above we have the solution driven by $Z^{(n)}$,   
$(v^{(n)}, w^{(n)})$,  converges to $(v,w)$ on $[0,t]$. Therefore for large $n$, $t <T_{\max}^{(n)}$ since $(v^{(n)},w^{(n)})$ exists on $[0,t]$ too. 
Namely $T_{\max} \le \liminf_{n\to \infty} T_{\max}^{(n)}.$ 
 \hfill{\Huge $\square$}

\

The remainder of the article (aside from the appendix) is devoted to the proof of Proposition~\ref{prop:conv-arbre}, that is, to the study of the convergence of the six diagrams associated with the model. As is customary, this convergence will this time rely on stochastic arguments.


 \section{Preliminary stochastic material}\label{sec:prelim-mat}

\subsection{Wiener chaos and multiple integrals} 

\

\smallskip

Recall that the regularization $(\xi^{(n)})$ of the noise is given by the formula
\begin{equation*} 
 \xi^{(n)}_t:= \frac{dW^{(n)}_t}{dt}, \quad W^{(n)}_t(x):=\sum_{k \geq 0} e^{-\eps_n \la_k} \beta^{(k)}_t  \vp_k(x)
\end{equation*} 
with $\eps_n:=2^{-n}$ and where $(\beta^{(k)})_{k\geq 0}$ is a family of independent Brownian motions.

\smallskip

For any function $f : \R_+\times \R^3 \longrightarrow \R$, we can write
\begin{align}\label{link-b-bn}
\iint f(t,x)\, W^{(n)}(dt,dx)&=\sum_{k \geq 0} \int d\beta^{(k)}_t \bigg(\int dx \, f(t,x) e^{-\varepsilon_n \la_k}\vp_k(x)\bigg) \nonumber\\
&= \sum_{k \geq 0} \int d\beta^{(k)}_t \langle e^{-\varepsilon_n H} f(t,.) ,\vp_k\rangle=\iint \Big(e^{-\varepsilon_n H}f(t,.)\Big)(x) \, W(dt,dx),
\end{align}
where $W$ stands for the underlying  cylindrical Wiener  process with formal expansion
\begin{equation*}
W_t(x)=\sum_{k \geq 0}  \beta^{(k)}_t  \vp_k(x).
\end{equation*} 
In particular, the representation of $\<Psi>^{(n)}$ in \eqref{def-L} can be rephrased as 
\begin{equation}\label{represen-psi-0}
\<Psi>^{(n)}_t(x)=\iint F^{(n)}_{t,x}(s,w) \, W(ds,dw), \quad \text{with} \ \ F^{(n)}_{t,x}(s,w):=\1_{[0,t]}(s) K_{t-s+\varepsilon_n}(x,w).
\end{equation}

\

This observation naturally invites us to reformulate the diagrams \eqref{def-L}-\eqref{ord5} in terms of multiple Wiener integrals (with respect to $W$), which will later facilitates the developments and moment calculations associated with these objects.

\smallskip

In the sequel, we denote by $(I^W_n)_{n\geq 0}$ the multiple integrals driven by $W$, as defined in \cite[Section~1.1.2]{nualart-book} (see also \cite[Section~2.2.7]{nourdin-peccati}). Recall in particular the fundamental orthogonality property: for all $f\in L^2\big((\R_+\times \R^3)^p\big)$ and $g\in L^2\big((\R_+\times \R^3)^q\big)$,
\begin{equation}\label{ortho-rule}
\mathbb{E}\Big[I^W_p(f) I^W_q(g)\Big]=\1_{\{p=q\}} p!\, \big\langle \text{Sym}(f),\text{Sym}(g)\big\rangle_{L^2},
\end{equation}
with
$$\text{Sym}(f)\big((t_1,x_1),\ldots, (t_p,x_p)\big):=\frac{1}{p!}\sum_\si f\big((t_{\si(1)},x_{\si(1)}),\ldots, (t_{\si(p)},x_{\si(p)})\big)$$
where the sum runs over all the permutations $\si$ of $\{1,\ldots,p\}$.

\smallskip

With this notation, the relation \eqref{represen-psi-0} reduces to
\begin{equation}\label{represen-psi}
\<Psi>^{(n)}_t(x)=I^W_1\big(F^{(n)}_{t,x}\big) .
\end{equation}

Let us now recall that the product of multiple integrals follows a well-established rule (see \cite[Proposition 1.1.3]{nualart-book}):
\begin{lemma}\label{lem:prod-rule-mult-int}
Given two symmetric functions $f\in L^2\big((\R_+\times \R^3)^p\big)$ and $g\in L^2\big((\R_+\times \R^3)^q\big)$, it holds that
\begin{equation}\label{prod-rule-mult-int}
I^W_p(f) I^W_q(g)=\sum_{r=0}^{p\wedge q} r!\binom{p}{r}\binom{q}{r} I_{p+q-2r}^W\big( f \otimes_r g\big),
\end{equation}
where $\otimes_r$ refers to the $r$-th contraction, that is
\begin{multline*}
\big(f\otimes_r g\big)\big((t_1,x_1),\ldots, (t_{p+q-2r},x_{p+q-2r})\big):=\\
 \int_{(\R_+\times \R^3)^r} d\mathbf{z} \, f\big((t_1,x_1),\ldots, (t_{p-r},x_{p-r}),\mathbf{z}\big) g\big(\mathbf{z}, (t_{p-r+1},x_{p-r+1}),\ldots, (t_{p+q-2r},x_{p+q-2r})  \big).
\end{multline*}
In the case $r=0$, the term $f\otimes_0 g=f\otimes g$ is the standard tensor product.
\end{lemma}

Starting from \eqref{represen-psi} and taking both rules \eqref{ortho-rule} and \eqref{prod-rule-mult-int} into account, we obtain the expression
\begin{equation}\label{wick-trees}
\<Psi2>^{(n)}_{t}(x) =I^W_1\big(F^{(n)}_{t,x}\big)^2-\mathbb{E}\big[ I^W_1\big(F^{(n)}_{t,x}\big)^2\big]=I^W_1\big(F^{(n)}_{t,x}\big)^2-F^{(n)}_{t,x}\otimes_1 F^{(n)}_{t,x}=I_2^W\big(F^{(n)}_{t,x} \otimes F^{(n)}_{t,x} \big) ,
\end{equation}
as well as
\begin{equation}\label{wick-trees-1}
 \<Psi3>^{(n)}_{t}(x) =I^W_1\big(F^{(n)}_{t,x}\big)^3-3\, \big(F^{(n)}_{t,x}\otimes_1 F^{(n)}_{t,x}\big)I^W_1\big(F^{(n)}_{t,x}\big)=I_3^W\big(F^{(n)}_{t,x} \otimes F^{(n)}_{t,x} \otimes F^{(n)}_{t,x} \big).
\end{equation}
Combining the representations \eqref{wick-trees}-\eqref{wick-trees-1} with the isometry property \eqref{ortho-rule}, we retrieve that
\begin{equation}\label{tree-carre}
\mathbb{E}\Big[ \<Psi2>^{(n)}_{s_1}(y_1) \<Psi2>^{(n)}_{s_2}(y_2)\Big]=2 \, \langle F^{(n)}_{s_1,y_1},F^{(n)}_{s_2,y_2} \rangle^2=2\Big(\mathbb{E}\Big[ \<Psi>^{(n)}_{s_1}(y_1) \<Psi>^{(n)}_{s_2}(y_2)\Big]\Big)^2 
\end{equation}
and similarly
\begin{equation}\label{tree-cube}
\mathbb{E}\Big[\<Psi3>^{(n)}_{s_1}(y_1)\<Psi3>^{(n)}_{s_2}(y_2)\Big]= 6 \Big(\mathbb{E}\Big[ \<Psi>^{(n)}_{s_1}(y_1) \<Psi>^{(n)}_{s_2}(y_2)\Big]\Big)^3.
\end{equation}

We will have several further opportunities to use the multiplication formula \eqref{prod-rule-mult-int} in the proof of Proposition \ref{prop:conv-arbre}.

\subsection{Estimate on the covariance of the linear solution}
 
We now turn to the presentation of three technical estimates associated with the covariance function of $\<Psi>^{(n)}$, and which will be extensively used in the sequel. For the sake of readability, we introduce the notation
$$\cac^{(n)}_{t_1,t_2}(y_1,y_2):= \mathbb{E}\Big[ \<Psi>^{(n)}_{t_1}(y_1) \<Psi>^{(n)}_{t_2}(y_2)\Big].$$
To expand the latter quantity, we can combine \eqref{represen-psi} and \eqref{ortho-rule}, which yields for all $0\leq t_1<t_2$
\begin{align*}
\mathbb{E}\Big[ \<Psi>^{(n)}_{t_1}(y_1) \<Psi>^{(n)}_{t_2}(y_2)\Big]=\langle F^{(n)}_{t_1,y_1} ,F^{(n)}_{t_2,y_2} \rangle&=\int_0^{t_1} ds\int  dw \, K_{t_1-s+\varepsilon_n}(y_1,w)K_{t_2-s+\varepsilon_n}(y_2,w)\\
&=\int_0^{t_1} ds\, K_{t_1+t_2-2s+2\varepsilon_n}(y_1,y_2)
\end{align*}
and so, for all $t_1,t_2\geq 0$,
\begin{equation}\label{def-E}
\cac^{(n)}_{t_1,t_2}(y_1,y_2)=\frac12\int^{t_1+t_2+2\varepsilon_n}_{|t_2-t_1|+2\varepsilon_n}d\si\,  K_{\si}(y_1,y_2).
\end{equation}

\smallskip

\subsubsection{First estimate}

\begin{lemma}  For all $1 \leq p \leq \infty$ and $t_1,t_2\geq 0$, one has
\begin{equation}\label{Cp}
\sup_{n\geq 1}\, \big\| \cac^{(n)}_{t_1,t_2} \big\|_{L^p{(\R^6)}}  \lesssim    \frac{1}{|t_2 -t_1|^{\frac12}},
\end{equation}
where the proportional constant does not depend on $t_1,t_2$.
\end{lemma}

\begin{proof}
Assume for instance that $ 0\leq  t_1 \leq t_2 $. For all $y_1,y_2 \in \R^3$, we have 
\begin{multline} \label{cac1}
 \cac^{(n)}_{t_1,t_2}(y_1,y_2)=\frac12\int_{t_2-t_1+2\varepsilon_n}^{t_2+t_1+2\varepsilon_n}d\si\,  K_{\si}(y_1,y_2) \lesssim \int_{t_2-t_1+2\varepsilon_n}^{t_2+t_1+2\varepsilon_n} \frac{d\sigma}{\sinh(2\sigma)^{\frac32}}  \lesssim \\
 \lesssim \int_{t_2-t_1 }^{+\infty} \frac{d\sigma}{\sinh(2\sigma)^{\frac32}} \lesssim \int_{t_2-t_1 }^{+\infty} \frac{d\sigma}{\sigma^{\frac32}}  \lesssim    \frac{1}{|t_2 -t_1|^{\frac12}},
\end{multline}
which corresponds to the claim for $p=+\infty$. \medskip

For $p=1$, we have the expression
\begin{align}
  \int dy_1 dy_2 \, \cac^{(n)}_{t_1,t_2}(y_1,y_2) =\frac12\int_{t_2-t_1+2\varepsilon_n}^{t_1+t_2+2\varepsilon_n} d\sigma \int dy_1 dy_2 \, K_{\sigma}(y_1,y_2) .\label{prelex}
	\end{align}
Then, using the Mehler formula \eqref{mehler2}, we can check that
\begin{multline*}
\int dy_1 dy_2 \, K_{\sigma}(y_1,y_2)=\frac{c}{(\sinh 2\si)^{\frac32}}\int dy_1dy_2\,  \exp\left(- \frac{ \vert y_1-y_2\vert^2}{4\tanh \si}-\frac{\tanh \si}{4}\vert y_1+y_2\vert^2\right)\\
=\frac{c}{(\sinh 2\si)^{\frac32}}\bigg(\int dz_1\,  \exp\left(- \frac{ \vert z_1\vert^2}{4\tanh \si}\right)\bigg)\bigg(\int dz_2\,  \exp\left(-\frac{\tanh \si}{4}\vert z_2\vert^2\right)\bigg)\lesssim \frac{1}{(\sinh 2\si)^{\frac32}} 
\end{multline*}
and so, going back to \eqref{prelex}, we conclude as in \eqref{cac1}  that
\begin{align}
 \int dy_1 dy_2 \, \cac^{(n)}_{t_1,t_2}(y_1,y_2)   \lesssim \int_{t_2-t_1+2\varepsilon_n}^{t_2+t_1+2\varepsilon_n}  \frac{d\sigma}{\sinh(2\sigma)^{\frac32}}    \lesssim    \frac{1}{|t_2 -t_1|^{\frac12}}.  \label{cacl1}
\end{align}

\smallskip

The general case $1 \leq p \leq \infty$ immediately follows from the combination of \eqref{cac1} and \eqref{cacl1}.
\end{proof}

\smallskip

\subsubsection{Second estimate}

For a clear statement of our second and third estimates on $\cac^{(n)}$, we introduce the space-translation operator $\tau=\tau_{y_1, y_2}$ defined for all function $F:\R^3\times \R^3 \to \R$ by 
\begin{equation}\label{def-tau}
\tau F(y_1,y_2)=F(y_1+y_2,y_2).
\end{equation}

 \begin{lemma}\label{lem-tokyo}
For every $0<\beta<\frac14$, there exists $\eta>0$ such that for all $s, t \geq 0$,
 \begin{equation} \label{eni-1}
  \sup_{n\geq 1}    \int  dy_1 \sup_{y_2\in \R^3}   \Big| \tau H_{y_1}^{\beta}H_{y_2}^{\beta}   \Big(\mathcal{C}^{(n)}_{t,t}(y_1,y_2)^2-\mathcal{C}^{(n)}_{t,s}(y_1,y_2)^2\Big)\Big|  \lesssim |t-s|^\eta,
 \end{equation}
where the proportional constant does not depend on $s,t$. Moreover, uniformly in $s,t \geq 0$, it holds that
 \begin{equation} \label{eni-2}
  \sup_{n\geq 1}    \int  dy_1 \sup_{y_2\in \R^3}   \Big| \tau H_{y_1}^{\beta}H_{y_2}^{\beta}   \Big(\mathcal{C}^{(n)}_{t,t}(y_1,y_2)^2-\mathcal{C}^{(n)}_{t,s}(y_1,y_2)^2\Big)\Big|  \lesssim  1.
 \end{equation}
 \end{lemma}

 \begin{proof}
Fix $0<\beta<\frac14$ and let us first prove \eqref{eni-1}.  Assume for instance that $ 0\leq s \leq t$. Using~\eqref{def-E} we can write
$$ H_{y_1}^{\beta}H_{y_2}^{\beta}   \Big(\mathcal{C}^{(n)}_{t,t}(y_1,y_2)^2-\mathcal{C}^{(n)}_{t,s}(y_1,y_2)^2\Big) =\frac14\int_{\cd_{s,t}^{(n)}} d\sigma_1 d\sigma_2\,  H_{x}^{\beta}H_{y}^{\beta}   \big(K_{\sigma_1}(y_1,y_2)    K_{\sigma_2 }(y_1,y_2) \big),$$
where we have set
$$\cd_{s,t}^{(n)} :=[\varepsilon_n,2t+\varepsilon_n]^2\, \backslash\,  [t-s+\varepsilon_n,t+s+\varepsilon_n]^2.$$
Then
\begin{multline*}
     \int   dy_1 \sup_{y_2\in \R^3}   \Big| \tau H_{y_1}^{\beta}H_{y_2}^{\beta}\Big(\mathcal{C}^{(n)}_{t,t}(y_1,y_2)^2-\mathcal{C}^{(n)}_{t,s}(y_1,y_2)^2\Big)\Big|      
\lesssim \\
\lesssim \int_{\cd_{s,t}^{(n)}}   d\sigma_1 d\sigma_2\,    \int    dy_1 \sup_{y_2\in \R^3}   \Big| \tau H_{y_1}^{\beta}H_{y_2}^{\beta}   \big(K_{\sigma_1}(y_1,y_2)K_{\sigma_2}(y_1,y_2)\big)\Big|  .
 \end{multline*}

\

It is readily checked that 
$$\tau H_{y_1} \tau^{-1}= -\Delta_{y_1}+|y_1+y_2|^2=: L_1$$ 
and 
$$\tau H_{y_2} \tau^{-1}=  -\Delta_{y_1}-\Delta_{y_2}+2\, \langle \nabla_{y_1},\nabla_{y_2}\rangle +|y_2|^2  =:L_2,$$
where we define the operator $\langle \nabla_{y_1},\nabla_{y_2}\rangle$ by the formula $\langle \nabla_{y_1},\nabla_{y_2}\rangle F:=\sum_{i=1}^3 \partial_{y_1^{(i)}}\partial_{y_2^{(i)}}F$. 

\smallskip

Observe that $[H_{y_1}, H_{y_2}]=0$, which immediately entails $[L_1, L_2]=0$. We set $A= L_1 L_2$, that is $A=  \tau H_{y_1}H_{y_2} \tau^{-1}$ and so $A^\beta=  \tau H_{y_1}^{\beta}H_{y_2}^{\beta} \tau^{-1}$. Using this notation, we need to show the existence of a constant $\eta>0$ such that for all $0 \leq s \leq t$
\begin{equation} 
\sup_{n\geq 1}\int_{\cd_{s,t}^{(n)}}  d\sigma_1 d\sigma_2\,  \int  dy_1 \sup_{y_2\in \R^3}   \Big| A^{\beta} \tau \big(K_{\sigma_1}(y_1,y_2)K_{\sigma_2}(y_1,y_2)\big)\Big|   \lesssim \min \big( |t-s|^\eta, 1\big), \label{a-montr}
\end{equation}
for some proportional constant independent of $s,t$.
\medskip

We observe that $A$ and its adjoint $A^*$ are dissipative, then $A$ generates a contraction semigroup. In this context, we can apply \cite[Theorem 8.1.]{Komatsu}. Namely, let $F \in \mathscr{S}(\R^6)$ then  for all $0< \beta< 1$ there exists $C_{\beta}>0$
\begin{equation}\label{interp}
\|A^\beta F\|_{L_{y_1}^1L_{y_2}^\infty} \leq C_\beta\|F\|^{1-\beta}_{L_{y_1}^1L_{y_2}^\infty} \|A F\|^{\beta}_{L_{y_1}^1L_{y_2}^\infty}.
\end{equation}
 Now we apply inequality \eqref{interp} to 
 \begin{eqnarray*} 
 F_{\sigma_1, \sigma_2}(y_1,y_2): &=&\tau  K_{\sigma_1}(y_1,y_2)K_{\sigma_2}(y_1,y_2)\\
 &=&    d(\sigma_1,\sigma_2) \exp\left(- a(\sigma_1,\sigma_2)\vert y_1\vert^2-b(\sigma_1,\sigma_2) \vert y_1+2y_2\vert^2\right),   
  \end{eqnarray*} 
where we have set $a(\sigma_1,\sigma_2)= \frac{1}{4\tanh \sigma_1}+\frac{1}{4\tanh \sigma_2}$,  $b(\sigma_1,\sigma_2)= \frac{\tanh \sigma_1}{4}+\frac{\tanh \sigma_2}{4}$ and 
$$d(\sigma_1,\sigma_2)= \big(4\pi^2(\sinh 2\sigma_1)(\sinh 2\sigma_2)\big)^{-\frac32}.$$ 

  In  order to prove the estimate \eqref{a-montr}, we will distinguish different cases with respect to  $\sigma_1$, $\sigma_2>0$. By symmetry, we can assume $\sigma_1<\sigma_2$. \medskip

$\bullet$ Case $0<\sigma_1\leq \sigma_2\leq 1$. In this case,  we will first prove  that for every $0<\beta<\frac14$
\begin{equation}\label{boundF-1}
  \int dy_1   \sup_{y_2\in \R^3}   \Big| A^{\beta}F_{\sigma_1, \sigma_2}(y_1,y_2)\Big| \lesssim  \frac{1}{(\sigma_1+\sigma_2)^{\frac32-2\beta}(\sigma_1 \sigma_2)^{2\beta}}+   \frac{1}{(\sigma_1+\sigma_2)^{\frac32+2\beta}} . 
\end{equation}
To begin with,  observe that for $0<\sigma_1<\sigma_2\leq 1$, we have for some $c>1$
  $$ \frac1c(\frac{1}{\sigma_1}+ \frac{1}{\sigma_2}) \leq    a(\sigma_1,\sigma_2) \leq c(\frac{1}{\sigma_1}+ \frac{1}{\sigma_2}) , \quad \frac1c(\sigma_1 +\sigma_2) \leq  b(\sigma_1,\sigma_2) \leq c(\sigma_1 +\sigma_2), $$
and $\dis \frac1{c(\sigma_1 \sigma_2)^{\frac32}} \leq d(\sigma_1,\sigma_2) \leq   \frac{c}{(\sigma_1 \sigma_2)^{\frac32}}$, then we get 
$$|F_{\sigma_1, \sigma_2}(y_1,y_2)| \leq  \frac{C}{(\sigma_1 \sigma_2)^{\frac32}}  e^{- a(\sigma_1,\sigma_2)\vert y_1\vert^2}, $$
therefore
\begin{equation} \label{b1}
\|F_{\sigma_1, \sigma_2}\|_{L_{y_1}^1L_{y_2}^\infty} \leq  C \big(\sigma_1 \sigma_2 a(\sigma_1,\sigma_2)\big)^{-\frac32} \leq C\big(\sigma_1+\sigma_2\big)^{-\frac32}.
\end{equation} 

\

Now we look for a bound on $\|AF_{\sigma_1, \sigma_2}\|_{L_{y_1}^1L_{y_2}^\infty}   $.
We can check that 
$$L_1 F_{\sigma_1, \sigma_2}(y_1,y_2)=d\Big(2(a+b)-4\big|(a+b)y_1+2by_2\big|^2+|y_1+y_2|^2\Big) e^{- a \vert y_1\vert^2-b  \vert y_1+2y_2\vert^2}.$$
Then, using the fact that $b \leq C a $, $b \leq c$ and $a\geq c$, we get through elementary computations that  
$$ \big | L_2 L_1 F_{\sigma_1, \sigma_2}(y_1,y_2)\big | \leq C d \big(a^2+a^4 |y_1|^4+ |y_2|^4 \big) e^{- a \vert y_1\vert^2-b  \vert y_1+2y_2\vert^2}.$$
We can bound
$$a^2 |y_1|^4  e^{- \frac{a}2 \vert y_1\vert^2-b  \vert y_1+2y_2\vert^2} \leq (a |y_1|^2)^2  e^{- \frac{a}2 \vert y_1\vert^2} \leq C,$$
and similarly
$$b^2 |y_1+2y_2|^4  e^{- \frac{a}2 \vert y_1\vert^2-b  \vert y_1+2y_2\vert^2} \leq     ( b|y_1+2y_2|^2)^2  e^{- b  \vert y_1+2y_2\vert^2} \leq C  .$$
Observe that $|y_2|^2 \leq C(|y_1|^2+|y_1+2y_2|^2)$, and so from the previous lines we infer that
$$b^2 |y_2|^4  e^{- \frac{a}2 \vert y_1\vert^2-b  \vert y_1+2y_2\vert^2} \leq  a^2 |y_1|^4  e^{- \frac{a}2 \vert y_1\vert^2}+b^2 |y_1+2y_2|^4  e^{-b  \vert y_1+2y_2\vert^2}  \leq C.$$
Putting the previous bounds together, we get 
\begin{equation}\label{AFb}
 \big | A F_{\sigma_1, \sigma_2}(y_1,y_2)\big | \leq C d(a^2+b^{-2} )  e^{- \frac{a}2 \vert y_1\vert^2},
 \end{equation}
and thus
$$\|AF_{\sigma_1, \sigma_2}\|_{L_{y_1}^1L_{y_2}^\infty}  \leq  C d(a^2+b^{-2} ) a^{-\frac32} \leq C \Big( \frac{(\sigma_1+\sigma_2)^{\frac12}}{(\sigma_1 \sigma_2)^2}+\frac1{(\sigma_1+\sigma_2)^{\frac72}}\Big).$$

As a consequence, by \eqref{interp}, we deduce that for all $0 \leq \beta <\frac14$
\begin{eqnarray*} 
  \int dy_1   \sup_{y_2\in \R^3}   \Big| A^{\beta}F_{\sigma_1, \sigma_2}(y_1,y_2)\big)\Big|  &\lesssim& \frac{1}{(\sigma_1+\sigma_2)^{\frac32(1-\beta)}  }\Big( \frac{(\sigma_1+\sigma_2)^{\frac{\beta}{2}}}{(\sigma_1\sigma_2)^{2\beta}}+\frac1{(\sigma_1+\sigma_2)^{\frac72 \beta}}\Big) \\
  &\lesssim& \frac{1}{(\sigma_1+\sigma_2)^{\frac32-2\beta}(\sigma_1 \sigma_2)^{2\beta}}+   \frac{1}{(\sigma_1+\sigma_2)^{\frac32+2\beta}}  ,
\end{eqnarray*}
which is \eqref{boundF-1}.
\medskip

Once endowed with this bound, we easily deduce that
\begin{align*}
&\int_{\substack{      \cd_{s,t}^{(n)}  \\  0<\sigma_1\leq \sigma_2\leq 1   }} d\sigma_1 d\sigma_2\,   \int  dy_1 \sup_{y_2\in \R^3}   \Big| A^{\beta}F_{\sigma_1, \sigma_2}(y_1,y_2)\Big| \\
&\lesssim \int_{\varepsilon_n}^{t-s+\varepsilon_n}\, d\si_1  { \bf 1 }_{\{0 \leq \sigma_1 \leq 1\}} \int_0^1 d\si_2\,  \bigg[\frac{1}{(\sigma_1+\sigma_2)^{\frac32-2\beta}(\sigma_1 \sigma_2)^{2\beta}}+   \frac{1}{(\sigma_1+\sigma_2)^{\frac32+2\beta}} \bigg]\\
&\hspace{3cm}+\int_{t+s+\varepsilon_n}^{2t+\varepsilon_n} d\si_2 { \bf 1 }_{\{0 \leq \sigma_2 \leq 1\}}  \int_0^1 d\si_1\,  \bigg[\frac{1}{(\sigma_1+\sigma_2)^{\frac32-2\beta}(\sigma_1 \sigma_2)^{2\beta}}+   \frac{1}{(\sigma_1+\sigma_2)^{\frac32+2\beta}} \bigg].
\end{align*}
Given $0<\beta<\frac14$, pick $\varepsilon>0$ such that $2\beta+\frac12+\varepsilon<1$. Then, on the one hand
\begin{align*}
&\int_{\varepsilon_n}^{t-s+\varepsilon_n} d\si_1 { \bf 1 }_{\{0 \leq \sigma_2 \leq 1\}}  \int_0^1 d\si_2\,  \bigg[\frac{1}{(\sigma_1+\sigma_2)^{\frac32-2\beta}(\sigma_1 \sigma_2)^{2\beta}}+   \frac{1}{(\sigma_1+\sigma_2)^{\frac32+2\beta}} \bigg]\\
&\lesssim \int_{\varepsilon_n}^{t-s+\varepsilon_n} \frac{d\si_1}{\si_1^{2\beta+\frac12+\varepsilon}}{ \bf 1 }_{\{0 \leq \sigma_1 \leq 1\}} \int_0^1 \frac{d\si_2}{\si_2^{1-\varepsilon}}\lesssim \int_{0}^{t-s} \frac{d\si_1}{\si_1^{2\beta+\frac12+\varepsilon}}{ \bf 1 }_{\{0 \leq \sigma_1 \leq 1\}}  \lesssim \min \big( |t-s|^{1-(2\beta+\frac12+\varepsilon)},1\big),
\end{align*}
uniformly over $n\geq 1$. Similarly,
\begin{multline*}
\int_{t+s+\varepsilon_n}^{2t+\varepsilon_n} d\si_2{ \bf 1 }_{\{0 \leq \sigma_2 \leq 1\}}\int_0^1 d\si_1\,  \bigg[\frac{1}{(\sigma_1+\sigma_2)^{\frac32-2\beta}(\sigma_1 \sigma_2)^{2\beta}}+   \frac{1}{(\sigma_1+\sigma_2)^{\frac32+2\beta}} \bigg]\\
\lesssim \int_{t+s+\varepsilon_n}^{2t+\varepsilon_n} \frac{d\si_2}{\si_2^{2\beta+\frac12+\varepsilon}} { \bf 1 }_{\{0 \leq \sigma_2 \leq 1\}} \lesssim \min \big( |t-s|^{1-(2\beta+\frac12+\varepsilon)},1\big),
\end{multline*}
uniformly over $n\geq 1$. We have thus shown that
\begin{equation}\label{estim1}
\sup_{n\geq 1}\int_{\substack{      \cd_{s,t}^{(n)}  \\  0<\sigma_1\leq \sigma_2\leq 1   }} d\sigma_1 d\sigma_2\,  \int  dy_1 \sup_{y_2\in \R^3}   \Big| A^{\beta} \tau \big(K_{\sigma_1}(y_1,y_2)K_{\sigma_2}(y_1,y_2)\big)\Big| \lesssim  \min \big( |t-s|^{1-(2\beta+\frac12+\varepsilon)},1\big).
\end{equation}

\

$\bullet$ Case $1\leq \sigma_1\leq \sigma_2$. Here,  we will show  that for every $0<\beta<\frac14$
\begin{equation}\label{boundF-2}
  \int dy_1   \sup_{y_2\in \R^3}   \Big| A^{\beta}F_{\sigma_1, \sigma_2}(y_1,y_2)\Big| \lesssim  e^{-3 (\sigma_1+\sigma_2)} . 
\end{equation}
In the regime $1\leq \sigma_1\leq \sigma_2$, we have for some $c>1$
  $$ \frac1c  \leq    a(\sigma_1,\sigma_2) \leq c , \qquad \frac1c \leq  b(\sigma_1,\sigma_2) \leq c , \qquad \frac1{c} e^{-3 (\sigma_1+\sigma_2)} \leq d(\sigma_1,\sigma_2) \leq  c e^{-3 (\sigma_1+\sigma_2)} ,$$
thus 
$$|F_{\sigma_1, \sigma_2}(y_1,y_2)| \lesssim  e^{-3 (\sigma_1+\sigma_2)}  e^{- \frac{\vert y_1\vert^2}{c}}. $$
Similarly, from \eqref{AFb} we get 
\begin{equation*} 
 \big | A F_{\sigma_1, \sigma_2}(y_1,y_2)\big | \lesssim  e^{-3 (\sigma_1+\sigma_2)}  e^{- \frac{ \vert y_1\vert^2}{2c}},
 \end{equation*}
which by \eqref{interp} implies \eqref{boundF-2}. \medskip

Now, from \eqref{boundF-2} we get 
\begin{multline}
\int_{\substack{      \cd_{s,t}^{(n)}  \\  1\leq \sigma_1\leq \sigma_2   }} d\sigma_1 d\sigma_2\,   \int  dy_1 \sup_{y_2\in \R^3}   \Big| A^{\beta}F_{\sigma_1, \sigma_2}(y_1,y_2)\Big| \lesssim  \\
\begin{aligned}
&\lesssim \int_{\varepsilon_n}^{t-s+\varepsilon_n} d\si_1\int_0^{+\infty} d\si_2\, e^{-3 (\sigma_1+\sigma_2)}+\int_{t+s+\varepsilon_n}^{2t+\varepsilon_n} d\si_2\int_0^{+\infty} d\si_1\,  e^{-3 (\sigma_1+\sigma_2)}\\
&\lesssim  \min \big( |t-s|, 1\big) .\label{estim2}
\end{aligned}
\end{multline}

\

$\bullet$ Case $\sigma_1 \leq 1\leq \sigma_2$. Here,  we will show that for all $0<\beta<\frac14$
\begin{equation}\label{boundF-3}
  \int dy_1   \sup_{y_2\in \R^3}   \Big| A^{\beta}F_{\sigma_1, \sigma_2}(y_1,y_2)\Big| \lesssim     \frac{1}{\sigma_1 ^{2\beta}}    e^{-3  \sigma_2} . 
\end{equation}
Under the condition $\sigma_1 \leq 1\leq \sigma_2$, we have for some $c>1$
  $$ \frac{1}{c\sigma_1}  \leq    a(\sigma_1,\sigma_2) \leq  \frac{c}{\sigma_1}  , \quad \frac1c  \leq  b(\sigma_1,\sigma_2) \leq c, \quad  \frac{1}{c\sigma_1 ^{\frac32}}   e^{-3  \sigma_2} \leq d(\sigma_1,\sigma_2) \leq   \frac{c}{\sigma_1 ^{\frac32}}   e^{-3  \sigma_2} ,$$
  and then we get 
$$|F_{\sigma_1, \sigma_2}(y_1,y_2)| \lesssim   \frac{1}{\sigma_1 ^{\frac32}}   e^{-3  \sigma_2} e^{- \frac{\vert y_1\vert^2}{c \sigma_1}}. $$
Next, from \eqref{AFb} we obtain 
\begin{equation*} 
 \big | A F_{\sigma_1, \sigma_2}(y_1,y_2)\big | \lesssim   \frac{1}{\sigma_1 ^{\frac72}}   e^{-3  \sigma_2}  e^{- \frac{\vert y_1\vert^2}{c \sigma_1}},
 \end{equation*}
which by \eqref{interp} implies \eqref{boundF-3}. \medskip

Now, from \eqref{boundF-3} we get 
\begin{multline}\label{estim3}
\int_{\substack{      \cd_{s,t}^{(n)}  \\  \sigma_1\leq 1 \leq \sigma_2   }} d\sigma_1 d\sigma_2\,   \int  dy_1 \sup_{y_2\in \R^3}   \Big| A^{\beta}F_{\sigma_1, \sigma_2}(y_1,y_2)\Big| \lesssim  \\
\begin{aligned}
&\lesssim \int_{\varepsilon_n}^{t-s+\varepsilon_n} \frac{d\si_1}{\sigma_1 ^{2\beta}} { \bf 1 }_{\{0 \leq \sigma_1 \leq 1\}}  \int_0^{+\infty} d\si_2\,    e^{-3  \sigma_2} +\int_{0}^{1} \frac{d\si_1}{\sigma_1 ^{2\beta}} \int_{t+s+\varepsilon_n}^{2t+\varepsilon_n} d\si_2\,       e^{-3  \sigma_2} \\
&\lesssim \min\big(  |t-s|^{1-2\beta} ,1 \big).
\end{aligned}
\end{multline}

\

Putting the estimates \eqref{estim1}, \eqref{estim2} and \eqref{estim3} together, we get  \eqref{a-montr}. 
\end{proof}

\smallskip

\subsubsection{Third estimate}

Recall that the translation operator $\tau$ has been introduced in \eqref{def-tau}.

\begin{lemma}\label{lem:cac3}
For all $0<\eps<1$ and  $t_1,t_2\geq 0$, it holds that
\begin{equation}\label{borne3}
\sup_{n\geq 1}\, \int dy_1 \Big( \sup_{y_2 \in \R^3}\big|\tau\cac^{(n)}_{t_1,t_2}(y_1,y_2)\big|^3\Big) \lesssim \frac{1}{|t_2-t_1|^{\varepsilon}},
\end{equation}              
where the proportional constant does not depend on $t_1,t_2$.
\end{lemma}

\begin{proof}
We assume for instance that $ 0\leq  t_1 \leq t_2 $. 

\smallskip

Observe that
$$K_\sigma(y_1+y_2,y_2)=\frac{c}{\sinh(2\sigma)^{\frac32}} \exp\bigg(-\frac{1}{4\tanh(\sigma)}|y_1|^2-\frac{\tanh(\sigma)}{4}|y_1+2y_2|^2\bigg).$$
Thus, 
\begin{equation}\label{decomp}
 \cac^{(n)}_{t_1,t_2}(y_1+y_2,y_2) \lesssim \int_{t_2-t_1}^{+\infty}\, \frac{d\si}{\sinh(2\si)^{\frac32}} \exp\bigg(-\frac{1}{4\tanh(\si)}|y_1|^2\bigg) ,
\end{equation}
$\bullet$ Assume that $t_2-t_1 \leq 1$. Then 
\begin{equation*}
 \cac^{(n)}_{t_1,t_2}(y_1+y_2,y_2) \lesssim \int_{t_2-t_1 }^{1}\, \frac{d\si}{\si^{\frac32}} \exp\bigg(-\frac{1}{4\tanh(\si)}|y_1|^2\bigg)+ \int_{1}^{+\infty}\, \frac{d\si}{\sinh(2\si)^{\frac32}} \exp\bigg(-\frac{1}{4\tanh(\si)}|y_1|^2\bigg) ,
\end{equation*}
and so
\begin{equation*}
 \cac^{(n)}_{t_1,t_2}(y_1+y_2,y_2)^3\lesssim \mathcal{L}^{\mathbf{1}}_{t_1,t_2}(y_1,y_2)+ \mathcal{L}^{\mathbf{2}}_{t_1,t_2}(y_1,y_2)
\end{equation*}
where
\begin{equation*}
 \mathcal{L}^{\mathbf{1}}_{t_1,t_2}(y_1,y_2):= \int_{[t_2-t_1 ,1]^3}\frac{d\si_1 d\si_2 d\si_3}{(\si_1 \si_2\si_3)^{\frac32}} \exp\bigg(-\Big(\frac{1}{4\tanh(\si_1)}+\frac{1}{4\tanh(\si_2)}+\frac{1}{4\tanh(\si_3)}\Big)|y_1|^2\bigg)
\end{equation*}
and
\begin{multline*}
 \mathcal{L}^{\mathbf{2}}_{t_1,t_2}(y_1,y_2):=\\
 =\int_{[ 1, +\infty)^3}\frac{d\si_1 d\si_2 d\si_3}{(\sinh(2\si_1) \sinh(2\si_2)\sinh(2\si_3))^{\frac32}} \exp\bigg(-\Big(\frac{1}{4\tanh(\si_1)}+\frac{1}{4\tanh(\si_2)}+\frac{1}{4\tanh(\si_3)}\Big)|y_1|^2\bigg).
\end{multline*}
Let us now study the contribution of each of the two terms. On the one hand, we have
\begin{eqnarray*}
\int dy_1 \,  \mathcal{L}^{\mathbf{1}}_{t_1,t_2}(y_1,y_2)&\lesssim &
 \int_{[t_2-t_1,1]^3}\frac{d\si_1 d\si_2 d\si_3}{(\si_1 \si_2\si_3)^{\frac32}}\bigg(\frac{1}{\si_1}+\frac{1}{\si_2}+\frac{1}{\si_3}\bigg)^{-\frac32}\\
&\lesssim&  \int_{[t_2-t_1 ,1]^3}\frac{d\si_1 d\si_2 d\si_3}{(\si_2 \si_3+\si_1\si_3+\si_1\si_2)^{\frac32}}.
\end{eqnarray*}
For all $1\leq i,j \leq 3$ and any $\varepsilon >0$,  we use the bound $\dis (t_2-t_1)^{\frac{2\varepsilon}{3}} (\si_i \sigma_j)^{1-\frac{\varepsilon}{3}} \leq  \si_i \sigma_j $, then
\begin{eqnarray*}
\int dy_1 \,  \mathcal{L}^{\mathbf{1}}_{t_1,t_2}(y_1,y_2)&\lesssim &
  \frac{1}{(t_2-t_1)^{\varepsilon}} \int_{0<\si_1<\si_2<\si_3<1}  \frac{d\si_1d\si_2 d\si_3}{\big(( \si_2 \si_3)^{1-\frac{\varepsilon}{3}}+(\si_1 \si_3)^{1-\frac{\varepsilon}{3}}+(\si_1 \si_2)^{1-\frac{\varepsilon}{3}} \big)^{\frac32}}\\
&\lesssim& \frac{1}{(t_2-t_1)^{\varepsilon}} \int_{0<\si_1<\si_2<\si_3<1}  \frac{d\si_1d\si_2 d\si_3}{\big(\si_2 \si_3+\si_1 \si_3+\si_1 \si_2 \big)^{\frac32-\frac{\varepsilon}{2}}} .
\end{eqnarray*}
  Using the spherical coordinates
$$\si_1=\rho\cos \theta, \quad \si_2=\rho\sin \theta \cos \varphi, \quad \si_3=\rho\sin \theta \sin \varphi,$$
we can then assert that
\begin{multline*}
\int_{0<\si_1<\si_2<\si_3<1}  \frac{d\si_1d\si_2 d\si_3}{\big(\si_2 \si_3+\si_1 \si_3+\si_1 \si_2 \big)^{\frac32-\frac{\varepsilon}{2}}} \lesssim \\
\begin{aligned}
&\lesssim \int_0^{2}\frac{d\rho}{\rho^{1-\varepsilon}}\int_{\frac{\pi}{4}}^{\frac{\pi}{2}}\frac{d\theta}{\big(\sin\theta\big)^{\frac12-\varepsilon}}\int_{\frac{\pi}{4}}^{\frac{\pi}{2}}\frac{d\varphi}{\big(\sin\theta \cos \varphi  \sin\varphi+ \cos\theta \sin \varphi + \cos\theta\cos \vp \big)^{\frac32-\frac{\varepsilon}{2}}}\nonumber \\
&\lesssim \int_0^{2}\frac{d\rho}{\rho^{1-\varepsilon}}\int_{\frac{\pi}{4}}^{\frac{\pi}{2}}d\theta\int_{\frac{\pi}{4}}^{\frac{\pi}{2}}\frac{d\varphi}{\big(\sin\theta \cos \varphi  + \cos\theta \big)^{\frac32-\frac{\varepsilon}{2}}} \\
&\lesssim \int_0^{2}\frac{d\rho}{\rho^{1-\varepsilon}}\int_{\frac{\pi}{4}}^{\frac{\pi}{2}}\frac{d\theta}{\big( \cos\theta \big)^{\frac34-\frac{\varepsilon}{4}}}\int_{\frac{\pi}{4}}^{\frac{\pi}{2}}\frac{d\varphi}{\big( \cos \varphi  \big)^{\frac34-\frac{\varepsilon}{4}}}\ < \ \infty.
\end{aligned}
\end{multline*}
As a consequence we have shown that $\dis \int dy_1 \,  \mathcal{L}^{\mathbf{1}}_{t_1,t_2}(y_1,y_2)\lesssim \frac{1}{(t_2-t_1)^{\varepsilon}} $.

On the other hand, we clearly have $\dis \int dy_1 \,  \mathcal{L}^{\mathbf{2}}_{t_1,t_2}(y_1,y_2)\lesssim 1 \lesssim \frac{1}{(t_2-t_1)^{\varepsilon}} $,  hence we have proven~\eqref{borne3}.\medskip

$\bullet$ Assume that $t_2-t_1 \geq 1$. Then from \eqref{decomp} we get
\begin{eqnarray*}
\int dy_1 \,  \cac^{(n)}_{t_1,t_2}(y_1+y_2,y_2)^3&\lesssim &
\int_{[t_2-t_1,+\infty)^3} d\si_1 d\si_2 d\si_3 e^{-3(\sigma_1+\sigma_2+\sigma_3)}\\
&\lesssim  &e^{-9(t_2-t_1)} \lesssim  \frac1{(t_2-t_1)^{\varepsilon}},
\end{eqnarray*}
therefore we obtain \eqref{borne3}.
\end{proof}

\section{First-order diagram}\label{sec:first-order-diagram}

We naturally start our diagrams investigations with the case of $\<Psi>^{(n)}$, as defined in \eqref{def-L} (or equivalently in \eqref{represen-psi-0}).  In accordance with the statement of Proposition \ref{prop:conv-arbre}, we have in this situation:
\begin{proposition}\label{Prop-luxo}
For every $\al>\frac12$, there exists $\ka>0$ such that for all $p\geq 1$ 
\begin{equation}\label{luxo1}
\sup_{r\geq 0}    \mathbb{E}\Big[ \big\| \<Psi>^{(n)}-\<Psi>^{(n+1)}\big\|_{\cac( [r,r+1]; \cb^{-\al}_x)}^{2p}\Big]   \lesssim 2^{-\ka np}. 
\end{equation}
 As a result, for every $T>0$, the sequence $(\<Psi>^{(n)})$ converges almost surely to an element  $\<Psi>$ in $\cac{\big( [0,T]; \cb^{-\al}_x\big)}$, for every $\al>\frac12$.
\end{proposition}

\begin{proof}
For more clarity, we set 
$$\<Psi>^{(n,n+1)}:=\<Psi>^{(n)}-\<Psi>^{(n+1)}.$$

Fix $\al>\frac12$. Firstly, note that by using the H\"older inequality $\mathbb{E}[|X|] \leq \big(\mathbb{E}\big[|X|^{2p}\big] \big)^{\frac1{2p}}$, it is enough to prove \eqref{luxo1} for $p \geq 1$ large. Then, for $\beta>0$ small enough  and $p\geq 1$ large enough,  one has thanks to~\eqref{est-GRR} and \eqref{sobo-beso}, for all $r\geq 0$
\begin{multline}\label{grr}
\Big\| \<Psi>^{(n,n+1)}\Big\|_{\cac{\big( [r,r+1]; \cb_x^{-\al}\big)}}^{2p}\lesssim \Big\| \<Psi>^{(n,n+1)}\Big\|_{\cac^\beta{\big( [r,r+1]; \cb_x^{-\al}\big)}}^{2p} \\
\begin{aligned}
&\lesssim \Big\| \<Psi>_{r}^{(n,n+1)}\Big\|_{ \cb_x^{-\al}}^{2p}+\int_r^{r+1}\int_r^{r+1} dt dt' \, \frac{\big\| \<Psi>^{(n,n+1)}_{t'}-\<Psi>^{(n,n+1)}_t\big\|_{\cb_x^{-\al}}^{2p}}{|t'-t|^{2\beta p+2}}  \\
&\lesssim \Big\| \<Psi>_{r}^{(n,n+1)}\Big\|_{ \cb_x^{-\al}}^{2p}+\int_r^{r+1}\int_r^{r+1} dt dt' \, \frac{\big\|\<Psi>^{(n,n+1)}_{t'}-\<Psi>^{(n,n+1)}_t\big\|_{\cw_x^{ -\al+\varepsilon,2p}}^{2p}}{|t'-t|^{2\beta p+2}}.
\end{aligned}
\end{multline}
Therefore it suffices to prove the existence of constants $\ka,\eta>0$ such that one has both
\begin{equation}\label{phir}
\sup_{r\geq 0}\,  \mathbb{E}\bigg[ \Big\| \<Psi>_{r}^{(n,n+1)}\Big\|_{ \cb_x^{-\al}}^{2p}\bigg]  \lesssim 2^{-\ka np}
 \end{equation} 
and
\begin{equation}\label{diff-luxo}
\mathbb{E}\bigg[ \Big\| \<Psi>^{(n,n+1)}_{t'}-\<Psi>^{(n,n+1)}_t\Big\|_{\cw_x^{-\al,2p}}^{2p}\bigg] \lesssim 2^{-\ka n p}|t-t'|^{\eta p} ,
\end{equation}
uniformly over $t,t'\geq 0$. After that, we can indeed choose  $\beta>0$ such that $2\beta< \eta$ and use \eqref{elem} to obtain~\eqref{luxo1}.
\medskip

Let us first prove \eqref{diff-luxo}. Due to the Gaussian hypercontractivity property, we can write for all $0\leq t<t'$, 
\begin{align*}
\mathbb{E} \Big[ \big\| \<Psi>^{(n,n+1)}_{t'}-\<Psi>^{(n,n+1)}_{t} \big\|^{2p}_{\mathcal{W}_x^{-\al,2p}} \Big]&=\int dx \, \mathbb{E} \bigg[ \Big| H^{-\frac{\al}{2}}\big( \<Psi>^{(n,n+1)}_{t'}-\<Psi>^{(n,n+1)}_{t}\big)(x) \Big|^{2p} \bigg]\nonumber\\
&\lesssim \int dx \, \mathbb{E} \bigg[ \Big| H^{-\frac{\al}{2}}\big( \<Psi>^{(n,n+1)}_{t'}-\<Psi>^{(n,n+1)}_{t}\big)(x) \Big|^{2} \bigg]^p.
\end{align*}
Then, starting from the expression
\begin{align*}
&\big(H^{-\frac{\al}{2}}\<Psi>^{(n)}_{t'}\big)(x)=\sum_{k \geq 0} \, \vp_k(x) \, \la_k^{-\frac{\al}{2}}e^{-\varepsilon_n \la_k}\int_0^{t'} e^{-\la_k(t'-s)} d\beta^{(k)}_s ,
\end{align*}
we get the decomposition
\begin{align}
&H^{-\frac{\al}{2}}\big( \<Psi>^{(n,n+1)}_{t'}-\<Psi>^{(n,n+1)}_{t}\big)(x) = \mathscr{A}^{(n)}_{t,t'}+\mathscr{B}^{(n)}_{t,t'},\label{decompo-fir-or}
\end{align}
with
\begin{align*}
&\mathscr{A}^{(n)}_{t,t'}(x):=\sum_{k \geq 0} \, \vp_k(x) \, \la_k^{-\frac{\al}{2}}\big(e^{-\varepsilon_n \la_k}-e^{-\varepsilon_{n+1} \la_k}\big)\int_{t}^{t'}   e^{-\la_k(t'-s)}  d\beta^{(k)}_s
\end{align*}
and
\begin{align*}
&\mathscr{B}^{(n)}_{t,t'}(x):=\sum_{k \geq 0} \, \vp_k(x) \, \la_k^{-\frac{\al}{2}}\big(e^{-\varepsilon_n \la_k}-e^{-\varepsilon_{n+1} \la_k}\big)\int_0^{t}  \big(e^{-\la_k(t'-s)}-e^{-\la_k(t-s)}\big)  d\beta^{(k)}_s .
\end{align*}

\smallskip

On the one hand, for all small $\ka,\eta>0$, it holds that, for all $r\leq t<t'\leq r+1$
\begin{align*}
\mathbb{E} \Big[ \big| \mathscr{A}^{(n)}_{t,t'}(x) \big|^{2} \Big]&=\sum_{k \geq 0} \, \vp^2_{k}(x) \, \la_{k}^{-\al}\big| e^{-\varepsilon_n \la_{k}}-e^{-\varepsilon_{n+1} \la_{k}}\big|^2\int_{t}^{t'} ds \, e^{-2\la_k(t'-s)}\nonumber\\
&\lesssim 2^{-n\ka}\sum_{k \geq 0} \, \vp^2_{k}(x) \, \la_{k}^{-\al+\ka} \int_{t}^{t'} ds\, e^{-2\la_k (t'-s)}\nonumber\\
&\lesssim 2^{-n\ka}\sum_{k \geq 0} \, \vp^2_{k}(x) \, \la_{k}^{-\al+\ka} \frac{1}{\la_k} \big|1-e^{-2\la_k (t'-t)}\big|\nonumber\\
&\lesssim 2^{-n\ka}|t-t'|^\eta \sum_{k \geq 0} \, \vp^2_{k}(x) \, \la_{k}^{-\al-1+\ka+\eta} \lesssim 2^{-n\ka}|t-t'|^\eta h_{-\al-1+\ka+\eta}(x,x). 
\end{align*}

  On the other hand,  
\begin{equation*}
\mathbb{E} \Big[ \big| \mathscr{B}^{(n)}_{t,t'}(x) \big|^{2} \Big]
=\sum_{k \geq 0} \, \vp^2_{k}(x) \, \la_{k}^{-\al}\big| e^{-\varepsilon_n \la_{k}}-e^{-\varepsilon_{n+1} \la_{k}}\big|^2\int_0^{t}ds\,  \big|e^{-\la_k(t'-s)}-e^{-\la_k(t-s)}\big|^2 .
\end{equation*}
Then we check that 
$$  
\int_0^{t}ds\,  \big|e^{-\la_k(t'-s)}-e^{-\la_k(t-s)}\big|^2 \lesssim \frac1{\lambda_k} \Big(1-e^{-\la_k (t'-t)}\Big)^2,
$$
and along with the same ideas
\begin{eqnarray*}
\mathbb{E} \Big[ \big| \mathscr{B}^{(n)}_{t,t'}(x) \big|^{2} \Big]
&\lesssim &2^{-n\ka}|t-t'|^\eta \sum_{k \geq 0} \, \vp^2_{k}(x) \, \la_{k}^{-\al-1+\ka+\eta} \\
&\lesssim &2^{-n\ka}|t-t'|^\eta h_{-\al-1+\ka+\eta}(x,x).
\end{eqnarray*}
Going back to the decomposition \eqref{decompo-fir-or}, we deduce that
\begin{align*}
\int dx \, \mathbb{E} \Big[ \big|H^{-\frac{\al}{2}}\big( \<Psi>^{(n,n+1)}_{t'}-\<Psi>^{(n,n+1)}_{t}\big)(x)  \big|^{2} \Big]^p& \lesssim 2^{-n\ka p}|t-t'|^{\eta p} \big\|h_{-\al-1+\ka+\eta}(.,.)\big\|_{L^p(\R^3)}^p ,
\end{align*}
and we can use Lemma \ref{Lem-normeLp} to conclude that if $\ka,\eta>0$ are picked small enough so that $-\al-1+\ka+\eta<-\frac32$, then for any $p\geq 1$ large enough
\begin{align*}
\int dx \, \mathbb{E} \Big[ \big| H^{-\frac{\al}{2}}\big( \<Psi>^{(n,n+1)}_{t'}-\<Psi>^{(n,n+1)}_{t}\big)(x)  \big|^{2} \Big]^p& \lesssim 2^{-\ka n p}|t-t'|^{\eta p} ,
\end{align*}
uniformly over $t,t'\geq 0$.\medskip

\smallskip

Let us then check \eqref{phir}. In fact, we have already seen that 
\begin{align*}
&H^{-\frac{\al}{2}}\big( \<Psi>^{(n,n+1)}_{t}\big)(x) = \sum_{k \geq 0} \, \vp_k(x) \, \la_k^{-\frac{\al}{2}}\big(e^{-\varepsilon_n \la_k}-e^{-\varepsilon_{n+1} \la_k}\big)\int_{0}^{t}   e^{-\la_k(t-s)}  d\beta^{(k)}_s,
\end{align*}
and based on this expression, we get that for all $\kappa >0$ and $t \geq 0$, 
\begin{align*}
\mathbb{E} \Big[ \big| H^{-\frac{\al}{2}}\big( \<Psi>^{(n,n+1)}_{t}\big)(x) \big|^{2} \Big]&=\sum_{k \geq 0} \, \vp^2_{k}(x) \, \la_{k}^{-\al}\big| e^{-\varepsilon_n \la_{k}}-e^{-\varepsilon_{n+1} \la_{k}}\big|^2\int_{0}^{t} ds \, e^{-2\la_k(t-s)}\nonumber\\
&\lesssim\sum_{k \geq 0} \, \vp^2_{k}(x) \, \la_{k}^{-\al-1}\big| e^{-\varepsilon_n \la_{k}}-e^{-\varepsilon_{n+1} \la_{k}}\big|^2 \nonumber\\
&\lesssim 2^{-\ka n}\sum_{k \geq 0} \, \vp^2_{k}(x) \, \la_{k}^{-\al-1+\ka} \nonumber\\
&\lesssim    2^{-\ka n}h_{-\al-1+\ka}(x,x),
\end{align*}
uniformly over  $t  \geq 0$. With the same arguments as previously, we deduce that 
\begin{align*}
\sup_{t\geq 0}\,  \mathbb{E}\Big[ \big\| \<Psi>_{t}^{(n,n+1)}\big\|_{ \cb_x^{-\al}}^{2p}\Big]  \lesssim \sup_{t\geq 0}\, \mathbb{E} \Big[ \big\| \<Psi>^{(n,n+1)}_{t} \big\|^{2p}_{\mathcal{W}_x^{-\al+\eps,2p}} \Big] \lesssim 2^{-\ka n p},
\end{align*}
which corresponds to \eqref{phir}. \medskip

Let $T>0$. Since $\cac\big( [0,T]; \cb^{-\al}_x\big)$ is a Banach space, we deduce the convergence of $\<Psi>^{(n)}$ in this space to some limit $\<Psi>$. Using \eqref{luxo1}, we deduce that 
\begin{equation*}
    \mathbb{E}\Big[ \big\| \<Psi>^{(n)}-\<Psi>\big\|_{\cac( [0,T]; \cb^{-\al}_x)}^{2p}\Big]   \lesssim 2^{-\ka np}. 
\end{equation*}
and from there, a classical use of the Borell-Cantelli lemma justifies the desired almost
sure convergence of $\<Psi>^{(n)}$ to $\<Psi>$  in $\cac{\big( [0,T]; \cb^{-\al}_x\big)}$, for every $\al>\frac12$.
\end{proof}


\section{Second order diagram}

We now turn to the case of the second-order process
$$\<Psi2>^{(n)}_t(x):=\big(\<Psi>^{(n)}_t(x)\big)^2-\mathbb{E}\Big[\big(\<Psi>^{(n)}_t(x)\big)^2\Big].$$

\begin{proposition}\label{prop:cherry}
For every $\al>\frac12$, there exists $\ka>0$ such that for all $p\geq1$
\begin{equation}\label{second-or}
\sup_{r\geq 0}     \mathbb{E}\Big[  \big\| \<Psi2>^{(n)}-\<Psi2>^{(n+1)}\big\|_{\cac( [r,r+1]; \cb^{-2\alpha}_x)}^{2p}\Big] \lesssim 2^{-\ka np}. 
\end{equation}
Therefore, for every $T>0$,  the sequence $(\<Psi2>^{(n)})$ converges almost surely to an element $\<Psi2>$  in the space $\cac{\big( [0,T]; \cb^{-2\al}_x\big)}$, for every $\al>\frac12$.

\end{proposition}

\begin{proof}
Fix $\al>\frac12$. For the sake of conciseness, we will only focus on the uniform bound
\begin{equation*}
\sup_{n\geq 1} \sup_{r\geq 0}\,      \mathbb{E}\Big[  \big\| \<Psi2>^{(n)}\big\|_{\cac( [r,r+1]; \cb_x^{-2\alpha})}^{2p}\Big] \lesssim  1.
\end{equation*}

With the same arguments as in \eqref{grr}, we obtain that for $\beta>0$ small enough, $p\geq 1$ large enough, and for all $r \geq 0$,
\begin{equation*}
\mathbb{E}\Big[  \big\| \<Psi2>^{(n)}\big\|_{\cac( [r,r+1]; \cb_x^{-2\alpha})}^{2p}\Big] 
\lesssim \mathbb{E}\bigg[  \big\| \<Psi2>_{r}^{(n)}\big\|_{  \cb_x^{-2\alpha}}^{2p}\bigg]+\int_r^{r+1}\int_r^{r+1} dt dt' \, \frac{\mathbb{E}\Big[\big\|\<Psi2>^{(n)}_{t'}-\<Psi2>^{(n)}_t\big\|_{\cw_x^{ -2\al+\varepsilon,2p}}^{2p}\Big]}{|t'-t|^{2\beta p+2}}  .
\end{equation*} 
  Therefore it suffices to prove the existence of a constant $\eta>0$ such that one has both
\begin{equation}\label{phir2}
 \sup_{n\geq 0}\sup_{r\geq 0}\,   \mathbb{E}\bigg[  \big\| \<Psi2>_{r}^{(n)}\big\|_{  \cb_x^{-2\alpha}}^{2p}\bigg]  \lesssim 1
 \end{equation} 
and 
\begin{equation}\label{diff2}
\sup_{n\geq 0} \, \mathbb{E}\Big[ \big\| \<Psi2>^{(n)}_{t'}-\<Psi2>^{(n)}_t\big\|_{\cw_x^{-2\al,2p}}^{2p}\Big] \lesssim |t-t'|^{\eta p} ,
\end{equation}
uniformly over $t,t'\geq 0$.

\smallskip

Let us start with \eqref{diff2}. Recall that the hypercontractivity property holds true in any finite Wiener chaos  (see {\it e.g.} \cite[Theorem 2.7.2]{nourdin-peccati}), which yields here
\begin{align*}
\mathbb{E}\Big[ \big\| \<Psi2>^{(n)}_{t'}-\<Psi2>^{(n)}_t\big\|_{\cw_x^{-2\al,2p}}^{2p}\Big]&=\int dx \, \mathbb{E}\bigg[\Big| \big(H^{-\al}(\<Psi2>^{(n)}_{t'}-\<Psi2>^{(n)}_t)\big)(x)\Big|^{2p}\bigg]\\
&\lesssim \int dx \, \mathbb{E}\bigg[\Big| \big(H^{-\al}(\<Psi2>^{(n)}_{t'}-\<Psi2>^{(n)}_t)\big)(x)\Big|^{2}\bigg]^p.
\end{align*}

Then, thanks to \eqref{tree-carre}, one has
\begin{align*}
&\mathbb{E}\bigg[\Big| \big(H^{-\al}(\<Psi2>^{(n)}_{t'}-\<Psi2>^{(n)}_t)\big)(x)\Big|^{2}\bigg]=\int dy_1 dy_2 \, h_{-\al}(x,y_1) h_{-\al}(x,y_2)\mathbb{E}\Big[ (\<Psi2>^{(n)}_{t'}-\<Psi2>^{(n)}_t)(y_1)(\<Psi2>^{(n)}_{t'}-\<Psi2>^{(n)}_t)(y_2)\Big]\nonumber\\
&=2\int dy_1 dy_2 \, h_{-\al}(x,y_1) h_{-\al}(x,y_2)\\
& \hspace{1.5cm}\Big( \mathbb{E}\big[ \<Psi2>^{(n)}_{t'}(y_1) \<Psi2>^{(n)}_{t'}(y_2)  \big]-\mathbb{E}\big[\<Psi2>^{(n)}_{t'}(y_1)\<Psi2>^{(n)}_{t}(y_2)  \big]+\mathbb{E}\big[ \<Psi2>^{(n)}_{t}(y_1) \<Psi2>^{(n)}_{t}(y_2)  \big]-\mathbb{E}\big[ \<Psi2>^{(n)}_{t}(y_1) \<Psi2>^{(n)}_{t'}(y_2)  \big]\Big)\nonumber\\
&=2\int dy_1 dy_2 \, h_{-\al}(x,y_1) h_{-\al}(x,y_2)\Big(\mathcal{C}_{t',t'}^{(n)}(y_1,y_2)^2- \mathcal{C}_{t',t}^{(n)}(y_1,y_2)^2 \Big)\\
&\hspace{1cm}+2\int dy_1 dy_2 \, h_{-\al}(x,y_1) h_{-\al}(x,y_2)\Big(\mathcal{C}_{t,t}^{(n)}(y_1,y_2)^2- \mathcal{C}_{t,t'}^{(n)}(y_1,y_2)^2 \Big)=:A^{(n)}_{t,t'}(x)+\tilde{A}^{(n)}_{t,t'}(x).
\end{align*}
In the sequel, we will only provide details for the estimate of $A^{(n)}_{t,t'}(x)$, but the term $\tilde{A}^{(n)}_{t,t'}(x)$ could clearly be treated in a similar way.

\smallskip

Let us write, for all $0<\beta<\frac14$,
\begin{align*}
A^{(n)}_{t,t'}(x)
&=   2 \int \int  dy_1 dy_2\big(H_{y_1}^{-\beta} h_{-\alpha}(x,y_1)\big) \big(H_{y_2}^{-\beta} h_{-\alpha}(x,y_2)\big)\Big(H_{y_1}^{\beta}H_{y_2}^{\beta}  \Big(\mathcal{C}_{t',t'}^{(n)}(y_1,y_2)^2- \mathcal{C}_{t',t}^{(n)}(y_1,y_2)^2 \Big)\Big).
\end{align*}
Observe that $H_{y_1}^{-\beta}  h_{-\alpha}(x,y_1)=h_{-\alpha-\beta}(x,y_1) $, and so
\begin{equation*}
A^{(n)}_{t,t'}(x) =2 \int \int dy_1 dy_2   h_{-\alpha-\beta}(x,y_1)   h_{-\alpha-\beta}(x,y_2)\Big(H_{y_1}^{\beta}H_{y_2}^{\beta} \Big(\mathcal{C}_{t',t'}^{(n)}(y_1,y_2)^2- \mathcal{C}_{t',t}^{(n)}(y_1,y_2)^2 \Big)\Big).
\end{equation*}

Now, by the Cauchy-Schwarz inequality and \eqref{normeL2}
\begin{align*}
\big|A^{(n)}_{t,t'}(x)\big| &\lesssim    \|    h_{-\alpha-\beta}(x,\cdot) \|_{L^2_{y_2}(\R^3)}   \Big\|     \int   h_{-\alpha-\beta}(x,y_1)   \Big(H_{y_1}^{\beta}H_{y_2}^{\beta} \Big(\mathcal{C}_{t',t'}^{(n)}(y_1,.)^2- \mathcal{C}_{t',t}^{(n)}(y_1,.)^2 \Big)\Big)dy_1 \Big\|_{L^2_{y_2}(\R^3)}\\
&\lesssim   \big( h_{-2(\alpha+\beta)}(x,x)   \big)^{\frac12}  \Big\|     \int   h_{-\alpha-\beta}(x,y_1)    \Big(H_{y_1}^{\beta}H_{y_2}^{\beta}  \Big(\mathcal{C}_{t',t'}^{(n)}(y_1,.)^2- \mathcal{C}_{t',t}^{(n)}(y_1,.)^2 \Big)\Big)dy_1 \Big\|_{L^2_{y_2}(\R^3)}.
\end{align*}

Recall the definition \eqref{def-tau} of the space translation $\tau$. We make the change of variables $y_1 \to y_1+y_2$ in the r.h.s. integral of the previous line. Then the Cauchy-Schwarz inequality gives 
\begin{multline*}
 \Big\|     \int  dy_1 h_{-\alpha-\beta}(x,y_1)    \Big(H_{y_1}^{\beta}H_{y_2}^{\beta}  \Big(\mathcal{C}_{t',t'}^{(n)}(y_1,.)^2- \mathcal{C}_{t',t}^{(n)}(y_1,.)^2 \Big)\Big) \Big\|_{L^2_{y_2}(\R^3)}=\\
 \begin{aligned}
  &= \Big\|     \int  dy_1\,  h_{-\alpha-\beta}(x,y_1+\cdot)  \tau  \Big(H_{y_1}^{\beta}H_{y_2}^{\beta}   \Big(\mathcal{C}_{t',t'}^{(n)}(y_1,.)^2- \mathcal{C}_{t',t}^{(n)}(y_1,.)^2 \Big)\Big) \Big\|_{L^2_{y_2}(\R^3)}\\
 &\leq        \int dy_1\,  \big\|    h_{-\alpha-\beta}(x,y_1+\cdot) \big\|_{L^2_{y_2}(\R^3)}    \sup_{y_2\in \R^3}   \Big| \tau H_{y_1}^{\beta}H_{y_2}^{\beta}   \Big(\mathcal{C}_{t',t'}^{(n)}(y_1,y_2)^2- \mathcal{C}_{t',t}^{(n)}(y_1,y_2)^2 \Big)\Big|  \\
&=     \big( h_{-2(\alpha+\beta)}(x,x)   \big)^{\frac12}       \int   dy_1 \, \sup_{y_2\in \R^3}   \Big| \tau H_{y_1}^{\beta}H_{y_2}^{\beta}  \Big(\mathcal{C}_{t',t'}^{(n)}(y_1,y_2)^2- \mathcal{C}_{t',t}^{(n)}(y_1,y_2)^2 \Big)\Big|.
\end{aligned}
\end{multline*}
We have thus shown that 
\begin{align*}
 \big|A^{(n)}_{t,t'}(x)\big|
&\lesssim   \big( h_{-2(\alpha+\beta)}(x,x)   \big)   \int   dy_1 \, \sup_{y_2\in \R^3}   \Big| \tau H_{y_1}^{\beta}H_{y_2}^{\beta}  \Big(\mathcal{C}_{t',t'}^{(n)}(y_1,y_2)^2- \mathcal{C}_{t',t}^{(n)}(y_1,y_2)^2 \Big)\Big|,
\end{align*}
and so
\begin{align*}
\bigg(\int dx  \,\big|A^{(n)}_{t,t'}(x)\big|^p\bigg)^{\frac{1}{p}}
&\lesssim    \big\|h_{-2(\alpha+\beta)}(\cdot,\cdot)\big\|_{L^p(\R^3)}  \int   dy_1 \, \sup_{y_2\in \R^3}   \Big| \tau H_{y_1}^{\beta}H_{y_2}^{\beta}  \Big(\mathcal{C}_{t',t'}^{(n)}(y_1,y_2)^2- \mathcal{C}_{t',t}^{(n)}(y_1,y_2)^2 \Big)\Big|.
\end{align*}

Since $\alpha >\frac12$, we can pick $0<\beta<\frac14$ such that $2(\alpha+\beta)>\frac32$, and for such a choice, we know by Lemma~\ref{Lem-normeLp} that $  h_{-2(\alpha+\beta)}(\cdot,\cdot) \in  L^p(\R^3)$  for every $p\geq 1$ large.  Besides, we can appeal to Lemma~\ref{lem-tokyo}  and guarantee the existence of $\eta>0$ such that, for all $r \leq t \leq t' \leq r+1$
$$\sup_{n\geq 1}\int   dy_1 \, \sup_{y_2\in \R^3}   \Big| \tau H_{y_1}^{\beta}H_{y_2}^{\beta}  \Big(\mathcal{C}_{t',t'}^{(n)}(y_1,y_2)^2- \mathcal{C}_{t',t}^{(n)}(y_1,y_2)^2 \Big)\Big|\lesssim |t'-t|^\eta,$$
uniformly in $r \geq 0$.

In brief, we have obtained that for $p\geq 1$ large enough and $\eta>0$ small enough, 
\begin{align*}
\sup_{n\geq 1} \int dx  \,\big|A^{(n)}_{t,t'}(x)\big|^p &\lesssim    |t'-t|^{\eta p}.
\end{align*}
The same estimate being true for $\tilde{A}$ as well, we can conclude that
$$\sup_{n\geq 1}\,  \int dx \, \mathbb{E}\bigg[\Big| \big(H^{-\al}(\<Psi2>^{(n)}_{t'}-\<Psi2>^{(n)}_t)\big)(x)\Big|^{2}\bigg]^p \lesssim |t'-t|^{\eta p}$$
uniformly over $t,t'\geq 0$, as desired. \medskip

The proof of \eqref{phir2} easily follows from the same procedure, by replacing the use of \eqref{eni-1} with that of~\eqref{eni-2}.
\end{proof}

\section{Third order diagram}

  Recall the definition \eqref{ord3} of $\<IPsi3>^{(n)}$. The main result of this section is the following:

\begin{proposition}\label{prop-ord3}
For all coefficients $0<\beta<\frac12$ and $0\leq \ga <\frac14-\frac{\beta}{2}$, there exists $\ka>0$ such that for all $p\geq 1$
\begin{equation*}
\sup_{r\geq 0}     \mathbb{E}\bigg[ \Big\| \<IPsi3>^{(n+1)}-\<IPsi3>^{(n)}\Big\|_{\cac^{\ga}( [r,r+1]; \cb^{\beta}_x)}^{2p}\bigg] \lesssim 2^{-\ka np}. 
\end{equation*}
As a result, for all such coefficients $\beta,\ga$, the sequence $(\<IPsi3>^{(n)})$ converges almost surely to an element~$\<IPsi3>$   in $\cac^\gamma{\big( [0,T]; \cb^{\beta}_x\big)}$, for every $T>0$.
\end{proposition}

Observe that the above assertion immediately implies the convergence result for $(\<IPsi3>^{(n)})$ in Proposition \ref{prop:conv-arbre}, by taking $\ga=\frac{\eps}{4}$, $\beta=\frac12-\eps$, and then $\ga=\frac14-\eps,\beta=\eps$.

\begin{proof}
Just as in the proof of Proposition \ref{prop:cherry}, we will only focus on the bound
\begin{align*}
\sup_{n\geq }\sup_{r\geq 0}     \mathbb{E}\bigg[ \Big\| \<IPsi3>^{(n)}\Big\|_{\cac^{\ga}( [r,r+1]; \cb_x^{\beta})}^{2p}\bigg]\lesssim 1,
\end{align*}
for all $0<\beta<\frac12$ and $0\leq \ga <\frac14-\frac{\beta}{2}$. 

\smallskip

For the same reasons as in \eqref{grr}, we have for all $r\geq 0$, all $\eps>0$ and all $p\geq 1$ large enough, 
\small
\begin{align*}
&\mathbb{E}\bigg[\Big\| \<IPsi3>^{(n)}\Big\|_{\cac^{\ga}( [r,r+1]; \cb_x^{\beta})}^{2p}\bigg]\lesssim \mathbb{E}\bigg[\Big\| \<IPsi3>_r^{(n)}\Big\|_{ \cb_x^{\beta}}^{2p}\bigg]+\int_r^{r+1}\int_r^{r+1} dt dt' \, \frac{\mathbb{E}\Big[\big\|\<IPsi3>^{(n)}_{t'}-\<IPsi3>^{(n)}_t\big\|_{\cw_x^{ \beta+\varepsilon,2p}}^{2p}\Big]}{|t'-t|^{2\ga p+2}}.
\end{align*}
\normalsize
Therefore, for all $0<\beta<\frac12$ and $0\leq \ga <\frac14-\frac{\beta}{2}$, it suffices to prove the existence of $\eta>0$ such that 
\begin{equation}\label{phir3}
 \sup_{n\geq 1}\sup_{r\geq 0}\,   \mathbb{E}\bigg[  \Big\| \<IPsi3>_r^{(n)}\Big\|_{ \cb_x^{\beta}}^{2p}   \bigg]  \lesssim 1
 \end{equation} 
 and 
\begin{equation}\label{delta3}
 \sup_{n\geq 1}\, \mathbb{E}\bigg[ \Big\| \<IPsi3>^{(n)}_{t'}-\<IPsi3>^{(n)}_t\Big\|_{\cw_x^{ \beta,2p}}^{2p}\bigg] \lesssim  |t'-t|^{(2\ga+\eta)p},
\end{equation}
uniformly over $t,t'\geq 0$.

\smallskip

Fix $0<\beta<\frac12$, $0\leq \ga <\frac14-\frac{\beta}{2}$, and let us prove \eqref{delta3}. Assume that $0\leq t<t'$ and set $$\<IPsi3>^{(n)}_{t,t'}:=\<IPsi3>^{(n)}_{t'}-\<IPsi3>^{(n)}_t.$$
 By hypercontractivity  (see e.g. \cite[Theorem 2.7.2]{nourdin-peccati}), it holds that
\begin{eqnarray*}
\mathbb{E}\bigg[ \Big\| \<IPsi3>^{(n)}_{t,t'}\Big\|_{\cw_x^{\beta,2p}}^{2p}\bigg]  \lesssim \int dx \, \mathbb{E}\Big[\big|H^{\frac{\beta}{2}}(\<IPsi3>^{(n)}_{t,t'})(x)\big|^2\Big]^p, 
\end{eqnarray*}
and therefore we are reduced to showing that for some $\eta>0$,
\begin{equation}\label{unif-glace3b}
\sup_{n\geq 1}\bigg(\int dx \, \mathbb{E}\Big[\big|H^{\frac{\beta}{2}}(\<IPsi3>^{(n)}_{t,t'})(x)\big|^2\Big]^p\bigg)^{\frac{1}{p}}\lesssim  |t'-t|^{ 2\ga+\eta},
\end{equation}
where the proportional constant does not depend on $t,t'$.

To this end, let us write first
\begin{equation}\label{beg}
\<IPsi3>^{(n)}_t(y)=\int_0^t \Big(e^{-H(t-s)}\<Psi3>^{(n)}_s\Big)(y)\, ds=\int_0^t ds \int_{\R^3} dz \, K_{t-s}(y,z)\<Psi3>^{(n)}_s(z),
\end{equation}
which immediately entails the decomposition
\begin{align*}
\<IPsi3>^{(n)}_{t,t'}
&=\int_t^{t'} ds \int_{\R^3} dz \, K_{t'-s}(y,z)\<Psi3>^{(n)}_s(z)+\int_0^t ds \int_{\R^3} dz \, \big(K_{t'-s}(y,z)-K_{t-s}(y,z)\big)\<Psi3>^{(n)}_s(z)\\
&=:\<IPsi3>^{\mathbf{1},(n)}_{t,t'}(y)+\<IPsi3>^{\mathbf{2},(n)}_{t,t'}(y).
\end{align*}

\

\noindent
\textit{Case of $\<IPsi3>^{\mathbf{1},(n)}_{t,t'}$.} One has in this case, using \eqref{tree-cube}
\small
\begin{align*}
&\mathbb{E}\Big[\big|H^{\frac{\beta}{2}}(\<IPsi3>^{\mathbf{1},(n)}_{t,t'})(x)\big|^2\Big]=\int dy_1 dy_2 \, h_{\frac{\beta}{2}}(x,y_1) h_{\frac{\beta}{2}}(x,y_2) \mathbb{E}\Big[\<IPsi3>^{\mathbf{1},(n)}_{t,t'}(y_1)\<IPsi3>^{\mathbf{1},(n)}_{t,t'}(y_2)\Big]\\
&=\int dy_1 dy_2 \, h_{\frac{\beta}{2}}(x,y_1) h_{\frac{\beta}{2}}(x,y_2) \int_t^{t'} ds_1\int_t^{t'} ds_2 \int dz_1\int dz_2 \, K_{t'-s_1}(y_1,z_1) K_{t'-s_2}(y_2,z_2)\mathbb{E}\Big[\<Psi3>^{(n)}_{s_1}(z_1)\<Psi3>^{(n)}_{s_2}(z_2)\Big]\\
&=\int dy_1 dy_2 \, h_{\frac{\beta}{2}}(x,y_1) h_{\frac{\beta}{2}}(x,y_2) \int_t^{t'} ds_1\int_t^{t'} ds_2 \int dz_1\int dz_2 \, K_{t'-s_1}(y_1,z_1) K_{t'-s_2}(y_2,z_2)\cac^{(n)}_{s_1,s_2}(z_1,z_2)^3,
\end{align*}
\normalsize
and so
$$\mathbb{E}\Big[\big|H^{\frac{\beta}{2}}(\<IPsi3>^{\mathbf{1},(n)}_{t,t'})(x)\big|^2\Big]=\int_{[t,t']^2} ds_1 ds_2 \int dz_1 dz_2\, k^{(\frac{\beta}{2})}_{t'-s_1}(x,z_1)k^{(\frac{\beta}{2})}_{t'-s_2}(x,z_2)\cac^{(n)}_{s_1,s_2}(z_1,z_2)^3,$$
where we have set
$$k^{(\frac{\beta}{2})}_{\tau}(x,z):=\int dy  \, h_{\frac{\beta}{2}}(x,y) K_{\tau}(y,z)=\sum_{k \geq 0} \la_k^{\frac{\beta}{2}} e^{-\la_k \tau} \vp_k(x)\vp_k(z).  $$


From here, let us write
\small
\begin{align}
&\mathbb{E}\Big[\big|H^{\frac{\beta}{2}}(\<IPsi3>^{\mathbf{1},(n)}_{t,t'})(x)\big|^2\Big]\leq 2\int_t^{t'} ds_1\int_t^{t'} ds_2 \int dz_1 dz_2\, \big|k^{(\frac{\beta}{2})}_{t'-s_1}(x,z_1)\big| \big|k^{(\frac{\beta}{2})}_{t'-s_2}(x,z_2)\big| \big|\cac^{(n)}_{s_1,s_2}(z_1,z_2)\big|^3\nonumber\\
&\leq 2\int_t^{t'} ds_1\int_t^{t'} ds_2  \int  dz_2\, \big| k^{(\frac{\beta}{2})}_{t'-s_2}(x,z_2)\big| \bigg(\int dz_1 \, \big|k^{(\frac{\beta}{2})}_{t'-s_1}(x,z_1)\big|\cac^{(n)}_{s_1,s_2}(z_1,z_2)\big|^3\bigg)\nonumber\\
&\leq 2\int_t^{t'} ds_1\int_t^{t'} ds_2 \, \bigg(\int  dz'_2\, \big| k^{(\frac{\beta}{2})}_{t'-s_2}(x,z'_2)\big|^2\bigg)^{\frac12} \bigg(\int  dz_2\bigg|\int dz_1 \, \big| k^{(\frac{\beta}{2})}_{t'-s_1}(x,z_1)\big|\big|\cac^{(n)}_{s_1,s_2}(z_1,z_2)\big|^3\bigg|^2  \bigg)^{\frac12}\nonumber\\
&\leq 2\int_t^{t'} ds_1\int_t^{t'} ds_2 \, \bigg(\int  dz'_2\, \big| k^{(\frac{\beta}{2})}_{t'-s_2}(x,z'_2)\big|^2\bigg)^{\frac12} \bigg(\int  dz_2\bigg|\int dz_1 \, \big| k^{(\frac{\beta}{2})}_{t'-s_1}(x,z_1+z_2)\big|\Big(\sup_{z_2\in \R^3}\big|\cac^{(n)}_{s_1,s_2}(z_1+z_2,z_2)\big|^3\Big)\bigg|^2  \bigg)^{\frac12}\nonumber\\
&\leq 2\int_t^{t'} ds_1\int_t^{t'} ds_2 \, \bigg(\int  dz'_2\, \big| k^{(\frac{\beta}{2})}_{t'-s_2}(x,z'_2)\big|^2\bigg)^{\frac12}\bigg(\int  dz''_2\, \big| k^{(\frac{\beta}{2})}_{t'-s_1}(x,z''_2)\big|^2\bigg)^{\frac12} \bigg(\int dz_1 \, \sup_{z_2\in \R^3}\big|\cac^{(n)}_{s_1,s_2}(z_1+z_2,z_2)\big|^3\bigg),\label{blue-calc}
\end{align}
\normalsize
where we have used Jensen's inequality to derive the last estimate.

\smallskip
At this point, recall that according to Lemma \ref{lem:cac3}, one has for all $\eps>0$ and $0\leq s_1<s_2$
\begin{equation}\label{borneeps}
\sup_{n\geq 1}\int dz_1 \, \sup_{z_2\in \R^3}\big|\cac^{(n)}_{s_1,s_2}(z_1+z_2,z_2)\big|^3  \lesssim  \frac{1}{|s_2-s_1|^{\varepsilon}},
\end{equation}
for some proportional constant that does not depend on $s_1,s_2$. Besides, it is readily checked that for every $\varepsilon >0$,
\begin{align}
\int  dz'\, \big| k^{(\frac{\beta}{2})}_{t'-s}(x,z')\big|^2&=\int  dz'\, \bigg|\sum_{k \geq 0} \la_k^{\frac{\beta}{2}} e^{-\la_k (t'-s)} \vp_k(x)\vp_k(z') \bigg|^2\nonumber\\
&=\sum_{k \geq 0} \la_k^{\beta} e^{-2\la_k (t'-s)} \vp^2_k(x) \nonumber\\
& \lesssim \frac{e^{-(t'-s)}}{(t'-s)^{\frac32+\beta+2\varepsilon}} \sum_{k \geq 0} \la_k^{-\frac32-2\varepsilon}\vp^2_k(x) \nonumber \\
&=  \frac{e^{-(t'-s)}}{(t'-s)^{\frac32+\beta+2\varepsilon}}h_{-\frac32-2\varepsilon}(x,x) ,\label{besides}
\end{align}
and so 
\begin{equation}\label{borneeps2}
\bigg(\int  dz'_2\, \big| k^{(\frac{\beta}{2})}_{t'-s_2}(x,z'_2)\big|^2\bigg)^{\frac12}\bigg(\int  dz''_2\, \big| k^{(\frac{\beta}{2})}_{t'-s_1}(x,z''_2)\big|^2\bigg)^{\frac12} \lesssim    \frac{e^{-(t'-s_1)}}{(t'-s_1)^{\frac34+\frac{\beta}{2}+\varepsilon}} \frac{e^{-(t'-s_2)}}{(t'-s_2)^{\frac34+\frac{\beta}{2}+\varepsilon}} h_{-\frac32-2\varepsilon}(x,x).
\end{equation}

For $\varepsilon>0$ and $p>1$ large enough, we know by Lemma \ref{Lem-normeLp} that $h_{-\frac32-2\varepsilon}(\cdot , \cdot) \in L^p(\R^3)$. Then, by injecting the estimate \eqref{borneeps2} into~\eqref{blue-calc} and using~\eqref{borneeps}, we derive that for all $\varepsilon>0$ small enough and $p>1$ large enough
\begin{multline*}
\sup_{n\geq 1}\, \bigg(\int dx \, \mathbb{E}\Big[\big|H^{\frac{\beta}{2}}(\<IPsi3>^{\mathbf{1},(n)}_{t,t'})(x)\big|^2\Big]^p\bigg)^{\frac{1}{p}} \lesssim\\
\begin{aligned}
&\lesssim  \big\| h_{-\frac32-2\varepsilon}(\cdot,\cdot)\big\|_{L^p(\R^3)}   \int_t^{t'} \frac{ds_1}{(t'-s_1)^{\frac34+\frac{\beta}{2}+\varepsilon}}\int_{t}^{t'} \frac{ds_2}{(t'-s_2)^{\frac34+\frac{\beta}{2}+\varepsilon}|s_2-s_1|^{\varepsilon}}   \\
&\lesssim  |t'-t|^{2-2(\frac34+\frac{\beta}{2}+\varepsilon)-\varepsilon}\int_0^{1} \frac{d\sigma_1}{(1-\sigma_1)^{\frac34+\frac{\beta}{2}+\varepsilon}}\int_{0}^{1} \frac{d\sigma_2}{(1-\sigma_2)^{\frac34+\frac{\beta}{2}+\varepsilon}|\sigma_2-\sigma_1|^{\varepsilon}} \lesssim  |t'-t|^{\frac12-\beta-3\varepsilon}.
\end{aligned}
\end{multline*}
Since $0\leq \ga <\frac14-\frac{\beta}{2}$, we can pick $\eps >0$ small enough (and accordingly $p\geq 1$ large enough in the above bound) so that $\eta:=\frac12-\beta-3\eps-2\ga>0$, which yields the desired rate in \eqref{unif-glace3b}.

\

\noindent
\textit{Case of $\<IPsi3>^{\mathbf{2},(n)}_{t,t'}$.} By following the same arguments as for $\<IPsi3>^{\mathbf{1},(n)}$, we easily obtain the estimate
\small
\begin{align*}
&\sup_{n\geq 1}\, \mathbb{E}\Big[\big|H^{\frac{\beta}{2}}(\<IPsi3>^{\mathbf{2},(n)}_{t,t'})(x)\big|^2\Big]\lesssim \int_{[0,t]^2} \frac{ds_1 ds_2}{|s_1-s_2|^{\frac{\varepsilon}{2}}} \, \bigg(\int  dz'\, \big| k^{(\frac{\beta}{2})}_{t'-s_2,t-s_2}(x,z')\big|^2\bigg)^{\frac12}\bigg(\int  dz''\, \big| k^{(\frac{\beta}{2})}_{t'-s_1,t-s_1}(x,z'')\big|^2\bigg)^{\frac12} ,
\end{align*}
\normalsize
for some proportional constant independent of $t,t'$, and where we have set
\begin{align*}
k^{(\frac{\beta}{2})}_{r',r}(x,z):=\int dy  \, h_{\frac{\beta}{2}}(x,y) \big(K_{r'}(y,z)-K_r(y,z)\big)=\sum_{k \geq 0} \la_k^{\frac{\beta}{2}}\big( e^{-\la_k r'}-e^{-\la_k r}\big) \vp_k(x)\vp_k(z).
\end{align*}
Then, just as in \eqref{besides}, we can write for every $s\in [0,t]$,
\begin{eqnarray*}
\int  dz'\, \big| k^{(\frac{\beta}{2})}_{t'-s,t-s}(x,z')\big|^2&=&\int  dz'\, \bigg|\sum_{k \geq 0} \la_k^{\frac{\beta}{2}} \big(e^{-\la_k (t'-s)}-e^{-\la_k (t-s)}\big) \vp_k(x)\vp_k(z') \bigg|^2\\
&=&\sum_{k \geq 0} \la_k^{\beta} \big|e^{-\la_k (t'-s)}-e^{-\la_k (t-s)}\big|^2 \vp^2_k(x).
\end{eqnarray*}
For all $0\leq s < t < t'$ and $\kappa>0$, we have trivially
$$ \big|e^{-\la_k (t'-s)}-e^{-\la_k (t-s)}\big| \leq \lambda_k |t'-t|,$$
as well as
$$ \big|e^{-\la_k (t'-s)}-e^{-\la_k (t-s)}\big| \leq e^{-\la_k (t-s)}= e^{-\frac{\la_k}{2} (t-s)}e^{-\frac{\la_k}{2} (t-s)} \leq \frac{1}{\lambda^\kappa_k |t-s|^\kappa}e^{-  (t-s)}.$$
By combining these two bounds, we get that for all $0 \leq \theta_1, \theta_2 \leq 1$ such that $\theta_1+\theta_2=1$,
$$ \big|e^{-\la_k (t'-s)}-e^{-\la_k (t-s)}\big| \leq \lambda^{\theta_1-\theta_2 \kappa}_k    \frac{|t'-t|^{ \theta_1}e^{- \theta_2(t-s)}}{|t-s|^{\theta_2 \kappa }}.  $$
For $\eps>0$ small enough, let us choose $\theta_1=\frac14 -\frac{\beta}2-2\varepsilon$, $\theta_2=1-    \theta_1$ and $\kappa=\frac{2-\eps}{2\theta_2}$, which yields 
\begin{equation*}
\int  dz'\, \big| k^{(\frac{\beta}{2})}_{t'-s,t-s}(x,z')\big|^2\lesssim \frac{|t'-t|^{\frac12-\beta-4\varepsilon}e^{-2\varepsilon (t-s)}}{|t-s|^{2-\varepsilon}} \sum_{k \geq 0} \la_k^{-\frac32-3\varepsilon}\vp^2_k(x) .
\end{equation*}
As a result, for every $\eps>0$ small enough,
\begin{align*}
\sup_{n\geq 1}\, \mathbb{E}\Big[\big|H^{\frac{\beta}{2}}(\<IPsi3>^{\mathbf{2},(n)}_{t,t'})(x)\big|^2\Big]&\lesssim |t'-t|^{\frac12-\beta-4\varepsilon} h_{-\frac32-3\varepsilon}(x,x) \int_{[0,t]^2} \frac{ds_1 ds_2}{|s_1-s_2|^{\frac{\varepsilon}{2}}} \frac{    e^{-\varepsilon (t-s_2)}  }{|t-s_2|^{1-\frac{\varepsilon}{2}}}\frac{e^{-\varepsilon (t-s_1)} }{|t-s_1|^{1-\frac{\varepsilon}{2}}}\\
&\lesssim |t'-t|^{\frac12-\beta-4\varepsilon} h_{-\frac32-3\varepsilon}(x,x) \int_{[0,+\infty)^2} \frac{d\sigma_1 d\sigma_2}{|\sigma_1-\sigma_2|^{\frac{\varepsilon}{2}}} \frac{    e^{-\varepsilon \sigma_2}  }{\sigma_2^{1-\frac{\varepsilon}{2}}}\frac{e^{-\varepsilon \sigma_1} }{\sigma_1^{1-\frac{\varepsilon}{2}}}\\
&\lesssim |t'-t|^{\frac12-\beta-4\varepsilon} h_{-\frac32-3\varepsilon}(x,x),
\end{align*}
uniformly over $t,t'\geq 0$. The desired bound (that is the one in \eqref{unif-glace3b}) now follows from the same arguments as in the previous case. 

\

The proof of \eqref{phir3} can then be derived along the same ideas. Namely, starting with formula~\eqref{beg} and similarly to \eqref{blue-calc}, we obtain
\begin{multline*}
\mathbb{E}\Big[\big|H^{\frac{\beta}{2}}(\<IPsi3>^{(n)}_{t})(x)\big|^2\Big] \lesssim 
 \int_0^{t} ds_1\int_0^{t} ds_2 \, \bigg(\int  dz'_2\, \big| k^{(\frac{\beta}{2})}_{t-s_2}(x,z'_2)\big|^2\bigg)^{\frac12} \\
 \bigg(\int  dz''_2\, \big| k^{(\frac{\beta}{2})}_{t-s_1}(x,z''_2)\big|^2\bigg)^{\frac12} \bigg(\int dz_1 \, \sup_{z_2\in \R^3}\big|\cac^{(n)}_{s_1,s_2}(z_1+z_2,z_2)\big|^3\bigg).
\end{multline*}
Then, with the help of \eqref{borneeps} and \eqref{besides}, we get that for every $\eps>0$ small enough,
\begin{eqnarray*}
\sup_{n\geq 1}\, \mathbb{E}\Big[\big|H^{\frac{\beta}{2}}(\<IPsi3>^{(n)}_{t})(x)\big|^2\Big] &\lesssim &
 h_{-\frac32-2\varepsilon}(x,x)\int_0^{t}\int_0^{t}    \frac{ds_1 ds_2}{|s_2-s_1|^\eps} \,\frac{e^{-(t-s_1)}}{(t-s_1)^{\frac34+\frac{\beta}{2}+\varepsilon}} \frac{e^{-(t-s_2)}}{(t-s_2)^{\frac34+\frac{\beta}{2}+\varepsilon}}\\
 &\lesssim &
 h_{-\frac32-2\varepsilon}(x,x)\int_0^{+\infty}\int_0^{+\infty}  \frac{d\si_1d\si_2 }{|\si_2-\si_1|^\eps}\,\frac{e^{-\si_1}}{\si_1^{\frac34+\frac{\beta}{2}+\varepsilon}} \frac{e^{-\si_2}}{\si_2^{\frac34+\frac{\beta}{2}+\varepsilon}} \\
 &\lesssim & h_{-\frac32-2\varepsilon}(x,x),
\end{eqnarray*}
uniformly over $t\geq 0$. Therefore, with similar arguments as before, 
\begin{equation*} 
\sup_{n\geq 1} \sup_{t\geq 0}\,  \mathbb{E}\bigg[  \Big\| \<IPsi3>_t^{(n)}\Big\|_{ \cb_x^{\beta}}^{2p}   \bigg]  \lesssim   \sup_{n\geq 1}\sup_{t\geq 0}\,  \mathbb{E}\bigg[  \Big\| \<IPsi3>_t^{(n)}\Big\|_{\cw_x^{ \beta+\varepsilon,2p}}^{2p}  \bigg]  \lesssim  1,
 \end{equation*} 
which corresponds to our claim.
\end{proof}


\section{Fourth order diagram 1}\label{Sect-fourth1}

By \eqref{ord4}, the diagram $\<PsiIPsi3>^{(n)}$ (which we will focus on in this section) is defined by
$$\<PsiIPsi3>^{(n)}:=\<Psi>^{(n)}\pe \<IPsi3>^{(n)}.  $$

Our main convergence statement for $(\<PsiIPsi3>^{(n)})$, implying the one in Proposition \ref{prop:conv-arbre}, reads as follows.

\begin{proposition}
Let $T>0$. For all $0<\eps,\eta<\frac12$, there exists $\ka>0$ such that for every $p\geq 1$,
\begin{eqnarray*}
 &\mathbb{E} \Big[ \Big\| \widetilde{ \<PsiIPsi3>}^{(n+1)}   - \widetilde{ \<PsiIPsi3>}^{(n)} \Big\|_{\cac^{1-\eps}([0,T]; \cb^{-\eta}_x)}^{2p} \Big]\lesssim 2^{-\ka n p } .
\end{eqnarray*}
As a result, the sequence $(\widetilde{ \<PsiIPsi3>}^{(n)})$ converges almost surely to an element $\widetilde{ \<PsiIPsi3>}$ in  $\cac^{1-\varepsilon}\big([0,T]; \cb^{-\eta}_x\big)$, for all $\varepsilon,\eta>0$.

Moreover, the following uniform in time estimate holds true:
\begin{equation}\label{L3}
\sup_{r\geq 0}    \mathbb{E} \Big[ \Big\|\widetilde{ \<PsiIPsi3>}^{(n+1)}   - \widetilde{ \<PsiIPsi3>}^{(n)}\Big\|_{{\ov \cac}^{1-\eps}([r,r+1]; \cb^{-\eta}_x)}^{2p} \Big]\lesssim 2^{-\ka n p }. 
\end{equation}
\end{proposition}

\medskip

For a more tractable expression of this process, let us write
\begin{eqnarray*}
\<PsiIPsi3>^{(n)}_t (x) &=&\sum_{i \sim i'} \delta_i\big(\<Psi>^{(n)}_{t}\big)(x) \delta_{i'}\big(\<IPsi3>^{(n)}_{t}\big)(x)\\
&=&\sum_{i \sim i'} \int dz dy \, \delta_i(x,z) \delta_{i'}(x,y)\<Psi>^{(n)}_{t}(z)\<IPsi3>^{(n)}_{t}(y) \\
& =&\sum_{i \sim i'} \int dz dy \, \delta_i(x,z) \delta_{i'}(x,y)\int_0^t ds \int dw \, K_{t-s}(y,w)\<Psi>^{(n)}_{t}(z)\<Psi3>^{(n)}_{s}(w)\\
&=&\sum_{i \sim i'} \int dz dy \, \delta_i(x,z) \delta_{i'}(x,y)\int_0^t ds \int dw \, K_{t-s}(y,w)I^W_1\big(F^{(n)}_{t,z}\big)I^W_3\big(F^{(n)}_{s,w} \otimes F^{(n)}_{s,w} \otimes F^{(n)}_{s,w} \big),
\end{eqnarray*}
which, by applying the product rule \eqref{prod-rule-mult-int}, yields the decomposition
\begin{align*}
&\<PsiIPsi3>^{(n)}_t (x) =\frakt^{\mathbf{1},(n)}_t (x) +\frakt^{\mathbf{2},(n)}_t (x) 
\end{align*}
with
\begin{align*}
\frakt^{\mathbf{1},(n)}_t (x) :=\sum_{i \sim i'} \int dz dy \, \delta_i(x,z) \delta_{i'}(x,y)\int_0^t ds \int dw \, K_{t-s}(y,w)I^W_4\Big(F^{(n)}_{t,z} \otimes  \big(F^{(n)}_{s,w} \otimes F^{(n)}_{s,w} \otimes F^{(n)}_{s,w} \big)\Big)
\end{align*}
and
\begin{align*}
\frakt^{\mathbf{2},(n)}_t (x) :=3 \int_0^t ds \int dw \, K_{t-s}(y,w) \,\cac^{(n)}_{t,s}(z,w) \,  I^W_2\big( F^{(n)}_{s,w} \otimes F^{(n)}_{s,w} \big).
\end{align*}

For $\mathbf{a}=1,2$, recall the notation (introduced in Definition \ref{defi:conv-c--la})
$$\frakti^{\mathbf{a},(n)}_t (x):=\int_0^t ds\, \frakt^{\mathbf{a},(n)}_s(x) .$$

\

For the sake of conciseness, we will only focus on the uniform bounds (for $\mathbf{a}=1,2$)
\begin{equation}\label{unif-fourth-1}
 \sup_{n\geq 1} \, \mathbb{E} \Big[ \big\|\frakti^{\mathbf{a},(n)} \big\|_{\cac^{1-\eps}([0,T];\cb_x^{-\eta})}^{2p} \Big] <\infty
\end{equation}
and 
\begin{equation}\label{unif-fourth-1bis}
 \sup_{r\geq 0} \sup_{n\geq 1} \, \mathbb{E} \Big[ \big\|\frakti^{\mathbf{a},(n)} \big\|_{{\ov \cac}^{1-\eps}([r,r+1];\cb_x^{-\eta})}^{2p} \Big] <\infty.
\end{equation}

\subsection{Study of $\frakti^{\mathbf{1},(n)}$}\label{subsec:ref-raiso}

 Fix $0<\varepsilon,\eta <\frac12$. For all $p\geq 1$, one has by \eqref{est-GRR} and the fact that $\frakti_0=0$,
 \begin{equation*}
\big\|\frakti^{\mathbf{1},(n)} \big\|_{\cac^{1-\varepsilon}([0,T];\cb_x^{-\eta})}^{2p}\lesssim \int_0^T\int_0^T dv_1 dv_2 \, \frac{\big\| \frakti^{\mathbf{1},(n)}_{v_2}-\frakti^{\mathbf{1},(n)}_{v_1}\big\|_{\cb_x^{-\eta}}^{2p}}{|v_2-v_1|^{2p(1-\varepsilon)+2}},
\end{equation*}
and then, just as in the previous sections, we can use the Sobolev embedding \eqref{sobo-beso} to assert that every $p\geq 1$ large enough,
 \begin{equation}
\big\|\frakti^{\mathbf{1},(n)} \big\|_{\cac^{1-\varepsilon}([0,T];\cb_x^{-\eta})}^{2p} \lesssim \int_0^T\int_0^T dv_1 dv_2 \, \frac{\big\| \frakti^{\mathbf{1},(n)}_{v_2}-\frakti^{\mathbf{1},(n)}_{v_1}\big\|_{L^{2p}_x}^{2p}}{|v_2-v_1|^{2p(1-\varepsilon)+2}}.\label{normlip}
\end{equation} 
  Similarly, under the same assumptions, with \eqref{est-GRR-bar}, we have that for all $r\geq 0$
 \begin{equation}\label{normlip1}
\big\|\frakti^{\mathbf{1},(n)} \big\|_{{\ov \cac}^{1-\varepsilon}([r,r+1];\cb_x^{-\eta})}^{2p} \lesssim \int_{r}^{r+1}\int_r^{r+1} dv_1 dv_2 \, \frac{\big\| \frakti^{\mathbf{1},(n)}_{v_2}-\frakti^{\mathbf{1},(n)}_{v_1}\big\|_{L^{2p}_x}^{2p}}{|v_2-v_1|^{2p(1-\varepsilon)+2}}.
\end{equation}

Now, using the hypercontractivity property,
\begin{align}
\mathbb{E}\Big[\big\| \frakti^{\mathbf{1},(n)}_{v_2}-\frakti^{\mathbf{1},(n)}_{v_1}\big\|_{L_x^{2p}}^{2p}\Big]&=\int dx \, \mathbb{E}\bigg[ \Big| \big(\frakti^{\mathbf{1},(n)}_{v_2}-\frakti^{\mathbf{1},(n)}_{v_1}\big)(x)\Big|^{2p}\bigg]\lesssim \int dx \, \mathbb{E}\bigg[ \Big| \big(\frakti^{\mathbf{1},(n)}_{v_2}-\frakti^{\mathbf{1},(n)}_{v_1}\big)(x)\Big|^{2}\bigg]^p,\label{hyperc}
\end{align}
and thus we only need to focus on the latter quantity.

\

In this setting, one has
\begin{align}
&\mathbb{E}\bigg[ \Big|  \big(\frakti^{\mathbf{1},(n)}_{v_2}-\frakti^{\mathbf{1},(n)}_{v_1}\big)(x)\Big|^{2}\bigg]=\mathbb{E}\bigg[ \bigg|\int_{v_1}^{v_2}dt\,    \frakt^{\mathbf{1},(n)}_t(x)\bigg|^2\bigg] =\int_{v_1}^{v_2}dt_1 \int_{v_1}^{v_2}dt_2\,  \mathbb{E}\Big[ \frakt^{\mathbf{1},(n)}_{t_1}(x)   \frakt^{\mathbf{1},(n)}_{t_2}(x) \Big],\label{uno}
\end{align}
with
\begin{align}\label{dos-0}
&\mathbb{E}\Big[ \frakt^{\mathbf{1},(n)}_{t_1}(x)   \frakt^{\mathbf{1},(n)}_{t_2}(x) \Big]=\int_0^{t_1}ds_1\int_0^{t_2}ds_2\, \sum_{i \sim i'} \sum_{j\sim j'}  \int dz_1 dy_1dz_2 dy_2 \, \delta_i(x,z_1) \delta_{i'}(x,y_1) \delta_j(x,z_2) \delta_{j'}(x,y_2)\,  \nonumber \\
&\hspace{0.5cm} \int dw_1 dw_2 \, K_{t_1-s_1}(y_1,w_2)K_{t_2-s_2}(y_2,w_2)\nonumber\\
&\hspace{1.5cm}\mathbb{E}\Big[I^W_4\Big(F^{(n)}_{t_1,z_1} \otimes  \big(F^{(n)}_{s_1,w_1} \otimes F^{(n)}_{s_1,w_1} \otimes F^{(n)}_{s_1,w_1} \big)\Big)I^W_4\Big(F^{(n)}_{t_2,z_2} \otimes  \big(F^{(n)}_{s_2,w_2} \otimes F^{(n)}_{s_2,w_2} \otimes F^{(n)}_{s_2,w_2} \big)\Big) \Big].
\end{align}

\

The latter expectation can easily be expanded as
\begin{align*}
&\mathbb{E}\Big[I^W_4\Big(F^{(n)}_{t_1,z_1} \otimes  \big(F^{(n)}_{s_1,w_1} \otimes F^{(n)}_{s_1,w_1} \otimes F^{(n)}_{s_1,w_1} \big)\Big)I^W_4\Big(F^{(n)}_{t_2,z_2} \otimes  \big(F^{(n)}_{s_2,w_2} \otimes F^{(n)}_{s_2,w_2} \otimes F^{(n)}_{s_2,w_2} \big)\Big) \Big]\\
&=c\, \Big\langle \text{Sym}\Big(F^{(n)}_{t,z_1} \otimes  \big(F^{(n)}_{s_1,w_1} \otimes F^{(n)}_{s_1,w_1} \otimes F^{(n)}_{s_1,w_1} \big)\Big), \text{Sym}\Big(F^{(n)}_{t,z_2} \otimes  \big(F^{(n)}_{s_2,w_2} \otimes F^{(n)}_{s_2,w_2} \otimes F^{(n)}_{s_2,w_2} \big)\Big) \Big\rangle_{L^2((\R_+\times \R^3)^4)}\\
&=\sum_{\mathbf{b}=1}^2 c_{\mathbf{b}} \cq^{\mathbf{1},\mathbf{b},(n)}_{t,s}({z,w})
\end{align*}
for some combinatorial coefficients $c,c_{\mathbf{b}}\geq 0$, and with
\begin{align*}
 \cq^{\mathbf{1},\mathbf{1},(n)}_{t,s}(z,w)&:=\cac^{(n)}_{t_1,t_2}(z_1,z_2)   \cac^{(n)}_{s_1,s_2}(w_1,w_2)^3,\\
 \cq^{\mathbf{1},\mathbf{2},(n)}_{t,s}(z,w)&:=\cac^{(n)}_{t_1,s_2}(z_1,w_2) \cac^{(n)}_{t_2,s_1}(z_2,w_1)\cac^{(n)}_{s_1,s_2}(w_1,w_2)^2.
\end{align*}

 Going back to \eqref{dos-0} and using the notation introduced in \eqref{def-op-T} for the operator $\mathcal{R}$, we obtain
\begin{align*}
&\mathbb{E}\Big[ \frakt^{\mathbf{1},(n)}_{t_1}(x)   \frakt^{\mathbf{1},(n)}_{t_2}(x) \Big]=\sum_{\mathbf{b}=1}^2 c_{\mathbf{b}}A^{\mathbf{1,b},(n)}_{t_1,t_2}(x),
\end{align*}
with
\begin{align*}
&A^{\mathbf{1,b},(n)}_{t_1,t_2}(x):=\int_0^{t_1} ds_1 \int_0^{t_2} ds_2 \, \mathcal{R}\Big(  \int dw_1 dw_2  K_{t_1-s_1}(y_1,w_1)   K_{t_2-s_2}(y_2,w_2)  \cq^{\mathbf{1},\mathbf{b},(n)}_{t,s}({z,w})\Big)(x).
\end{align*}

With this notation in hand, the problem can now be summed up as follows.
\begin{lemma}\label{Lem-defT}
Assume that for both $\mathbf{b}=1$ and $\mathbf{b}=2$, for all $\delta>0$ and all $p\geq 1$ large enough, one has
\begin{equation} \label{defi-Ap}
\sup_{n\geq 1} \, \| A^{\mathbf{1,b},(n)}_{t_1,t_2}\|_{L^p(\R^3)} \lesssim |t_2-t_1|^{-\delta}.
\end{equation}
Then, for all $T>0$, for all $0< \eps <\frac12$ and $\eta>0$, it holds that
\begin{eqnarray}\label{est-T}
&\dis \sup_{n\geq 1} \, \mathbb{E} \Big[ \big\|\frakti^{\mathbf{1},(n)} \big\|_{\cac^{1-\eps}([0,T];\cb_x^{-\eta})}^{2p} \Big] <\infty.
\end{eqnarray}
  Similarly, one has 
\begin{equation} \label{est-T1}
 \sup_{r\geq 0} \sup_{n\geq 1} \, \mathbb{E} \Big[ \big\|\frakti^{\mathbf{1},(n)} \big\|_{{\ov \cac}^{1-\eps}([r,r+1];\cb_x^{-\eta})}^{2p} \Big] <\infty.
\end{equation}
\end{lemma}

\begin{proof}
 Assume that \eqref{defi-Ap} holds true. Then we deduce that $\dis \sup_{n\geq 1}\,  \Big\| \mathbb{E}\big[ \frakt^{\mathbf{1},(n)}_{t_1}(\cdot )   \frakt^{\mathbf{1},(n)}_{t_2}(\cdot) \big]\Big\|_{L^p(\R^3)}  \lesssim |t_2-t_1|^{-\delta}$. After that,  by \eqref{elem},
\begin{equation}\label{dd}
 \sup_{n\geq 1} \,\Big\| \mathbb{E}\big[ \big|  \frakti^{\mathbf{1},(n)}_{v_2}-\frakti^{\mathbf{1},(n)}_{v_1}\big|^{2}\big] \Big\|_{L^p(\R^3)}   \lesssim |v_2-v_1|^{2-\delta}.
\end{equation}
We then plug \eqref{dd} into \eqref{normlip} and we can deduce \eqref{est-T}, choosing $\delta<\frac{\eps}{2}$.

\smallskip

The bound \eqref{est-T1} is obtained by the same argument, using \eqref{normlip1} instead of \eqref{normlip}. 
\end{proof}

\

The next two subsections are devoted to the proof of \eqref{defi-Ap} for $\mathbf{b}=1$ and $\mathbf{b}=2$.


\subsubsection{Estimation of $A^{\mathbf{1},\mathbf{1},(n)}_{t_1,t_2}$}

Given the expression \eqref{def-E} of $\cac^{(n)}$, we can recast $\cq^{\mathbf{1},\mathbf{1},(n)}_{t,s}({z,w})$ into
\begin{align}
\cq^{\mathbf{1},\mathbf{1},(n)}_{t,s}({z,w})&= I_{t,s}^{\mathbf{1},\mathbf{1},(n)} K^{\mathbf{1},\mathbf{1}}_{\sigma,\tau, \eta}(z,w) \nonumber\\
&:= \frac1{16}   \int_{|t_1-t_2|+\varepsilon_n}^{t_1+t_2+\varepsilon_n} d\sigma_1 \int_{|s_1-s_2|+\varepsilon_n}^{s_1+s_2+\varepsilon_n} d\tau_1   \int_{|s_1-s_2|+\varepsilon_n}^{s_1+s_2+\varepsilon_n} d\tau_2    \int_{|s_1-s_2|+\varepsilon_n}^{s_1+s_2+\varepsilon_n} d\tau_3 \, K^{\mathbf{1},\mathbf{1}}_{\sigma,\tau, \eta}(z,w),\label{repres-q11}
 \end{align}
 where 
 \begin{align*}
K^{\mathbf{1},\mathbf{1}}_{\sigma,\tau, \eta}(z,w) &:=K_{\sigma_1}(z_1,z_2) K_{\tau_1}(w_1,w_2)K_{\tau_2}(w_1,w_2)K_{\tau_3}(w_1,w_2).
 \end{align*}
With this notation, one has
\begin{eqnarray}\label{def-A11}
A^{\mathbf{1},\mathbf{1},(n)}_{t_1,t_2}(x)&=&  \int_0^{t_1} ds_1 \int_0^{t_2} ds_2   \, \mathcal{R}\Big( \int dw_1  dw_2  K_{t_1-s_1}(y_1,w_1)   K_{t_2-s_2}(y_2,w_2) \cq^{\mathbf{1},\mathbf{1},(n)}_{t,s}({ z,w})\big)(x)\nonumber  \\
&=& \int_0^{t_1} ds_1 \int_0^{t_2} ds_2\,  I_{t,s}^{\mathbf{1},\mathbf{1},(n)} \big[\mathcal{R}\cf^{\mathbf{1},\mathbf{1}}\big](x),
 \end{eqnarray}
where  the function $\cf^{\mathbf{1},\mathbf{1}}$ is defined by 
\begin{equation} \label{defi-F}
 \cf^{\mathbf{1},\mathbf{1}}(y_1,y_2,z_1,z_2):=   \int dw_1 dw_2 \,K_{t_1-s_1}(y_1,w_1)K_{t_2-s_2}(y_2,w_2)K^{\mathbf{1},\mathbf{1}}_{\sigma,\tau,\eta}(z_1,w_1,z_2,w_2).
\end{equation} 

\

 We now have to bound the latter function with respect to the norms involved in  \eqref{norme-LP} and Corollary~\ref{coro:interpol-T}. To this end, set  
\begin{equation}\label{def-k}
\kappa_i=  2\pi \sinh \big(2(t_i-s_i)\big), \quad \nu_i=  2\pi \sinh \big(2\sigma_i\big), \quad \rho_i =\tanh(t_i-s_i), \quad \mu_i=\tanh(\sigma_i).
\end{equation}
 Notice that in the sequel, all these quantities are positive, and since $0<s_i < t_i \leq T$ and $0<\sigma_i \leq t_1+t_2 \leq 2T+2$  we have $\kappa_i  \sim \rho_i  \sim t_i-s_i $ and $\nu_i  \sim  \mu_i \sim \sigma_i$. Since all these constants are $\lesssim 1$, below we will be able to systematically use the bound $\rho^{-1}_i+\rho_i \lesssim \rho^{-1}_i$ and so on. Observe however that the constants in the estimates will depend on the  time $T>0$. \medskip

By the Mehler formula \eqref{mehler2} and the notations \eqref{def-k}
\begin{equation*}  
 K_{t_i-s_i}(x,y)  =  \kappa_i^{-\frac32} \exp\left(- \frac{ \vert x-y\vert^2}{4\rho_i}-\frac{\rho_i}{4}\vert x+y\vert^2\right)
\end{equation*}
\begin{equation*}  
K_{\sigma_i}(x,y)  =\nu_i^{-\frac32} \exp\left(- \frac{ \vert x-y\vert^2}{4\mu_i}-\frac{\mu_i}{4}\vert x+y\vert^2\right).
\end{equation*}

We can prove the following result.
     \begin{lemma}  \label{lem-F1}
     The following bounds hold true 
     \begin{equation}  \label{f1-inf}
  \|\cf^{\mathbf{1},\mathbf{1}}\|_{L^{\infty}(\R^{12})}   \lesssim  \sigma^{-\frac32}_1 (\tau_1 \tau_2 \tau_3)^{-\frac54} (t_2-s_2)^{-\frac34},
 \end{equation}
and for all $q>\frac32$
\begin{equation}  \label{grad-f1}
 \sup_{y_1,y_2,z_2 \in \R^3} \big(\int dz_1 \big| H_{y_1}  \cf^{\mathbf{1},\mathbf{1}}\big|^q\big)^{\frac{1}{q}}  \lesssim \sigma^{\frac3{2q}}_1  (t_1-s_1)^{-1}   \sigma^{-\frac32}_1 (\tau_1 \tau_2 \tau_3)^{-\frac54} (t_2-s_2)^{-\frac34}.
 \end{equation}
 \end{lemma}

 Observe that in \eqref{f1-inf}, the contribution of $\sigma^{-\frac32}_1$ is bad and can not be controlled by a term appearing the minimum, that is why we need to switch derivatives in the bounds of Lemma \ref{conti}.\medskip

\begin{proof}
We first prove  \eqref{f1-inf}. 
For all $w_1,y_2 \in \R^3$
\begin{multline}\label{est-e1}
\int dw_2  \, K_{t_2-s_2}(y_2,w_2)  K_{\tau_1}(w_1,w_2)K_{\tau_2}(w_1,w_2)K_{\tau_3}(w_1,w_2) \leq \\
\begin{aligned}
 & \leq \|K_{t_2-s_2} {(y_2, \cdot)   \|_{L^2_{w_2}}} \| K_{\tau_1}(w_1,\cdot)K_{\tau_2}(w_1,\cdot)K_{\tau_3}(w_1,\cdot)\|_{L_{w_2}^2} \\
  & \lesssim \|K_{t_2-s_2} {(y_2, \cdot)   \|_{L^2_{w_2}}} \| K_{\tau_1}(w_1,\cdot)\|_{L_{w_2}^6} \| K_{\tau_2}(w_1,\cdot)\|_{L_{w_2}^6} \| K_{\tau_3}(w_1,\cdot)\|_{L_{w_2}^6} \\
  & \lesssim (t_2-s_2)^{-\frac34} (\tau_1 \tau_2 \tau_3)^{-\frac54},
\end{aligned}
\end{multline}
thanks to \eqref{normeLpp}. Then for all $y_1,y_2,z_1,z_2 \in \R^3$, 
\begin{eqnarray}\label{est-e2}
\cf^{\mathbf{1},\mathbf{1}}(y_1,y_2,z_1,z_2)&\lesssim & { \|K_{\sigma_1}\|_{L^\infty_{z_1,z_2}}  \|K_{t_1-s_1}(y_1, \cdot)   \|_{L_{w_1}^1} }(t_2-s_2)^{-\frac34} (\tau_1 \tau_2 \tau_3)^{-\frac54}\\
&\lesssim &   \sigma^{-\frac32}_1(t_2-s_2)^{-\frac34} (\tau_1 \tau_2 \tau_3)^{-\frac54}, \nonumber
\end{eqnarray}
which was the claim. \medskip

 We now turn to the proof of \eqref{grad-f1}. Using the expression of $\cf^{\mathbf{1},\mathbf{1}}$, we have 
\begin{multline*} 
 \big(H_{y_1}\cf^{\mathbf{1},\mathbf{1}} \big)(y_1,y_2,z_1,z_2)=\\
 =  \int dw_1 dw_2 \,\big( H_{y_1}K_{t_1-s_1}\big)(y_1,w_1)K_{t_2-s_2}(y_2,w_2)K_{\sigma_1}(z_1,z_2) K_{\tau_1}(w_1,w_2)K_{\tau_2}(w_1,w_2)K_{\tau_3}(w_1,w_2) ,
\end{multline*} 
 and thus 
 \begin{multline*} 
 \|\big(H_{y_1}\cf^{\mathbf{1},\mathbf{1}} \big)(y_1,y_2,\cdot ,z_2)\|_{L_{z_1}^q(\R^3)}\lesssim \\
    \int dw_1 dw_2 \,\Big| \big( H_{y_1}K_{t_1-s_1}\big)(y_1,w_1)\Big| K_{t_2-s_2}(y_2,w_2)K_{\tau_1}(w_1,w_2)K_{\tau_2}(w_1,w_2)K_{\tau_3}(w_1,w_2)  \big\| K_{\sigma_1}(\cdot,z_2)   \big\|_{L_{z_1}^q(\R^3)}.
\end{multline*}
Now we make the same estimates as in \eqref{est-e1} and \eqref{est-e2} (using \eqref{normeLpp}) and get 
$$ \|\big(H_{y_1}\cf^{\mathbf{1},\mathbf{1}} \big)(y_1,y_2,\cdot ,z_2)\|_{L_{z_1}^q(\R^3)}\lesssim    \sigma^{\frac3{2q}-\frac32}_1    (t_2-s_2)^{-\frac34} (\tau_1 \tau_2 \tau_3)^{{ -\frac54}}  \Big( \int dw_1   \,\Big| \big( H_{y_1}K_{t_1-s_1}\big)(y_1,w_1)\Big|  \Big). $$
 To complete the proof of \eqref{grad-f1}, it remains to check that 
\begin{equation}\label{L1Linf}
  \int dw_1   \,\Big| \big( H_{y_1}K_{t_1-s_1}\big)(y_1,w_1)\Big|   \lesssim (t_1-s_1)^{-1}.
  \end{equation}
Since $|t_1-s_1| \le 1$, this is in fact a consequence of the point-wise bound
\begin{equation*} 
\Big| \big( H_{y_1}K_{t_1-s_1}\big)(y_1,w_1)\Big| \lesssim (t_1-s_1)^{-\frac52} \exp{\Big( - \frac{1}{8 \rho_1}|y_1-w_1|^2  \Big)}
 \end{equation*}
 which in turn follows from \eqref{borne-F22}.
 \end{proof}

 \begin{remark}\label{rema-grad}
 Observe that  the estimate \eqref{f1-inf} is of the form $ \|\cf\|_{L^{\infty}(\R^{12})}   \lesssim L(\sigma, \tau, \eta, t-s)$ for some positive function $L$, and has been obtained using only the H\"older inequality and the bounds \eqref{normeLpp} on the functions~$K$ appearing in the definition of $\cf$. Then one can deduce the estimate \eqref{grad-f1} which reads
 $$ \sup_{y_1,y_2,z_2 \in \R^3} \big(\int dz_1 \big| H_{y_1}  \cf\big|^q\big)^{\frac1q}  \lesssim \sigma^{\frac3{2q}}_1  (t_1-s_1)^{-1} L(\sigma, \tau, \eta, t-s).$$
 This latter bound was obtained using the same H\"older estimates, together with the bounds \eqref{L1Linf} and
 $$\|K_{\sigma_1}(\cdot , z_2)\|_{L^q_{z_1}} \lesssim  \sigma^{\frac3{2q}}_1 \|K_{\sigma_1}\|_{L^\infty_{z_1,z_2}}. $$
 Similarly, one can prove 
  $$ \sup_{y_1,y_2,z_1 \in \R^3} \big(\int dz_2 \big| H_{y_2}  \cf\big|^q\big)^{\frac1q}  \lesssim \sigma^{\frac3{2q}}_1  (t_2-s_2)^{-1} L(\sigma, \tau, \eta, t-s),$$
  but this estimate is not needed in the case $\cf=\cf^{\mathbf{1},\mathbf{1}}$.
  \end{remark}\medskip

 We now state some crude estimates for $\cf^{\mathbf{1},\mathbf{1}}$.
 
  \begin{lemma}\label{lemma-crude}
 There exists $N\geq 1$ such that 
        \begin{multline}  \label{724}
 \max\Big(\|\cf^{\mathbf{1},\mathbf{1}}\|_{\mathcal{H}^{16}(\R^{12})}  ,\|H_{y_1}  H_{y_2}\cf^{\mathbf{1},\mathbf{1}}\|_{L^{\infty}(\R^{12})}, \|H^{2}_{y_1}H^{-1}_{z_1} H_{y_2}\cf^{\mathbf{1},\mathbf{1}}\|_{L^{\infty}(\R^{12})} \Big)   \lesssim  \\
  \lesssim  \big(\sigma_1 \tau_1 \tau_2 \tau_3 (t_1-s_1)(t_2-s_2)   \big)^{-N}.
 \end{multline}
   \end{lemma} 
 
 Notice that we do not need to evaluate precisely the exponent $N\geq 1$ which appears in \eqref{724}, since through our subsequent application of Corollary \ref{coro:interpol-T}, these bounds will only be handled with a small power.

   \begin{proof} 
We simply observe that, thanks to \eqref{borne-F23}, any power $H^n\cf^{\mathbf{1},\mathbf{1}} $ can be point-wise controlled by a term of the form $  \big(\sigma_1 \tau_1 \tau_2 \tau_3 (t_1-s_1)(t_2-s_2)   \big)^{-N}\sqrt{\cf^{\mathbf{1},\mathbf{1}}} $, hence the result. 
   \end{proof}

 \begin{lemma} \label{lem:a11}
  Let $0<\eps<\frac12$. Then if $p\geq 1$ is large enough  
          \begin{equation}\label{TF1}
   \|\mathcal{R}\cf^{\mathbf{1},\mathbf{1}}\|_{ {L}^{p}({ \R^{3}})} \lesssim \sigma^{-1-\eps}_1 (\tau_1\tau_2\tau_3 )^{-\frac54-\eps}     (t_1-s_1)^{-\frac12-\eps}(t_2-s_2)^{ -\frac34-\eps} ,    \end{equation}
   and 
             \begin{equation}\label{intA1}
\sup_{n\geq 1}\, \| A^{\mathbf{1},\mathbf{1},(n)}_{t_1,t_2}\|_{L^p(\R^3)} \lesssim |t_2-t_1| ^{-6\eps}.
   \end{equation}
 \end{lemma}
 
\

 \begin{remark}
  At this point, let us explain the strategy to control the different terms $\|\mathcal{R}\cf\|_{ {L}^{p}({ \R^{3}})} $ (see also Section \ref{sect-fourth}). In each of the cases, we prove a bound of the form  
\begin{multline}\label{TFm}
\|\mathcal{R}\cf\|_{ {L}^{p}({ \R^{3}})} \lesssim \\
\lesssim \sigma^{-\alpha_1-\eps}_1 \sigma^{-\alpha_2-\eps}_2 \tau^{-\beta_1-\eps}_1 \tau^{-\beta_2-\eps}_2 \eta^{-\gamma_1-\eps}_1 \eta^{-\gamma_2-\eps}_2\eta^{-\gamma_3-\eps}_3 \eta^{-\gamma_4-\eps}_4 (t_1-s_1)^{-\delta_1-\eps}(t_2-s_2)^{-\delta_2-\eps}, 
\end{multline}
for some parameters $0 \leq \alpha_j, \beta_j, \gamma_j , \delta_j\leq \frac32$ which are independent of $\eps>0$ and where $\alpha_1+ \alpha_2+\beta_1+ \beta_2+ \gamma_1+\gamma_2+\gamma_3+\gamma_4+ \delta_1+\delta_2=6$. By the formula \eqref{defi-F}, at most~4 of the parameters $\alpha_j, \beta_j, \gamma_j $ are different from~0. These parameters have to be chosen in such a way that one gets the bound $\dis \| A_{t_1,t_2}\|_{L^p(\R^3)} \lesssim |t_2-t_1| ^{-6\eps}$ after integration of \eqref{TFm}, thanks to the use of Lemma \ref{Lem-prod} or  Lemma \ref{Lem-prod2}. For instance, this  in particular imposes that   $\alpha_1, \alpha_2 \leq 1$, $\beta_1+\beta_2 <3$, and $\delta_1, \delta_2 <1$, but is seems difficult to give here easy sufficient and necessary conditions on all the parameters; we will handle these bounds case by case.
 \end{remark}

\smallskip
 
  \begin{proof}  
 Fix $0<\eps<\frac12$. By applying Corollary \ref{coro:interpol-T} with $\la_1=\la_2=\frac12$ and $\la_3=0$, we get that for every $q>\frac32$ and every $p\geq 1$ large enough,
 \begin{multline}\label{quantif-eps}
\|\mathcal{R}\cf^{\mathbf{1},\mathbf{1}}\|_{L^p(\R^3)} \lesssim \Big( 1\vee\|\cf^{\mathbf{1},\mathbf{1}}\|_{L^{\infty}(\R^{12})}\Big)^{\frac12}   \Big(1\vee \sup_{y_1,y_2,z_2 \in \R^3} \big(\int dz_1 \big| H_{y_1}  \cf^{\mathbf{1},\mathbf{1}}\big|^q\big)^{\frac1q} \Big)^{\frac12} \\
\Big(\|H_{y_1}  H_{y_2}\cf^{\mathbf{1},\mathbf{1}}\|_{L^{\infty}(\R^{12})}\Big)^{\frac{\eps}{6N}} \Big(\|H^{2}_{y_1}H^{-1}_{z_1} H_{y_2}\cf^{\mathbf{1},\mathbf{1}}\|_{L^{\infty}(\R^{12})} \Big)^{\frac{\eps}{6N}} 
\Big(1\vee \|\cf^{\mathbf{1},\mathbf{1}}\|_{\mathcal{H}^{16}(\R^{12})}\Big)^{\frac{\eps}{6N}}.
\end{multline}
We can now inject the estimates of Lemma \ref{lem-F1} and Lemma \ref{lemma-crude}, which gives, for every $q>\frac32$ and every $p\geq 1$ large enough,
\begin{align*}
&\|\mathcal{R}\cf^{\mathbf{1},\mathbf{1}}\|_{L^p(\R^3)} 
 \lesssim  \sigma^{-\frac12(3-\frac{3}{2q})}_1 \big(  \tau_1 \tau_2 \tau_3  \big)^{-\frac54}  (t_1-s_1)^{-\frac12}  (t_2-s_2)^{-\frac34}   
\Big(\sigma_1 \tau_1 \tau_2  \tau_3  (t_1-s_1)(t_2-s_2)   \Big)^{-\frac{\eps}{2}} .
\end{align*}
Finally, by choosing $q>\frac32$ such that $\frac12(3-\frac{3}{2q}) = 1+\frac{\eps}{2}$, we deduce that for every $p\geq 1$ large enough,
\begin{align*}
&\|\mathcal{R}\cf^{\mathbf{1},\mathbf{1}}\|_{L^p(\R^3)} \lesssim \sigma_1^{-1-\frac{\eps}{2}}\big( \tau_1 \tau_2 \tau_3 \big)^{-\frac54}  (t_1-s_1)^{-\frac12}  (t_2-s_2)^{-\frac34}  \Big(\sigma_1 \tau_1 \tau_2 \tau_3 (t_1-s_1)(t_2-s_2)   \Big)^{-\frac{\eps}{2}} ,
\end{align*}
and \eqref{TF1} immediately follows.
  \medskip

The estimate \eqref{intA1} can then be derived from successive integrations of \eqref{TF1} and using \eqref{xayb}. Namely, recalling~\eqref{def-A11}, we integrate in    $\tau_1, \tau_2, \tau_2 , \sigma_1$ and get 
\begin{align*} 
\sup_{n\geq 1}\, \big\|     I_{t,s}^{\mathbf{1},\mathbf{1},(n)} \big[\mathcal{R}\cf^{\mathbf{1},\mathbf{1}}\big]\big\|_{L^p(\R^3)}&\lesssim \int_{|t_1-t_2|}^{+\infty} d\sigma_1 \int_{|s_1-s_2|}^{+\infty} d\tau_1   \int_{|s_1-s_2|}^{+\infty} d\tau_2    \int_{|s_1-s_2|}^{+\infty} d\tau_3 \, \big\|   \mathcal{R}\cf^{\mathbf{1},\mathbf{1}}\big\|_{L^p(\R^3)}\\
& \lesssim  |t_1-t_2|^{-\eps}|s_1-s_2 |^{-\frac34-3\eps}   (t_1-s_1)^{-\frac12-\eps}(t_2-s_2)^{ -\frac34-\eps}.
\end{align*}
Then by \eqref{xayb}, an integration in the variable $s_1$ gives 
\begin{equation*} 
\sup_{n\geq 1}\, \int_{0}^{t_1} ds_1 \big\|     I_{t,s}^{\mathbf{1},\mathbf{1},(n)} \big[\mathcal{R}\cf^{\mathbf{1},\mathbf{1}}\big]\big\|_{L^p(\R^3)} \lesssim  |t_1-t_2|^{-\eps}|t_1-s_2 |^{-\frac14-4\eps}    (t_2-s_2)^{ -\frac34-\eps}.
\end{equation*}
 Finally, an integration in the variable $s_2$ and by Lemma \ref{Lem-prod} again, we get 
\begin{equation*} 
\sup_{n\geq 1}\,\| A^{\mathbf{1},\mathbf{1},(n)}_{t_1,t_2}\|_{L^p(\R^3)} \lesssim |t_2-t_1|^{-6\eps},
\end{equation*}
   which was the claim.
\end{proof}

\smallskip
  
\subsubsection{Estimation of $A^{\mathbf{1},\mathbf{2},(n)}_{t_1,t_2}$} 

We follow the same overall procedure as for $A^{\mathbf{1},\mathbf{1},(n)}$: along the pattern of~\eqref{repres-q11}, we first write
\begin{align*}
&\cq^{\mathbf{1},\mathbf{2},(n)}_{t,s}(z,w)=\\
&= I_{t,s}^{\mathbf{1},\mathbf{2},(n)} K^{\mathbf{1},\mathbf{2}}_{\sigma,\tau, \eta}(z,w) :=\frac1{16}   \int_{|t_1-s_2|+\eps_n}^{t_1+s_2+\eps_n} d\eta_1 \int_{|t_2-s_1|+\eps_n}^{t_2+s_1+\eps_n} d\eta_2   \int_{|s_1-s_2|+\eps_n}^{s_1+s_2+\eps_n} d\tau_1    \int_{|s_1-s_2|+\eps_n}^{s_1+s_2+\eps_n} d\tau_2 \, K^{\mathbf{1},\mathbf{2}}_{\sigma,\tau, \eta}(z,w),
 \end{align*}
 where 
 \begin{align*}
&K^{\mathbf{1},\mathbf{2}}_{\sigma,\tau, \eta}(z,w) :=K_{\eta_1}(z_1,w_2) K_{\eta_2}(z_2,w_1)K_{\tau_1}(w_1,w_2)K_{\tau_2}(w_1,w_2).
 \end{align*}
This leads us to the expression
\begin{align}\label{def-A12}
A^{\mathbf{1},\mathbf{2},(n)}_{t_1,t_2}(x)
&= \int_0^{t_1} ds_1 \int_0^{t_2} ds_2\, I^{\mathbf{1},\mathbf{2},(n)}_{t,s} \big[\mathcal{R}\cf^{\mathbf{1},\mathbf{2}}\big](x),
 \end{align}
with 
\begin{equation*} 
\cf^{\mathbf{1},\mathbf{2}}(y_1,y_2,z_1,z_2):=   \int dw_1 dw_2 \,K_{t_1-s_1}(y_1,w_1)K_{t_2-s_2}(y_2,w_2)K^{\mathbf{1},\mathbf{2}}_{\sigma,\tau,\eta}(z_1,w_1,z_2,w_2) .
\end{equation*}

\

 \begin{lemma}  \label{lem:f2-inf}
For any  $0<\delta< \frac14$, the following bound holds true: 
     \begin{equation}\label{f2-inf}  
 \|\cf^{\mathbf{1},\mathbf{2}}\|_{L^{\infty}(\R^{12})}   \lesssim  (t_1-s_1)^{-3\delta} (t_2-s_2)^{-3\delta} \eta_1^{-\frac32+\delta}\eta_2^{-\frac32+3\delta}(\tau_1\tau_2)^{-\frac32+\delta}.
 \end{equation}
  Moreover, there exists $N\geq 1$ such that 
        \begin{equation*} 
 \max\Big(\|\cf^{\mathbf{1},\mathbf{2}}\|_{\mathcal{H}^{16}(\R^{12})} ,\|H_{y_1}  H_{y_2}\cf^{\mathbf{1},\mathbf{2}}\|_{L^{\infty}(\R^{12})}  \Big) \lesssim    \big(\eta_1 \eta_2 \tau_1 \tau_2 (t_1-s_1)(t_2-s_2)   \big)^{-N}.
 \end{equation*}

 \end{lemma}

\begin{proof} 
We only prove \eqref{f2-inf}, since the other estimates are obtained as in Lemma \ref{lemma-crude}.  Let $1<r_1, r_2<\infty$ be such that $\frac{1}{r_1}+\frac{1}{r_2}=1$. Then for all $y_2,z_1,w_1 \in \R^3$,
\begin{multline*} 
\int dw_2  \,  K_{t_2-s_2}(y_2,w_2)K_{\eta_1}(z_1,w_2) K_{\tau_1}(w_1,w_2)K_{\tau_2}(w_1,w_2)  \leq \\
\begin{aligned}
 & \leq    \|K_{t_2-s_2} (y_2, \cdot)   \|_{L^{r_1}_{w_2}}  \|  K_{\eta_1} (z_1, \cdot) K_{\tau_1}(w_1,\cdot)K_{\tau_2}(w_1,\cdot)\|_{L_{w_2}^{r_2}} \\
  & \lesssim  (t_2-s_2)^{-\frac32+\frac{3}{2r_1}}  \| K_{\eta_1}(w_1,\cdot)\|_{L_{w_2}^{3r_2}} \| K_{\tau_1}(w_1,\cdot)\|_{L_{w_2}^{3r_2}}   \| K_{\tau_2}(w_1,\cdot)\|_{L_{w_2}^{3r_2}} \\
  & \lesssim (t_2-s_2)^{-\frac32+\frac{3}{2r_1}} \eta_1^{-\frac32+\frac{1}{2r_2}}(\tau_1\tau_2)^{-\frac32+\frac{1}{2r_2}} .
\end{aligned}
\end{multline*}
Consequently, for all $y_1,y_2,z_1,z_2$, 
\begin{eqnarray*} 
\cf^{\mathbf{1},\mathbf{2}}(y_1,y_2,z_1,z_2)&\lesssim & { \|K_{\eta_2}(z_2,.)\|_{L^{r_2}_{w_1}}  \|K_{t_1-s_1}(y_1, \cdot)   \|_{L_{w_1}^{r_1}} }(t_2-s_2)^{-\frac32+\frac{3}{2r_1}} \eta_1^{-\frac32+\frac{1}{2r_2}}(\tau_1\tau_2)^{-\frac32+\frac{1}{2r_2}}\\
&\lesssim &  (t_1-s_1)^{-\frac32+\frac{3}{2r_1}} (t_2-s_2)^{-\frac32+\frac{3}{2r_1}} \eta_1^{-\frac32+\frac{1}{2r_2}}\eta_2^{-\frac32+\frac{3}{2r_2}}(\tau_1\tau_2)^{-\frac32+\frac{1}{2r_2}}, \nonumber
\end{eqnarray*}
and we get the result by setting $\delta=\frac1{2r_2}$.
\end{proof}

 \begin{lemma} \label{lem:a12}
Let $0<\delta< \frac14$ and $\eps>0$. Then if $p\geq 1$ is large enough  
\begin{equation}\label{TF2}
  \|\mathcal{R}\cf^{\mathbf{1},\mathbf{2}}\|_{ {L}^{p}({ \R^{3}})} \lesssim  (t_1-s_1)^{-3\delta-\eps} (t_2-s_2)^{-3\delta-\eps} \eta_1^{-\frac32+\delta-\eps}\eta_2^{-\frac32+3\delta-\eps}(\tau_1\tau_2)^{-\frac32+\delta-\eps} ,  
\end{equation}
   and 
\begin{equation}\label{intA2}
\sup_{n\geq 1}\, \| A^{\mathbf{1},\mathbf{2},(n)}_{t_1,t_2}\|_{L^p(\R^3)} \lesssim |t_2-t_1| ^{-6\eps}.
 \end{equation}
 \end{lemma}
 
 \begin{proof}
 Fix $0<\delta< \frac14$ and $\eps>0$. By applying Corollary \ref{coro:interpol-T} with $\la_1=1$, $\la_2=\la_3=0$, and then using the bounds contained in Lemma \ref{lem:f2-inf}, we deduce that for every $p\geq 1$ large enough,
\begin{align*}
\|\mathcal{R}\cf^{\mathbf{1},\mathbf{2}}\|_{L^p(\R^3)}& \lesssim \Big(1 \vee \|\cf^{\mathbf{1},\mathbf{2}}\|_{L^{\infty}(\R^{12})} \Big)\Big(1\vee \|\cf^{\mathbf{1},\mathbf{2}}\|_{\mathcal{H}^{16}(\R^{12})}\Big)^{\frac{\eps}{2N}} \Big(\|H_{y_1}  H_{y_2}\cf^{\mathbf{1},\mathbf{2}}\|_{L^{\infty}(\R^{12})}\Big)^{\frac{\eps}{2N}}  \\
&\lesssim (t_1-s_1)^{-3\delta} (t_2-s_2)^{-3\delta} \eta_1^{-\frac32+\delta}\eta_2^{-\frac32+3\delta}(\tau_1\tau_2)^{-\frac32+\delta} \Big(\eta_1 \eta_2 \tau_1 \tau_2 (t_1-s_1)(t_2-s_2) \Big)^{-\eps},
\end{align*}
which gives \eqref{TF2}.
	
	\medskip
  
The bound \eqref{intA2} is then obtained by successive integrations of \eqref{TF2}. Namely, with \eqref{def-A12} in mind, we integrate in $\tau_1, \tau_2, \eta_1 , \eta_2$ and get 
\begin{align*} 
&\sup_{n\geq 1}\, \big\| I_{t,s}^{\mathbf{1},\mathbf{2},(n)} \big[\mathcal{R}\cf^{\mathbf{1},\mathbf{2}}\big]\big\|_{L^p(\R^3)} \lesssim   \int_{|t_1-s_2|}^{+\infty} d\eta_1 \int_{|t_2-s_1|}^{+\infty} d\eta_2   \int_{|s_1-s_2|}^{+\infty} d\tau_1    \int_{|s_1-s_2|}^{+\infty} d\tau_2 \, \big\| \mathcal{R}\cf^{\mathbf{1},\mathbf{2}}\big\|_{L^p(\R^3)}\\
&\hspace{3cm}\lesssim  (t_1-s_1)^{-3\delta-\eps} (t_2-s_2)^{-3\delta-\eps}    |t_2-s_1|^{-\frac12+\delta-\eps}  |t_1-s_2|^{-\frac12+3\delta-\eps}   |s_1 -s_2 |^{-1+2\delta-2\eps}.
\end{align*}
We now apply Lemma \ref{Lem-prod2} to derive
  \begin{equation*} 
\sup_{n\geq 1}\, \| A^{\mathbf{1},\mathbf{2},(n)}_{t_1,t_2}\|_{L^p(\R^3)}\leq \sup_{n\geq 1}\,\int_0^{t_1} ds_1 \int_0^{t_2} ds_2 \; \big\| I_{t,s}^{\mathbf{1},\mathbf{2},(n)}\big[\mathcal{R}\cf^{\mathbf{1},\mathbf{2}}\big]\big\|_{L^p(\R^3)}    \lesssim |t_2-t_1|^{-6\eps},
   \end{equation*}
   which corresponds to our assertion.  
 \end{proof}

\

Lemmas \ref{lem:a11} and \ref{lem:a12} thus guarantee that the condition \eqref{defi-Ap} is indeed satisfied. Therefore we can appeal to Lemma \ref{Lem-defT} and deduce the desired bound for $ \frakti^{\mathbf{1},(n)}$, that is
\begin{equation}\label{conc-frak1}
\sup_{n\geq 1}\, \mathbb{E}\Big[\big\|\frakti^{\mathbf{1},(n)} \big\|_{\cac^{1-\varepsilon}([0,T]; \cb_x^{-\eta})}^{2p}\Big]<\infty,
\end{equation}
and  
\begin{equation*} 
 \sup_{r\geq 0} \sup_{n\geq 1} \, \mathbb{E} \Big[ \big\|\frakti^{\mathbf{1},(n)} \big\|_{{\ov \cac}^{1-\eps}([r,r+1];\cb_x^{-\eta})}^{2p} \Big] <\infty,
\end{equation*}
for all $T>0$, $0<\eps,\eta<\frac12$.

\
	
\subsection{Study of $\frakti^{\mathbf{2},(n)}$} 

Let $T>0$. For the same reasons as in \eqref{normlip}-\eqref{hyperc}-\eqref{uno}, it holds that 
\begin{multline}
\mathbb{E}\Big[\big\|\frakti^{\mathbf{2},(n)} \big\|_{\cac^{1-\varepsilon}([0,T]; \cb_x^{-\eta})}^{2p}\Big] \lesssim \\
\lesssim \int_0^T\int_0^T  \frac{dv_1 dv_2}{|v_2-v_1|^{2p(1-\varepsilon)+2}}\bigg( \int_{[v_1,v_2]^2}dt_1 dt_2\, \bigg(\int dx \,  \mathbb{E}\Big[ \frakt^{\mathbf{2},(n)}_{t_1}(x)   \frakt^{\mathbf{2},(n)}_{t_2}(x) \Big]^p\bigg)^{\frac1p} \bigg)^p,\label{frakt2-main}
\end{multline} 
for all $0<\eps,\eta<\frac12$ and $p\geq 2$ large enough.  Similarly, under the same assumptions, with \eqref{est-GRR-bar}, we have that for all $r\geq 0$
\begin{multline}
\mathbb{E}\Big[\big\|\frakti^{\mathbf{2},(n)} \big\|_{{\ov\cac}^{1-\varepsilon}([r,r+1]; \cb_x^{-\eta})}^{2p}\Big] \lesssim \\
\lesssim \int_r^{r+1}\int_r^{r+1}  \frac{dv_1 dv_2}{|v_2-v_1|^{2p(1-\varepsilon)+2}}\bigg( \int_{[v_1,v_2]^2}dt_1 dt_2\, \bigg(\int dx \,  \mathbb{E}\Big[ \frakt^{\mathbf{2},(n)}_{t_1}(x)   \frakt^{\mathbf{2},(n)}_{t_2}(x) \Big]^p\bigg)^{\frac1p} \bigg)^p.\label{frakt2-main-bis}
\end{multline}

Then one has  
\begin{align*}
&\mathbb{E}\Big[ \frakt^{\mathbf{2},(n)}_{t_1}(x)   \frakt^{\mathbf{2},(n)}_{t_2}(x) \Big]
=c\int_0^{t_1}ds_1\int_0^{t_2}ds_2\, \sum_{i \sim i'} \sum_{j\sim j'}  \int dz_1 dy_1dz_2 dy_2 \, \delta_i(x,z_1) \delta_{i'}(x,y_1) \delta_j(x,z_2) \delta_{j'}(x,y_2)\,  \nonumber \\
&\hspace{7cm} \int dw_1 dw_2 \, K_{t_1-s_1}(y_1,w_2)K_{t_2-s_2}(y_2,w_2)  \cq^{\mathbf{2},(n)}_{t,s}(z,w),
\end{align*}
for some combinatorial coefficient $c\geq 0$ and where
\begin{align*}
 \cq^{\mathbf{2},(n)}_{t,s}(z,w)&:=\cac^{(n)}_{t_1,s_1}(z_1,w_1) \cac^{(n)}_{t_2,s_2}(z_2,w_2) \cac^{(n)}_{s_1,s_2}(w_1,w_2)^2.
\end{align*}
In other words,
\begin{align*}
&\mathbb{E}\Big[ \frakt^{\mathbf{2},(n)}_{t_1}(x)   \frakt^{\mathbf{2},(n)}_{t_2}(x) \Big]=c\int_0^{t_1} ds_1 \int_0^{t_2} ds_2 \, \mathcal{R}\Big(  \int dw_1 dw_2  K_{t_1-s_1}(y_1,w_1)   K_{t_2-s_2}(y_2,w_2)  \cq^{\mathbf{2},(n)}_{t,s}({z,w})\Big)(x).
\end{align*}
Just as in \eqref{repres-q11}, we can recast $ \cq^{\mathbf{2},(n)}$ into
\small
\begin{align*}
 \cq^{\mathbf{2},(n)}_{t,s}(z,w)&= I_{t,s}^{\mathbf{2},(n)} K^{\mathbf{2}}_{\sigma,\tau, \eta}(z,w) :=\frac1{16}   \int_{|t_1-s_1|+\eps_n}^{t_1+s_1+\eps_n} d\eta_1 \int_{|t_2-s_2|+\eps_n}^{t_2+s_2+\eps_n} d\eta_2   \int_{|s_1-s_2|+\eps_n}^{s_1+s_2+\eps_n} d\tau_1    \int_{|s_1-s_2|+\eps_n}^{s_1+s_2+\eps_n} d\tau_2 \, K^{\mathbf{2}}_{\sigma,\tau, \eta}(z,w),
 \end{align*}
\normalsize
 where 
 \begin{align*}
K^{\mathbf{2}}_{\sigma,\tau, \eta}(z,w) &:=K_{\eta_1}(z_1,w_1) K_{\eta_2}(z_2,w_2)K_{\tau_1}(w_1,w_2)K_{\tau_2}(w_1,w_2) .
 \end{align*}
Finally, we get the representation
\begin{align*}
\mathbb{E}\Big[ \frakt^{\mathbf{2},(n)}_{t_1}(x)   \frakt^{\mathbf{2},(n)}_{t_2}(x) \Big]
&= \int_0^{t_1} ds_1 \int_0^{t_2} ds_2\, I^{\mathbf{2},(n)}_{t,s} \big[\mathcal{R}\cf^{\mathbf{2}}\big](x),
 \end{align*}
where the function $\cf^{\mathbf{2}}$ is here defined as
\begin{align*} 
&\cf^{\mathbf{2}}(y_1,y_2,z_1,z_2)=  \int dw_1 dw_2 \,K_{t_1-s_1}(y_1,w_1)K_{t_2-s_2}(y_2,w_2)K^{\mathbf{2}}_{\sigma,\tau, \eta}(z,w).
\end{align*} 
The contribution of this term can be estimated as we did for $\cf^{\mathbf{1},\mathbf{2}}$, and accordingly we only state the intermediate bounds below without more details.

\begin{lemma}  
For any  $0<\delta< \frac14$, the following bound holds true:
\begin{equation*} 
 \|\cf^{\mathbf{2}}\|_{L^{\infty}(\R^{12})}   \lesssim  (t_1-s_1)^{-3\delta} (t_2-s_2)^{-3\delta} \eta_1^{-\frac32+\delta}\eta_2^{-\frac32+3\delta}(\tau_1\tau_2)^{-\frac32+\delta} .
 \end{equation*}
  Moreover, there exists $N\geq 1$ such that 
        \begin{equation*}   
 \max\Big(\|\cf^{\mathbf{2}}\|_{\mathcal{H}^{16}(\R^{12})} ,\|H_{y_1}  H_{y_2}\cf^{\mathbf{2}}\|_{L^{\infty}(\R^{12})} \Big)  \lesssim    \big(\eta_1 \eta_2 \tau_1 \tau_2 (t_1-s_1)(t_2-s_2)   \big)^{-N}.
 \end{equation*}

  \end{lemma}
	
	\smallskip

By mimicking the proof of Lemma \ref{lem:a12}, we obtain the desired estimate on $\frakt^{\mathbf{2},(n)}$.
\begin{lemma} 
Let $0<\delta< \frac14$ and $\eps>0$. Then if $p\geq 1$ is large enough, 
\begin{equation*} 
  \|\mathcal{R}\cf^{\mathbf{2}}\|_{ {L}^{p}({ \R^{3}})} \lesssim (t_1-s_1)^{-3\delta-\eps} (t_2-s_2)^{-3\delta-\eps} \eta_1^{-\frac32+\delta-\eps}\eta_2^{-\frac32+3\delta-\eps}(\tau_1\tau_2)^{-\frac32+\delta-\eps} , 
 \end{equation*}
and 
\begin{equation}\label{intA3}
\sup_{n\geq 1}\, \bigg(\int dx \,  \mathbb{E}\Big[ \frakt^{\mathbf{2},(n)}_{t_1}(x)   \frakt^{\mathbf{2},(n)}_{t_2}(x) \Big]^p\bigg)^{\frac1p} \lesssim |t_2-t_1| ^{-6\eps}.
\end{equation}
 \end{lemma}

\

We are now in a position to inject \eqref{intA3} into \eqref{frakt2-main} and deduce the uniform estimate
$$\sup_{n\geq 1}\, \mathbb{E}\Big[\big\|\frakti^{\mathbf{2},(n)} \big\|_{\cac^{1-\varepsilon}([0,T];\cb_x^{-\eta})}^{2p}\Big]<\infty,$$
for all $T>0$, $0<\eps,\eta<\frac12$ and $p\geq 1$. Combined with \eqref{conc-frak1}, this achieves the proof of \eqref{unif-fourth-1}.

  The proof of 
$$\sup_{r\geq 0} \sup_{n\geq 1}\, \mathbb{E}\Big[\big\|\frakti^{\mathbf{2},(n)} \big\|_{{\ov \cac}^{1-\varepsilon}([r,r+1];\cb_x^{-\eta})}^{2p}\Big]<\infty,$$
 is similar, using \eqref{frakt2-main-bis}. This in turn implies \eqref{unif-fourth-1bis}.

 
\

\section{Fourth order diagram 2}\label{sect-fourth}

We turn here to the analysis of the sequence of diagrams defined for all $n\geq 1$ \big(see \eqref{ord4}\big) by
$$\<Psi2IPsi2>^{(n)}:=\<Psi2>^{(n)} \pe \<IPsi2>^{(n)}- \frakc^{\mathbf{2},(n)}$$
where
$$\frakc_t^{\mathbf{2},(n)}(x) := \mathbb{E}\Big[ \<Psi2>^{(n)}_t(x) \<IPsi2>^{(n)}_t(x)\Big].$$

Our main convergence statement for $(\<Psi2IPsi2>^{(n)})$,  reads as follows.

\begin{proposition}\label{prop-9.1}
Let $T>0$. For all $0<\eps,\eta<\frac12$, there exists $\ka>0$ such that
\begin{eqnarray*}
&\mathbb{E} \Big[ \Big\|\widetilde{\<Psi2IPsi2>}^{(n+1)} -\widetilde{\<Psi2IPsi2>}^{(n)}\Big\|_{\cac^{1-\eps}([0,T]; \cb_x^{-\eta})}^{2p} \Big]\lesssim 2^{-\ka n p} .
\end{eqnarray*}
Consequently, the sequence $(\widetilde{\<Psi2IPsi2>}^{(n)})$ converges almost surely to an element $\widetilde{\<Psi2IPsi2>}$ in  $\cac^{1-\varepsilon}\big([0,T];\cb^{-\eta}_x\big)$, for all $\varepsilon,\eta>0$.

Moreover, the following uniform in time estimate holds true:
\begin{equation*} 
\sup_{r\geq 0}   \mathbb{E} \Big[ \Big\|\widetilde{\<Psi2IPsi2>}^{(n+1)} -\widetilde{\<Psi2IPsi2>}^{(n)}\Big\|_{{\ov \cac}^{1-\eps}([r,r+1]; \cb_x^{-\eta})}^{2p} \Big]\lesssim 2^{-\ka n p}. 
\end{equation*}

\end{proposition}

\medskip

As a first step, observe that this element can be expanded as
\begin{align*}
\<Psi2IPsi2>^{(n)}=\Big(\<Psi2>^{(n)} \pe \<IPsi2>^{(n)}-\mathbb{E}\big[\<Psi2>^{(n)} \pe \<IPsi2>^{(n)}\big]\Big)+\Big(\mathbb{E}\big[\<Psi2>^{(n)} \pe \<IPsi2>^{(n)}\big]-  \mathbb{E}\big[\<Psi2>^{(n)}  \<IPsi2>^{(n)}\big]\Big),
\end{align*} 
and then
\begin{align*}
&\big(\<Psi2>^{(n)}_t \pe \<IPsi2>^{(n)}_t\big)(x)-\mathbb{E}\Big[\big(\<Psi2>^{(n)}_t \pe \<IPsi2>^{(n)}_t\big)(x)\Big]\\
&=\sum_{i\sim i'}\int dz dy \, \delta_i(x,z) \delta_{i'}(x,y) \Big\{ \<Psi2>^{(n)}_t(z) \<IPsi2>^{(n)}_t(y)-\mathbb{E}\Big[ \<Psi2>^{(n)}_t(z)\<IPsi2>^{(n)}_t(y)\Big]\Big\}\\
&=\int_0^t ds\sum_{i\sim i'}\int dz dy \, \delta_i(x,z) \delta_{i'}(x,y) \int dw  \, K_{t-s}(y,w)\Big\{ \<Psi2>^{(n)}_t(z) \<Psi2>^{(n)}_s(w)-\mathbb{E}\Big[ \<Psi2>^{(n)}_t(z)\<Psi2>^{(n)}_s(w)\Big]\Big\}\\
&=\int_0^t ds\sum_{i\sim i'}\int dz dy \, \delta_i(x,z) \delta_{i'}(x,y)  \int dw  \, K_{t-s}(y,w)\\
&\hspace{4cm}\Big\{ I^W_2\big(F^{(n)}_{t,z}\otimes F^{(n)}_{t,z}\big) I^W_2\big(F^{(n)}_{s,w}\otimes F^{(n)}_{s,w} \big)-\mathbb{E}\Big[I^W_2\big(F^{(n)}_{t,z}\otimes F^{(n)}_{t,z}\big) I^W_2\big(F^{(n)}_{s,w}\otimes F^{(n)}_{s,w} \big) \Big]\Big\}
\end{align*}
where we have used the representation \eqref{wick-trees} of $\<Psi2>^{(n)}$. By applying the product rule \eqref{prod-rule-mult-int}, we immediately derive the decomposition
\begin{equation*}
\<Psi2IPsi2>^{(n)}_t(x)=\scret^{\mathbf{1},(n)}_t(x)+\scret^{\mathbf{2},(n)}_t(x){+\scret^{\mathbf{3},(n)}_t(x)},
\end{equation*}
with
\begin{align*}
\scret^{\mathbf{1},(n)}_t(x):=\int_0^t ds\sum_{i\sim i'}\int dz dy \, \delta_i(x,z) \delta_{i'}(x,y)  \int dw  \, K_{t-s}(y,w)I^W_4\big(F^{(n)}_{t,z}\otimes F^{(n)}_{t,z}\otimes F^{(n)}_{s,w}\otimes F^{(n)}_{s,w}\big) ,
\end{align*}
\begin{align*}
\scret^{\mathbf{2},(n)}_t(x):=4\int_0^t ds\sum_{i\sim i'}\int dz dy \, \delta_i(x,z) \delta_{i'}(x,y)  \int dw  \, K_{t-s}(y,w)\cac^{(n)}_{t,s}(z,w)I^W_2\big(F^{(n)}_{t,z}\otimes F^{(n)}_{s,w}\big)
\end{align*}
and
\begin{align*}
\scret^{\mathbf{3},(n)}_t(x):=\mathbb{E}\Big[\big(\<Psi2>^{(n)}_t \pe \<IPsi2>^{(n)}_t\big)(x)\Big]-  \mathbb{E}\Big[\<Psi2>^{(n)}_t(x)  \<IPsi2>^{(n)}_t(x)\Big].
\end{align*}

Recall that following the statement of Proposition \ref{prop-9.1}, we are interested in the convergence of $(\<Psi2IPsi2>^{(n)})$ in the space $\cac^{-\eps}_T\cb_x^{-\frac{\eps}{2}}$, for $\eps>0$. Thus, at least for $\mathbf{a}=\mathbf{1},\mathbf{2}$, we consider the time-integrated process
\begin{align*}
\screti^{\mathbf{a},(n)}_t(x):= \int_0^t ds \,  \scret^{\mathbf{a},(n)}_s(x) .
\end{align*}

\begin{proposition}
Let $T>0$. For $\mathbf{a}=1,2$ and for all $0<\eps,\eta<\frac12$, there exists $\ka>0$ such that
\begin{eqnarray}\label{est-prop92}
&\mathbb{E} \Big[ \big\|\screti^{\mathbf{a},(n+1)} -\screti^{\mathbf{a},(n)}\big\|_{\cac^{1-\eps}([0,T]; \cb_x^{-\eta})}^{2p} \Big]\lesssim 2^{-\ka n p} .
\end{eqnarray}
Consequently, the sequence $(\screti^{\mathbf{a},(n)})$ converges almost surely to an element $\screti^{\mathbf{a}}$ in  $\cac^{1-\varepsilon}\big([0,T];\cb^{-\eta}_x\big)$, for all $\varepsilon,\eta>0$.

Moreover, 
\begin{equation}\label{est-prop92bis}
\sup_{r\geq 0}  \mathbb{E} \Big[ \big\|\screti^{\mathbf{a},(n+1)} -\screti^{\mathbf{a},(n)}\big\|_{{\ov\cac}^{1-\eps}([r,r+1]; \cb_x^{-\eta})}^{2p} \Big]\lesssim 2^{-\ka n p}. 
\end{equation}
\end{proposition}

\smallskip

The proof of this proposition will occupy Sections \ref{subsec:secreti1} and \ref{subsec:secreti2} below. The control of the remaining (deterministic) diagram $\scret^{\mathbf{3},(n)}$ will then be exhibited in Section \ref{subsec:secret3}. \medskip

  We provide the detailed proof of \eqref{est-prop92} only, as the estimate \eqref{est-prop92bis} follows from similar arguments (we refer to the previous section for details).

\subsection{Study of $\screti^{\mathbf{1},(n)}$} \label{subsec:secreti1}

For more clarity in the presentation of our arguments, we will stick to the proof of the uniform bound
\begin{eqnarray}
&\dis \sup_{n\geq 1} \, \mathbb{E} \Big[ \big\|\screti^{\mathbf{a},(n)} \big\|_{\cac^{1-\eps}([0,T];\cb_x^{-\eta})}^{2p} \Big] <\infty, \quad \quad \mathbf{a}=1,2.\label{unif-fourth-2}
\end{eqnarray}

Let $T>0$. Observe that $\screti^{\mathbf{a},(n)}_0=0$. With the same successive arguments as in~\eqref{normlip}-\eqref{hyperc}-\eqref{uno}, we deduce that
\begin{multline}
\mathbb{E}\Big[\big\|\screti^{\mathbf{1},(n)} \big\|_{\cac^{1-\varepsilon}([0,T]; \cb_x^{-\eta})}^{2p}\Big] \lesssim \\
\lesssim \int_0^T\int_0^T  \frac{dv_1 dv_2}{|v_2-v_1|^{2p(1-\varepsilon)+2}}\bigg( \int_{[v_1,v_2]^2}dt_1 dt_2\, \bigg(\int dx \,  \mathbb{E}\Big[ \scret^{\mathbf{1},(n)}_{t_1}(x)   \scret^{\mathbf{1},(n)}_{t_2}(x) \Big]^p\bigg)^{\frac1p} \bigg)^p,\label{scret1-main}
\end{multline} 
for all $0<\eps,\eta<\frac12$ and $p\geq 2$ large enough. Then observe that
\begin{multline*}
\mathbb{E} \Big[I^W_4\big(F^{(n)}_{t_1,z_1} \otimes F^{(n)}_{t_1,z_1} \otimes F^{(n)}_{s_1,w_1} \otimes F^{(n)}_{s_1,w_1}\big)I^W_4\big(F^{(n)}_{t_2,z_2} \otimes F^{(n)}_{t_2,z_2} \otimes F^{(n)}_{s_2,w_2} \otimes F^{(n)}_{s_2,w_2}\big)\Big]=\\
\begin{aligned}
&=c\, \Big\langle \text{Sym} \big( F^{(n)}_{t_1,z_1} \otimes F^{(n)}_{t_1,z_1} \otimes F^{(n)}_{s_1,w_1} \otimes F^{(n)}_{s_1,w_1}\big), \text{Sym}\big(F^{(n)}_{t_2,z_2} \otimes F^{(n)}_{t_2,z_2} \otimes F^{(n)}_{s_2,w_2} \otimes F^{(n)}_{s_2,w_2}\big) \Big\rangle\\
&=\sum_{\mathbf{b}=1}^3 c_{\mathbf{b}}\,  \scrq^{\mathbf{1},\mathbf{b},(n)}_{t,s}(z,w)
\end{aligned}
\end{multline*}
for some combinatorial coefficients $c_{\mathbf{b}} \geq 0$ and with
\begin{align*}
\scrq^{\mathbf{1},\mathbf{1},(n)}_{t,s}(z,w)&:= \cac^{(n)}_{t_1,t_2}(z_1,z_2)^2 \cac^{(n)}_{s_1,s_2}(w_1,w_2)^2,\\
\scrq^{\mathbf{1},\mathbf{2},(n)}_{t,s}(z,w)&:=\cac^{(n)}_{t_1,s_2}(z_1,w_2)^2 \cac^{(n)}_{t_2,s_1}(z_2,w_1)^2 ,\\
\scrq^{\mathbf{1},\mathbf{3},(n)}_{t,s}(z,w)&:=\cac^{(n)}_{t_1,t_2}(z_1,z_2) \cac^{(n)}_{t_1,s_2}(z_1,w_2) \cac^{(n)}_{t_2,s_1}(z_2,w_1) \cac^{(n)}_{s_1,s_2}(w_1,w_2).
\end{align*}
Using also the operator $\mathcal{R}$ introduced in \eqref{def-op-T}, we deduce the expression
\begin{multline}\label{dos1}
\mathbb{E}\Big[ \scret^{\mathbf{1},(n)}_{t_1}(x)   \scret^{\mathbf{1},(n)}_{t_2}(x) \Big]=\\
=\sum_{\mathbf{b}=1}^3 c_{\mathbf{b}}\int_0^{t_1} ds_1 \int_0^{t_2} ds_2\, \mathcal{R}\Big(  \int dw_1 dw_2 \,  K_{t_1-s_1}(y_1,w_1)   K_{t_2-s_2}(y_2,w_2)   \scrq^{\mathbf{1},\mathbf{b},(n)}_{t,s}(z,w)\Big)(x).
\end{multline}
As in the previous section (see {\it e.g.} \eqref{repres-q11}), let us go further by representing the quantities $\scrq^{\mathbf{1},\mathbf{b},(n)}$ as 
\small
\begin{align*}
\scrq^{\mathbf{1},\mathbf{1},(n)}_{t,s}(z,w)&= \scri_{t,s}^{\mathbf{1},\mathbf{1},(n)} \scrk^{\mathbf{1},\mathbf{1}}_{\sigma,\tau, \eta}(z,w) := \frac1{16}   \int_{|t_1-t_2|+\eps_n}^{t_1+t_2+\eps_n} d\sigma_1\int_{|t_1-t_2|+\eps_n}^{t_1+t_2+\eps_n} d\sigma_2   \int_{|s_1-s_2|+\eps_n}^{s_1+s_2+\eps_n} d\tau_1    \int_{|s_1-s_2|+\eps_n}^{s_1+s_2+\eps_n} d\tau_2  \, \scrk^{\mathbf{1},\mathbf{1}}_{\sigma,\tau, \eta}(z,w),\\
\scrq^{\mathbf{1},\mathbf{2},(n)}_{t,s}(z,w)&= \scri_{t,s}^{\mathbf{1},\mathbf{2},(n)} \scrk^{\mathbf{1},\mathbf{2}}_{\sigma,\tau, \eta}(z,w) :=\frac1{16}   \int_{|t_1-s_2|+\eps_n}^{t_1+s_2+\eps_n} d\eta_1 \int_{|t_1-s_2|+\eps_n}^{t_1+s_2+\eps_n} d\eta_2   \int_{|t_2-s_1|+\eps_n}^{t_2+s_1+\eps_n} d\eta_3    \int_{|t_2-s_1|+\eps_n}^{t_2+s_1+\eps_n} d\eta_4  \, \scrk^{\mathbf{1},\mathbf{2}}_{\sigma,\tau, \eta}(z,w),\\
\scrq^{\mathbf{1},\mathbf{3},(n)}_{t,s}(z,w)&= \scri_{t,s}^{\mathbf{1},\mathbf{3},(n)} \scrk^{\mathbf{1},\mathbf{3}}_{\sigma,\tau, \eta}(z,w) :=\frac1{16}   \int_{|t_1-t_2|+\eps_n}^{t_1+t_2+\eps_n} d\sigma_1 \int_{|t_1-s_2|+\eps_n}^{t_1+s_2+\eps_n} d\eta_1   \int_{|t_2-s_1|+\eps_n}^{t_2+s_1+\eps_n} d\eta_2    \int_{|s_1-s_2|+\eps_n}^{s_1+s_2+\eps_n} d\tau_1  \, \scrk^{\mathbf{1},\mathbf{3}}_{\sigma,\tau, \eta}(z,w),
 \end{align*}
\normalsize
 where 
 \begin{align*}
\scrk^{\mathbf{1},\mathbf{1}}_{\sigma,\tau, \eta}(z,w) &:=K_{\sigma_1}(z_1,z_2)K_{\sigma_2}(z_1,z_2) K_{\tau_1}(w_1,w_2)K_{\tau_2}(w_1,w_2),\\  
\scrk^{\mathbf{1},\mathbf{2}}_{\sigma,\tau, \eta}(z,w) &:=K_{\eta_1}(z_1,w_2) K_{\eta_2}(z_1,w_2)K_{\eta_3}(z_2,w_1)K_{\eta_4}(z_2,w_1),\\ 
\scrk^{\mathbf{1},\mathbf{3}}_{\sigma,\tau, \eta}(z,w)&:=K_{\sigma_1}(z_1,z_2) K_{\eta_1}(z_1,w_2)K_{\eta_2}(z_2,w_1)K_{\tau_1}(w_1,w_2).
 \end{align*}

\

Going back to \eqref{dos1}, this yields the decomposition 
\begin{equation}
\bigg|\mathbb{E}\Big[ \scret^{\mathbf{1},(n)}_{t_1}(x)   \scret^{\mathbf{1},(n)}_{t_2}(x) \Big]\bigg|=\bigg|\sum_{\mathbf{b}=1}^3 c_{\mathbf{b}}\, \scra^{\mathbf{1},\mathbf{b},(n)}_{t_1,t_2}(x)\bigg|\leq \sum_{\mathbf{b}=1}^3 c_{\mathbf{b}} \, \big|\scra^{\mathbf{1},\mathbf{b},(n)}_{t_1,t_2}(x)\big|,\label{decomposcret1}
\end{equation}
with
\begin{align*}
\scra^{\mathbf{1},\mathbf{b},(n)}_{t_1,t_2}(x)&:=\int_0^{t_1} ds_1 \int_0^{t_2} ds_2 \, \scri^{\mathbf{1},\mathbf{b},(n)}_{t,s} \big[\mathcal{R}\scrf^{\mathbf{1},\mathbf{b}}\big](x),
 \end{align*}
and 
\begin{equation*}  
\scrf^{\mathbf{1},\mathbf{b}}(y_1,y_2,z_1,z_2):=   \int dw_1 dw_2 \,K_{t_1-s_1}(y_1,w_1)K_{t_2-s_2}(y_2,w_2)\scrk^{\mathbf{1},\mathbf{b}}_{\sigma,\tau,\eta}(z_1,w_1,z_2,w_2) .
\end{equation*}

 \

 \subsubsection{Estimation of $\scra^{\mathbf{1},\mathbf{1},(n)}$}  

Let us start with the following controls on $\scrf^{\mathbf{1},\mathbf{1}}$.
     \begin{lemma}\label{Lem-F4}

For all $q>\frac32$, it holds that
\begin{equation}  \label{Grad-F4}
 \sup_{y_1,y_2,z_2 \in \R^3} \big(\int dz_1 \big| H_{y_1}  \scrf^{\mathbf{1},\mathbf{1}}\big|^q\big)^{\frac1q}  \lesssim \sigma^{\frac3{2q}}_1  (t_1-s_1)^{-1}  (\sigma_1 \sigma_2 )^{-\frac32}\tau^{-\frac32}_2\tau^{-\frac54}_1(t_1-s_1)^{-\frac14},
 \end{equation}
 \begin{equation}  \label{Grad-F41}
 \sup_{y_1,y_2,z_1 \in \R^3} \big(\int dz_2 \big| H_{y_2}  \scrf^{\mathbf{1},\mathbf{1}}\big|^q\big)^{\frac1q}  \lesssim \sigma^{\frac3{2q}}_2  (t_2-s_2)^{-1}  (\sigma_1 \sigma_2 )^{-\frac32}\tau^{-\frac32}_2\tau^{-\frac54}_1(t_1-s_1)^{-\frac14}.
 \end{equation}
Moreover, there exists $N\geq 1$ such that 
\begin{align}   
&\max\Big( \|\scrf^{\mathbf{1},\mathbf{1}}\|_{\mathcal{H}^{16}(\R^{12})} ,\|H^{2}_{y_1}H^{-1}_{z_1} H_{y_2}\scrf^{\mathbf{1},\mathbf{1}}\|_{L^{\infty}(\R^{12})}, \|H^{2}_{y_2}H^{-1}_{z_2} H_{y_1}\scrf^{\mathbf{1},\mathbf{1}}\|_{L^{\infty}(\R^{12})}  \Big)\nonumber\\
& \hspace{7cm} \lesssim    \big(\sigma_1 \sigma_2 \tau_1 \tau_2 (t_1-s_1)(t_2-s_2)   \big)^{-N}.\label{crude-scrf}
 \end{align}

 \end{lemma}

\begin{proof}

To show \eqref{Grad-F4}, we adopt the same strategy as in the proof of \eqref{grad-f1} (see Remark \ref{rema-grad}). Using the expression of $\scrf^{\mathbf{1},\mathbf{1}}$, we have 
\begin{multline*} 
 \big(H_{y_1}\scrf^{\mathbf{1},\mathbf{1}} \big)(y_1,y_2,z_1,z_2)=\\
 =  \int dw_1 dw_2 \, K_{\sigma_1}(z_1,z_2)  K_{\sigma_2}(z_1,z_2) K_{\tau_2}(w_1,w_2)  \big( H_{y_1}K_{t_1-s_1}\big)(y_1,w_1)  K_{t_2-s_2}(y_2,w_2)K_{\tau_1}(w_1,w_2),
\end{multline*} 
 and thus for any $q>\frac32$,
 \begin{multline*} 
\big \|\big(H_{y_1}\scrf^{\mathbf{1},\mathbf{1}} \big)(y_1,y_2,\cdot ,z_2)\big\|_{L_{z_1}^q(\R^3)}\lesssim \\
 \begin{aligned}
& \lesssim     \|K_{\sigma_1}\|_{L^q}  \|K_{\sigma_2}\|_{L^\infty}  \|K_{\tau_2}\|_{L^\infty}  \int dw_1 dw_2 \,\Big| \big( H_{y_1}K_{t_1-s_1}\big)(y_1,w_1)\Big| K_{t_2-s_2}(y_2,w_2)K_{\tau_1}(w_1,w_2)    \\
 &\lesssim   \|K_{\sigma_1}\|_{L^q}  \|K_{\sigma_2}\|_{L^\infty}  \|K_{\tau_2}\|_{L^\infty}  \| (H_{y_1} K_{t_1-s_1})(y_1, \cdot) \|_{L^{\frac65}_{w_1}}   \| K_{\tau_1}(\cdot,w_2) \|_{L^6_{w_1}}  \|   K_{t_2-s_2}(y_2,\cdot)\|_{L^1_{w_2}} \\
&\lesssim  \sigma^{\frac3{2q}}_1   (\sigma_1 \sigma_2 )^{-\frac32}\tau^{-\frac32}_2\tau^{-\frac54}_1   \| (H_{y_1} K_{t_1-s_1})(y_1, \cdot) \|_{L^{\frac65}_{w_1}(\R^3)}.
  \end{aligned}
\end{multline*}
Now by \eqref{borne-F22} we deduce
 $$ \| (H_{y_1} K_{t_1-s_1})(y_1, \cdot) \|_{L^{\frac65}_{w_1}(\R^3)}  \lesssim (t_1-s_1)^{-\frac54},$$
 which in turn yields \eqref{Grad-F4}. The estimate \eqref{Grad-F41} is obtained similarly.

\

The proof of the estimate \eqref{crude-scrf} follows from the same (crude) arguments as those in the proof of Lemma~\ref{lemma-crude}.
 \end{proof}

 \begin{lemma} 
Let $\eps>0$. Then if $p\geq 1$ is large enough  
\begin{equation}\label{TF4}
 \|\mathcal{R}\scrf^{\mathbf{1},\mathbf{1}}\|_{ {L}^{p}(\R^{3})} \lesssim (\sigma_1 \sigma_2)^{-1-\eps} \tau_1^{-\frac54-\eps} \tau_2^{-\frac32-\eps} (t_1-s_1)^{-\frac34-\eps}(t_2-s_2)^{ -\frac12-\eps} ,
\end{equation}
and 
\begin{equation}\label{intA4}
\sup_{n\geq 1}\, \| \scra^{\mathbf{1},\mathbf{1},(n)}_{t_1,t_2}\|_{L^p(\R^3)} \lesssim |t_2-t_1| ^{-6\eps}.
\end{equation}
 \end{lemma}
 
  \begin{proof}
Fix $\eps>0$. By applying Corollary \ref{coro:interpol-T} with $\la_1=0$, $\la_2=\la_3=\frac12$, we get that for every $q>\frac32$ and every $p\geq 1$ large enough,
\small
\begin{align*}
&\|\mathcal{R}\scrf^{\mathbf{1},\mathbf{1}}\|_{L^p(\R^3)} \lesssim \Big(1 \vee \sup_{y_1,y_2,z_2 \in \R^3}\Big(\int dz_1 \big| H_{y_1}  \scrf^{\mathbf{1},\mathbf{1}}\big|^q\Big)^{\frac1q}  \Big)^{\frac12}\Big(1 \vee \sup_{y_1,y_2,z_1 \in \R^3}\Big(\int dz_2 \big| H_{y_2}  \scrf^{\mathbf{1},\mathbf{1}}\big|^q\Big)^{\frac1q}  \Big)^{\frac12}\\
&\hspace{2cm}\Big(1\vee \|\scrf^{\mathbf{1},\mathbf{1}}\|_{\mathcal{H}^{16}(\R^{12})}\Big)^{\frac{\eps}{6N}}  \Big(\|H^{2}_{y_1}H^{-1}_{z_1} H_{y_2}\scrf^{\mathbf{1},\mathbf{1}}\|_{L^{\infty}(\R^{12})}\Big)^{\frac{\eps}{6N}} \Big(\|H^{2}_{y_2}H^{-1}_{z_2} H_{y_1}\scrf^{\mathbf{1},\mathbf{1}}\|_{L^{\infty}(\R^{12})}\Big)^{\frac{\eps}{6N}}\\
&\lesssim \sigma^{-\frac32+\frac3{4q}}_1\sigma^{-\frac32+\frac3{4q}}_2 \tau^{-\frac54}_1\tau^{-\frac32}_2(t_1-s_1)^{-\frac34}   (t_2-s_2)^{-\frac12} (t_1-s_1)^{-\frac14} \Big(\sigma_1 \sigma_2 \tau_1 \tau_2 (t_1-s_1)(t_2-s_2)\Big)^{-\frac{\eps}{2}},
\end{align*}
\normalsize
where we have naturally used the bounds contained in Lemma \ref{Lem-F4} to derive the second inequality. To reach \eqref{TF4}, it only remains us to choose $q>\frac32$ such that $\frac32-\frac3{4q}=1+\frac{\eps}{2}$.

 \medskip

As for \eqref{intA4}, we integrate \eqref{TF4} in $\sigma_1, \sigma_2, \tau_1, \tau_2$ and get 
 \begin{equation*} 
\sup_{n\geq 1}\, \big\|     \scri^{\mathbf{1},\mathbf{1},(n)}_{t,s} \big[\mathcal{R}\scrf^{\mathbf{1},\mathbf{1}}\big]\big\|_{L^p(\R^3)}  
\lesssim  |t_1-t_2|^{-2\eps}   |s_1 -s_2 |^{-\frac34-2\eps}   (t_1-s_1)^{-\frac34-\eps}(t_2-s_2)^{ -\frac12-\eps}.
   \end{equation*}
   The desired bound easily follows from an integration in $s_1, s_2$ and by applying Lemma \ref{Lem-prod} twice.
    \end{proof}


 \subsubsection{Estimation of $\scra^{\mathbf{1},\mathbf{2},(n)}$}  

The key technical result for this estimation reads as follows.

     \begin{lemma}  \label{Lem-F5}
It holds that
   \begin{equation}  \label{F5-inf}
 \|\scrf^{\mathbf{1},\mathbf{2}}\|_{L^{\infty}(\R^{12})}   \lesssim  (t_1-s_1)^{-\frac34}  (t_2-s_2)^{-\frac34}    (\eta_1 \eta_2\eta_3 \eta_4)^{\frac98}.
 \end{equation}
Moreover, there exists $N\geq 1$ such that 
       \begin{equation*}   
\max\Big(\|\scrf^{\mathbf{1},\mathbf{2}}\|_{\mathcal{H}^{16}(\R^{12})}, \|H_{y_1}  H_{y_2}\scrf^{\mathbf{1},\mathbf{2}}\|_{L^{\infty}(\R^{12})}\Big)   \lesssim    \big(\eta_1 \eta_2 \eta_3 \eta_4 (t_1-s_1)(t_2-s_2)   \big)^{-N}.
 \end{equation*}

 \end{lemma}

\begin{proof}
For all $w_1,y_2 \in \R^3$,
\begin{eqnarray*} 
\int dw_2  \,  K_{t_2-s_2}(y_2,w_2)K_{\eta_1}(z_1,w_2) K_{\eta_2}(z_1,w_2)  &\leq &
 \|K_{t_2-s_2} (y_2, \cdot)   \|_{L^{2}_{w_2}(\R^3)}  \|  K_{\eta_1} (z_1, \cdot) K_{\eta_2}(z_1,\cdot) \|_{L_{w_2}^{2}(\R^3)} \\
  & \lesssim&  (t_2-s_2)^{-\frac34 }  \|  K_{\eta_1} (z_1, \cdot) \|_{L_{w_2}^{4}(\R^3)}  \|   K_{\eta_2}(z_1,\cdot) \|_{L_{w_2}^{4}(\R^3)} \\
  &   \lesssim&   (t_2-s_2)^{-\frac34 } (\eta_1 \eta_2)^{-\frac98} ,
\end{eqnarray*}
 and similarly
$$  \int dw_1  \,K_{t_1-s_1}(y_1,w_1)  K_{\eta_3}(z_2,w_1)K_{\eta_4}(z_2,w_1)   \lesssim   (t_1-s_1)^{-\frac34 } (\eta_1 \eta_2)^{-\frac98},  $$
which implies \eqref{F5-inf}.
\end{proof}

 \begin{lemma} 
 Let $\eps>0$. For every $p\geq 1$ large enough, one has 
\begin{equation}\label{TF5}
 \|\mathcal{R}\scrf^{\mathbf{1},\mathbf{2}}\|_{ {L}^{p}(\R^{3})} \lesssim  (t_1-s_1)^{-\frac34-\eps}  (t_2-s_2)^{-\frac34-\eps}    (\eta_1 \eta_2\eta_3 \eta_4)^{\frac98-\eps},
 \end{equation}
and 
\begin{equation}\label{intA5}
\sup_{n\geq 1}\, \| \scra^{\mathbf{1},\mathbf{2},(n)}_{t_1,t_2}\|_{L^p(\R^3)} \lesssim |t_2-t_1| ^{-6\eps}.
\end{equation}
 \end{lemma}

\begin{proof}
For \eqref{TF5}, we apply Corollary \ref{coro:interpol-T} with $\la_1=1$, $\la_2=\la_3=0$, and then inject the estimates of Lemma \ref{Lem-F5}, along the same model as in the proof of Lemma \ref{lem:a12}.

\smallskip

As for \eqref{intA5}, we integrate \eqref{TF4} in $\eta_1, \eta_2, \eta_3, \eta_4$ and get first
 \begin{equation*} 
\sup_{n\geq 1}\, \big\|     \scri_{t,s}^{\mathbf{1},\mathbf{2},(n)} \big[\mathcal{R}\scrf^{\mathbf{1},\mathbf{2}}\big]\big\|_{L^p(\R^3)}  
\lesssim  (t_1-s_1)^{-\frac34-\eps}(t_2-s_2)^{-\frac34-\eps}      (t_1-s_2)^{-\frac14-2\eps}  (t_2-s_1)^{-\frac14-2\eps}.
   \end{equation*}
   The claimed estimate is derived by integrating in $s_1, s_2$ and applying Lemma \ref{Lem-prod} twice.
\end{proof}
 
	
  \smallskip

 \subsubsection{Estimation of $\scra^{\mathbf{1},\mathbf{3},(n)}$}

     \begin{lemma}  \label{lem:F6-inf}
     The following bounds hold true:
     \begin{equation}  \label{F6-inf}
 \|\scrf^{\mathbf{1},\mathbf{3}}\|_{L^{\infty}(\R^{12})}   \lesssim   \big(\sigma_1 \eta_1 \eta_2 \tau_1\big)^{-\frac32},
 \end{equation}
and for all $q>\frac32$,
\begin{equation}  \label{Grad-F6}
 \sup_{y_1,y_2,z_2 \in \R^3} \big(\int dz_1 \big| H_{y_1}  \scrf^{\mathbf{1},\mathbf{3}}\big|^q\big)^{\frac1q}  \lesssim \sigma^{\frac3{2q}}_1  (t_1-s_1)^{-1}  \big(\sigma_1 \eta_1 \eta_2 \tau_1\big)^{-\frac32}.
 \end{equation}
Moreover, there exists $N\geq 1$ such that 
\begin{multline*}   
 \max\Big(\|\scrf^{\mathbf{1},\mathbf{3}}\|_{\mathcal{H}^{16}(\R^{12})} ,\|H_{y_1}  H_{y_2}\scrf^{\mathbf{1},\mathbf{3}}\|_{L^{\infty}(\R^{12})}, \|H^{2}_{y_1}H^{-1}_{z_1} H_{y_2}\scrf^{\mathbf{1},\mathbf{3}}\|_{L^{\infty}(\R^{12})} \Big)   \lesssim  \\
\lesssim  \big(\sigma_1 \eta_1 \eta_2 \tau_1 (t_1-s_1)(t_2-s_2)   \big)^{-N}.
\end{multline*}

 \end{lemma}

\begin{proof}
We first prove  \eqref{F6-inf}. For all $y_1,y_2,z_1,z_2 \in \R^3$, 
\begin{multline*}
\scrf^{\mathbf{1},\mathbf{3}}(y_1,y_2,z_1,z_2) \leq \\
\begin{aligned}
&\leq  \|K_{\sigma_1}\|_{L^\infty(\R^6)}  \|K_{\eta_1}\|_{L^\infty(\R^6)}  \|K_{\eta_2}\|_{L^\infty(\R^6)}  \|K_{\tau_1}\|_{L^\infty(\R^6)}   \int dw_1 dw_2 \,K_{t_1-s_1}(y_1,w_1)K_{t_2-s_2}(y_2,w_2)\\
&\lesssim   \big(\sigma_1 \eta_1 \eta_2 \tau_1\big)^{-\frac32} \Big( \int dw_1   \,K_{t_1-s_1}(y_1,w_1) \Big)\Big( \int   dw_2 \, K_{t_2-s_2}(y_2,w_2) \Big) \\
&\lesssim   \big(\sigma_1 \eta_1 \eta_2 \tau_1\big)^{-\frac32},
\end{aligned}
\end{multline*}
which was the claim. \medskip

The estimate \eqref{Grad-F6} is obtained with the arguments detailed in Remark \ref{rema-grad}.
 \end{proof}

 \begin{lemma} 
  Let $\eps>0$. Then if $p\geq 1$ is large enough  
          \begin{equation}\label{TF6}
 \|\mathcal{R}\scrf^{\mathbf{1},\mathbf{3}}\|_{ {L}^{p}(\R^{3})} \lesssim \sigma^{-1-\eps}_1 \eta^{-\frac32-\eps}_1 \eta^{-\frac32-\eps}_2 \tau^{-\frac32-\eps}_1 (t_1-s_1)^{-\frac12-\eps}(t_2-s_2)^{ -\eps} ,
   \end{equation}
   and 
             \begin{equation}\label{intA6}
\sup_{n\geq 1}\, \|\scra^{\mathbf{1},\mathbf{3},(n)}_{t_1,t_2}\|_{L^p(\R^3)} \lesssim |t_2-t_1| ^{-6\eps}.
   \end{equation}
 \end{lemma}
 
  \begin{proof}
The transition from Lemma \ref{lem:F6-inf} to the estimate \eqref{TF6} can be done exactly as in the proof of Lemma \ref{lem:a11}, and so we do not repeat the details. \medskip
 
 We can then integrate \eqref{TF6} in $\sigma_1, \eta_1, \eta_2, \tau_1$ and get 
 \begin{align*} 
&\sup_{n\geq 1}\, \big\|    \scri_{t,s}^{\mathbf{1},\mathbf{3},(n)} \big[\mathcal{R}\scrf^{\mathbf{1},\mathbf{3}}\big]\big\|_{L^p(\R^3)}  \\
&\lesssim |t_2-t_1|^{-\eps}|t_1-s_2|^{-\frac12-\eps} |t_2-s_1|^{-\frac12-\eps}      |s_2-s_1|^{-\frac12-\eps}  (t_1-s_1)^{-\frac12-\eps}(t_2-s_2)^{ -\eps}.
   \end{align*}
  The bound \eqref{intA6} now follows from an application of  Lemma \ref{Lem-prod2}.
  \end{proof}
	
	\

By injecting \eqref{intA4}, \eqref{intA5} and \eqref{intA6} into \eqref{scret1-main}-\eqref{decomposcret1}, we deduce the claimed bound for $\screti^{\mathbf{1},(n)}$, that is
\begin{equation}
\dis \sup_{n\geq 1} \, \mathbb{E} \Big[ \big\|\screti^{\mathbf{1},(n)} \big\|_{\cac^{1-\eps}([0,T];\cb_x^{-\eta})}^{2p} \Big] <\infty,\label{conclu-iti-1}
\end{equation}
for all $0<\eps,\eta<\frac12$ and $p\geq 1$.

\

\subsection{Study of $\screti^{\mathbf{2},(n)}$}\label{subsec:secreti2}

Let $T>0$. We follow the same scheme as in the previous section: namely, we start with
\begin{multline}
\mathbb{E}\Big[\big\|\screti^{\mathbf{2},(n)} \big\|_{\cac^{1-\varepsilon}([0,T]; \cb_x^{-\eta}}^{2p}\Big] \lesssim \\
\lesssim \int_0^T\int_0^T  \frac{dv_1 dv_2}{|v_2-v_1|^{2p(1-\varepsilon)+2}}\bigg( \int_{[v_1,v_2]^2}dt_1 dt_2\, \bigg(\int dx \,  \mathbb{E}\Big[ \scret^{\mathbf{2},(n)}_{t_1}(x)   \scret^{\mathbf{2},(n)}_{t_2}(x) \Big]^p\bigg)^{\frac1p} \bigg)^p,\label{scret2-main}
\end{multline} 
for all $0<\eps,\eta<\frac12$ and $p\geq 2$ large enough. Then one has time
\begin{align*}
&\cac^{(n)}_{t_1,s_1}(z_1,w_1)\cac^{(n)}_{t_2,s_2}(z_2,w_2)\mathbb{E} \Big[I^W_2\big(F^{(n)}_{t_1,z_1}\otimes F^{(n)}_{s_1,w_1}\big)I^W_2\big(F^{(n)}_{t_2,z_2}\otimes F^{(n)}_{s_2,w_2}\big)\Big]\\
&=2\,\cac^{(n)}_{t_1,s_1}(z_1,w_1)\cac^{(n)}_{t_2,s_2}(z_2,w_2) \Big\langle \text{Sym} \big( F^{(n)}_{t_1,z_1} \otimes F^{(n)}_{s_1,w_1} \big), \text{Sym}\big(F^{(n)}_{t_2,z_2} \otimes F^{(n)}_{s_2,w_2} \big) \Big\rangle=\sum_{\mathbf{b}=1}^2 c_{\mathbf{b}}\,  \scrq^{\mathbf{2},\mathbf{b},(n)}_{t,s}(z,w)
\end{align*}
for some combinatorial coefficients $c_{\mathbf{b}} \geq 0$ and with
\begin{align*}
\scrq^{\mathbf{2},\mathbf{1},(n)}_{t,s}(z,w)&:= \cac^{(n)}_{t_1,s_1}(z_1,w_1)\cac^{(n)}_{t_2,s_2}(z_2,w_2)\cac^{(n)}_{t_1,t_2}(z_1,z_2) \cac^{(n)}_{s_1,s_2}(w_1,w_2),\\
\scrq^{\mathbf{2},\mathbf{2},(n)}_{t,s}(z,w)&:=\cac^{(n)}_{t_1,s_1}(z_1,w_1)\cac^{(n)}_{t_2,s_2}(z_2,w_2)\cac_{t_1,s_2}(z_1,w_2)\cac^{(n)}_{t_2,s_1}(z_2,w_1) ,
\end{align*}
which yields
\begin{multline}
\mathbb{E}\Big[ \scret^{\mathbf{2},(n)}_{t_1}(x)   \scret^{\mathbf{2},(n)}_{t_2}(x) \Big]=  \\
=\sum_{\mathbf{b}=1}^2 c_{\mathbf{b}}\int_0^{t_1} ds_1 \int_0^{t_2} ds_2\, \mathcal{R}\Big(  \int dw_1 dw_2 \,  K_{t_1-s_1}(y_1,w_1)   K_{t_2-s_2}(y_2,w_2)   \scrq^{\mathbf{2},\mathbf{b},(n)}_{t,s}(z,w)\Big)(x).\label{dos2}
\end{multline}

As before, let us expand the quantities $\scrq^{\mathbf{2},\mathbf{b},(n)}$ as 
\small
\begin{align*}
\scrq^{\mathbf{2},\mathbf{1},(n)}_{t,s}(z,w)&= \scri_{t,s}^{\mathbf{2},\mathbf{1},(n)} \scrk^{\mathbf{2},\mathbf{1}}_{\sigma,\tau, \eta}(z,w) := \frac1{16}   \int_{|t_1-s_1|+\eps_n}^{t_1+s_1+\eps_n} d\eta_1 \int_{|t_2-s_2|+\eps_n}^{t_2+s_2+\eps_n} d\eta_2   \int_{|t_1-t_2|+\eps_n}^{t_1+t_2+\eps_n} d\sigma_1    \int_{|s_1-s_2|+\eps_n}^{s_1+s_2+\eps_n} d\tau_1  \, \scrk^{\mathbf{2},\mathbf{1}}_{\sigma,\tau, \eta}(z,w),\\
\scrq^{\mathbf{2},\mathbf{2},(n)}_{t,s}(z,w)&= \scri_{t,s}^{\mathbf{2},\mathbf{2},(n)} \scrk^{\mathbf{2},\mathbf{2}}_{\sigma,\tau, \eta}(z,w) :=\frac1{16}   \int_{|t_1-s_1|+\eps_n}^{t_1+s_1+\eps_n} d\eta_1 \int_{|t_2-s_2|+\eps_n}^{t_2+s_2+\eps_n} d\eta_2   \int_{|t_1-s_2|+\eps_n}^{t_1+s_2+\eps_n} d\eta_3    \int_{|t_2-s_1|+\eps_n}^{t_2+s_1+\eps_n} d\eta_4 \, \scrk^{\mathbf{2},\mathbf{2}}_{\sigma,\tau, \eta}(z,w),
 \end{align*}
\normalsize
 where 
 \begin{align*}
\scrk^{\mathbf{2},\mathbf{1}}_{\sigma,\tau, \eta}(z,w) &:=K_{\eta_1}(z_1,w_1) K_{\eta_2}(z_2,w_2)K_{\sigma_1}(z_1,z_2)K_{\tau_1}(w_1,w_2),\\  
\scrk^{\mathbf{2},\mathbf{2}}_{\sigma,\tau, \eta}(z,w) &:=K_{\eta_1}(z_1,w_1) K_{\eta_2}(z_2,w_2)K_{\eta_3}(z_1,w_2)K_{\eta_4}(z_2,w_1).
 \end{align*}
We have thus obtained that  
\begin{align}
\Big|\mathbb{E}\Big[ \scret^{\mathbf{2},(n)}_{t_1}(x)   \scret^{\mathbf{2},(n)}_{t_2}(x) \Big]\Big|&\leq \sum_{\mathbf{b}=1}^2 c_{\mathbf{b}} \, \big|\scra^{\mathbf{2},\mathbf{b},(n)}_{t_1,t_2}(x)\big|,\label{decomposcret2}
\end{align}
with
\begin{align*}
\scra^{\mathbf{2},\mathbf{b},(n)}_{t_1,t_2}(x)&:=\int_0^{t_1} ds_1 \int_0^{t_2} ds_2 \, \scri^{\mathbf{2},\mathbf{b},(n)}_{t,s} \big[\mathcal{R}\scrf^{\mathbf{2},\mathbf{b}}\big](x),
 \end{align*}
and 
\begin{equation*} 
\scrf^{\mathbf{2},\mathbf{b}}(y_1,y_2,z_1,z_2):=   \int dw_1 dw_2 \,K_{t_1-s_1}(y_1,w_1)K_{t_2-s_2}(y_2,w_2)\scrk^{\mathbf{2},\mathbf{b}}_{\sigma,\tau,\eta}(z_1,w_1,z_2,w_2) .
\end{equation*} 

\

 \subsection{Estimation of $\scra^{\mathbf{2},\mathbf{1},(n)}_{t_1,t_2}$}  The following bounds can be derived with similar arguments as those in the proof of Lemma \ref{Lem-F4}, and thus we omit the details.

 \begin{lemma} \label{lem:F7}
It holds that
   \begin{equation*} 
 \|\scrf^{\mathbf{2},\mathbf{1}}\|_{L^{\infty}(\R^{12})}   \lesssim   \big(\sigma_1 \eta_1 \eta_2 \tau_1\big)^{-\frac32},
 \end{equation*}
and for all $q>\frac32$
\begin{equation*} 
 \sup_{y_1,y_2,z_2 \in \R^3} \big(\int dz_1 \big| H_{y_1}  \scrf^{\mathbf{2},\mathbf{1}}\big|^q\big)^{\frac1q}  \lesssim \sigma^{\frac3{2q}}_1  (t_1-s_1)^{-1}  \big(\sigma_1 \eta_1 \eta_2 \tau_1\big)^{-\frac32},
 \end{equation*}
 \begin{equation*}  
 \sup_{y_1,y_2,z_1 \in \R^3} \big(\int dz_2 \big| H_{y_2}  \scrf^{\mathbf{2},\mathbf{1}}\big|^q\big)^{\frac1q}  \lesssim \sigma^{\frac3{2q}}_1  (t_2-s_2)^{-1}  \big(\sigma_1 \eta_1 \eta_2 \tau_1\big)^{-\frac32}.
 \end{equation*} 
Moreover, there exists $N\geq 1$ such that 
\begin{align*}   
 &\max\Big(\|\scrf^{\mathbf{2},\mathbf{1}}\|_{\mathcal{H}^{16}(\R^{12})} ,\|H_{y_1}  H_{y_2}\scrf^{\mathbf{2},\mathbf{1}}\|_{L^{\infty}(\R^{12})} , \|H^{2}_{y_1}H^{-1}_{z_1} H_{y_2}\scrf^{\mathbf{2},\mathbf{1}}\|_{L^{\infty}(\R^{12})}, \|H^{2}_{y_2}H^{-1}_{z_2} H_{y_1}\scrf^{\mathbf{2},\mathbf{1}}\|_{L^{\infty}(\R^{12})}   \Big)\\
&\hspace{4cm}\lesssim    \big(\sigma_1 \tau_1 \eta_1 \eta_2 (t_1-s_1)(t_2-s_2)   \big)^{-N}.
\end{align*}

 \end{lemma}

\smallskip

\begin{lemma} 
Let $\eps>0$. Then for every $p\geq 1$ large enough,  
\begin{equation}\label{TF7}
\|\mathcal{R}\scrf^{\mathbf{2},\mathbf{1}}\|_{ {L}^{p}(\R^{3})} \lesssim  \sigma^{-1-\eps}_1 \big(\eta_1 \eta_2 \tau_1\big)^{-\frac32-\eps}(t_1-s_1)^{-\frac14-\eps}(t_2-s_2)^{-\frac14-\eps},
\end{equation}
and 
\begin{equation}\label{intA7}
\sup_{n\geq 1}\, \| \scra^{\mathbf{2},\mathbf{1},(n)}_{t_1,t_2}\|_{L^p(\R^3)} \lesssim |t_2-t_1| ^{-6\eps}.
\end{equation}
\end{lemma}
 
\begin{proof}
Fix $\eps>0$ and apply Corollary \ref{coro:interpol-T} with $\la_1=\frac12$ and $\la_2=\la_3=\frac14$. This yields that for every $q>\frac32$ and every $p\geq 1$ large enough,
\small
\begin{align*}
&\|TF\|_{L^p(\R^3)} \lesssim \Big(1 \vee \|F\|_{L^{\infty}(\R^{12})} \Big)^{\frac12}\Big(1 \vee \sup_{y_1,y_2,z_2 \in \R^3}\Big(\int dz_1 \big| H_{y_1}  F\big|^q\Big)^{\frac1q}  \Big)^{\frac14}\Big(1 \vee \sup_{y_1,y_2,z_1 \in \R^3}\Big(\int dz_2 \big| H_{y_2}  F\big|^q\Big)^{\frac1q}  \Big)^{\frac14}\\
&\hspace{1cm} \Big(1\vee \|F\|_{\mathcal{H}^{16}(\R^{12})}\Big)^{\frac{\eps}{8N}} \Big(\|H_{y_1}  H_{y_2}F\|_{L^{\infty}(\R^{12})}\Big)^{\frac{\eps}{8N}} \Big(\|H^{2}_{y_1}H^{-1}_{z_1} H_{y_2}F\|_{L^{\infty}(\R^{12})}\Big)^{\frac{\eps}{8N}} \Big(\|H^{2}_{y_2}H^{-1}_{z_2} H_{y_1}F\|_{L^{\infty}(\R^{12})}\Big)^{\frac{\eps}{8N}} \\
& \lesssim \big( \eta_1 \eta_2 \tau_1\big)^{-\frac32}\sigma^{-\frac32+\frac3{4q}}_1 (t_1-s_1)^{-\frac14}  (t_2-s_2)^{-\frac14} \Big(\sigma_1 \tau_1 \eta_1 \eta_2 (t_1-s_1)(t_2-s_2) \Big)^{-\frac{\eps}{2}} 
\end{align*}
\normalsize
thanks to the estimates of Lemma \ref{lem:F7}. By choosing $q>\frac32$ such that $\frac32-\frac3{4q}=1+\frac{\eps}{2}$, we get \eqref{TF7}.

\

As for \eqref{intA7}, we integrate \eqref{TF7} in $\sigma_1, \eta_1, \eta_2, \tau_1$ and get first
 \begin{equation*} 
\sup_{n\geq 1}\, \big\|    \scri_{t,s}^{\mathbf{2},\mathbf{1},(n)} \big[\mathcal{R}\scrf^{\mathbf{2},\mathbf{1}}\big]\big\|_{L^p(\R^3)}  
\lesssim  |t_1-t_2|^{-2\eps}   |s_1 -s_2 |^{-\frac12-2\eps}   (t_1-s_1)^{-\frac34-\eps}(t_2-s_2)^{ -\frac34-\eps}.
   \end{equation*}
 The desired bound is then a consequence of Lemma \ref{Lem-prod2}.
    \end{proof}


 \subsection{Estimation of $\scra^{\mathbf{2},\mathbf{2},(n)}_{t_1,t_2}$}  We proceed as in the proof of Lemma \ref{Lem-F5} to get:

 \begin{lemma}\label{lem:F8-inf}
It holds that
\begin{equation*}  
\|\scrf^{\mathbf{2},\mathbf{2}}\|_{L^{\infty}(\R^{12})}   \lesssim   (t_1-s_1)^{-\frac34}  (t_2-s_2)^{-\frac34}    (\eta_1 \eta_2\eta_3 \eta_4)^{\frac98}.
\end{equation*}
Moreover, there exists $N\geq 1$ such that 
\begin{equation*}   
\max\Big( \|\scrf^{\mathbf{2},\mathbf{2}}\|_{\mathcal{H}^{16}(\R^{12})},\|H_{y_1}  H_{y_2}\scrf^{\mathbf{2},\mathbf{2}}\|_{L^{\infty}(\R^{12})}\Big)   \lesssim    \big(\eta_1 \eta_2 \eta_3 \eta_4 (t_1-s_1)(t_2-s_2)   \big)^{-N}.
\end{equation*}

 \end{lemma}

\

 \begin{lemma} 
  Let $\eps>0$. Then if $p\geq 1$ is large enough  
          \begin{equation}\label{TF8}
 \|\mathcal{R}\scrf^{\mathbf{2},\mathbf{2}}\|_{ {L}^{p}(\R^{3})} \lesssim   (t_1-s_1)^{-\frac34-\eps}  (t_2-s_2)^{-\frac34-\eps}    (\eta_1 \eta_2\eta_3 \eta_4)^{\frac98-\eps},
   \end{equation}
   and 
             \begin{equation}\label{intA8}
\sup_{n\geq 1}\, \| \scra^{\mathbf{2},\mathbf{2},(n)}_{t_1,t_2}\|_{L^p(\R^3)} \lesssim |t_2-t_1| ^{-6\eps}.
   \end{equation}
 \end{lemma}
 
 \begin{proof}
The bound \eqref{TF8} is derived from Lemma \ref{lem:F8-inf} along the same interpolation procedure as the one in the proof of Lemma \ref{lem:a12}.

\smallskip

For \eqref{intA8}, we integrate \eqref{TF8} in $\eta_1, \eta_2, \eta_3, \eta_4$ and get first
 \begin{equation*} 
\sup_{n\geq 1}\, \big\|    \scri^{\mathbf{2},\mathbf{2},(n)}_{t,s} \big[\mathcal{R}\scrf^{\mathbf{2},\mathbf{2}}\big]\big\|_{L^p(\R^3)}  
\lesssim  (t_1-s_1)^{-\frac78-2\eps}(t_2-s_2)^{-\frac78-2\eps}      (t_1-s_2)^{-\frac18-\eps}  (t_2-s_1)^{-\frac18-\eps}.
   \end{equation*}
    The claimed estimate is then obtained thanks to Lemma \ref{Lem-prod2}.
  \end{proof}

\
 
The combination of \eqref{intA7}-\eqref{intA8} and \eqref{scret2-main}-\eqref{decomposcret2} allows us to assert that
\begin{equation}
\dis \sup_{n\geq 1} \, \mathbb{E} \Big[ \big\|\screti^{\mathbf{2},(n)} \big\|_{\cac^{1-\eps}([0,T];\cb_x^{-\eta})}^{2p} \Big] <\infty, \quad \text{for all} \ 0<\eps,\eta<\frac12 \ \text{and} \ p\geq 1.\label{conclu-iti-2}
\end{equation}
 Gathering \eqref{conclu-iti-1} and \eqref{conclu-iti-2} finally provides us with the desired conclusion in \eqref{unif-fourth-2}.

\

\subsection{Study of $\scret^{\mathbf{3},(n)}$}\label{subsec:secret3}
 
Note that $\scret^{\mathbf{3},(n)}$ is a purely deterministic object. The following convergence statement is then clearly sufficient for our purpose.

\begin{proposition}
Let $T>0$. For every $0<\eta<\frac12$, there exists $\ka>0$ such that
\begin{eqnarray}
&\big\|\scret^{\mathbf{3},(n+1)} -\scret^{\mathbf{3},(n)}\big\|_{L^\infty([0,T]; \cb_x^{-\eta})}\lesssim 2^{-\ka n } .\label{mogener}
\end{eqnarray}
Consequently, the sequence $(\scret^{\mathbf{3},(n)})$ converges almost surely to an element $\scret^{\mathbf{3}}$ in  $L^{\infty}\big([0,T];\cb^{-\eta}_x\big)$, for every $\eta>0$.
\end{proposition}

\begin{proof}
Consider the function $F^{(n)}$ defined by
$$F^{(n)}_{t,s}(z,z'):=\int dw \, K_{t-s}(z',w) \cac^{(n)}_{t,s}(z,w)^2.$$
Then
\begin{align*}
&\scret^{\mathbf{3},(n)}_t(y)=\mathbb{E}\Big[\big(\<Psi2>^{(n)}_t \pe \<IPsi2>^{(n)}_t\big)(y)\Big]-  \mathbb{E}\Big[\<Psi2>^{(n)}_t(y)  \<IPsi2>^{(n)}_t(y)\Big]\\
&=\int_0^t ds\sum_{i\sim i'}\int dz dz'\, \delta_i(y,z)\delta_{i'}(y,z') \int dw\,  K_{t-s}(z',w) \mathbb{E}\Big[\<Psi2>^{(n)}_t(z) \<Psi2>^{(n)}_s(w)\Big]\\
&\hspace{5cm}-\int_0^t ds\int dw\,  K_{t-s}(y,w) \mathbb{E}\Big[\<Psi2>^{(n)}_t(y) \<Psi2>^{(n)}_s(w)\Big]\\
&=2 \int_0^t ds\bigg[\sum_{i\sim i'}\int dz dz'\, \delta_i(y,z)\delta_{i'}(y,z') \int dw\,  K_{t-s}(z',w) \cac^{(n)}_{t,s}(z,w)^2-\int dw\,  K_{t-s}(y,w) \cac^{(n)}_{t,s}(y,w)^2\bigg]\\
&=2 \int_0^t ds \bigg[\sum_{i\leq i'-4}\int dz dz'\, \delta_i(y,z)\delta_{i'}(y,z') F^{(n)}_{t,s}(z,z')+\sum_{i'\leq i-4}\int dz dz'\, \delta_i(y,z)\delta_{i'}(y,z') F^{(n)}_{t,s}(z,z')\bigg]\\
&=2  \Big[\scret^{\mathbf{3,1},(n)}_t(y)+\scret^{\mathbf{3,2},(n)}_t(y)\Big],
\end{align*}
where we have set
$$\scret^{\mathbf{3,1},(n)}_t(y):= \int_0^t ds  \sum_{i\leq i'-4}\int dz dz'\, \delta_i(y,z)(H^\eta_y\delta_{i'})(y,z') (H^{-\eta}_{z'}F^{(n)}_{t,s})(z,z')$$
and
$$\scret^{\mathbf{3,2},(n)}_t(y):= \int_0^t ds \sum_{i'\leq i-4}\int dz dz'\, (H^\eta_y\delta_i)(y,z)\delta_{i'}(y,z') (H^{-\eta}_z F^{(n)}_{t,s})(z,z').$$
Using the operator $\cm^{(\eta)}$ introduced in \eqref{defMal}, we can write
$$2^{-2j\eta}\delta_j \big( \scret^{\mathbf{3,1},(n)}_t\big)(x)= \int_0^t ds \, \cm^{(\eta)}_{j,(z,z')\to x}\big(H^{-\eta}_{z'}F^{(n)}_{t,s}\big)(x)$$
and
$$2^{-2j\eta}\delta_j \big( \scret^{\mathbf{3,2},(n)}_t\big)(x)= \int_0^t ds\, \cm^{(\eta)}_{j,(z,z')\to x}\big(H^{-\eta}_{z}F^{(n)}_{t,s}\big)(x).$$
Thanks to Lemma \ref{Lem-Comp}, we obtain that
\begin{align}
& \big\|\scret^{\mathbf{3},(n)}_t \big\|_{\cb_x^{-2\eta}}\lesssim \int_0^t ds\, \big\|H^{-\eta}_{z'}F^{(n)}_{t,s}\big\|_{L^\infty_{z,z'}}+\int_0^t ds\, \big\|H^{-\eta}_{z}F^{(n)}_{t,s}\big\|_{L^\infty_{z,z'}}.\label{appli-lem-comp}
\end{align}
Now, on the one hand, for $q\geq 1$ defined by the relation $\frac{1}{q}=1-\frac{\eta}{3}$, one has
\begin{align}
\big|(H^{-\eta}_{z'}F^{(n)}_{t,s})(z,z')\big|&=\int dw \, \big|(H^{-\eta}_{z'}K_{t-s})(z',w)\big| \big| \cac^{(n)}_{t,s}(z,w)\big|^2\nonumber\\
&\leq \Big( \int dw_1 \, \big|(H^{-\eta}_{z'}K_{t-s})(z',w_1)\big|^q\Big)^{\frac{1}{q}}\Big(\int dw_2\, \big| \cac^{(n)}_{t,s}(z,w_2)\big|^{\frac{6}{\eta}} \Big)^{\frac{\eta}{3}}\nonumber\\
&\lesssim \frac{1}{|t-s|^{1-\frac{\eta}{6}}} \Big( \int dw_1 \, \big|(H^{-\eta}_{z'}K_{t-s})(z',w_1)\big|^q\Big)^{\frac{1}{q}}\Big(\int dw_2\, \big| \cac^{(n)}_{t,s}(z,w_2)\big|\Big)^{\frac{\eta}{3}}\nonumber\\
&\lesssim \frac{1}{|t-s|^{1-\frac{\eta}{6}}} \big\| K_{t-s}(z',.)\big\|_{L^1_w}\lesssim \frac{1}{|t-s|^{1-\frac{\eta}{6}}} ,\label{appli-lem-comp-1}
\end{align}
uniformly over $n,z,z'$, and where we have used the Sobolev embedding $L^1(\R^3) \subset \cw^{-2\eta,q}(\R^3)$ to deduce the fourth inequality.

\smallskip

In a similar way, using the Sobolev embedding $ L^{\frac{3}{\eta}}(\R^3) \subset \cw^{-2\eta,\infty}(\R^3)$, we get that
\begin{align}
\big|(H^{-\eta}_{z}F^{(n)}_{t,s})(z,z')\big|&=\int dw \, \big|K_{t-s}(z',w)\big| \big| \big(H^{-\eta}_{z}(\cac^{(n)}_{t,s})^2\big)(z,w)\big|\nonumber\\
&\lesssim \sup_{w} \big| \big(H^{-\eta}_{z}(\cac^{(n)}_{t,s})^2\big)(z,w)\big|\nonumber\\
&\lesssim \sup_{w}  \Big(\int dz \, \big|\cac^{(n)}_{t,s}(z,w)\big|^{\frac{6}{\eta}} \Big)^{\frac{\eta}{3}}\nonumber\\
&\lesssim \frac{1}{|t-s|^{1-\frac{\eta}{6}}} \sup_{w}  \Big(\int dz \, \big|\cac^{(n)}_{t,s}(z,w)\big| \Big)^{\frac{1}{p}}\lesssim \frac{1}{|t-s|^{1-\frac{\eta}{6}}}. \label{appli-lem-comp-2}
\end{align}
By injecting \eqref{appli-lem-comp-1} and \eqref{appli-lem-comp-2} into \eqref{appli-lem-comp}, we obtain the uniform control
\begin{align*}
& \sup_{n\geq 1}\, \big\|\scret^{\mathbf{3},(n)} \big\|_{L^\infty([0,T];\cb_x^{-2\eta})}<\infty .
\end{align*}
For the sake of conciseness, we leave the reader to check that the more general bound \eqref{mogener} could be derived from the same steps.
\end{proof}


\section{Fifth order diagram}\label{section:fifth-order-diagram}

We finally consider the case of the fifth-order diagram introduced in \eqref{ord5}, that is,
$$\<Psi2IPsi3>^{(n)}_t(x):=\Big(\<Psi2>^{(n)}_t \pe \<IPsi3>^{(n)}_t\Big)(x) - 3\, \frakc^{\mathbf{2},(n)}_t(x) \, \<Psi>_t^{(n)}{(x)} ,$$
where the deterministic sequence $\frakc^{\mathbf{2},(n)}$ is given by
$$\frakc^{\mathbf{2},(n)}_t(x):=\mathbb{E}\Big[ \<Psi2>^{(n)}_t(x) \<IPsi2>^{(n)}_t(x)\Big].$$

Our main convergence statement for $(\<Psi2IPsi3>^{(n)})$,  reads as follows.

\begin{proposition}\label{prop-10.1}
Let $T>0$. For all $0<\eps,\eta<\frac12$, there exists $\ka>0$ such that
\begin{eqnarray*}
&\mathbb{E} \Big[ \Big\|\widetilde{\<Psi2IPsi3>}^{(n+1)} -\widetilde{\<Psi2IPsi3>}^{(n)}\Big\|_{\cac^{\frac34-\eps}([0,T]; \cb_x^{-\frac14-\eta})}^{2p} \Big]\lesssim 2^{-\ka n p} .
\end{eqnarray*}
Consequently, the sequence $(\widetilde{\<Psi2IPsi3>}^{(n)})$ converges almost surely to an element $\widetilde{\<Psi2IPsi3>}$ in  $\cac^{\frac34-\eps}([0,T]; \cb_x^{-\frac14-\eta})$, for all $\varepsilon,\eta>0$.

Moreover, the following uniform in time estimate holds true:
\begin{equation*} 
\sup_{r\geq 0}    \mathbb{E} \Big[ \Big\|\widetilde{ \<Psi2IPsi3>}^{(n+1)}   - \widetilde{ \<Psi2IPsi3>}^{(n)}\Big\|_{{\ov \cac}^{\frac34-\eps}([r,r+1];  \cb_x^{-\frac14-\eta})}^{2p} \Big]\lesssim 2^{-\ka n p }. 
\end{equation*}
\end{proposition}

\medskip

Observe first that $\frakc^{\mathbf{2},(n)}_t(x)$ can be recast into 
\begin{align}
\frakc^{\mathbf{2},(n)}_t(x)&= \int_0^t ds \int dw \, K_{t-s}(x,w)\mathbb{E}\Big[ \<Psi2>^{(n)}_t(x) \<Psi2>^{(n)}_s(w)\Big] =2\int_0^t ds \int dw \, K_{t-s}(x,w)\cac^{(n)}_{t,s}(x,w)^2. \label{identif-c2}
\end{align}

\smallskip

Besides, one has
\begin{align*}
&\Big(\<Psi2>^{(n)}_t\pe\<IPsi3>^{(n)}_t\Big)(x)=\sum_{i\sim i'}\delta_{i}\big(\<Psi2>^{(n)}_t\big)(x) \delta_{i'}\big(\<IPsi3>^{(n)}_t\big)(x) \\
&=\sum_{i\sim i'}\int dz dy \, \delta_i(x,z)\delta_{i'}(x,y) \<Psi2>^{(n)}_t(z) \int_0^t ds \int dw \, K_{t-s}(y,w) \<Psi3>^{(n)}_s(w)\\
&=\sum_{i\sim i'}\int dz dy \, \delta_i(x,z)\delta_{i'}(x,y)  \int_0^t ds \int dw \, K_{t-s}(y,w) I^W_2\big(F^{(n)}_{t,z}\otimes F^{(n)}_{t,z}\big)I^W_3\big(F^{(n)}_{s,w}\otimes F^{(n)}_{s,w}\otimes F^{(n)}_{s,w}\big).
\end{align*}
Using the multiplication rule \eqref{prod-rule-mult-int} and the identity \eqref{identif-c2}, we obtain the decomposition
\begin{align*}
&\<Psi2IPsi3>^{(n)}_t(x)=\calt^{\mathbf{1},(n)}_t(x)+6\, \calt^{\mathbf{2},(n)}_t(x)+6\, \calt^{\mathbf{3},(n)}_t(x)+6\, \calt^{\mathbf{4},(n)}_t(x)+6\, \calt^{\mathbf{5},(n)}_t(x),
\end{align*}
with
\begin{align*}
\calt^{\mathbf{1},(n)}_t(x):=\sum_{i\sim i'}\int dz dy \, \delta_i(x,z)\delta_{i'}(x,y)  \int_0^t ds \int dw \, K_{t-s}(y,w) I^W_5\big(F^{(n)}_{t,z}\otimes F^{(n)}_{t,z}\otimes F^{(n)}_{s,w}\otimes F^{(n)}_{s,w}\otimes F^{(n)}_{s,w}\big),
\end{align*}
\begin{align*}
\calt^{\mathbf{2},(n)}_t(x):=\sum_{i\sim i'}\int dz dy \, \delta_i(x,z)\delta_{i'}(x,y)  \int_0^t ds \int dw \, K_{t-s}(y,w) \cac^{(n)}_{t,s}(z,w)I^W_3\big(F^{(n)}_{t,z}\otimes F^{(n)}_{s,w}\otimes F^{(n)}_{s,w}\big),
\end{align*}
\begin{align*}
&\calt^{\mathbf{3},(n)}_t(x):=\\
&\sum_{i\sim i'}\int dz dy \, \delta_i(x,z)\delta_{i'}(x,y)  \int_0^t ds \int dw \, K_{t-s}(y,w)\cac^{(n)}_{t,s}(z,w)^2 \<Psi>^{(n)}_s(w) -\int_0^t ds \int dw \, K_{t-s}(x,w)\cac^{(n)}_{t,s}(x,w)^2 \<Psi>^{(n)}_s(w),
\end{align*}
\begin{align*}
\calt^{\mathbf{4},(n)}_t(x)&:=\int_0^t ds \int dw \, K_{t-s}(x,w)\cac^{(n)}_{t,s}(x,w)^2 \big(\<Psi>^{(n)}_s(w)-\<Psi>^{(n)}_t(w)\big),
\end{align*}
\begin{align*}
\calt^{\mathbf{5},(n)}_t(x)&:= \int_0^t ds \int dw \, K_{t-s}(x,w)\cac^{(n)}_{t,s}(x,w)^2 \big(\<Psi>^{(n)}_t(w)-\<Psi>^{(n)}_t(x)\big).
\end{align*}

Along our recurring convention, we set
$$\calti^{\mathbf{a},(n)}_t(x):=\int_0^t \, \calt^{\mathbf{a},(n)}_s(x)\, ds.$$

The convergence result for $(\<Psi2IPsi3>^{(n)})$ in Proposition \ref{prop-10.1} is then an immediate consequence of the following statement:

\begin{proposition}
Let $T>0$. For $\mathbf{a}=1,\ldots,5$ and for all $0<\eps,\eta<\frac12$, there exists $\ka>0$ such that
\begin{eqnarray}\label{sp1}
&\mathbb{E} \Big[ \big\|\calti^{\mathbf{a},(n+1)} -\calti^{\mathbf{a},(n)}\big\|_{\cac^{\frac34-\eps}([0,T]; \cb_x^{-\frac14-\eta})}^{2p} \Big]\lesssim 2^{-\ka n p} .
\end{eqnarray}
Consequently, the sequence $(\calti^{\mathbf{a},(n)})$ converges almost surely to an element $\calti^{\mathbf{a}}$ in  $\cac^{\frac34-\eps}([0,T]; \cb_x^{-\frac14-\eta})$, for all $\varepsilon,\eta>0$.

Moreover, 
\begin{equation}\label{sp}
\sup_{r\geq 0}    \mathbb{E} \Big[ \Big\|\calti^{\mathbf{a},(n+1)} -\calti^{\mathbf{a},(n)}\Big\|_{{\ov \cac}^{\frac34-\eps}([r,r+1]; \cb_x^{-\frac14-\eta})}^{2p} \Big]\lesssim 2^{-\ka n p }. 
\end{equation}

\end{proposition}

\smallskip

For the sake of conciseness, and just as in the previous sections, we will focus on the proof of the uniform bounds
\begin{eqnarray}\label{gene-bound-5}
&\dis \sup_{n\geq 1} \, \mathbb{E} \Big[ \big\|\calti^{\mathbf{a},(n)} \big\|_{\cac^{\frac34-\eps}([0,T]; \cb_x^{-\frac14-\eta})}^{2p} \Big] <\infty, \quad \quad \mathbf{a}=1,\ldots,5.
\end{eqnarray}

In the same manner, we do not give the details of the proof of \eqref{sp} since it is obtained similarly as \eqref{sp1} (see Section \ref{Sect-fourth1}     for the details).
 
\smallskip

The rest of the section is thus devoted to the proof of \eqref{gene-bound-5}. To be more specific, the result can be derived from the combination of Corollary \ref{corol-t1}, Corollary \ref{corol-t2}, Proposition \ref{prop-fifth-order-3} and Proposition~\ref{prop-fifth-order-4-5} below.

\

\subsection{Study of $\calti^{\mathbf{1},(n)}$}\label{subsec:tildet1}

 Fix $0<\ga <\frac34$ and let $0<\varepsilon,\eta <\frac12$. For the same reasons as in Section~\ref{subsec:ref-raiso} (see \eqref{normlip}-\eqref{hyperc}-\eqref{uno}), we can directly assert that for every $p\geq 2$ large enough,
\begin{align}
&\mathbb{E}\Big[\big\|\calti^{\mathbf{1},(n)} \big\|_{\cac^{\ga}([0,T];\cb_x^{-\frac14-\eta})}^{2p}\Big]\lesssim \mathbb{E}\Big[\big\|\calti^{\mathbf{1},(n)} \big\|_{\cac^{\ga}([0,T];\cb_x^{-\eta})}^{2p}\Big]\nonumber \\
&\hspace{2cm}\lesssim \int_0^T\int_0^T  \frac{dv_1 dv_2}{|v_2-v_1|^{2p\ga+2}}\bigg( \int_{[v_1,v_2]^2}dt_1 dt_2\, \bigg(\int dx \,  \mathbb{E}\Big[ \calt^{\mathbf{1},(n)}_{t_1}(x)   \calt^{\mathbf{1},(n)}_{t_2}(x) \Big]^p\bigg)^{\frac1p} \bigg)^p.\label{combi0-5}
\end{align}

\

Given the definition of $\calt^{\mathbf{1},(n)}$, one has
\begin{align}\label{dos-0-5}
&\mathbb{E}\Big[ \calt^{\mathbf{1},(n)}_{t_1}(x)   \calt^{\mathbf{1},(n)}_{t_2}(x) \Big]=\int_0^{t_1}ds_1\int_0^{t_2}ds_2\, \sum_{i \sim i'} \sum_{j\sim j'}  \int dz_1 dy_1dz_2 dy_2 \, \delta_i(x,z_1) \delta_{i'}(x,y_1) \delta_j(x,z_2) \delta_{j'}(x,y_2)\,  \nonumber \\
&\hspace{0.5cm} \int dw_1 dw_2 \, K_{t_1-s_1}(y_1,w_2)K_{t_2-s_2}(y_2,w_2)\nonumber\\
&\mathbb{E}\Big[I^W_5\Big(F^{(n)}_{t_1,z_1}\otimes F^{(n)}_{t_1,z_1}\otimes F^{(n)}_{s_1,w_1}\otimes F^{(n)}_{s_1,w_1}\otimes F^{(n)}_{s_1,w_1}\Big)I^W_5\Big(F^{(n)}_{t_2,z_2}\otimes F^{(n)}_{t_2,z_2}\otimes F^{(n)}_{s_2,w_2}\otimes F^{(n)}_{s_2,w_2}\otimes F^{(n)}_{s_2,w_2}\Big) \Big].
\end{align}
By \eqref{ortho-rule}, the latter expectation can be readily expanded as
\small
\begin{align*}
&\mathbb{E}\Big[I^W_5\Big(F^{(n)}_{t_1,z_1}\otimes F^{(n)}_{t_1,z_1}\otimes F^{(n)}_{s_1,w_1}\otimes F^{(n)}_{s_1,w_1}\otimes F^{(n)}_{s_1,w_1}\Big)I^W_5\Big(F^{(n)}_{t_2,z_2}\otimes F^{(n)}_{t_2,z_2}\otimes F^{(n)}_{s_2,w_2}\otimes F^{(n)}_{s_2,w_2}\otimes F^{(n)}_{s_2,w_2}\Big) \Big]\\
&=c\, \Big\langle \text{Sym}\Big(F^{(n)}_{t_1,z_1}\otimes F^{(n)}_{t_1,z_1}\otimes F^{(n)}_{s_1,w_1}\otimes F^{(n)}_{s_1,w_1}\otimes F^{(n)}_{s_1,w_1}\Big), \text{Sym}\Big(F^{(n)}_{t_2,z_2}\otimes F^{(n)}_{t_2,z_2}\otimes F^{(n)}_{s_2,w_2}\otimes F^{(n)}_{s_2,w_2}\otimes F^{(n)}_{s_2,w_2}\Big) \Big\rangle_{L^2((\R_+\times \R^3)^5)}\\
&=\sum_{\mathbf{b}=1}^3 c_{\mathbf{b}} \cq^{\mathbf{1},\mathbf{b},(n)}_{t,s}({z,w})
\end{align*}
\normalsize
for some combinatorial coefficients $c,c_{\mathbf{b}}\geq 0$, and with
\begin{align*}
 \cq^{\mathbf{1},\mathbf{1},(n)}_{t,s}(z,w)&:=\cac^{(n)}_{t_1,t_2}(z_1,z_2)^2 \cac^{(n)}_{s_1,s_2}(w_1,w_2)^3,\\
 \cq^{\mathbf{1},\mathbf{2},(n)}_{t,s}(z,w)&:=\cac^{(n)}_{t_1,t_2}(z_1,z_2) \cac^{(n)}_{t_1,s_2}(z_1,w_2) \cac^{(n)}_{t_2,s_1}(z_2,w_1) \cac^{(n)}_{s_1,s_2}(w_1,w_2)^2,\\
 \cq^{\mathbf{1},\mathbf{3},(n)}_{t,s}(z,w)&:=\cac^{(n)}_{t_1,s_2}(z_1,w_2)^2 \cac^{(n)}_{t_2,s_1}(z_2,w_1)^2 \cac^{(n)}_{s_1,s_2}(w_1,w_2).
\end{align*}
Going back to \eqref{dos-0-5} and using the operator $\mathcal{R}$ (see \eqref{def-op-T}), we obtain that
\begin{align}
&\mathbb{E}\Big[ \calt^{\mathbf{1},(n)}_{t_1}(x)   \calt^{\mathbf{1},(n)}_{t_2}(x) \Big]=\sum_{\mathbf{b}=1}^3 c_{\mathbf{b}}\ca^{\mathbf{1,b},(n)}_{t_1,t_2}(x),\label{combi1-5}
\end{align}
with
\begin{align*}
&\ca^{\mathbf{1,b},(n)}_{t_1,t_2}(x):=\int_0^{t_1} ds_1 \int_0^{t_2} ds_2 \, \mathcal{R}\Big(  \int dw_1 dw_2\,   K_{t_1-s_1}(y_1,w_1)   K_{t_2-s_2}(y_2,w_2)  \cq^{\mathbf{1},\mathbf{b},(n)}_{t,s}({z,w})\Big)(x).
\end{align*}

\begin{proposition}\label{prop-defT-51}
Fix $\mathbf{b}=1,2,3$. Then for every small $\eps>0$ and every $p\geq 1$ large enough, one has
\begin{equation} \label{combi2-5}
\sup_{n\geq 1} \, \| \ca^{\mathbf{1,b},(n)}_{t_1,t_2}\|_{L^p(\R^3)} \lesssim \frac{1}{|t_2-t_1|^{\frac12+\eps}}.
\end{equation}

\end{proposition}

\

Before we turn to the proof of this technical result, observe that by combining \eqref{combi0-5}-\eqref{combi1-5} with \eqref{combi2-5}, we obtain, for every $\eps>0$ and every $p\geq 1$ large enough,
\begin{align*}
\sup_{n\geq 1} \, \mathbb{E} \Big[ \big\|\calti^{\mathbf{1},(n)} \big\|_{\cac^{\ga}([0,T];\cb_x^{-\eta})}^{2p} \Big]&\lesssim \int_0^T\int_0^T  \frac{dv_1 dv_2}{|v_2-v_1|^{2p\ga+2}}\bigg( \int_{[v_1,v_2]^2}\frac{dt_1 dt_2}{|t_2-t_1|^{\frac12+\eps}}\bigg)^p\\
& \lesssim \int_0^T\int_0^T dv_1 dv_2 \, \frac{|v_2-v_1|^{(\frac32-\eps)p}}{|v_2-v_1|^{2\ga p+2}}  ,
\end{align*}
which leads us to the desired statement:

\begin{corollary}\label{corol-t1}
Fix $T>0$. For all $0< \ga <\frac34$ and $\eta>0$, it holds that
\begin{eqnarray*}
&\dis \sup_{n\geq 1} \, \mathbb{E} \Big[ \big\|\calti^{\mathbf{1},(n)} \big\|_{\cac^{\ga}([0,T];\cb_x^{-\eta})}^{2p} \Big] <\infty.
\end{eqnarray*}
\end{corollary}

\

Proposition \ref{prop-defT-51} is now derived from the combination of Propositions \ref{prop:a11-5}, \ref{prop:a12-5} and \ref{prop:a13-5} below.


\subsubsection{Estimation of $\ca^{\mathbf{1},\mathbf{1},(n)}_{t_1,t_2}$}

Based on the representation \eqref{def-E} of $\cac^{(n)}$, we can write
\begin{multline*}
\cq^{\mathbf{1},\mathbf{1},(n)}_{t,s}({z,w})= I_{t,s}^{\mathbf{1},\mathbf{1},(n)} K^{\mathbf{1},\mathbf{1}}_{\sigma,\tau, \eta}(z,w) \\
:= \frac1{32}   \int_{|t_1-t_2|+\varepsilon_n}^{t_1+t_2+\varepsilon_n} d\sigma_1\int_{|t_1-t_2|+\varepsilon_n}^{t_1+t_2+\varepsilon_n} d\sigma_2 \int_{|s_1-s_2|+\varepsilon_n}^{s_1+s_2+\varepsilon_n} d\tau_1   \int_{|s_1-s_2|+\varepsilon_n}^{s_1+s_2+\varepsilon_n} d\tau_2    \int_{|s_1-s_2|+\varepsilon_n}^{s_1+s_2+\varepsilon_n} d\tau_3 \, K^{\mathbf{1},\mathbf{1}}_{\sigma,\tau, \eta}(z,w), 
 \end{multline*}
 where 
 \begin{align*}
K^{\mathbf{1},\mathbf{1}}_{\sigma,\tau, \eta}(z,w) &:=K_{\sigma_1}(z_1,z_2) K_{\sigma_2}(z_1,z_2)K_{\tau_1}(w_1,w_2)K_{\tau_2}(w_1,w_2)K_{\tau_3}(w_1,w_2).
 \end{align*}
With this notation, one has
\begin{eqnarray*} 
\ca^{\mathbf{1},\mathbf{1},(n)}_{t_1,t_2}(x)&=&  \int_0^{t_1} ds_1 \int_0^{t_2} ds_2   \, \mathcal{R}\Big( \int dw_1  dw_2 \,  K_{t_1-s_1}(y_1,w_1)   K_{t_2-s_2}(y_2,w_2) \cq^{\mathbf{1},\mathbf{1},(n)}_{t,s}({ z,w})\big)(x)\nonumber  \\
&=& \int_0^{t_1} ds_1 \int_0^{t_2} ds_2\,  I_{t,s}^{\mathbf{1},\mathbf{1},(n)} \big[\mathcal{R}\cf^{\mathbf{1},\mathbf{1}}\big](x),
 \end{eqnarray*}
where the function $\cf^{\mathbf{1},\mathbf{1}}$ is defined by 
\begin{equation*} 
 \cf^{\mathbf{1},\mathbf{1}}(y_1,y_2,z_1,z_2):=   \int dw_1 dw_2 \,K_{t_1-s_1}(y_1,w_1)K_{t_2-s_2}(y_2,w_2)K^{\mathbf{1},\mathbf{1}}_{\sigma,\tau,\eta}(z_1,w_1,z_2,w_2).
\end{equation*}
In this way,
\begin{align}\label{repre}
\big\|\ca^{\mathbf{1},\mathbf{1},(n)}_{t_1,t_2}\|_{L^p(\R^3)} &\leq \int_0^{t_1} ds_1 \int_0^{t_2} ds_2\, \big\| I_{t,s}^{\mathbf{1},\mathbf{1},(n)} \big[\mathcal{R}\cf^{\mathbf{1},\mathbf{1}}\big]\big\|_{L^p(\R^3)}.
 \end{align} 

\smallskip

\begin{lemma}  \label{lem-F1-5}
The following bounds hold true:
\begin{equation}  \label{f1-inf-5}
\|\cf^{\mathbf{1},\mathbf{1}}\|_{L^{\infty}(\R^{12})}   \lesssim  \sigma_1^{-\frac32} \sigma_2^{-\frac32} (\tau_1 \tau_2 \tau_3)^{-\frac54}\inf\big((t_1-s_1)^{-\frac34},(t_2-s_2)^{-\frac34}\big),
\end{equation}
and for all $q>\frac32$
\begin{equation}  \label{grad-f1-5}
 \sup_{y_1,y_2,z_2 \in \R^3} \big(\int dz_1 \big| H_{y_1}  \cf^{\mathbf{1},\mathbf{1}}\big|^q\big)^{\frac{1}{q}}  \lesssim \sigma_1^{-\frac32+\frac3{4q}} \sigma_2^{-\frac32+\frac3{4q}}    (t_2-s_2)^{-\frac34} (\tau_1 \tau_2 \tau_3)^{{ -\frac54}}(t_1-s_1)^{-1},
 \end{equation}
\begin{equation}  \label{grad-f1-bis-5}
 \sup_{y_1,y_2,z_1 \in \R^3} \big(\int dz_2 \big| H_{y_2}  \cf^{\mathbf{1},\mathbf{1}}\big|^q\big)^{\frac{1}{q}}  \lesssim \sigma_1^{-\frac32+\frac3{4q}} \sigma_2^{-\frac32+\frac3{4q}}    (t_1-s_1)^{-\frac34} (\tau_1 \tau_2 \tau_3)^{{ -\frac54}}(t_2-s_2)^{-1}.
 \end{equation}
 \end{lemma}

\begin{proof}
Recall that by \eqref{est-e1}, one has for all $w_1,y_2 \in \R^3$
\begin{align}\label{est-e1-5}
\int dw_2  \, K_{t_2-s_2}(y_2,w_2)  K_{\tau_1}(w_1,w_2)K_{\tau_2}(w_1,w_2)K_{\tau_3}(w_1,w_2) \lesssim (t_2-s_2)^{-\frac34} (\tau_1 \tau_2 \tau_3)^{-\frac54}.
\end{align}
Then for all $y_1,y_2,z_1,z_2 \in \R^3$, 
\begin{align*}
\cf^{\mathbf{1},\mathbf{1}}(y_1,y_2,z_1,z_2)&\lesssim   \|K_{\sigma_1}\|_{L^\infty_{z_1,z_2}}\|K_{\sigma_2}\|_{L^\infty_{z_1,z_2}}  \|K_{t_1-s_1}(y_1, \cdot)   \|_{L_{w_1}^1} (t_2-s_2)^{-\frac34} (\tau_1 \tau_2 \tau_3)^{-\frac54}\\
&\lesssim    \sigma_1^{-\frac32}\sigma_2^{-\frac32}(t_2-s_2)^{-\frac34} (\tau_1 \tau_2 \tau_3)^{-\frac54}. \nonumber
\end{align*}
With symmetric arguments, we easily obtain that
\begin{align*}
\cf^{\mathbf{1},\mathbf{1}}(y_1,y_2,z_1,z_2)&\lesssim    \sigma_1^{-\frac32}\sigma_2^{-\frac32}(t_1-s_1)^{-\frac34} (\tau_1 \tau_2 \tau_3)^{-\frac54}, \nonumber
\end{align*}
and thus the proof of \eqref{f1-inf-5} is complete.

\medskip
 
As far as \eqref{grad-f1-5} is concerned, observe first that
\begin{multline*} 
 \big(H_{y_1}\cf^{\mathbf{1},\mathbf{1}} \big)(y_1,y_2,z_1,z_2)=\\
 = K_{\sigma_1}(z_1,z_2) K_{\sigma_2}(z_1,z_2) \int dw_1 dw_2 \,\big( H_{y_1}K_{t_1-s_1}\big)(y_1,w_1)K_{t_2-s_2}(y_2,w_2) K_{\tau_1}(w_1,w_2)K_{\tau_2}(w_1,w_2)K_{\tau_3}(w_1,w_2) ,
\end{multline*} 
so that
 \begin{multline*} 
 \|\big(H_{y_1}\cf^{\mathbf{1},\mathbf{1}} \big)(y_1,y_2,\cdot ,z_2)\|_{L_{z_1}^q(\R^3)}\lesssim \big\| K_{\sigma_1}(\cdot,z_2)  \big\|_{L_{z_1}^{2q}(\R^3)}\big\| K_{\sigma_2}(\cdot,z_2)   \big\|_{L_{z_1}^{2q}(\R^3)}\\
    \int dw_1 \,\Big| \big( H_{y_1}K_{t_1-s_1}\big)(y_1,w_1)\Big| \int dw_2 \, K_{t_2-s_2}(y_2,w_2)K_{\tau_1}(w_1,w_2)K_{\tau_2}(w_1,w_2)K_{\tau_3}(w_1,w_2)  .
\end{multline*}
By \eqref{est-e1-5} and \eqref{normeLpp}, we deduce that 
$$ \|\big(H_{y_1}\cf^{\mathbf{1},\mathbf{1}} \big)(y_1,y_2,\cdot ,z_2)\|_{L_{z_1}^q(\R^3)}\lesssim    \sigma_1^{\frac3{4q}-\frac32} \sigma_2^{\frac3{4q}-\frac32}    (t_2-s_2)^{-\frac34} (\tau_1 \tau_2 \tau_3)^{{ -\frac54}}  \Big( \int dw_1   \,\Big| \big( H_{y_1}K_{t_1-s_1}\big)(y_1,w_1)\Big|  \Big), $$
and we obtain \eqref{grad-f1-5} thanks to \eqref{L1Linf}. The bound \eqref{grad-f1-bis-5} can then be derived from symmetric arguments.
\end{proof}

With the same arguments as in Lemma \ref{lemma-crude}, we also derive the following rough estimates about~$\cf^{\mathbf{1},\mathbf{1}}$.
 
\begin{lemma}\label{lemma-crude-5}
There exists $N\geq 1$ such that 
\begin{multline*}  
\max\Big(\|\cf^{\mathbf{1},\mathbf{1}}\|_{\mathcal{H}^{16}(\R^{12})}  ,\|H_{y_1}  H_{y_2}\cf^{\mathbf{1},\mathbf{1}}\|_{L^{\infty}(\R^{12})}, \|H^{2}_{y_1}H^{-1}_{z_1} H_{y_2}\cf^{\mathbf{1},\mathbf{1}}\|_{L^{\infty}(\R^{12})}, \|H^{2}_{y_2}H^{-1}_{z_2} H_{y_1}\cf^{\mathbf{1},\mathbf{1}}\|_{L^{\infty}(\R^{12})} \Big) \\
\lesssim  \big(\sigma_1\si_2 \tau_1 \tau_2 \tau_3 (t_1-s_1)(t_2-s_2)   \big)^{-N}.
\end{multline*}

\end{lemma}

\begin{proposition} \label{prop:a11-5}
Let $0<\eps<\frac12$. Then for every $p\geq 1$ large enough, it holds that
\begin{equation}\label{intA1-5}
\sup_{n\geq 1}\, \| \ca^{\mathbf{1},\mathbf{1},(n)}_{t_1,t_2}\|_{L^p(\R^3)} \lesssim \frac{1}{|t_1-t_2|^{\frac12+\eps}}.
\end{equation}
\end{proposition}

\begin{proof}  
Fix $0<\eps<\frac12$. By applying Corollary \ref{coro:interpol-T} with $\la_1=\frac12$ and $\la_2=\la_3=\frac14$, we get that for every $q>\frac32$ and every $p\geq 1$ large enough,
\small
\begin{align*}
&\|\mathcal{R}\cf^{\mathbf{1},\mathbf{1}}\|_{L^p(\R^3)}\\
& \lesssim \Big( 1\vee\|\cf^{\mathbf{1},\mathbf{1}}\|_{L^{\infty}(\R^{12})}\Big)^{\frac12}   \Big(1\vee \sup_{y_1,y_2,z_2 \in \R^3} \big(\int dz_1 \big| H_{y_1}  \cf^{\mathbf{1},\mathbf{1}}\big|^q\big)^{\frac1q} \Big)^{\frac14}\Big(1 \vee \sup_{y_1,y_2,z_1 \in \R^3}\Big(\int dz_2 \big| H_{y_2}  \cf^{\mathbf{1},\mathbf{1}}\big|^q\Big)^{\frac1q}  \Big)^{\frac14} \\
&\Big(1\vee \|\cf^{\mathbf{1},\mathbf{1}}\|_{\mathcal{H}^{16}(\R^{12})}\Big)^{\frac{\eps}{8N}}\Big(\|H_{y_1}  H_{y_2}\cf^{\mathbf{1},\mathbf{1}}\|_{L^{\infty}(\R^{12})}\Big)^{\frac{\eps}{8N}} \Big(\|H^{2}_{y_1}H^{-1}_{z_1} H_{y_2}\cf^{\mathbf{1},\mathbf{1}}\|_{L^{\infty}(\R^{12})} \Big)^{\frac{\eps}{8N}} \Big(\|H^{2}_{y_2}H^{-1}_{z_2} H_{y_1}\cf^{\mathbf{1},\mathbf{1}}\|_{L^{\infty}(\R^{12})} \Big)^{\frac{\eps}{8N}} ,
\end{align*}
\normalsize
where $N$ is the fixed integer introduced in Lemma \ref{lemma-crude-5}.

\smallskip

We can now inject the estimates of Lemma \ref{lem-F1-5} and Lemma \ref{lemma-crude-5} to obtain that for every $q>\frac32$ and every $p\geq 1$ large enough,
\begin{align*}
&\|\mathcal{R}\cf^{\mathbf{1},\mathbf{1}}\|_{L^p(\R^3)} 
 \lesssim  \sigma_1^{-\frac12(3-\frac{3}{4q})} \sigma_2^{-\frac12(3-\frac{3}{4q})}\big(  \tau_1 \tau_2 \tau_3  \big)^{-\frac54}  (t_1-s_1)^{-\frac{5}{8}}  (t_2-s_2)^{-\frac{5}{8}}   
\Big(\sigma_1\si_2 \tau_1 \tau_2  \tau_3  (t_1-s_1)(t_2-s_2)   \Big)^{-\frac{\eps}{2}} .
\end{align*}

\smallskip

Then, going back to \eqref{repre}, we deduce that for every $q>\frac32$ and every $p\geq 1$ large enough,
\begin{align*} 
&\sup_{n\geq 1}\, \big\|\ca^{\mathbf{1},\mathbf{1},(n)}_{t_1,t_2}\|_{L^p(\R^3)}\lesssim \int_0^{t_1}ds_1\int_0^{t_2}ds_2\int_{|t_1-t_2|}^{+\infty} d\sigma_1 \int_{|s_1-s_2|}^{+\infty} d\tau_1   \int_{|s_1-s_2|}^{+\infty} d\tau_2    \int_{|s_1-s_2|}^{+\infty} d\tau_3 \, \big\|   \mathcal{R}\cf^{\mathbf{1},\mathbf{1}}\big\|_{L^p(\R^3)}\\
& \lesssim  \int_0^{t_1}\frac{ds_1}{(t_1-s_1)^{\frac{5}{8}+\frac{\varepsilon}{2}}}\int_0^{t_2}\frac{ds_2}{(t_2-s_2)^{\frac{5}{8}+\frac{\varepsilon}{2}}}\frac{1}{|s_1-s_2|^{\frac34+\frac{3\eps}{2}}}\bigg(\int_{|t_1-t_2|}^{+\infty}\frac{ d\sigma}{\sigma^{\frac32-\frac{3}{8q}+\frac{\eps}{2}}}\bigg)^2 \\
& \lesssim  \frac{ 1}{|t_1-t_2|^{1-\frac{3}{4q}+\eps}}\int_0^{t_1}\frac{ds_1}{(t_1-s_1)^{\frac{5}{8}+\frac{\varepsilon}{2}}}\frac{1}{|t_2-s_1|^{\frac38+2\eps}}\lesssim  \frac{ 1}{|t_1-t_2|^{1-\frac{3}{4q}+\frac72\eps}}.
\end{align*}
The claim \eqref{intA1-5} is finally obtained by picking $q>\frac32$ close enough to $\frac32$.
\end{proof}

\subsubsection{Estimation of $\ca^{\mathbf{1},\mathbf{2},(n)}_{t_1,t_2}$} \label{subsec:ca12n}

Recall that
\begin{align}
&\ca^{\mathbf{1,2},(n)}_{t_1,t_2}(x):=\int_0^{t_1} ds_1 \int_0^{t_2} ds_2 \, \mathcal{R}\big(  \cf^{\mathbf{1,2},(n)}\big)(x),\label{rappel12}
\end{align}
with
\begin{multline*}
\cf^{\mathbf{1,2},(n)}(y_1,y_2,z_1,z_2):=\\
 \cac^{(n)}_{t_1,t_2}(z_1,z_2)\int dw_1 dw_2\,   K_{t_1-s_1}(y_1,w_1)   K_{t_2-s_2}(y_2,w_2)  \cac^{(n)}_{t_1,s_2}(z_1,w_2) \cac^{(n)}_{t_2,s_1}(z_2,w_1) \cac^{(n)}_{s_1,s_2}(w_1,w_2)^2.
\end{multline*}

\begin{lemma}  \label{lem:cf12}
For every $\eps>0$, it holds that
\begin{equation}  \label{cf12-inf}
\sup_{n \geq 1} \|\cf^{\mathbf{1,2},(n)}\|_{L^{\infty}(\R^{12})}   \lesssim  \frac{1}{|t_1-t_2|^{\frac12}}\frac{1}{|t_1-s_2|^{\frac12}}\frac{1}{|t_2-s_1|^{\frac12}}\frac{1}{|t_1-s_1|^{\frac12}}\frac{1}{|t_2-s_2|^{\frac12}}\frac{1}{|s_1-s_2|^{\eps}}.
\end{equation}
Besides, there exists $N\geq 1$ such that 
\begin{equation*}   
\max\Big(\|\cf^{\mathbf{1,2},(n)}\|_{\mathcal{H}^{16}(\R^{12})}, \|H_{y_1}  H_{y_2}\cf^{\mathbf{1,2},(n)}\|_{L^{\infty}(\R^{12})}\Big)   \lesssim    \big(|t_1-t_2| |t_1-s_1| |t_2-s_2||t_1-s_2| |t_2-s_1| |s_1-s_2|   \big)^{-N},
\end{equation*}
uniformly over $n\geq 1$.
 \end{lemma}

\begin{proof}
We only prove \eqref{cf12-inf}, since the second assertion follows from the same general arguments as in Lemma \ref{lemma-crude}.

\smallskip

In fact, with the convention introduced in \eqref{Ntilde}, one has here
\begin{align*}
&\big|\cf^{\mathbf{1,2},(n)}(y_1,y_2,z_1,z_2)\big|\\
&\lesssim \frac{1}{|t_1-t_2|^{\frac12}}\frac{1}{|t_1-s_2|^{\frac12}}\frac{1}{|t_2-s_1|^{\frac12}}\int dw_1 dw_2\, K_{t_1-s_1}(y_1,w_1)K_{t_2-s_2}(y_2,w_2)  \cac^{(n)}_{s_1,s_2}(w_1,w_2)^2\\
&\lesssim \frac{1}{|t_1-t_2|^{\frac12}}\frac{1}{|t_1-s_2|^{\frac12}}\frac{1}{|t_2-s_1|^{\frac12}}\big\| {K}_{t_1-s_1}(y_1, \cdot)\big\|_{L^{\frac32}(\R^3)} \big\| {K}_{t_2-s_2}(y_2, \cdot)\big\|_{L^{\frac32}(\R^3)}\big\|(\widetilde{\cac}^{(n)}_{s_1,s_2})^2 \big\|_{L^{\frac32}(\R^3)}\\
&\lesssim \frac{1}{|t_1-t_2|^{\frac12}}\frac{1}{|t_1-s_2|^{\frac12}}\frac{1}{|t_2-s_1|^{\frac12}}\frac{1}{|t_1-s_1|^{\frac12}}\frac{1}{|t_2-s_2|^{\frac12}}\frac{1}{|s_1-s_2|^{\varepsilon}}
\end{align*}
for every $\eps>0$, where we have used Lemma \ref{lem:cac3} to derive the last inequality.
\end{proof}

\begin{proposition} \label{prop:a12-5}
For every $\eps>0$ and every $p\geq 1$ large enough, it holds that  
\begin{equation*} 
\sup_{n\geq 1}\, \| \ca^{\mathbf{1},\mathbf{2},(n)}_{t_1,t_2}\|_{L^p(\R^3)} \lesssim \frac{1}{|t_2-t_1| ^{\frac12+\eps}}.
 \end{equation*}
\end{proposition}

\begin{proof}
By applying Corollary \ref{coro:interpol-T} with $\la_1=1$, $\la_2=\la_3=0$, and then using the bounds contained in Lemma \ref{lem:cf12}, we get that for every $\eps>0$ and every $p\geq 1$ large enough,
\begin{align*}
&\|\mathcal{R}\big(  \cf^{\mathbf{1,2},(n)}\big)\|_{L^p(\R^3)} \lesssim \Big(1 \vee \| \cf^{\mathbf{1,2},(n)}\|_{L^{\infty}(\R^{12})} \Big)\Big(1\vee \| \cf^{\mathbf{1,2},(n)}\|_{\mathcal{H}^{16}(\R^{12})}\Big)^{\frac{\eps}{2N}} \Big(\|H_{y_1}  H_{y_2} \cf^{\mathbf{1,2},(n)}\|_{L^{\infty}(\R^{12})}\Big)^{\frac{\eps}{2N}}  \\
&\lesssim \frac{1}{|t_1-t_2|^{\frac12+\eps}}\frac{1}{|t_1-s_2|^{\frac12+\eps}}\frac{1}{|t_2-s_1|^{\frac12+\eps}}\frac{1}{|t_1-s_1|^{\frac12+\eps}}\frac{1}{|t_2-s_2|^{\frac12+\eps}}\frac{1}{|s_1-s_2|^{2\varepsilon}},
\end{align*}
uniformly over $n\geq 1$. Going back to \eqref{rappel12}, we immediately deduce that
\begin{align*} 
&\sup_{n\geq 1}\, \big\| \ca^{\mathbf{1},\mathbf{2},(n)}_{t_1,t_2}\big\|_{L^p(\R^3)} \lesssim \frac{1}{|t_1-t_2|^{\frac12+\eps}}\int_0^{t_1} \frac{ds_1}{|t_1-s_1|^{\frac12+\eps}}\int_0^{t_2} \frac{ds_2}{|t_2-s_2|^{\frac12+\eps}}  \frac{1}{|t_1-s_2|^{\frac12+\eps}}\frac{1}{|t_2-s_1|^{\frac12+\eps}}\frac{1}{|s_1-s_2|^{2\varepsilon}}.
\end{align*}
We are now in a position to apply Lemma \ref{Lem-prod2}, which immediately yields
\begin{equation*} 
\sup_{n\geq 1}\, \| \ca^{\mathbf{1},\mathbf{2},(n)}_{t_1,t_2}\|_{L^p(\R^3)} \lesssim \frac{1}{|t_2-t_1|^{\frac12+7\eps}},
\end{equation*}
and the proof is complete.  
\end{proof}

\subsubsection{Estimation of $\ca^{\mathbf{1},\mathbf{3},(n)}_{t_1,t_2}$} 

We mimic the procedure of the previous section. One has in this case
\begin{align}
&\ca^{\mathbf{1,3},(n)}_{t_1,t_2}(x):=\int_0^{t_1} ds_1 \int_0^{t_2} ds_2 \, \mathcal{R}\big(  \cf^{\mathbf{1,3},(n)}\big)(x),\label{rappel-13}
\end{align}
with
\begin{multline*}
\cf^{\mathbf{1,3},(n)}(y_1,y_2,z_1,z_2):=\\
=\int dw_1 dw_2\,   K_{t_1-s_1}(y_1,w_1)   K_{t_2-s_2}(y_2,w_2)  \cac^{(n)}_{t_1,s_2}(z_1,w_2)^2 \cac^{(n)}_{t_2,s_1}(z_2,w_1)^2 \cac^{(n)}_{s_1,s_2}(w_1,w_2).
\end{multline*}

\begin{lemma}  \label{lem:cf13}
For every $\eps>0$, it holds that
\begin{equation}  \label{cf13-inf}
\sup_{n \geq 1} \|\cf^{\mathbf{1,3},(n)}\|_{L^{\infty}(\R^{12})}   \lesssim   \frac{1}{|s_1-s_2|^{\frac12}}\frac{1}{|t_1-s_2|^{\frac12+\eps}}\frac{1}{|t_2-s_1|^{\frac12+\eps}}\frac{1}{|t_1-s_1|^{\frac12}}\frac{1}{|t_2-s_2|^{\frac12}}.
\end{equation}
Besides, there exists $N\geq 1$ such that 
\begin{equation}   \label{cf13-crude}
\max\Big(\|\cf^{\mathbf{1,3},(n)}\|_{\mathcal{H}^{16}(\R^{12})}, \|H_{y_1}  H_{y_2}\cf^{\mathbf{1,3},(n)}\|_{L^{\infty}(\R^{12})}\Big)   \lesssim    \big(|t_1-s_2| |t_1-s_1| |t_2-s_2||t_2-s_1| |s_1-s_2|  \big)^{-N},
\end{equation}
uniformly over $n\geq 1$.

 \end{lemma}

\begin{proof}
The arguments toward \eqref{cf13-crude} are again the same as in Lemma \ref{lemma-crude}, and so we only focus on \eqref{cf13-inf}. 

\smallskip

Still using the notation in \eqref{Ntilde}, we can write
\begin{align*}
&\big|\cf^{\mathbf{1,3},(n)}(y_1,y_2,z_1,z_2)\big|\\
&\lesssim \frac{1}{|s_1-s_2|^{\frac12}}\frac{1}{|t_2-s_1|^{\frac12}} \frac{1}{|t_1-s_2|^{\frac12}} \int dw_1 \,   K_{t_1-s_1}(y_1,w_1) \big|\widetilde{\cac}^{(n)}_{t_2,s_1}(z_2-w_1)\big| \int dw_2\,   K_{t_2-s_2}(y_2,w_2)  \big|\widetilde{\cac}^{(n)}_{t_1,s_2}(z_1-w_2)\big|  \\
&\lesssim \frac{1}{|s_1-s_2|^{\frac12}}\frac{1}{|t_2-s_1|^{\frac12}} \frac{1}{|t_1-s_2|^{\frac12}}\big\| {K}_{t_1-s_1}(y_1, \cdot)\big\|_{L^{\frac32}(\R^3)}\big\|\widetilde{\cac}^{(n)}_{t_2,s_1} \big\|_{L^{3}(\R^3)} \big\| {K}_{t_2-s_2}(y_2, \cdot)\big\|_{L^{\frac32}(\R^3)}\big\|\widetilde{\cac}^{(n)}_{t_1,s_2} \big\|_{L^{3}(\R^3)}\\
&\lesssim \frac{1}{|s_1-s_2|^{\frac12}}\frac{1}{|t_1-s_2|^{\frac12+\eps}}\frac{1}{|t_2-s_1|^{\frac12+\eps}}\frac{1}{|t_1-s_1|^{\frac12}}\frac{1}{|t_2-s_2|^{\frac12}}
\end{align*}
for every $\eps>0$, where we have used \eqref{normeLpp} and Lemma \ref{lem:cac3} to derive the last inequality.
\end{proof}

\begin{proposition} \label{prop:a13-5}
For every $\eps>0$ and every $p\geq 1$ large enough, it holds that  
\begin{equation*} 
\sup_{n\geq 1}\, \| \ca^{\mathbf{1},\mathbf{3},(n)}_{t_1,t_2}\|_{L^p(\R^3)} \lesssim \frac{1}{|t_2-t_1| ^{\frac12+\eps}}.
 \end{equation*}
\end{proposition}

\begin{proof}
Just as in the proof of Proposition \ref{prop:a12-5}, we apply Corollary \ref{coro:interpol-T} with $\la_1=1$, $\la_2=\la_3=0$, and then inject the bounds established in Lemma \ref{lem:cf13}, which gives for every $\eps>0$ and every $p\geq 1$ large enough,
\begin{align*}
&\|\mathcal{R}\big(  \cf^{\mathbf{1,3},(n)}\big)\|_{L^p(\R^3)} \lesssim  \frac{1}{|s_1-s_2|^{\frac12+\eps}}\frac{1}{|t_1-s_2|^{\frac12+2\eps}}\frac{1}{|t_2-s_1|^{\frac12+2\eps}}\frac{1}{|t_1-s_1|^{\frac12+\eps}}\frac{1}{|t_2-s_2|^{\frac12+\eps}},
\end{align*}
uniformly over $n\geq 1$. In light of \eqref{rappel-13}, we get that
\begin{align*} 
&\sup_{n\geq 1}\, \big\| \ca^{\mathbf{1},\mathbf{3},(n)}_{t_1,t_2}\big\|_{L^p(\R^3)} \lesssim \int_0^{t_1} \frac{ds_1}{|t_1-s_1|^{\frac12+\eps}}\int_0^{t_2} \frac{ds_2}{|t_2-s_2|^{\frac12+\eps}}  \frac{1}{|t_1-s_2|^{\frac12+2\eps}}\frac{1}{|t_2-s_1|^{\frac12+2\eps}}\frac{1}{|s_1-s_2|^{\frac12+\varepsilon}},
\end{align*}
and the desired conclusion easily follows from Lemma \ref{Lem-prod2}.
\end{proof}

\subsection{Study of $\calti^{\mathbf{2},(n)}$}

Fix $0<\ga <\frac34$ and let $0<\eta <\frac12$.  Fix $0<\ga <\frac34$ and let $0<\eta <\frac12$. We start as in \eqref{combi0-5}: for every $p\geq 2$ large enough,
\begin{multline}\label{combi5-2}
\mathbb{E}\Big[\big\|\calti^{\mathbf{2},(n)} \big\|_{\cac^{\ga}([0,T];\cb_x^{-\frac14-\eta})}^{2p}\Big] \lesssim \\
\lesssim\int_0^T\int_0^T  \frac{dv_1 dv_2}{|v_2-v_1|^{2p\ga+2}}\bigg( \int_{[v_1,v_2]^2}dt_1 dt_2\, \bigg(\int dx \,  \mathbb{E}\Big[ \calt^{\mathbf{2},(n)}_{t_1}(x)   \calt^{\mathbf{2},(n)}_{t_2}(x) \Big]^p\bigg)^{\frac1p} \bigg)^p.
\end{multline} 
Then one has this time
\begin{align*}
&\mathbb{E}\Big[ \calt^{\mathbf{2},(n)}_{t_1}(x)   \calt^{\mathbf{2},(n)}_{t_2}(x) \Big]=\int_0^{t_1}ds_1\int_0^{t_2}ds_2\, \sum_{i \sim i'} \sum_{j\sim j'}  \int dz_1 dy_1dz_2 dy_2 \, \delta_i(x,z_1) \delta_{i'}(x,y_1) \delta_j(x,z_2) \delta_{j'}(x,y_2)\,  \\
&\hspace{0.5cm} \int dw_1 dw_2 \, K_{t_1-s_1}(y_1,w_2)K_{t_2-s_2}(y_2,w_2)\cac^{(n)}_{t_1,s_1}(z_1,w_1) \cac^{(n)}_{t_2,s_2}(z_2,w_2)\\
&\hspace{4cm}\mathbb{E}\Big[I^W_3\big(F^{(n)}_{t_1,z_1}\otimes F^{(n)}_{s_1,w_1}\otimes F^{(n)}_{s_1,w_1}\big) I^W_3\big(F^{(n)}_{t_2,z_2}\otimes F^{(n)}_{s_2,w_2}\otimes F^{(n)}_{s_2,w_2}\big)\Big].
\end{align*}
By \eqref{ortho-rule}, the above expectation can be explicitly computed as
\begin{align*}
&\mathbb{E}\Big[I^W_3\big(F^{(n)}_{t_1,z_1}\otimes F^{(n)}_{s_1,w_1}\otimes F^{(n)}_{s_1,w_1}\big) I^W_3\big(F^{(n)}_{t_2,z_2}\otimes F^{(n)}_{s_2,w_2}\otimes F^{(n)}_{s_2,w_2}\big)\Big]=\sum_{\mathbf{b}=1}^2 c_{\mathbf{b}} \cq^{\mathbf{2},\mathbf{b},(n)}_{t,s}({z,w})
\end{align*}
for some combinatorial coefficients $c,c_{\mathbf{b}}\geq 0$, with
\begin{align*}
 \cq^{\mathbf{2},\mathbf{1},(n)}_{t,s}(z,w)&:=\cac^{(n)}_{t_1,t_2}(z_1,z_2) \cac^{(n)}_{s_1,s_2}(w_1,w_2)^2,\\
 \cq^{\mathbf{2},\mathbf{2},(n)}_{t,s}(z,w)&:=\cac^{(n)}_{t_1,s_2}(z_1,w_2) \cac^{(n)}_{t_2,s_1}(z_2,w_1) \cac^{(n)}_{s_1,s_2}(w_1,w_2).
\end{align*}

As a result, we can here rely on the decomposition
\begin{align}
&\mathbb{E}\Big[ \calt^{\mathbf{2},(n)}_{t_1}(x)   \calt^{\mathbf{2},(n)}_{t_2}(x) \Big]=\sum_{\mathbf{b}=1}^2 c_{\mathbf{b}}\ca^{\mathbf{2,b},(n)}_{t_1,t_2}(x),\label{combi521}
\end{align}
where $\mathcal{R}$ is still the operator introduced in \eqref{def-op-T}, and where we have set
\begin{align*}
&\ca^{\mathbf{2,b},(n)}_{t_1,t_2}(x):=\\
&\int_0^{t_1} ds_1 \int_0^{t_2} ds_2 \, \mathcal{R}\Big(  \int dw_1 dw_2\,  K_{t_1-s_1}(y_1,w_1)   K_{t_2-s_2}(y_2,w_2) \cac^{(n)}_{t_1,s_1}(z_1,w_1) \cac^{(n)}_{t_2,s_2}(z_2,w_2) \cq^{\mathbf{2},\mathbf{b},(n)}_{t,s}({z,w})\Big)(x).
\end{align*}

\begin{proposition}\label{prop-defT-52}
Fix $\mathbf{b}=1,2$. Then for every small $\eps>0$ and every $p\geq 1$ large enough, one has
\begin{equation} \label{combi522}
\sup_{n\geq 1} \, \| \ca^{\mathbf{2,b},(n)}_{t_1,t_2}\|_{L^p(\R^3)} \lesssim \frac{1}{|t_2-t_1|^{\frac12+\eps}}.
\end{equation}

\end{proposition}

By gathering \eqref{combi5-2}, \eqref{combi521} and \eqref{combi522}, we deduce (as in Section \ref{subsec:tildet1}): 

\begin{corollary}\label{corol-t2}
Fix $T>0$. For all $0< \ga <\frac34$ and $\eta>0$, it holds that
\begin{eqnarray*}
&\dis \sup_{n\geq 1} \, \mathbb{E} \Big[ \big\|\calti^{\mathbf{2},(n)} \big\|_{\cac^{\ga}([0,T];\cb_x^{-\eta})}^{2p} \Big] <\infty.
\end{eqnarray*}

\end{corollary}

\

We are thus left with the proof of Proposition \ref{prop-defT-52}.

\subsubsection{Estimation of $\ca^{\mathbf{2},\mathbf{1},(n)}_{t_1,t_2}$} 

The procedure is overall the same as in Section \ref{subsec:ca12n}. One has in this case
\begin{align}
&\ca^{\mathbf{2,1},(n)}_{t_1,t_2}(x):=\int_0^{t_1} ds_1 \int_0^{t_2} ds_2 \, \mathcal{R}\big(  \cf^{\mathbf{2,1},(n)}\big)(x),\label{rappel-21}
\end{align}
with
\begin{align*}
&\cf^{\mathbf{2,1},(n)}(y_1,y_2,z_1,z_2):=\\
& \cac^{(n)}_{t_1,t_2}(z_1,z_2)\int dw_1 dw_2\,  K_{t_1-s_1}(y_1,w_1)   K_{t_2-s_2}(y_2,w_2) \cac^{(n)}_{t_1,s_1}(z_1,w_1) \cac^{(n)}_{t_2,s_2}(z_2,w_2) \cac^{(n)}_{s_1,s_2}(w_1,w_2)^2.
\end{align*}

\begin{lemma}  
For every $\eps>0$, it holds that
\begin{equation*}  
\sup_{n \geq 1} \|\cf^{\mathbf{2,1},(n)}\|_{L^{\infty}(\R^{12})}   \lesssim  \frac{1}{|t_1-t_2|^{\frac12}}\frac{1}{|t_1-s_1|^{\frac34}}\frac{1}{|t_2-s_2|^{\frac34}}\frac{1}{|s_1-s_2|^{\frac12+\eps}}.
\end{equation*}
Besides, there exists $N\geq 1$ such that 
\begin{equation*}  
\max\Big(\|\cf^{\mathbf{2,1},(n)}\|_{\mathcal{H}^{16}(\R^{12})}, \|H_{y_1}  H_{y_2}\cf^{\mathbf{2,1},(n)}\|_{L^{\infty}(\R^{12})}\Big)   \lesssim    \big(|t_1-t_2| |t_1-s_1| |t_2-s_2| |s_1-s_2|  \big)^{-N},
\end{equation*}
uniformly over $n\geq 1$.
 \end{lemma}

\begin{proof}
In the same spirit as in the proof of Lemma \ref{lem:cf12}, one has
\begin{align*}
&\big|\cf^{\mathbf{2,1},(n)}(y_1,y_2,z_1,z_2)\big|\\
&\lesssim \frac{1}{|t_1-t_2|^{\frac12}}\frac{1}{|t_1-s_1|^{\frac12}} \frac{1}{|t_2-s_2|^{\frac12}} \frac{1}{|s_1-s_2|^{\frac12}} \int dw_1 dw_2\,   K_{t_1-s_1}(y_1,w_1)    K_{t_2-s_2}(y_2,w_2)  \big|\widetilde{\cac}^{(n)}_{s_1,s_2}(w_1-w_2)\big|  \\
&\lesssim \frac{1}{|t_1-t_2|^{\frac12}}\frac{1}{|t_1-s_1|^{\frac12}} \frac{1}{|t_2-s_2|^{\frac12}} \frac{1}{|s_1-s_2|^{\frac12}} \big\| {K}_{t_1-s_1}(y_1, \cdot)\big\|_{L^{\frac65}(\R^3)} \big\| {K}_{t_2-s_2}(y_2, \cdot)\big\|_{L^{\frac65}(\R^3)}\big\|\widetilde{\cac}^{(n)}_{t_1,s_2} \big\|_{L^{3}(\R^3)}\\
&\lesssim \frac{1}{|t_1-t_2|^{\frac12}}\frac{1}{|t_1-s_1|^{\frac34}}\frac{1}{|t_2-s_2|^{\frac34}}\frac{1}{|s_1-s_2|^{\frac12+\eps}},
\end{align*}
where we have used \eqref{normeLpp} and Lemma \ref{lem:cac3} to derive the last inequality.
\end{proof}

 \begin{lemma} 
For every $\eps>0$ and every $p\geq 1$ large enough, it holds that  
\begin{equation*}
\sup_{n\geq 1}\, \| \ca^{\mathbf{2},\mathbf{1},(n)}_{t_1,t_2}\|_{L^p(\R^3)} \lesssim \frac{1}{|t_2-t_1| ^{\frac12+\eps}}.
 \end{equation*}
 \end{lemma}

\begin{proof}
We apply Corollary \ref{coro:interpol-T} with $\la_1=1$, $\la_2=\la_3=0$, and then inject the bounds contained in Lemma \ref{lem:cf13}, which yields for every $\eps>0$ and every $p\geq 1$ large enough,
\begin{align*}
&\|\mathcal{R}\big(  \cf^{\mathbf{2,1},(n)}\big)\|_{L^p(\R^3)} \lesssim \frac{1}{|t_1-t_2|^{\frac12+\eps}}\frac{1}{|t_1-s_1|^{\frac34+\eps}}\frac{1}{|t_2-s_2|^{\frac34+\eps}}\frac{1}{|s_1-s_2|^{\frac12+2\eps}},
\end{align*}
uniformly over $n\geq 1$. Going back to \eqref{rappel-21}, we deduce
\begin{align*} 
\sup_{n\geq 1}\, \big\| \ca^{\mathbf{2},\mathbf{1},(n)}_{t_1,t_2}\big\|_{L^p(\R^3)} &\lesssim \frac{1}{|t_1-t_2|^{\frac12+\eps}}\int_0^{t_1} \frac{ds_1}{|t_1-s_1|^{\frac34+\eps}}\int_0^{t_2} \frac{ds_2}{|t_2-s_2|^{\frac34+\eps}}  \frac{1}{|s_1-s_2|^{\frac12+2\varepsilon}}\\
&\lesssim \frac{1}{|t_1-t_2|^{\frac12+\eps}}\int_0^{t_1} \frac{ds_1}{|t_1-s_1|^{\frac34+\eps}}\frac{1}{|t_2-s_1|^{\frac14+3\eps}}\lesssim \frac{1}{|t_1-t_2|^{\frac12+5\eps}},
\end{align*}
as desired.
\end{proof}

\subsubsection{Estimation of $\ca^{\mathbf{2},\mathbf{2},(n)}_{t_1,t_2}$} 

Recall that
\begin{align}
&\ca^{\mathbf{2,2},(n)}_{t_1,t_2}(x):=\int_0^{t_1} ds_1 \int_0^{t_2} ds_2 \, \mathcal{R}\big(  \cf^{\mathbf{2,2},(n)}\big)(x),\label{rappel-22}
\end{align}
with
\begin{align*}
&\cf^{\mathbf{2,2},(n)}(y_1,y_2,z_1,z_2):=\\
&\int dw_1 dw_2\,  K_{t_1-s_1}(y_1,w_1)   K_{t_2-s_2}(y_2,w_2) \cac^{(n)}_{t_1,s_1}(z_1,w_1) \cac^{(n)}_{t_2,s_2}(z_2,w_2) \cac^{(n)}_{t_1,s_2}(z_1,w_2) \cac^{(n)}_{t_2,s_1}(z_2,w_1) \cac^{(n)}_{s_1,s_2}(w_1,w_2).
\end{align*}

\begin{lemma} 
For every $\eps>0$, it holds that
\begin{equation*}  
\sup_{n \geq 1} \|\cf^{\mathbf{2,2},(n)}\|_{L^{\infty}(\R^{12})}   \lesssim  \frac{1}{|t_1-s_1|^{\frac12}}\frac{1}{|t_2-s_2|^{\frac12}}\frac{1}{|t_1-s_2|^{\frac12}}\frac{1}{|t_2-s_1|^{\frac12}}\frac{1}{|s_1-s_2|^{\frac12}}.
\end{equation*}
Besides, there exists $N\geq 1$ such that 
\begin{equation*}   
\max\Big(\|\cf^{\mathbf{2,2},(n)}\|_{\mathcal{H}^{16}(\R^{12})}, \|H_{y_1}  H_{y_2}\cf^{\mathbf{2,2},(n)}\|_{L^{\infty}(\R^{12})}\Big)   \lesssim    \big( |t_1-s_1| |t_2-s_2| |t_1-s_2| |t_2-s_1| |s_1-s_2|  \big)^{-N},
\end{equation*}
uniformly over $n\geq 1$.

 \end{lemma}

\begin{proof}
One has easily
\begin{align*}
&\big|\cf^{\mathbf{2,2},(n)}(y_1,y_2,z_1,z_2)\big|\\
&\lesssim \frac{1}{|t_1-s_1|^{\frac12}}\frac{1}{|t_2-s_2|^{\frac12}}\frac{1}{|t_1-s_2|^{\frac12}}\frac{1}{|t_2-s_1|^{\frac12}}\frac{1}{|s_1-s_2|^{\frac12}}\int dw_1 \, K_{t_1-s_1}(y_1,w_1) \int dw_2\, K_{t_2-s_2}(y_2,w_2)   \\
&\lesssim \frac{1}{|t_1-s_1|^{\frac12}}\frac{1}{|t_2-s_2|^{\frac12}}\frac{1}{|t_1-s_2|^{\frac12}}\frac{1}{|t_2-s_1|^{\frac12}}\frac{1}{|s_1-s_2|^{\frac12}}.
\end{align*}
\end{proof}

 \begin{lemma}
For every $\eps>0$ and every $p\geq 1$ large enough, it holds that  
\begin{equation}\label{bou22}
\sup_{n\geq 1}\, \| \ca^{\mathbf{2},\mathbf{2},(n)}_{t_1,t_2}\|_{L^p(\R^3)} \lesssim \frac{1}{|t_2-t_1| ^{\frac12+\eps}}.
 \end{equation}
 \end{lemma}

\begin{proof}
Combining Corollary \ref{coro:interpol-T} and Lemma \ref{lem:cf13} as before, we get for every $\eps>0$ and every large $p\geq 1$,
\begin{align*}
&\|\mathcal{R}\big(  \cf^{\mathbf{2,2},(n)}\big)\|_{L^p(\R^3)} \lesssim \frac{1}{|t_1-s_1|^{\frac12+\eps}}\frac{1}{|t_2-s_2|^{\frac12+\eps}}\frac{1}{|t_1-s_2|^{\frac12+\eps}}\frac{1}{|t_2-s_1|^{\frac12+\eps}}\frac{1}{|s_1-s_2|^{\frac12+\eps}},
\end{align*}
uniformly over $n\geq 1$. Going back to \eqref{rappel-22}, this yields
\begin{align*} 
\sup_{n\geq 1}\, \big\| \ca^{\mathbf{2},\mathbf{2},(n)}_{t_1,t_2}\big\|_{L^p(\R^3)} &\lesssim \int_0^{t_1} \frac{ds_1}{|t_1-s_1|^{\frac12+\eps}}\int_0^{t_2} \frac{ds_2}{|t_2-s_2|^{\frac12+\eps}} \frac{1}{|t_1-s_2|^{\frac12+\eps}}\frac{1}{|t_2-s_1|^{\frac12+\eps}}\frac{1}{|s_1-s_2|^{\frac12+\eps}},
\end{align*}
and the bound \eqref{bou22} readily follows from Lemma \ref{Lem-prod2}.
\end{proof}

\subsection{Study of $\calt^{\mathbf{3},(n)}_t(x)$}
Let us decompose this term into
\begin{align*}
\calt^{\mathbf{3},(n)}_t(x)&=\sum_{i\sim i'}\int dz dy \, \delta_i(x,z)\delta_{i'}(x,y)  \int_0^t ds \int dw \, K_{t-s}(y,w)\cac^{(n)}_{t,s}(z,w)^2 \<Psi>^{(n)}_s(w)\nonumber\\
&\hspace{2cm} -\sum_{i, i'}\int dz dy \, \delta_i(x,z)\delta_{i'}(x,y) \int_0^t ds \int dw \, K_{t-s}(y,w)\cac^{(n)}_{t,s}(z,w)^2 \<Psi>^{(n)}_s(w)\nonumber\\
&=- \Big( \calt^{\mathbf{3,1},(n)}_t(x) +\calt^{\mathbf{3,2},(n)}_t(x)  \Big), 
\end{align*}
with
$$\calt^{\mathbf{3},\mathbf{1},(n)}_t(x)=\sum_{i\leq i'-2}\int dz dy \, \delta_i(x,z)\delta_{i'}(x,y)  \int_0^t ds \int dw \, K_{t-s}(y,w)\cac^{(n)}_{t,s}(z,w)^2 \<Psi>^{(n)}_s(w)$$
and
$$\calt^{\mathbf{3},\mathbf{2},(n)}_t(x)=\sum_{i'\leq i-2}\int dz dy \, \delta_i(x,z)\delta_{i'}(x,y)  \int_0^t ds \int dw \, K_{t-s}(y,w)\cac^{(n)}_{t,s}(z,w)^2 \<Psi>^{(n)}_s(w).$$
We also set, for $\mathbf{b}=1,2$,
$$\calti^{\mathbf{3},\mathbf{b},(n)}_t(x):=\int_0^t \calt^{\mathbf{3},\mathbf{b},(n)}_s(x) \, ds.$$

\smallskip

\begin{proposition}\label{prop-fifth-order-3}
Let $T>0$. For each $\mathbf{b}=1,2$ and for all $\eps,\eta >0$, it holds that
\begin{eqnarray}
&\dis \sup_{n\geq 1} \, \mathbb{E} \Big[ \big\|\calti^{\mathbf{3},\mathbf{b},(n)} \big\|_{\cac^{\frac34-\eps}([0,T]; \cb_x^{-\frac14-4\eta})}^{2p} \Big] <\infty. \label{unif-fifth-3-b}
\end{eqnarray}
\end{proposition}

\

\noindent
\textit{Proof of Proposition \ref{prop-fifth-order-3}}.

\

Fix $\mathbf{b}=1,2$ and $0\leq \ga<\frac34$. For every $p\geq 1$ large enough, it holds that 
\begin{align}
\mathbb{E}\Big[ \big\|\calti^{\mathbf{3},\mathbf{b},(n)} \big\|_{\cac^{\ga}([0,T]; \cb_x^{-\frac14-4\eta})}^{2p}\Big]
&\lesssim \int_0^T\int_0^T dv_1 dv_2 \, \frac{\mathbb{E}\Big[\big\| \calti^{\mathbf{3},\mathbf{b},(n)}_{v_2}-\calti^{\mathbf{3},\mathbf{b},(n)}_{v_1}\big\|_{\cb_x^{-\frac14-4\eta}}^{2p}\Big]}{|v_2-v_1|^{2\ga p+2}}\nonumber\\
&\lesssim \int_0^T\int_0^T dv_1 dv_2 \, \frac{\mathbb{E}\Big[\big\| \calti^{\mathbf{3},\mathbf{b},(n)}_{v_2}-\calti^{\mathbf{3},\mathbf{b},(n)}_{v_1}\big\|_{\cb_{2p,2p}^{-\frac14-2\eta}}^{2p}\Big]}{|v_2-v_1|^{2\ga p+2}}.\label{t3}
\end{align}
Then write
\begin{align}
&\mathbb{E}\Big[\big\| \calti^{\mathbf{3},\mathbf{b},(n)}_{v_2}-\calti^{\mathbf{3},\mathbf{b},(n)}_{v_1}\big\|_{\cb_{2p,2p}^{-\frac14-2\eta}}^{2p}\Big]=\sum_j 2^{-2pj(\frac14+2\eta)}\int dx \, \mathbb{E}\Big[ \big| \delta_j\big(\calti^{\mathbf{3},\mathbf{b},(n)}_{v_2}-\calti^{\mathbf{3},\mathbf{b},(n)}_{v_1}\big)(x)\big|^{2p}\Big]\nonumber\\
&\lesssim \sum_j 2^{-2pj(\frac14+2\eta)}\int dx \, \mathbb{E}\Big[ \big| \delta_j\big(\calti^{\mathbf{3},\mathbf{b},(n)}_{v_2}-\calti^{\mathbf{3},\mathbf{b},(n)}_{v_1}\big)(x)\big|^{2}\Big]^p\nonumber\\
&\lesssim \sum_j \int dx \, \bigg|\int_{v_1}^{v_2}dt_1 \int_{v_1}^{v_2}dt_2\, 2^{-2j(\frac14+2\eta)} \mathbb{E}\Big[  \delta_j\big(\calt^{\mathbf{3},\mathbf{b},(n)}_{t_1}\big)(x)\delta_j\big(\calt^{\mathbf{3},\mathbf{b},(n)}_{t_2}\big)(x)\Big]\bigg|^p\nonumber\\
&\lesssim \bigg(\int_{v_1}^{v_2}dt_1 \int_{v_1}^{v_2}dt_2\,\bigg(\sum_j 2^{-2jp(\frac14+2\eta)} \int dx \,  \Big|\mathbb{E}\Big[  \delta_j\big(\calt^{\mathbf{3},\mathbf{b},(n)}_{t_1}\big)(x)\delta_j\big(\calt^{\mathbf{3},\mathbf{b},(n)}_{t_2}\big)(x)\Big]\Big|^p \bigg)^{\frac1p}\bigg)^p.\label{t3-topo}
\end{align}

\

\subsubsection{First case: $\mathbf{b}=1$}\label{subsec:bfb1}

One has in this case
\begin{align*}
&\mathbb{E}\Big[  \delta_j\big(\calt^{\mathbf{3},\mathbf{1},(n)}_{t_1}\big)(x)\delta_j\big(\calt^{\mathbf{3},\mathbf{1},(n)}_{t_2}\big)(x)\Big]=\int dy_1 dy_2\, \delta_j(x,y_1)\delta_j(x,y_2)\mathbb{E}\Big[  \calt^{\mathbf{3},\mathbf{1},(n)}_{t_1}(y_1)\calt^{\mathbf{3},\mathbf{1},(n)}_{t_2}(y_2)\Big]\\
&=\int dy_1 dy_2\, \delta_j(x,y_1)\delta_j(x,y_2)\sum_{i_1\leq i_1'-2}\int dz_1 dz_1' \, \delta_{i_1}(y_1,z_1)\delta_{i_1'}(y_1,z_1')  \int_0^{t_1} ds_1 \int dw_1 \, K_{t_1-s_1}(z_1',w_1)\cac^{(n)}_{t_1,s_1}(z_1,w_1)^2 \\
&\hspace{1cm}\sum_{i_2\leq i_2'-2}\int dz_2 dz_2' \, \delta_{i_2}(y_2,z_2)\delta_{i_2'}(y_2,z_2')  \int_0^{t_2} ds_2 \int dw_2 \, K_{t_2-s_2}(z_2',w_2)\cac^{(n)}_{t_2,s_2}(z_2,w_2)^2   \, \mathbb{E}\Big[\<Psi>^{(n)}_{s_1}(w_1)\<Psi>^{(n)}_{s_2}(w_2)\Big]\\
&=\int_0^{t_1} ds_1 \int_0^{t_2} ds_2\int dz_1 dz_1'dz_2 dz_2' \, \bigg[\int dy_1 \, \delta_j(x,y_1)\sum_{i_1\leq i_1'-2}\delta_{i_1}(y_1,z_1)\delta_{i_1'}(y_1,z_1') \bigg]\\
&\hspace{3cm}\bigg[\int dy_2 \, \delta_j(x,y_2)\sum_{i_2\leq i_2'-2}\delta_{i_2}(y_2,z_2)\delta_{i_2'}(y_2,z_2') \bigg]  \cf^{(n)}_{s_1,t_1,s_2,t_2}(z_1,z_1',z_2,z_2')
\end{align*}
where we have set
\begin{multline}\label{cfn}
\cf^{(n)}_{s_1,t_1,s_2,t_2}(z_1,z_1',z_2,z_2'):= \\
\int dw_1 \, K_{t_1-s_1}(z_1',w_1)\cac^{(n)}_{t_1,s_1}(z_1,w_1)^2 \int dw_2 \, K_{t_2-s_2}(z_2',w_2)\cac^{(n)}_{t_2,s_2}(z_2,w_2)^2  \cac^{(n)}_{s_1,s_2}(w_1,w_2) .
\end{multline}
We immediately deduce the expression: for every $\al\in (0,1)$,
\begin{align*}
&\mathbb{E}\Big[  \delta_j\big(\calt^{\mathbf{3},\mathbf{1},(n)}_{t_1}\big)(x)\delta_j\big(\calt^{\mathbf{3},\mathbf{1},(n)}_{t_2}\big)(x)\Big]=2^{4j\al}\int_0^{t_1} ds_1 \int_0^{t_2} ds_2\,  \cp^{(\al)}_j\big( H^{-\al}_{z'_1} H^{-\al}_{z'_2}\cf^{(n)}_{s_1,t_1,s_2,t_2}\big)(x)
\end{align*}
where we define the operator $\cp^{(\al)}$ by the formula
\begin{align}
&\big(\cp^{(\al)}_jF\big)(x):=\int dz_1 dz_1' dz_2 dz_2' \, F(z_1,z_1',z_2,z_2')\nonumber\\
&\hspace{1.5cm}\bigg[2^{-2j\al}\int dy_1 \, \delta_{j}(x,y_1)\Big(\sum_{i_1\leq i_1'-2}  \delta_{i_1}(y_1,z_1) H^\al_{y_1}\big(\delta_{i_1'}(y_1,z_1')\big) \Big)\bigg]\nonumber\\
&\hspace{2cm}\bigg[2^{-2j\al} \int dy_2\, \delta_j(x,y_2)\Big(\sum_{i_2\leq i_2'-2}  \delta_{i_2}(y_2,z_1) H^\al_{y_2}\big(\delta_{i_2'}(y_2,z_2')\big) \Big)\bigg].\label{deficpn}
\end{align}
Going back to \eqref{t3-topo} and taking $\al=\frac18$, we obtain that for every $\eta>0$
\small
\begin{align}
&\bigg(\mathbb{E}\Big[\big\| \calti^{\mathbf{3},\mathbf{1},(n)}_{v_2}-\calti^{\mathbf{3},\mathbf{1},(n)}_{v_1}\big\|_{\cb_{2p,2p}^{-\frac14-\eta}}^{2p}\Big]\bigg)^{\frac1p}\nonumber\\
&\lesssim \int_{v_1}^{v_2}dt_1 \int_{v_1}^{v_2}dt_2\,\bigg(\sum_j 2^{-2jp(\frac14+2\eta)} \int dx \,  \bigg|2^{\frac{j}{2}}\int_0^{t_1} ds_1 \int_0^{t_2} ds_2\,  \cp^{(\frac18)}_j\big( H^{-\frac18}_{z'_1} H^{-\frac18}_{z'_2}\cf^{(n)}_{s_1,t_1,s_2,t_2}\big)(x)\bigg|^p \bigg)^{\frac1p}\nonumber\\
&\lesssim \int_{v_1}^{v_2}dt_1 \int_{v_1}^{v_2}dt_2\int_0^{t_1} ds_1 \int_0^{t_2} ds_2\,\bigg(\sum_j 2^{-4jp\eta} \int dx \,  \Big|  \cp^{(\frac18)}_j\big( H^{-\frac18}_{z'_1} H^{-\frac18}_{z'_2}\cf^{(n)}_{s_1,t_1,s_2,t_2}\big)(x)\Big|^p \bigg)^{\frac1p}\nonumber\\
&\lesssim \int_{v_1}^{v_2}dt_1 \int_{v_1}^{v_2}dt_2\int_0^{t_1} ds_1 \int_0^{t_2} ds_2\,\sup_j \Big\|  \cp^{(\frac18)}_j\big( H^{-\frac18}_{z'_1} H^{-\frac18}_{z'_2}\cf^{(n)}_{s_1,t_1,s_2,t_2}\big)\Big\|_{L^p_x}\nonumber \\
&\lesssim \int_{v_1}^{v_2}dt_1 \int_{v_1}^{v_2}dt_2\int_0^{t_1} ds_1 \int_0^{t_2} ds_2\,\bigg(\sup_j \Big\|  \cp^{(\frac18)}_j\big( H^{-\frac18}_{z'_1} H^{-\frac18}_{z'_2}\cf^{(n)}_{s_1,t_1,s_2,t_2}\big)\Big\|_{L^\infty_x}\bigg)^{1-\frac{1}{p}} \bigg(\sup_j \Big\|  \cp^{(\frac18)}_j\big( H^{-\frac18}_{z'_1} H^{-\frac18}_{z'_2}\cf^{(n)}_{s_1,t_1,s_2,t_2}\big)\Big\|_{L^1_x}\bigg)^{\frac{1}{p}} .\label{t3-l1linf}
\end{align}
\normalsize

\

The control of the first term into brackets is precisely the topic of Lemma \ref{lem:cpal}: thanks to the latter, we can assert that
\begin{align*}
&\sup_j \Big\|  \cp^{(\frac18)}_j\big( H^{-\frac18}_{z'_1} H^{-\frac18}_{z'_2}\cf^{(n)}_{s_1,t_1,s_2,t_2}\big)\Big\|_{L^\infty_x}\lesssim  \Big\|   H^{-\frac18}_{z'_1} H^{-\frac18}_{z'_2}\cf^{(n)}_{s_1,t_1,s_2,t_2}\Big\|_{L^\infty_{z_1,z_1',z_2,z_2'}}.
\end{align*}
Then
\small
\begin{align*}
&\big|H^{-\frac18}_{z'_1} H^{-\frac18}_{z'_2}\cf^{(n)}_{s_1,t_1,s_2,t_2}(z_1,z_1',z_2,z_2')\big|\\
&=\bigg|\int dw_1 \, (H^{-\frac18}_{z'_1} K_{t_1-s_1})(z_1',w_1)\cac^{(n)}_{t_1,s_1}(z_1,w_1)^2 \int dw_2 \, (H^{-\frac18}_{z'_2} K_{t_2-s_2})(z_2',w_2)\cac^{(n)}_{t_2,s_2}(z_2,w_2)^2   \, \cac^{(n)}_{s_1,s_2}(w_1,w_2)\bigg|\\
&\lesssim \frac{1}{|s_1-s_2|^{\frac12}}\bigg( \int dw_1 \, \big|(H^{-\frac18}_{z'_1} K_{t_1-s_1})(z_1',w_1)\big|\big|\cac^{(n)}_{t_1,s_1}(z_1,w_1)\big|^2\bigg)\bigg( \int dw_2 \, \big|(H^{-\frac18}_{z'_2} K_{t_2-s_2})(z_2',w_2)\big|\big|\cac^{(n)}_{t_2,s_2}(z_2,w_2)\big|^2 \bigg)\\
&\lesssim  \frac{1}{|s_1-s_2|^{\frac12}}\bigg( \int dw_1 \, \big|(H^{-\frac18}_{z'_1} K_{t_1-s_1})(z_1',w_1)\big|^{\frac{12}{11}}\bigg)^{\frac{11}{12}}\bigg( \int dw'_1 \, \big|\cac^{(n)}_{t_1,s_1}(z_1,w'_1)\big|^{24}\bigg)^{\frac{1}{12}}\\
&\hspace{3cm}\bigg( \int dw_2 \, \big|(H^{-\frac18}_{z'_2} K_{t_2-s_2})(z_2',w_2)\big|^{\frac{12}{11}}\bigg)^{\frac{11}{12}}\bigg( \int dw'_2 \, \big|\cac^{(n)}_{t_2,s_2}(z_2,w'_2)\big|^{24}\bigg)^{\frac{1}{12}}\\ 
& \lesssim  \frac{1}{|s_1-s_2|^{\frac12}}\bigg( \int dw_1 \, \big|K_{t_1-s_1}(z_1',w_1)\big|^{1+\theta}\bigg)^{\frac1{1+\theta}}\bigg( \int dw'_1 \, \big|\cac^{(n)}_{t_1,s_1}(z_1,w'_1)\big|^{24}\bigg)^{\frac{1}{12}}\\
&\hspace{3cm}\bigg( \int dw_2 \, \big|K_{t_2-s_2}(z_2',w_2)\big|^{1+\theta}\bigg)^{\frac1{1+\theta}}\bigg( \int dw'_2 \, \big|\cac^{(n)}_{t_2,s_2}(z_2,w'_2)\big|^{24}\bigg)^{\frac{1}{12}},
\end{align*}
\normalsize
where we have used the Sobolev embedding $L^{1+\theta}(\R^3)\subset \cw^{-\frac14,\frac{12}{11}}(\R^3)$, for all $\theta >0$.  Since $\dis \Big(\int dw \, |K_r(z,w)|^{1+\theta}\Big)^{\frac1{1+\theta}} \lesssim r^{-\frac{3\theta}{2(1+\theta)}}$ and
$$\bigg( \int dw' \, \big|\cac^{(n)}_{t,s}(z,w')\big|^{24}\bigg)^{\frac{1}{12}} \lesssim \frac{1}{|t-s|^{\frac{23}{24}}}\bigg( \int dw' \, \big|\cac^{(n)}_{t,s}(z,w')\big|\bigg)^{\frac{1}{12}}\lesssim \frac{1}{|t-s|^{\frac{23}{24}}} ,$$
we deduce that, for $\theta >0$ small enough
\begin{align*}
\Big\|   H^{-\frac18}_{z'_1} H^{-\frac18}_{z'_2}\cf^{(n)}_{s_1,t_1,s_2,t_2}\Big\|_{L^\infty_{z_1,z_1',z_2,z_2'}}&\lesssim \frac{1}{|s_1-s_2|^{\frac12}}\frac{1}{|t_1-s_1|^{\frac{47}{48}}}\frac{1}{|t_2-s_2|^{\frac{47}{48}}}. 
\end{align*}
As for the second term into brackets in \eqref{t3-l1linf}, we can use similar arguments as in the proof of Lemma~\ref{conti-T} to show the existence of a (large) integer $L$ such that 
\begin{equation}\label{argu1}
\sup_j \Big\|  \cp^{(\frac18)}_j\big( H^{-\frac18}_{z'_1} H^{-\frac18}_{z'_2}\cf^{(n)}_{s_1,t_1,s_2,t_2}\big)\Big\|_{L^1_x}\lesssim \big\|\cf^{(n)}_{s_1,t_1,s_2,t_2} \big\|_{\ch^L(\R^{12})},
\end{equation}
and then, just as in Lemma \ref{lemma-crude}, we easily get that
\begin{equation}\label{argu2}
\sup_{n\geq 1}\big\|\cf^{(n)}_{s_1,t_1,s_2,t_2} \big\|_{\ch^L(\R^{12})} \lesssim \frac{1}{|t_1-s_1|^N}\frac{1}{|t_2-s_2|^N} \frac{1}{|s_1-s_2|^N}
\end{equation}
for some large (but fixed) integer $N$.

\smallskip

By injecting the above estimates into \eqref{t3-l1linf}, we derive that 
\begin{align*}
&\sup_{n\geq 1}\bigg(\mathbb{E}\Big[\big\| \calti^{\mathbf{3},\mathbf{1},(n)}_{v_2}-\calti^{\mathbf{3},\mathbf{1},(n)}_{v_1}\big\|_{\cb_{2p,2p}^{-\frac14-\eta}}^{2p}\Big]\bigg)^{\frac1p}\\
&\lesssim \int_{v_1}^{v_2}dt_1 \int_{v_1}^{v_2}dt_2\int_0^{t_1} ds_1 \int_0^{t_2} ds_2\, \frac{1}{|s_1-s_2|^{\frac12 (1-\frac{1}{p})+\frac{N}{p}}}\frac{1}{|t_1-s_1|^{\frac{47}{48}(1-\frac{1}{p})+\frac{N}{p}}}\frac{1}{|t_2-s_2|^{\frac{47}{48}(1-\frac{1}{p})+\frac{N}{p}}}.
\end{align*}
As a result, for every small $\eps>0$ and every $p\geq 1$ large enough, one has
\begin{align*}
&\sup_{n\geq 1}\bigg(\mathbb{E}\Big[\big\| \calti^{\mathbf{3},\mathbf{1},(n)}_{v_2}-\calti^{\mathbf{3},\mathbf{1},(n)}_{v_1}\big\|_{\cb_{2p,2p}^{-\frac14-\eta}}^{2p}\Big]\bigg)^{\frac1p}\\
&\lesssim \int_{v_1}^{v_2}dt_1 \int_{v_1}^{v_2}dt_2\int_0^{t_1} ds_1 \int_0^{t_2} ds_2\, \frac{1}{|s_1-s_2|^{\frac12+\eps}}\frac{1}{|t_1-s_1|^{1-\eps}}\frac{1}{|t_2-s_2|^{1-\eps}}\\
&\lesssim \int_{v_1}^{v_2}dt_1 \int_{v_1}^{v_2}dt_2\int_0^{t_1} ds_1 \frac{1}{|t_1-s_1|^{1-\eps}} \frac{1}{|t_2-s_1|^{\frac12}}\lesssim \int_{v_1}^{v_2}\int_{v_1}^{v_2}\frac{dt_1 dt_2}{|t_1-t_2|^{\frac12-\eps}} \lesssim |v_2-v_1|^{\frac32}.
\end{align*}
Putting this estimate into \eqref{t3}, we can conclude that for every $p\geq 1$ large enough,
\begin{align*}
\sup_{n\geq 1}\mathbb{E}\Big[ \big\|\calti^{\mathbf{3},\mathbf{1},(n)} \big\|_{\cac^{\ga}([0,T]; \cb_x^{-\frac14-4\eta})}^{2p}\Big]
&\lesssim \int_0^T\int_0^T  \frac{dv_1 dv_2}{|v_2-v_1|^{2-(2\ga-\frac32)p}} < \infty,
\end{align*}
due to $\ga <\frac34$.

\

\subsubsection{Second case: $\mathbf{b}=2$}

One has here
\begin{align*}
&\mathbb{E}\Big[  \delta_j\big(\calt^{\mathbf{3},\mathbf{2},(n)}_{t_1}\big)(x)\delta_j\big(\calt^{\mathbf{3},\mathbf{2},(n)}_{t_2}\big)(x)\Big]=\int dy_1 dy_2\, \delta_j(x,y_1)\delta_j(x,y_2)\mathbb{E}\Big[  \calt^{\mathbf{3},\mathbf{2},(n)}_{t_1}(y_1)\calt^{\mathbf{3},\mathbf{2},(n)}_{t_2}(y_2)\Big]\\
&=\int dy_1 dy_2\, \delta_j(x,y_1)\delta_j(x,y_2)\sum_{i'_1\leq i_1-2}\int dz_1 dz_1' \, \delta_{i_1}(y_1,z_1)\delta_{i_1'}(y_1,z_1')  \int_0^{t_1} ds_1 \int dw_1 \, K_{t_1-s_1}(z_1',w_1)\cac^{(n)}_{t_1,s_1}(z_1,w_1)^2 \\
&\hspace{1cm}\sum_{i'_2\leq i_2-2}\int dz_2 dz_2' \, \delta_{i_2}(y_2,z_2)\delta_{i_2'}(y_2,z_2')  \int_0^{t_2} ds_2 \int dw_2 \, K_{t_2-s_2}(z_2',w_2)\cac^{(n)}_{t_2,s_2}(z_2,w_2)^2   \, \mathbb{E}\Big[\<Psi>^{(n)}_{s_1}(w_1)\<Psi>^{(n)}_{s_2}(w_2)\Big]\\
&=\int_0^{t_1} ds_1 \int_0^{t_2} ds_2\int dz_1 dz_1'dz_2 dz_2' \, \bigg[\int dy_1 \, \delta_j(x,y_1)\sum_{i'_1\leq i_1-2}\delta_{i_1}(y_1,z_1)\delta_{i_1'}(y_1,z_1') \bigg]\\
&\hspace{3cm}\bigg[\int dy_2 \, \delta_j(x,y_2)\sum_{i'_2\leq i_2-2}\delta_{i_2}(y_2,z_2)\delta_{i_2'}(y_2,z_2') \bigg]  \cf^{(n)}_{s_1,t_1,s_2,t_2}(z_1,z_1',z_2,z_2')
\end{align*}
where the function $\cf^{(n)}$ is the same as in \eqref{cfn}. Just as in the previous section, we deduce that for every $\al\in (0,1)$,
\begin{align*}
&\mathbb{E}\Big[  \delta_j\big(\calt^{\mathbf{3},\mathbf{2},(n)}_{t_1}\big)(x)\delta_j\big(\calt^{\mathbf{3},\mathbf{2},(n)}_{t_2}\big)(x)\Big]=2^{4j\al}\int_0^{t_1} ds_1 \int_0^{t_2} ds_2\,  \cp^{(\al)}_j\big( H^{-\al}_{z_1} H^{-\al}_{z_2}\cf^{(n)}_{s_1,t_1,s_2,t_2}\big)(x)
\end{align*}
where the operator $\cp^{(\al)}$ is the same as in \eqref{deficpn}.

\

Going back to \eqref{t3-topo} and taking $\al=\frac18$, we obtain that for every $\eta>0$,
\small
\begin{align*}
&\bigg(\mathbb{E}\Big[\big\| \calti^{\mathbf{3},\mathbf{2},(n)}_{v_2}-\calti^{\mathbf{3},\mathbf{2},(n)}_{v_1}\big\|_{\cb_{2p,2p}^{-\frac14-\eta}}^{2p}\Big]\bigg)^{\frac1p}\\
&\lesssim \int_{v_1}^{v_2}dt_1 \int_{v_1}^{v_2}dt_2\int_0^{t_1} ds_1 \int_0^{t_2} ds_2\,\bigg(\sum_j 2^{-4jp\eta} \int dx \,  \Big|  \cp^{(\frac18)}_j\big( H^{-\frac18}_{z_1} H^{-\frac18}_{z_2}\cf^{(n)}_{s_1,t_1,s_2,t_2}\big)(x)\Big|^p \bigg)^{\frac1p}\\
&\lesssim \int_{v_1}^{v_2}dt_1 \int_{v_1}^{v_2}dt_2\int_0^{t_1} ds_1 \int_0^{t_2} ds_2\,\bigg(\sup_j \Big\|  \cp^{(\frac18)}_j\big( H^{-\frac18}_{z'_1} H^{-\frac18}_{z'_2}\cf^{(n)}_{s_1,t_1,s_2,t_2}\big)\Big\|_{L^\infty_x}\bigg)^{1-\frac{1}{p}} \bigg(\sup_j \Big\|  \cp^{(\frac18)}_j\big( H^{-\frac18}_{z_1} H^{-\frac18}_{z_2}\cf^{(n)}_{s_1,t_1,s_2,t_2}\big)\Big\|_{L^1_x}\bigg)^{\frac{1}{p}} .
\end{align*}
\normalsize

\

By Lemma \ref{lem:cpal}, it holds that 
\begin{align*}
&\sup_j \Big\|  \cp^{(\frac18)}_j\big( H^{-\frac18}_{z_1} H^{-\frac18}_{z'_2}\cf^{(n)}_{s_1,t_1,s_2,t_2}\big)\Big\|_{L^\infty_x}\lesssim  \Big\|   H^{-\frac18}_{z_1} H^{-\frac18}_{z_2}\cf^{(n)}_{s_1,t_1,s_2,t_2}\Big\|_{L^\infty_{z_1,z_1',z_2,z_2'}}.
\end{align*}
Then 
\small
\begin{align*}
&\big|H^{-\frac18}_{z_1} H^{-\frac18}_{z_2}\cf^{(n)}_{s_1,t_1,s_2,t_2}(z_1,z_1',z_2,z_2')\big|\\
&=\bigg|\int dw_1 \,  K_{t_1-s_1}(z_1',w_1)H^{-\frac18}_{z_1}\big(\cac^{(n)}_{t_1,s_1}(z_1,w_1)^2\big) \int dw_2 \, K_{t_2-s_2}(z_2',w_2)H^{-\frac18}_{z_2}\big( \cac^{(n)}_{t_2,s_2}(z_2,w_2)^2\big)   \, \mathbb{E}\Big[\<Psi>^{(n)}_{s_1}(w_1)\<Psi>^{(n)}_{s_2}(w_2)\Big]\bigg|\\
&\lesssim \frac{1}{|s_1-s_2|^{\frac12}}\sup_{w_1,w_2}\, \big|H^{-\frac18}_{z_1}\big(\cac^{(n)}_{t_1,s_1}(z_1,w_1)^2\big) \big| \big|H^{-\frac18}_{z_2}\big( \cac^{(n)}_{t_2,s_2}(z_2,w_2)^2\big)\big| \\
& \lesssim \frac{1}{|s_1-s_2|^{\frac12}}\sup_{w_1,w_2}\, \bigg(\int dz_1\, \big|\cac^{(n)}_{t_1,s_1}(z_1,w_1)\big|^{48} \bigg)^{\frac{1}{24}} \bigg(\int dz_2\, \big|\cac^{(n)}_{t_2,s_2}(z_2,w_2)\big|^{48} \bigg)^{\frac{1}{24}}\\
& \lesssim \frac{1}{|s_1-s_2|^{\frac12}}\frac{1}{|t_1-s_1|^{\frac{47}{48}}}\frac{1}{|t_2-s_2|^{\frac{47}{48}}}\sup_{w_1,w_2}\, \bigg(\int dz_1\, \big|\cac^{(n)}_{t_1,s_1}(z_1,w_1)\big| \bigg)^{\frac{1}{24}} \bigg(\int dz_2\, \big|\cac^{(n)}_{t_2,s_2}(z_2,w_2)\big| \bigg)^{\frac{1}{24}}\\
& \lesssim \frac{1}{|s_1-s_2|^{\frac12}}\frac{1}{|t_1-s_1|^{\frac{47}{48}}}\frac{1}{|t_2-s_2|^{\frac{47}{48}}}.
\end{align*}
\normalsize

\

On the other hand, with similar considerations as in \eqref{argu1}-\eqref{argu2}, we get that
$$\sup_j \Big\|  \cp^{(\frac18)}_j\big( H^{-\frac18}_{z_1} H^{-\frac18}_{z_2}\cf^{(n)}_{s_1,t_1,s_2,t_2}\big)\Big\|_{L^1_x}\lesssim \frac{1}{|t_1-s_1|^N}\frac{1}{|t_2-s_2|^N} \frac{1}{|s_1-s_2|^N}$$
for some large (but fixed) $N\geq 1$.

\smallskip

We can now repeat the arguments of Section \ref{subsec:bfb1} and assert that for every $p\geq 1$ large enough,
\begin{align*}
\sup_{n\geq 1}\mathbb{E}\Big[ \big\|\calti^{\mathbf{3},\mathbf{1},(n)} \big\|_{\cac^{\ga}([0,T]; \cb_x^{-\frac14-4\eta})}^{2p}\Big]
&\lesssim \int_0^T\int_0^T  \frac{dv_1 dv_2}{|v_2-v_1|^{2-(2\ga-\frac32)p}} < \infty,
\end{align*}
which concludes the proof of \eqref{unif-fifth-3-b}.

\

\subsection{Study of $\calt^{\mathbf{4},(n)}$ and $\calt^{\mathbf{5},(n)}$}

Recall the notation
$$\calti^{\mathbf{a},(n)}_t(x):=\int_0^t \calt^{\mathbf{a},(n)}_s(x) \, ds.$$

\begin{proposition}\label{prop-fifth-order-4-5}
Let $T>0$. For $\mathbf{a}=4,5$ and for all $\eps,\eta >0$, it holds that
\begin{equation}\label{unif-fifth-4-5}
 \sup_{n\geq 1} \, \mathbb{E} \Big[ \big\|\calti^{\mathbf{a},(n)} \big\|_{\cac^{\frac34-\eps}([0,T]; \cb_x^{-\frac14-4\eta})}^{2p} \Big] <\infty.
\end{equation}
\end{proposition}

\

Fix $0\leq \ga<\frac34$. For both $\mathbf{a}=4$ and $\mathbf{a}=5$, let us first write, for every $p\geq 1$ large enough,   
\begin{align}
 \big\|\calti^{\mathbf{a},(n)} \big\|_{\cac^{\ga}([0,T]; \cb_x^{-\frac14-4\eta})}^{2p}
&\lesssim \int_0^T\int_0^T dv_1 dv_2 \, \frac{\big\|H^{-\frac18-\eta}\big( \calti^{\mathbf{a},(n)}_{v_2}-\calti^{\mathbf{a},(n)}_{v_1}\big)\big\|_{L^{2p}_x}^{2p}}{|v_2-v_1|^{2\ga p+2}}.\label{normlipbis}
\end{align}
where we have used the embedding $\cw^{-\frac14-2\eta,2p}  \subset \cb_x^{-\frac14-4\eta}$, and then
\begin{align}
&\mathbb{E}\Big[\big\|H^{-\frac18-\eta}\big( \calti^{\mathbf{a},(n)}_{v_2}-\calti^{\mathbf{a},(n)}_{v_1}\big)\big\|_{L^{2p}_x}^{2p}\Big]\lesssim \bigg(\int_{v_1}^{v_2}dt_1 \int_{v_1}^{v_2}dt_2\,  \bigg(\int dy_1 dy_2\, \Big|\mathbb{E}\big[ \calt^{\mathbf{a},(n)}_{t_1}(y_1)\calt^{\mathbf{a},(n)}_{t_2}(y_2) \big]\Big|^{12}\bigg)^{\frac{1}{12}} \bigg)^p,\label{18}
\end{align}
due to $L^{12} \subset \cw^{-\frac14-2\eta,2p}$. For the sake of clarity, let us now treat the two cases $\mathbf{a}=4$ and $\mathbf{a}=5$ separately.

\smallskip

\subsubsection{First case: $\mathbf{a}=\mathbf{4}$}

One has in this case, for all $0\leq t_1<t_2\leq T$,
\begin{multline*}
\mathbb{E}\big[ \calt^{\mathbf{4},(n)}_{t_1}(y_1)\calt^{\mathbf{4},(n)}_{t_2}(y_2) \big]= c\int_0^{t_1}ds_1\int_0^{t_2}ds_2\int dw_1 dw_2 \, K_{t_1-s_1}(y_1,w_1) K_{t_2-s_2}(y_2,w_2)   \\
\cac^{(n)}_{t_1,s_1}(y_1,w_1)^2 \cac^{(n)}_{t_2,s_2}(y_2,w_2)^2 \mathbb{E}\Big[\big(\<Psi>^{(n)}_{s_1}(w_1)-\<Psi>^{(n)}_{t_1}(w_1)\big)\big(\<Psi>^{(n)}_{s_2}(w_2)-\<Psi>^{(n)}_{t_2}(w_2)\big)\Big],
\end{multline*}
for some  combinatorial coefficient $c\in \N$, and so
\begin{align}
\bigg(\int dy_1 dy_2\, \Big|\mathbb{E}\big[ \calt^{\mathbf{4},(n)}_{t_1}(y_1)\calt^{\mathbf{4},(n)}_{t_2}(y_2) \big]\Big|^{12}\bigg)^{\frac{1}{12}}\lesssim  \int_0^{t_1}ds_1\int_0^{t_2}ds_2 \, \big\|\ca^{\mathbf{4},(n)}_{t,s}\big\|_{L^{12}(\R^6)}\label{defca4}
\end{align}
with
\begin{multline*}
 \ca^{\mathbf{4},(n)}_{t,s}(y_1,y_2):=\int dw_1 dw_2 \, K_{t_1-s_1}(y_1,w_1) K_{t_2-s_2}(y_2,w_2) \\
  \cac^{(n)}_{t_1,s_1}(y_1,w_1)^2 \cac^{(n)}_{t_2,s_2}(y_2,w_2)^2 \mathbb{E}\Big[\big(\<Psi>^{(n)}_{s_1}(w_1)-\<Psi>^{(n)}_{t_1}(w_1)\big)\big(\<Psi>^{(n)}_{s_2}(w_2)-\<Psi>^{(n)}_{t_2}(w_2)\big)\Big].
\end{multline*}

\smallskip

By \eqref{Cp}, we have
\small
\begin{align}
&\int dy_1 dy_2 \, \big|\ca^{\mathbf{4},(n)}_{t,s}(y_1,y_2)\big|\leq \int dw_1 dw_2 \Big|\mathbb{E}\Big[\big(\<Psi>^{(n)}_{s_1}(w_1)-\<Psi>^{(n)}_{t_1}(w_1)\big)\big(\<Psi>^{(n)}_{s_2}(w_2)-\<Psi>^{(n)}_{t_2}(w_2)\big)\Big]\Big|\nonumber\\
&\hspace{6cm}\int dy_1 dy_2\,  K_{t_1-s_1}(y_1,w_1) K_{t_2-s_2}(y_2,w_2)  \cac^{(n)}_{t_1,s_1}(y_1,w_1)^2 \cac^{(n)}_{t_2,s_2}(y_2,w_2)^2\nonumber \\
& \lesssim  \frac{1}{|t_1-s_1| |t_2-s_2|}\bigg( \int dy_1 \,  {K}_{t_1,s_1}(y_1,w_1) \bigg)\bigg(\int  dy_2\,   {K}_{t_2,s_2}(y_2,w_2)\bigg)\nonumber\\
& \hspace{6cm} \int dw_1 dw_2\, \Big|\mathbb{E}\Big[\big(\<Psi>^{(n)}_{s_1}(w_1)-\<Psi>^{(n)}_{t_1}(w_1)\big)\big(\<Psi>^{(n)}_{s_2}(w_2)-\<Psi>^{(n)}_{t_2}(w_2)\big)\Big]\Big|\nonumber\\
&\lesssim  \frac{1}{|t_1-s_1| |t_2-s_2|}\cq^{(\frac12)}_{t,s},\label{a3l1}
\end{align}
\normalsize
where for every $0<\nu<1$, we define
\begin{align}\label{defi:q-nu}
\cq^{(\nu)}_{t,s}:=\frac{1}{|s_1-s_2|^{\nu}} + \frac{1}{|t_2-s_1|^{\nu}}+\frac{1}{|t_1-s_2|^{\nu}} +\frac{1}{|t_1-t_2|^{\nu}}.
\end{align}

\

On the other hand, thanks to the subsequent Lemma \ref{lem:inc-gau}, we immediately obtain that for every $0<\eta<1$,
\small
\begin{align}
&\big|\ca^{\mathbf{4},(n)}_{t,s}(y_1,y_2)\big| \lesssim \nonumber\\
&\lesssim |t_1-s_1|^\eta |t_2-s_2|^\eta \big(\cq^{(\frac12)}_{t,s}\big)^{1-2\eta} \big(\cq^{(\frac32)}_{t,s}\big)^{2\eta}\int dw_1 dw_2 \, K_{t_1-s_1}(y_1,w_1) K_{t_2-s_2}(y_2,w_2)  \cac^{(n)}_{t_1,s_1}(y_1,w_1)^2 \cac^{(n)}_{t_2,s_2}(y_2,w_2)^2\nonumber\\
& \lesssim \frac{1}{|t_1-s_1|^{1-\eta} |t_2-s_2|^{1-\eta}} \big(\cq^{(\frac12)}_{t,s}\big)^{1-2\eta} \big(\cq^{(\frac32)}_{t,s}\big)^{2\eta} \int dw_1 \,  {K}_{t_1-s_1}(y_1,w_1) \int dw_2 {K}_{t_2-s_2}(y_2,w_2) \nonumber \\
&\lesssim \frac{1}{|t_1-s_1|^{1-\eta} |t_2-s_2|^{1-\eta}} \big(\cq^{(\frac12)}_{t,s}\big)^{1-2\eta} \big(\cq^{(\frac32)}_{t,s}\big)^{2\eta}.\label{a3sup}
\end{align}
\normalsize

\

Therefore, using a basic interpolation procedure, we get from \eqref{a3l1} and \eqref{a3sup} that
\begin{eqnarray*}
\sup_{n\geq 1}\, \big\|\ca^{\mathbf{4},(n)}_{t,s}\big\|_{L^{12}(\R^6)} &\lesssim &
 \bigg( \frac{1}{|t_1-s_1| |t_2-s_2|}\cq^{(\frac12)}_{t,s} \bigg)^{\frac{1}{12}} \bigg(\frac{1}{|t_1-s_1|^{1-\eta} |t_2-s_2|^{1-\eta}} \big(\cq^{(\frac12)}_{t,s}\big)^{1-2\eta} \big(\cq^{(\frac32)}_{t,s}\big)^{2\eta}  \bigg)^{\frac{11}{12}} \\
&\lesssim& \frac{1}{|t_1-s_1|^{1-\frac{11\eta}{12}} |t_2-s_2|^{1-\frac{11\eta}{12}}} \big( \cq^{(\frac12)}_{t,s} \big)^{1-\frac{\eta}{6}} \big( \cq^{(\frac32)}_{t,s}  \big)^{\frac{11\eta }{6}}\\
&\lesssim& \frac{1}{|t_1-s_1|^{1-\frac{11\eta}{12}} |t_2-s_2|^{1-\frac{11\eta}{12}}}  \cq^{(\frac12-\frac{\eta}{12})}_{t,s} \cq^{(\frac{11\eta }{4})}_{t,s}  .
\end{eqnarray*}
Fix $\varepsilon>0$. It is easy to see that by taking $\eta>0$ small enough in the above bound, we can obtain 
\begin{equation*}
\sup_{n\geq 1}\, \big\|\ca^{\mathbf{4},(n)}_{t,s}\big\|_{L^{12}(\R^6)}
\lesssim \sum_{\si\in \ce_\varepsilon}\frac{1}{|t_1-s_1|^{1-\varepsilon} |t_2-s_2|^{1-\varepsilon}}  \frac{1}{|s_1-s_2|^{\si_1}}  \frac{1}{|t_2-s_1|^{\si_2}}\frac{1}{|t_1-s_2|^{\si_3}} \frac{1}{|t_1-t_2|^{\si_4}}  ,
\end{equation*}
where 
$$\ce_{\varepsilon}:=\big\{(\frac12+\varepsilon,\varepsilon,\varepsilon,\varepsilon),(\varepsilon,\frac12+\varepsilon,\varepsilon,\varepsilon),(\varepsilon,\varepsilon,\frac12+\varepsilon,\varepsilon),(\varepsilon,\varepsilon,\varepsilon,\frac12+\varepsilon)\big\}.$$

We now check that for $\sigma \in \ce_{\varepsilon}$ the parameters $\gamma= \sigma_1$, $\alpha_1=1-\eps$,  $\alpha_2=\sigma_2$,  $\beta_1=\sigma_3$, and $\beta_2=1-\eps$ satisfy the assumptions of Lemma \ref{Lem-prod2}. As a result, for all $0\leq t_1<t_2\leq T$ and $\varepsilon>0$ small enough,
\begin{align}
\sup_{n\geq 1}\, \int_0^{t_1} ds_1 \int_0^{t_2} ds_2\, \big\|\ca^{\mathbf{4},(n)}_{t,s}\big\|_{L^{12}(\R^6)}& \lesssim  \frac1{|t_2-t_1|^{2-2\eps+\sigma_1+\sigma_2+\sigma_3+\sigma_4-2}} \lesssim  \frac1{|t_2-t_1|^{\frac12+2\eps}}.\label{refa4a5}
\end{align}

\smallskip

By injecting this estimate into \eqref{18}-\eqref{defca4}, we deduce that
\begin{align*}
\sup_{n\geq 1}\, \mathbb{E}\Big[\big\|H^{-\frac18}\big( \calti^{\mathbf{a},(n)}_{v_2}-\calti^{\mathbf{a},(n)}_{v_1}\big)\big\|_{L^{2p}_x}^{2p}\Big]
&\lesssim \bigg(\int_{v_1}^{v_2}dt_1 \int_{v_1}^{v_2}dt_2\,  \frac{1}{|t_2-t_1|^{\frac12+\eps}} \bigg)^p\lesssim |v_2-v_1|^{(\frac32-\eps)p},
\end{align*}
and accordingly, by \eqref{normlipbis}, for $p\geq 1$ large enough 
\begin{align*}
\sup_{n\geq 1}\, \mathbb{E}\Big[\big\|\calti^{\mathbf{a},(n)} \big\|_{\cac^{\ga}([0,T];\cb_x^{-\frac14-\varepsilon})}^{2p}\Big]&\lesssim \int_0^T\int_0^T   \frac{dv_1 dv_2}{|v_2-v_1|^{2-(\frac32-2\ga-\eps) p}},
\end{align*}
which, since $\ga <\frac34$, is indeed finite for $\eps>0$ small enough and $p\geq 1$ large enough. This concludes the proof of \eqref{unif-fifth-4-5} for $\mathbf{a}=\mathbf{4}$.

\

\subsubsection{Second case: $\mathbf{a}=\mathbf{5}$}

For this final diagram, one has for all $0\leq t_1<t_2\leq T$,
\begin{multline*}
\mathbb{E}\big[ \calt^{\mathbf{5},(n)}_{t_1}(y_1)\calt^{\mathbf{5},(n)}_{t_2}(y_2) \big]=c \int_0^{t_1}ds_1\int_0^{t_2}ds_2 dw_1 dw_2 \, K_{t_1-s_1}(y_1,w_1) K_{t_2-s_2}(y_2,w_2) \\
   \cac^{(n)}_{t_1,s_1}(y_1,w_1)^2 \cac^{(n)}_{t_2,s_2}(y_2,w_2)^2 \mathbb{E}\Big[\big(\<Psi>^{(n)}_{t_1}(w_1)-\<Psi>^{(n)}_{t_1}(y_1)\big)\big(\<Psi>^{(n)}_{t_2}(w_2)-\<Psi>^{(n)}_{t_2}(y_2)\big)\Big],
\end{multline*}
 for some  combinatorial coefficient $c\in \N$, and so
\begin{align*}
\bigg(\int dy_1 dy_2\, \Big|\mathbb{E}\big[ \calt^{\mathbf{5},(n)}_{t_1}(y_1)\calt^{\mathbf{5},(n)}_{t_2}(y_2) \big]\Big|^{12}\bigg)^{\frac{1}{12}} \lesssim  \int_0^{t_1}ds_1\int_0^{t_2}ds_2 \, \big\|\ca^{\mathbf{5},(n)}_{t,s}\big\|_{L^{12}(\R^6)}
\end{align*}
with
\begin{multline*}
\ca^{\mathbf{5},(n)}_{t,s}(y_1,y_2):=\int dw_1 dw_2 \, K_{t_1-s_1}(y_1,w_1) K_{t_2-s_2}(y_2,w_2) \\
 \cac^{(n)}_{t_1,s_1}(y_1,w_1)^2 \cac^{(n)}_{t_2,s_2}(y_2,w_2)^2 \mathbb{E}\Big[\big(\<Psi>^{(n)}_{t_1}(w_1)-\<Psi>^{(n)}_{t_1}(y_1)\big)\big(\<Psi>^{(n)}_{t_2}(w_2)-\<Psi>^{(n)}_{t_2}(y_2)\big)\Big].
\end{multline*}

\

On the one hand, we get by elementary changes of variables that
\begin{multline}
 \big\|\ca^{\mathbf{5},(n)}_{t,s}\big\|_{L^{1}(\R^6)}\lesssim \\
 \begin{aligned}
&\lesssim \int dy_1 dy_2\int dw_1 dw_2 \,  {K}_{t_1-s_1}(y_1,w_1)  {K}_{t_2-s_2}(y_2,w_2)  \cac^{(n)}_{t_1,s_1}(y_1,w_1)^2 \cac^{(n)}_{t_2,s_2}(y_2,w_2)^2  \\
&\hspace{4cm}\Big(\cac^{(n)}_{t_1,t_2}(w_1,w_1) +\cac^{(n)}_{t_1,t_2}(w_1,y_2)+\cac^{(n)}_{t_1,t_2}(y_1,w_2)+\cac^{(n)}_{t_1,t_2}(y_1,y_2) \Big)    \\
&\lesssim  \frac{1}{|t_1-s_1| |t_2-s_2|}\bigg( \int dy_1 \,   {K}_{t_1-s_1}(y_1,w_1) \bigg)\bigg(\int  dy_2\,   {K}_{t_2-s_2}(y_2,w_2)\bigg) \int dw_1 dw_2\, \cac^{(n)}_{t_1,t_2}(w_1,w_1) \\
&\lesssim  \frac{1}{|t_2-t_1|^{\frac12}}\frac{1}{|t_1-s_1| |t_2-s_2|},\label{a4l1}
 \end{aligned}
\end{multline}
uniformly over $n\geq 1$.

\

On the other hand, thanks to Lemma \ref{lem:inc-gau-2} below, we get that for every small $\eta>0$,
\small
\begin{multline}
\sup_{n\geq 1} \,\big|\ca^{\mathbf{5},(n)}_{t,s}(y_1,y_2)\big|\lesssim  \\
 \begin{aligned}
&\lesssim \frac{1}{|t_2-t_1|^{\frac12+\eta}}\frac{1}{|t_1-s_1|}\frac{1}{|t_2-s_2|} \bigg(\int dw_1  \, {K}_{t_1-s_1}(w_1,y_1)|w_1-y_1|^\eta\bigg) \bigg(\int dw_2\,   {K}_{t_2-s_2}(w_2,y_2) |w_2-y_2|^\eta\bigg)  \\
&\lesssim \frac{1}{|t_2-t_1|^{\frac12+\eta}}\frac{1}{|t_1-s_1|^{1-\frac{\eta}{2}}}\frac{1}{|t_2-s_2|^{1-\frac{\eta}{2}}}.\label{a4sup}
 \end{aligned}
\end{multline}
\normalsize

\

Combining \eqref{a4l1} and \eqref{a4sup}, we deduce that
\begin{eqnarray*}
 \sup_{n\geq 1}\, \big\|\ca^{\mathbf{5},(n)}_{t,s}\big\|_{L^{12}(\R^6)}&\lesssim & \bigg(\frac{1}{|t_2-t_1|^{\frac12}} \frac{1}{|t_1-s_1| |t_2-s_2|}\bigg)^{\frac{1}{12}} \bigg(\frac{1}{|t_2-t_1|^{\frac12+\eta}}\frac{1}{|t_1-s_1|^{1-\frac{\eta}{2}}}\frac{1}{|t_2-s_2|^{1-\frac{\eta}{2}}} \bigg)^{\frac{11}{12}}\\
&  \lesssim &\frac{1}{|t_2-t_1|^{\frac12+\eta}}\frac{1}{|t_1-s_1|^{1-\frac{11\eta}{24}} |t_2-s_2|^{1-\frac{11\eta}{24}}}  ,
\end{eqnarray*}
and so, for every small $\eps>0$,
\begin{align*}
\sup_{n\geq 1}\, \int_0^{t_1}ds_1\int_0^{t_2}ds_2 \,  \big\|\ca^{\mathbf{5},(n)}_{t,s}\big\|_{L^{12}(\R^6)}\lesssim \frac{1}{|t_2-t_1|^{\frac12+\eps}}.
\end{align*}
We are here in the same position as in \eqref{refa4a5}, and therefore we can use the same arguments as in the previous section to conclude that
\begin{eqnarray*}
&\dis \sup_{n\geq 1} \, \mathbb{E} \Big[ \big\|\calti^{\mathbf{5},(n)} \big\|_{\cac^{\frac34-\eps}([0,T]; \cb_x^{-\frac14-4\eta})}^{2p} \Big] <\infty.
\end{eqnarray*}
for all $\eps,\eta>0$. The proof of Proposition \ref{prop-fifth-order-4-5} is thus complete.

\

\subsubsection{Technical lemmas}

\begin{lemma}\label{lem:inc-gau}
For all $0\leq t_1<t_2\leq T$, $0< s_1<t_1$, $0<s_2<t_2$ and for every $0<\eta<\frac12$, one has 
\begin{align*}
&\sup_{n\geq 1}\sup_{w_1,w_2 \in \R^3}\Big|\mathbb{E}\Big[ \Big(\<Psi>^{(n)}_{s_1}(w_1)-\<Psi>^{(n)}_{t_1}(w_1) \Big)\Big(\<Psi>^{(n)}_{s_2}(w_2)-\<Psi>^{(n)}_{t_2}(w_2)\Big)\Big]\Big|\lesssim |t_1-s_1|^\eta |t_2-s_2|^\eta\big(\cq^{(\frac12)}_{t,s}\big)^{1-2\eta} \big(\cq^{(\frac32)}_{t,s}\big)^{2\eta},
\end{align*}
where the notation $\cq^{(\nu)}$ has been introduced in \eqref{defi:q-nu}.

\end{lemma}

\begin{proof}[Proof of Lemma \ref{lem:inc-gau}]
Note that
\begin{multline*}
\mathbb{E}\Big[ \Big(\<Psi>^{(n)}_{s_1}(w_1)-\<Psi>^{(n)}_{t_1}(w_1) \Big)\Big(\<Psi>^{(n)}_{s_2}(w_2)-\<Psi>^{(n)}_{t_2}(w_2)\Big)\Big]=\\
=\cac^{(n)}_{s_1,s_2}(w_1,w_2)-\cac^{(n)}_{s_1,t_2}(w_1,w_2)-\cac^{(n)}_{t_1,s_2}(w_1,w_2)+\cac^{(n)}_{t_1,t_2}(w_1,w_2).
\end{multline*}
From here, we have first
\begin{multline}
\Big| \mathbb{E}\Big[ \Big(\<Psi>^{(n)}_{s_1}(w_1)-\<Psi>^{(n)}_{t_1}(w_1) \Big)\Big(\<Psi>^{(n)}_{s_2}(w_2)-\<Psi>^{(n)}_{t_2}(w_2)\Big)\Big]\Big|\leq \\
\leq \big|\cac^{(n)}_{s_1,s_2}(w_1,w_2)\big|+\big|\cac^{(n)}_{s_1,t_2}(w_1,w_2)\big|+\big|\cac^{(n)}_{t_1,s_2}(w_1,w_2)\big|+\big|\cac^{(n)}_{t_1,t_2}(w_1,w_2)\big| \lesssim \cq^{(\frac12)}_{t,s},\label{cq12}
\end{multline}
uniformly over $n\geq 1$ (thanks to \eqref{Cp}). Moreover, one has trivially  
\begin{multline}
\Big| \mathbb{E}\Big[ \Big(\<Psi>^{(n)}_{s_1}(w_1)-\<Psi>^{(n)}_{t_1}(w_1) \Big)\Big(\<Psi>^{(n)}_{s_2}(w_2)-\<Psi>^{(n)}_{t_2}(w_2)\Big)\Big]\Big|\leq  \\
\leq \big|\cac^{(n)}_{s_1,s_2}(w_1,w_2)-\cac^{(n)}_{s_1,t_2}(w_1,w_2)\big|+\big|\cac^{(n)}_{t_1,s_2}(w_1,w_2)-\cac^{(n)}_{t_1,t_2}(w_1,w_2)\big|.\label{trivi}
\end{multline}
If $s_2<s_1<t_2$, then we can write
\begin{eqnarray*}
\big|\cac^{(n)}_{s_1,s_2}(w_1,w_2)-\cac^{(n)}_{s_1,t_2}(w_1,w_2)\big|&\lesssim& \big|\cac^{(n)}_{s_1,s_2}(w_1,w_2)\big|+\big|\cac^{(n)}_{s_1,t_2}(w_1,w_2)\big|\\
&\lesssim &\frac{1}{|s_1-s_2|^{\frac12}}+\frac{1}{|t_2-s_1|^{\frac12}}\\
&\lesssim& \frac{|s_1-s_2|}{|s_1-s_2|^{\frac32}}+\frac{|t_2-s_1|}{|t_2-s_1|^{\frac32}}\lesssim |t_2-s_2| \cq^{(\frac32)}_{t,s},
\end{eqnarray*}
uniformly over $n\geq 1$ (by \eqref{Cp} again). On the other hand, if $s_1<s_2<t_2$, and since
\begin{align*}
&\cac^{(n)}_{s_1,r}(y_1,y_2)=\frac12\int_{r-s_1+2\varepsilon_n}^{r+s_1+2\varepsilon_n}d\si\,  K_{\si}(y_1,y_2) \quad \text{for} \ r\geq s_1,
\end{align*}
it holds that
\begin{align*}
\big|\cac^{(n)}_{s_1,s_2}(w_1,w_2)-\cac^{(n)}_{s_1,t_2}(w_1,w_2)\big|&\lesssim |t_2-s_2| \, \Big(\sup_{r\in [s_2,t_2]} \big|\partial_r \cac^{(n)}_{s_1,r}(w_1,w_2)\big| \Big)\\
&\lesssim |t_2-s_2| \, \Big(\sup_{r\in [s_2,t_2]} \big|K_{r-s_1+2\varepsilon_n}(w_1,w_2)\big|+\sup_{r\in [s_2,t_2]} \big|K_{r+s_1+2\varepsilon_n}(w_1,w_2)\big| \Big)\\
&\lesssim \frac{|t_2-s_2|}{|s_2-s_1|^{\frac32}} \lesssim |t_2-s_2| \cq^{(\frac32)}_{t,s},
\end{align*}
uniformly over $n\geq 1$. With similar elementary arguments, we get that
\begin{align*}
\sup_{n\geq 1}\, \big|\cac^{(n)}_{t_1,s_2}(w_1,w_2)-\cac^{(n)}_{t_1,t_2}(w_1,w_2)\big|&\lesssim |t_2-s_2| \cq^{(\frac32)}_{t,s},
\end{align*}
and so, going back to \eqref{trivi}, we deduce that
\begin{align}
&\sup_{n\geq 1}\,  \Big| \mathbb{E}\Big[ \Big(\<Psi>^{(n)}_{s_1}(w_1)-\<Psi>^{(n)}_{t_1}(w_1) \Big)\Big(\<Psi>^{(n)}_{s_2}(w_2)-\<Psi>^{(n)}_{t_2}(w_2)\Big)\Big]\Big|\lesssim |t_2-s_2| \cq^{(\frac32)}_{t,s}.\label{cq32-1}
\end{align}
By symmetry,
\begin{align}
&\sup_{n\geq 1}\,  \Big| \mathbb{E}\Big[ \Big(\<Psi>^{(n)}_{s_1}(w_1)-\<Psi>^{(n)}_{t_1}(w_1) \Big)\Big(\<Psi>^{(n)}_{s_2}(w_2)-\<Psi>^{(n)}_{t_2}(w_2)\Big)\Big]\Big|\lesssim |t_1-s_1| \cq^{(\frac32)}_{t,s}.\label{cq32-2}
\end{align}

\

Interpolating between \eqref{cq12}, \eqref{cq32-1} and \eqref{cq32-2} immediately yields the desired bound.
\end{proof}

\

\begin{lemma}\label{lem:inc-gau-2}
For all $0\leq t_1,t_2\leq T$, $w_1,y_1,w_2,y_2\in \R^3$ and for every $\eta>0$ small enough, one has
\begin{align*}
&\sup_{n\geq 1}\, \Big|\mathbb{E}\Big[ \Big(\<Psi>^{(n)}_{t_1}(w_1)-\<Psi>^{(n)}_{t_1}(y_1) \Big)\Big(\<Psi>^{(n)}_{t_2}(w_2)-\<Psi>^{(n)}_{t_2}(y_2)\Big)\Big]\Big| \lesssim  \frac{|w_1-y_1|^\eta |w_2-y_2|^\eta}{|t_2-t_1|^{\frac12+\eta}}.
\end{align*}
\end{lemma}

\begin{proof}
Observe that
\begin{multline}
\mathbb{E}\Big[ \Big(\<Psi>^{(n)}_{t_1}(w_1)-\<Psi>^{(n)}_{t_1}(y_1) \Big)\Big(\<Psi>^{(n)}_{t_2}(w_2)-\<Psi>^{(n)}_{t_2}(y_2)\Big)\Big]=\\
=\cac^{(n)}_{t_1,t_2}(w_1,w_2)-\cac^{(n)}_{t_1,t_2}(w_1,y_2)-\cac^{(n)}_{t_1,t_2}(y_1,w_2)+\cac^{(n)}_{t_1,t_2}(y_1,y_2).\label{expexp}
\end{multline}
Thus, by \eqref{Cp}, one has immediately
\begin{align}
&\sup_{n\geq 1}\, \Big|\mathbb{E}\Big[ \Big(\<Psi>^{(n)}_{t_1}(w_1)-\<Psi>^{(n)}_{t_1}(y_1) \Big)\Big(\<Psi>^{(n)}_{t_2}(w_2)-\<Psi>^{(n)}_{t_2}(y_2)\Big)\Big] \Big|\lesssim \sup_{n\geq 1}\sup_{w,y \in \R^3} \big|\cac^{(n)}_{t_1,t_2}(w,y)\big|\lesssim \frac{1}{|t_1-t_2|^{\frac12}}.\label{estinc-1}
\end{align}
On the other hand, we can write for instance
\begin{align*}
\big|\cac^{(n)}_{t_1,t_2}(w_1,w_2)-\cac^{(n)}_{t_1,t_2}(w_1,y_2)\big| \leq |w_2-y_2| \Big( \sup_{w,y\in \R^3} \big| (\nabla_y \cac^{(n)}_{t_1,t_2})(w,y) \big|\Big).
\end{align*}
It is easy to check that for all $w,y$, 
\begin{align*}
\big| (\nabla_y \cac^{(n)}_{t_1,t_2})(w,y) \big|&\lesssim \int_{|t_2-t_1|+2\varepsilon_n}^{t_2+t_1+2\varepsilon_n}d\si\,  \big| (\nabla_y K_\si)(w,y) \big|\\
&\lesssim \int_{|t_2-t_1|}^{2T+2}d\si\, \bigg( \frac{1}{\si^{\frac52}} |w-y| \exp\Big(-\frac{1}{4\si}|w-y|^2\Big)+\frac{1}{\si^{\frac12}} |w+y| \exp\Big(-\frac{\si}{4}|w+y|^2\Big)  \bigg)\\
&\lesssim \int_{|t_2-t_1|}^{2T+2}d\si\, \bigg( \frac{\si^{\frac12}}{\si^{\frac52}} +\frac{1}{\si^{\frac12}\si^{\frac12}}  \bigg)\lesssim \int_{|t_2-t_1|}^{2T+2}\frac{d\si}{\si^{2}}\lesssim \frac{1}{|t_2-t_1|},
\end{align*}
uniformly over $n\geq 1$, and so
\begin{align*}
\sup_{n\geq 1}\, \big|\cac^{(n)}_{t_1,t_2}(w_1,w_2)-\cac^{(n)}_{t_1,t_2}(w_1,y_2)\big| \lesssim  \frac{|w_2-y_2|}{|t_2-t_1|}.
\end{align*}
With expansion \eqref{expexp} in mind, we derive in this way
\begin{align}
&\sup_{n\geq 1}\, \Big|\mathbb{E}\Big[ \Big(\<Psi>^{(n)}_{t_1}(w_1)-\<Psi>^{(n)}_{t_1}(y_1) \Big)\Big(\<Psi>^{(n)}_{t_2}(w_2)-\<Psi>^{(n)}_{t_2}(y_2)\Big)\Big]\lesssim \frac{1}{|t_1-t_2|} \min \Big( |w_1-y_1|,|w_2-y_2|\Big).\label{estinc-2}
\end{align}

\

Interpolating between \eqref{estinc-1} and \eqref{estinc-2} provides us with the desired estimate.
\end{proof}


\

\section{Renormalization sequences}\label{section:renormalization}

At this point, we have established the convergence of the sequence $(X^{(n)})$ defined by \eqref{appro-equ-introd}, for $(\frakc^{(n)})$ given by
$$\frakc^{(n)}_t(x):=3\frakc^{\mathbf{1},(n)}_t(x)-9\frakc^{\mathbf{2},(n)}_t(x),$$
where
$$\frakc^{\mathbf{1},(n)}_t(x):=\mathbb{E}\Big[ \big| \<Psi>^{(n)}_t(x)\big|^2\Big] \quad \text{and} \quad \frakc^{\mathbf{2},(n)}_t(x):=\mathbb{E}\Big[ \<Psi2>^{(n)}_t(x) \<IPsi2>^{(n)}_t(x)\Big].$$

Therefore, it only remains us to verify that the sequence $(\frakc^{(n)})$ satisfies the conditions appearing in items $(i)$ and $(ii)$ of Theorem \ref{theo:main-intro0}:

\begin{lemma}
In the above setting:

\smallskip

\noindent
$(i)$ For all fixed $t>0$, there exist constants $c_1,c_2,c_3>0$ such that for all  $x\in \R^3$
\begin{align}\label{bounding-c1c2}
  c_1 e^{-\frac{4 |x|^2}{2^n}}2^{\frac{n}{2}}   \leq \frakc^{\mathbf{1},(n)}_t(x) \leq  c_2  e^{-\frac{2 |x|^2}{2^n}}  2^{\frac{n}{2}} \quad \text{and}\quad 
0\leq \frakc^{\mathbf{2},(n)}_t(x) \leq c_3 n.
\end{align}

\smallskip

\noindent
$(ii)$ For all positive functions $\vp\in \cd(\R_+),\psi\in \cd(\R^3)$, with $\vp \equiv 1$ on $[1,2]$ and $\psi\neq 0$, one has
$$\int dt dx \, \vp(t)\psi(x) \frakc^{(n)}_t(x) \stackrel{n\to\infty}{\longrightarrow} \infty.$$ 
\end{lemma}

\begin{proof}

In the sequel, the notation $c$ stands for a generic strictly positive constant (independent from $t$ and $x$) whose exact value may vary from line to line.

\smallskip

\noindent
$(i)$ According to \eqref{def-E}, one has
\begin{align*}
\frakc^{\mathbf{1},(n)}_t(x)=\mathbb{E}\Big[ \big| \<Psi>^{(n)}_t(x)\big|^2\Big]=\cac^{(n)}_{t,t}(x,x)=\frac12\int_{2\varepsilon_n}^{2t+2\varepsilon_n}d\si\,  K_{\si}(x,x)=\int_{\varepsilon_n}^{t+\varepsilon_n}d\si\,  K_{2\si}(x,x) ,
\end{align*}
and thanks to Mehler's formula, we deduce the expression
\begin{align}\label{expressc1}
\frakc^{\mathbf{1},(n)}_t(x)&=c \int_{\varepsilon_n}^{t+\varepsilon_n}\frac{d\si}{\si^{\frac32}} \exp\big(-2\si |x|^2\big).
\end{align}
  The estimates for $\frakc^{\mathbf{1},(n)}$ in \eqref{bounding-c1c2} now immediately follow  from the elementary bounds
\begin{align}\label{encadrement}
  \int_{\varepsilon_n}^{t+\varepsilon_n}\frac{d\si}{\si^{\frac32}} \exp\big(-2\si |x|^2\big) \geq \int_{\varepsilon_n}^{2 \varepsilon_n  }\frac{d\si}{\si^{\frac32}} \exp\big(-2\si |x|^2\big)\geq c {\eps^{-\frac12}_n} \exp \big(- 4 \eps_n  |x|^2\big) 
\end{align}
and 
\begin{align*} 
  \int_{\varepsilon_n}^{t+\varepsilon_n}\frac{d\si}{\si^{\frac32}} \exp\big(-2\si |x|^2\big) \leq \exp\big(-2\eps_n |x|^2\big)  \int_{\varepsilon_n}^{+\infty  }\frac{d\si}{\si^{\frac32}} \leq c {\eps^{-\frac12}_n} \exp \big(- 2 \eps_n  |x|^2\big) .
\end{align*}
\

As far as $\frakc^{\mathbf{2},(n)}$ is concerned, we have first (see \eqref{identif-c2}) 
\begin{align*}
&\frakc^{\mathbf{2},(n)}_t(x)=\mathbb{E}\Big[ \<Psi2>^{(n)}_t(x) \<IPsi2>^{(n)}_t(x)\Big]=2\int_0^t ds\int dw \, K_{t-s}(x,w) \cac^{(n)}_{t,s}(x,w)^2 \geq 0.
\end{align*}
Then, still with expression \eqref{def-E} of $\cac^{(n)}$ in mind, we can expand the above quantity as
\begin{align*}
\frakc^{\mathbf{2},(n)}_t(x)&=\frac12\int_0^t ds\int dw \, K_{t-s}(x,w) \bigg( \int_{t-s+2\varepsilon_n}^{t+s+2\varepsilon_n}d\si\,  K_\si(x,w)\bigg)^2,
\end{align*}
and accordingly
\begin{align*}
\frakc^{\mathbf{2},(n)}_t(x)&\lesssim \int_0^t ds\int dw \, K_{s}(x,w) \bigg( \int_{s+2\varepsilon_n}^{\infty}\frac{d\si}{\si^{\frac32}}\bigg)^2\lesssim \int_0^t \frac{ds}{(s+2\varepsilon_n)}\lesssim |\log \eps_n|,
\end{align*}
and the bound follows since $\eps_n=2^{-n}$.

\

\noindent
$(ii)$ The claimed divergence is an elementary consequence of \eqref{expressc1}-\eqref{encadrement}: indeed, since $\vp \equiv 1$ on $[1,2]$ and $\psi\neq 0$, one has   for $n\geq 1$ large enough
\begin{align*}
\int dt dx \, \vp(t)\psi(x) \frakc^{\mathbf{1},(n)}_t(x)&\geq c {\eps^{-\frac12}_n} \int dt dx \, \vp(t)\psi(x) e^{- |x|^2} \\
& \geq c_\psi\, \eps_n^{-\frac12}= c_\psi\, 2^{\frac{n}2},
\end{align*}
  with $c_\psi>0$. The claim then follows from the bound $0\leq \frakc^{\mathbf{2},(n)}_t(x) \leq c_4 n$.
\end{proof}

\

\appendix

 \section{Microlocal analysis for the harmonic oscillator and estimates on the heat kernel}\label{appendix:microlocal}



    \subsection{Microlocal analysis for the harmonic oscillator on $\R^d$}

We will use semi-classical analysis, based on the standard quantization. For an introduction to the subject, see the books \cite{Robert, Martinez, Zworski}. We also refer to the pedagogical work of \cite[Apprendix A \& Appendix B]{Imekraz} for some results on pseudo-differential calculus for the Laplacian with confining polynomial  potential.\medskip
  
  Since the results of this section do not rely on the dimension, we state them in general dimension~$d\geq 1$.\medskip

Let $x=(x_1, \dots, x_d) \in \R^d$ and $\alpha= (\alpha_1, \dots, \alpha_d) \in \N^d$ and we define $\partial^\alpha_x= \partial^{\alpha_1}_{x_1} \cdots \partial^{\alpha_d}_{x_d} $. Then 
for $ m \in \R$, we define $ T^m $ as the vector space of symbols $ q(x,\xi) \in \mathcal{C}^\infty( \R^d \times \R^d ) $ such that  for all $ \alpha \in \N^d $ and $ \beta \in \N^d $, there exists a  constant $ C_{\alpha, \beta} >0$ such that for all $ ( x , \xi ) \in \R^d \times \R^d $, we have 
\begin{equation*}
  | \partial ^\alpha_x  \partial ^\beta_\xi q(x, \xi ) | \leq C_{\alpha, \beta}  (1+|x|+|\xi|)^{m-\beta}.
 \end{equation*} 
 
 Similarly, let $S^m $ be the vector space of symbols satisfying
\begin{equation*}
  | \partial ^\alpha_x  \partial ^\beta_\xi q(x, \xi ) | \leq C_{\alpha, \beta}  (1+|\xi|)^{m-\beta}.
 \end{equation*} 
 
For $ q \in S^m \cup T^m $ and $ 0<h \leq 1 $, let $ Op_h(q) $ be the operator defined by
\begin{eqnarray}
Op_h (q) f(x) &=& (2 \pi h )^{-d} \int_{\R^d \times \R^d}  dy d \xi  \,  e^{i\frac{(x-y)\cdot  \xi}{h} } q(x , \xi ) f(y)\nonumber \\
&=& (2 \pi h )^{-d} \int_{\R^d } d \xi  \, e^{i\frac{x \cdot \xi}h } q(x , \xi ) \hat{f}(\frac{\xi}h). \label{defi-2}  
\end{eqnarray}

With this definition, we observe that if $\psi : x \mapsto \psi(x)$ is a function which only depends on $x$, then 
$$Op_h (\psi q) =\psi Op_h ( q).$$  

It is well known that the composition of two pseudo-differential operators is also a pseudo-differential operator. The next result (see Martinez \cite{Martinez}) gives a quantification of this fact and a development of the symbol  in powers of $h$.

 \begin{theorem}  \label{1-compose}
If $ q_1 \in S^{m_1} $ (respectively $ T^{m_1} $) and $ q_2 \in S ^{m_2} $ (respectively $ T^{m_2} $) then there exists a symbol $ q \in S^{m_1+m_2} $ (respectively $ T^{m_1+m_2} $) such that
\begin{equation*}
 Op_h(q_1) \circ Op_h(q_2) = Op_h(q) 
 \end{equation*}
with
 \begin{equation*}
q = \sum _{| \alpha | \leq N } \frac{ h^{| \alpha | } } { i^{| \alpha| }} \ \partial^\alpha_\xi q_1 \partial^\alpha_x q_2 + h^{N+1} r_N  
\end{equation*}
where $r_N \in S^{m_1+m_2-(N+1) }$  (respectively  $T^{m_1+m_2-(N+1) }$).
\end{theorem} 

A pseudo-differential of order 0 is continuous in any $L^p$ space. Namely, we have the following result (see Martinez \cite{Martinez}):
 \begin{theorem} \label{1-operator}
If $ q(x,\xi) \in S^0 $ then for any $ 1\leq p\leq \infty $, there exists a constant $ C > 0 $ such that for any $ h \in ]0,1] $ and any  $ u \in L^p(\R^d) $,
\begin{equation*}
\|   Op_h( q ) u \|_{L^p(\R^d)} \leq C \|u\|_{L^p(\R^d)}.
\end{equation*}
\end{theorem}

We also recall (see \cite[Proposition A.4]{BPT}) the following result which makes the link between functional calculus and pseudo-differential calculus.

 \begin{proposition}  \label{expension} Let $ \Theta \in \mathcal{C}^\infty_0 ( \R ) $ and $\theta_0 \in \mathcal{C}^\infty_0 ( \R )$ such that for all $(x, \xi) \in\R^d \times \R^d$, 
 $$\theta_0(x) \Theta(|x|^2 +  |\xi| ^2 )= \Theta(|x|^2 +  |\xi| ^2 ).$$ Then, for all $N \geq 1$ there exist $(\Psi_j )_{0 \leq j\leq N}$ with $ \Psi_j \in T^{-j} \subset S^0 $  where $ \Psi_0(x, \xi ) = \Theta ( |x|^2 +  |\xi| ^2  ) $ and $ Supp \, \Psi_j  \subset \big\{ (x, \xi ) : \,  |x|^2 + |\xi|^2 \in Supp \, \Theta  \big\} $,  such that
  $$\Theta \big( -h^2 \Delta+ |x|^2  \big)u = \sum_{j=0}^{N-1} h^j  Op_h(  \Psi_j ) (\theta_0  u)+r_N(u),$$
 and $r_N$  satisfies: for all $\alpha, s\geq 0$, $1\leq p\leq \infty$ there is a constant $ C_{N} > 0  $ such that for all $ h \in ]0,1 ] $ and $ u \in L^p(\R^d) $,
\begin{equation*}
 \|\langle x\rangle^\alpha  r_N(u)\|_{L^p(\R^d)} \leq C_{N} h^{N} \|u\|_{L ^p(\R^d)}.
 \end{equation*} 
\end{proposition}

Recall the definition \eqref{def-delta} of the operator $\dis \delta_k=\chi(\frac{\sqrt{H}}{2^{k}})$ which is involved in the definition of the Besov spaces associated with the harmonic oscillator.   \medskip

For $k\geq 1$, set $h_k=2^{-2k}$, and for any function $f$, we define $\widetilde{f}_k$ by 
$$\dis \widetilde{f}_k(x)=f(\frac{x}{\sqrt{h_k}}).$$ Then we have  
\begin{equation}\label{cht} 
\delta_k f(x)= \big(\theta(h_kH)f\big)(x)=\big(\theta(-h_k^2\Delta+|x|^2)\widetilde{f}_k\big)(\sqrt{h}_kx).
\end{equation}
This representation will allow to use Proposition \ref{expension} in the study of terms involving $\delta_k$. \medskip

We have the  following continuity results :

 \begin{proposition}\label{prop.conti}
  Let $\chi \in \mathcal{C}_0^\infty(\R)$  and for $k \geq 0$ set    $\dis \delta_k=\chi(\frac{\sqrt{H}}{2^{k}})$. Let $1 \leq p \leq \infty$. Then there exists a constant $C>0$ such that for all $k \geq 1$ and all  $ u \in L^p(\R^d) $
\begin{equation}\label{conti-delta}
\|   \delta_k u \|_{L^p(\R^d)} \leq C \|u\|_{L^p(\R^d)},
\end{equation}
and
\begin{equation}\label{conti-sk}
\|   S_k u \|_{L^p(\R^d)} \leq C \|u\|_{L^p(\R^d)}.
\end{equation}
\end{proposition}

 \begin{proof}
The estimate \eqref{conti-delta} is a consequence of   Theorem \ref{1-operator} and Proposition \ref{expension}. We refer to \cite[Proposition 4.1]{BTT} for a self-contained proof.  

Let us turn to \eqref{conti-sk}. Define  $\psi \in \mathcal{C}_0^{\infty}(\R)$ by $\psi(\xi)=1-\sum_{j=1}^{+\infty}\chi(\frac{\xi}{2^j})=\chi_{-1}(\xi)+\chi(\xi)$. Then $\chi(\xi)=\psi(\xi)-\psi(2\xi)$ and we deduce that 
$$S_ku = \chi_{-1}(\sqrt{H})u +\psi(\frac{\sqrt{H}}{2^{k-1}})-\psi(2\sqrt{H})u.$$
The result now follows from \eqref{conti-delta}.
\end{proof}

 Let $\theta' \in \mathcal{C}_0^\infty(\R)$ such that $\theta' \equiv 1$ on a neighborhood of $\text{Supp} \, \theta$ (and so $\theta' \theta= \theta$). Assume moreover that $\text{Supp} \,\theta'  \subset \big\{\xi \in \R_+:\; \;  \big(\frac34\big)^2 \leq \xi \leq \big(\frac83\big)^2  \big\}$, and for $k \geq 0$, set  
  \begin{equation}\label{def-deltap}
  \delta'_k=\theta'(\frac{H}{2^{2k}}).
  \end{equation}
By construction, it holds that $\delta'_k\delta_k=\delta_k$. Therefore, in the term $\delta_k u$, the function $u$ can be assumed to be localised in frequencies, namely $u = \delta'_k u$. Observe also that the change of the cutoff function is harmless, since the~$\B^\alpha_{p,q}(\R^d)$ norm does not depend on the choice of the cutoff function  (see \cite[Remark~2.17)]{BCD}).
  \medskip
 
 We end this paragraph with an analogous result to the classical Littlewood-Paley theorem (see {\it e.g.} \cite{Stein}[Theorem 5]), which gives a characterization of the $L^p$ norm, $1<p<\infty$.
 
 \begin{lemma}[\cite{PX}, Proposition 4.3]
\label{L-P}
Let $1<p <\infty$, then 
\begin{equation*}
 \big\| u \big\|_{L^p(\R^d)}   \lesssim \Big\|     \big(\sum_{ j \in \Z} |\delta_j u|^2\big)^{\frac12} \Big\|_{L^p(\R^d)} \lesssim \big\|u  \big\|_{L^p(\R^d)} .
\end{equation*}
\end{lemma}


 \subsection{Paraproducts}\label{parap}
 
 As in \cite{GIP, MW} we will use the theory of paraproducts. For $f,g \in \mathscr{S}(\R^3)$, we define 
 $$S_k f=\sum_{j=-1}^ {k-1} \delta_j f .$$
Then, we define the  paraproduct of $f$ and $g$ by
\begin{equation}\label{def-para}
 f \pl g= \sum_{\substack{k, j \geq -1 \\ j \leq k-4   }}(\delta_j f) ( \delta_k g)=\sum_{k =-1}^{+\infty}(S_{k-3}f)( \delta_k g).
 \end{equation}
 The resonant term is defined by 
  $$f \pe g= \sum_ { k \sim j }(\delta_j f) ( \delta_k g),$$
   where  we use the notation
$$\{k \sim j\}=\big\{k,j \geq -1: \; |k-j| \leq 3\big\}.$$ 
  We write $f \pg g= g  \pl f $. Then the Bony decomposition reads 
  $$fg = f \pl g + f \pe g + f \pg g.$$
  We also define $\ple=\pl +\pe$ and  $\pge=\pg +\pe$.  \medskip
  
  We have modified  a bit the definition of the paraproduct compared to the usual one,  because we need a slightly  stronger assumption on the frequencies to control the interactions, see Lemma~\ref{lem.SC} and Lemma~\ref{LemC14}.


 \subsection{Some results on the harmonic  Besov spaces   $\B^\alpha_{p,q}(\R^d)$}
 
 The Besov spaces for the harmonic operator are defined in \eqref{def-besov}  (with obvious modifications  for the general case $d \geq 1)$. To have an insight of the harmonic Besov spaces compared to the usual Besov spaces (for which the operator $\delta_k$ is replaced with $\dis \widetilde{\delta}_k=\chi\big(\frac {\sqrt{1-\Delta} }{2^k}\big)$ in the definition), we quote the following result from~\cite[Lemma~A.5.1]{BT}:
 
  \begin{lemma}\label{compa-Besov}
Let $d \geq 1$. For any $1\leq p, q \leq + \infty$, and any $\rho\geq 0$, 
$$ \mathcal{B}^{\rho} _{p,q}(\R^d) \subset B^{\rho} _{p,q}(\R^d), \qquad {B}^{-\rho} _{p,q}(\R^d) \subset \mathcal{B}^{-\rho} _{p,q}(\R^d)$$
with continuous injections.
 \end{lemma}

  By \cite[Theorem 2.36]{BCD}, if $0<\rho<1$, we have the equivalence $\mathcal{C}^\rho(\R^d) = B^{\rho} _{\infty,\infty}(\R^d)$ 
 where $ \mathcal{C}^\rho(\R^d)$ is the space of  H\"older functions with exponent $\rho$.  
 One then see the strict inclusion in case of the harmonic Besov space  
 $$\mathcal{B}^{\rho} _{\infty,\infty}(\R^d) \subset \mathcal{C}^\rho(\R^d).$$ 
 For more results on classical Besov spaces, we refer for example to  the book of Bahouri-Chemin-Danchin \cite[Chapter~2]{BCD}. \medskip

   We will need the following result (see  \cite[Lemma 2.49 and Lemma 2.84]{BCD} for the case the usual Besov spaces).

 \begin{lemma}\label{lem-supp}
 Let $\sigma>0$,  $1 \leq p, q \leq \infty$ and $k_0\geq 0$. Let $(u_k)_{k \geq -1}$ be a
sequence of smooth functions such that $S_{k+k_0}u_k=u_k$ and 
 $$
 \Bigl( \sum_{k \geq -1} \|2^{k\sigma}  u_k \| ^q_{L^p(\R^d)} \Bigr)^{\frac 1 q} <\infty.
$$
We assume that the series $\dis \sum_{k\geq -1} u_k$ converges to $u$ in $\mathscr{S}'(\R^d)$. We then have $u \in \mathcal{B}^{\sigma}_{p,q}(\R^d)$ and if $1 \leq q <\infty$
 \begin{equation*}
 \| u\|_{\mathcal{B}^{\sigma}_{p,q}(\R^d)} \lesssim  \Bigl( \sum_{k \geq -1} \|2^{k\sigma}  u_k \| ^q_{L^p(\R^d)} \Bigr)^{\frac 1 q},
  \end{equation*}
  while for $q=\infty$
   \begin{equation*}
 \| u\|_{\mathcal{B}^{\sigma}_{p,\infty}(\R^d)} \lesssim   \sup_{k \geq -1}   \|2^{k\sigma}  u_k \|_{L^p(\R^d)}.
  \end{equation*}
 \end{lemma}
 
 \begin{proof} The proof is close to the proof of  \cite[Lemma 2.49]{BCD}, and we repeat the argument. Since  $\dis u=  \sum_{k\geq -1} u_k$, for all $j \geq k+k_0$,
 $$\delta_{j}u =\sum_{k\geq -1} \delta_{j}u_k=\sum_{k\geq -1} \delta_{j}\big(S_{k+k_0}u_k \big)= \sum_{k> j-k_0} \delta_{j}\big(S_{k+k_0}u_k\big),$$ 
 where we used that $ \delta_{j}S_{k+k_0}=0$ when $j \geq k+k_0+1$. Thus
 \begin{eqnarray*}
 2^{j \sigma} \|\delta_{j}u \|_{L^p(\R^d)} &\lesssim & \sum_{k> j-k_0-1}  2^{j \sigma}  \|u_k \|_{L^p(\R^d)} \\
  &\lesssim &  \sum_{k> j-k_0-1}  2^{(j-k) \sigma} \big(  2^{k \sigma}   \|u_k \|_{L^p(\R^d)} \big).
 \end{eqnarray*}
 We can then conclude using the Young convolution inequality.
 \end{proof}

We have also the following inclusion properties of classical type (see {\it e.g.} \cite{FI}).  
\begin{lemma} \label{lem:inclusion-besov}
Let $\alpha, \beta \in \mathbb R$ and $1 \leq p,q,r \leq \infty$. 
\begin{enumerate}
\item [$(i)$] If  $\alpha \le \beta$, then 
$$\| u\|_{ \mathcal{B}^{\alpha}_{p,q}(\R^d)} \leq C \|u\|_{\mathcal{B}^{\beta}_{p,q}(\R^d)}.$$
\item [$(ii)$] If $\beta>0$, then  
$$ \| u\|_{\mathcal{B}^0_{p,\infty}(\R^d)} \le C \|u\|_{L^{p}(\R^d)} \le C' \|u\|_{\mathcal{B}^{\beta}_{p,1}(\R^d)} \le C'' \|u\|_{\mathcal{B}^\beta_{p,\infty}(\R^d)}.$$
\item [$(iii)$] If $r \leq p$, then 
$$\|u\|_{\mathcal{B}^{\beta}_{p,q}(\R^d)}  \le C  \|u\|_{\mathcal{B}^{\beta+d(\frac{1}{r}-\frac{1}{p})}_{r,q}(\R^d)} . $$
\end{enumerate}
\end{lemma} 

\
   \begin{lemma} 
 Let $\alpha \in \R$. Let $1 \leq p<\infty$ and $\epsilon >0$ such that $\eps>\frac{d}{p}$, then there exists $C>0$ such that 
 \begin{equation}\label{sobo-beso}
 \| u\|_{\mathcal{B}^{\alpha} _{\infty,\infty}(\R^d) } \leq C \| u\|_{\mathcal{W}^{\alpha+\eps,p}(\R^d) }. 
  \end{equation}
 \end{lemma}
 
 \begin{proof}
 By the Sobolev embedding $\mathcal{W}^{\eps,p}(\R^d) \subset L^{\infty}(\R^d)$ for $\eps>\frac{d}{p}$, we have 
 \begin{eqnarray*} 
  \| u\|_{\mathcal{B}^{\alpha} _{\infty,\infty}(\R^d) } &= & \sup_{j \geq -1}   \|2^{j\alpha}  {\delta}_j u \|_{L^\infty(\R^d)}\nonumber\\
  &\lesssim  & \sup_{j \geq -1}   \|2^{j\alpha}  {\delta}_j u \|_{\W^{\eps,p}(\R^d)}= \sup_{j \geq -1}   \|2^{j\alpha} H^{\frac{\eps}2} {\delta}_j u \|_{L^{p}(\R^d)}.
 \end{eqnarray*}
 Let $j \geq 0$ (the case $j=-1$ is similar), then 
 $$\dis 2^{j\alpha} H^{\frac{\eps}2} {\delta}_j  =H^{\frac{\alpha+\eps}2}\big(\frac{2^{2j}}{H}\big)^{\frac{\alpha}2} \theta(\frac{H}{2^{2j}}) =H^{\frac{\alpha+\eps}2} \theta' (\frac{H}{2^{2j}})= H^{\frac{\alpha+\eps}2}\delta'_j,$$
 where $\theta'  \in \mathcal{C}_0^{\infty}(\R)$ is the function defined by $\dis \theta': \xi \mapsto \frac{\theta(\xi)}{\xi^{\frac{\alpha}2}}$, and $\delta'_j=\theta'(\frac{H}{2^{2j}})$. As a result,
 \begin{multline*}
  \| u\|_{\mathcal{B}^{\alpha} _{\infty,\infty}(\R^d) } \lesssim \sup_{j \geq -1}   \|  H^{\frac{\alpha+ \eps}2} {\delta}'_j u \|_{L^{p}(\R^d)} \lesssim \\
  \lesssim  \sup_{j \geq -1}   \| {\delta}'_j H^{\frac{\alpha+ \eps}2}  u \|_{L^{p}(\R^d)}  \lesssim \sup_{j \geq -1}   \|  H^{\frac{\alpha+ \eps}2}  u \|_{L^{p}(\R^d)}  \lesssim \| u\|_{\mathcal{W}^{\alpha+\eps,p}(\R^d) }
  \end{multline*}
where we have applied Proposition~\ref{prop.conti}  to derive the third inequality. 
  \end{proof}
  
 The next result shows the action  in harmonic Besov spaces of the multiplication or derivation with respect to the space variable.

  \begin{lemma}\label{lemBesov} Let $d\geq 1$, $\alpha >0$ and $1\leq p,q \leq \infty$. Then 
  $$\|\langle x\rangle\,  u\|_{\B^\alpha_{p,q}(\R^d)} \leq C \| u\|_{\B^{\alpha+1}_{p,q}(\R^d)}$$
  and 
    \begin{equation}\label{p-nabla}
    \|\nabla u\|_{\B^\alpha_{p,q}(\R^d)} \leq C \| u\|_{\B^{\alpha+1}_{p,q}(\R^d)}.
        \end{equation}
    More generally, for any $s \geq 0$
      \begin{equation}\label{multbesov}
      \|\langle x\rangle^s u\|_{\B^\alpha_{p,q}(\R^d)} \leq C \| u\|_{\B^{\alpha+s}_{p,q}(\R^d)}
           \end{equation}
  and 
    $$\|\langle -\Delta\rangle^{\frac{s}2} u\|_{\B^\alpha_{p,q}(\R^d)} \leq C \| u\|_{\B^{\alpha+s}_{p,q}(\R^d)}.$$
\end{lemma}

 \begin{proof} 
 For instance, we prove the estimate \eqref{multbesov}, the others are obtained similarly. Write $\dis u=\sum_{k \geq -1} \delta_k u$, so that for all $j \geq -1$
 $$2^{j\alpha} \big \| \delta_j(\langle x\rangle^s u)\big\|_{L^p(\R^d)} \leq 2^{j\alpha} \sum_{k\geq -1}  \| \delta_j(\langle x\rangle^s \delta_k u)\|_{L^p(\R^d)}=\mathcal{A}_j(u)+\mathcal{B}_j(u),$$
where $\dis {\mathcal{A}}_j(u):=2^{j\alpha} \sum_{k\geq j-3}  \| \delta_j(\langle x\rangle^s \delta_k u)\|_{L^p(\R^d)}$ and $\dis \mathcal{B}_j(u):=2^{j\alpha} \sum_{k\leq  j-4}  \| \delta_j(\langle x\rangle^s \delta_k u)\|_{L^p(\R^d)}$. \medskip

  For $n \in \big\{k, j\big\}$, set $h_n=2^{-2n}$, and for any function $f$, we define $\widetilde{f}_n$ by $\dis \widetilde{f}_n(x)=f(\frac{x}{\sqrt{h_n}})$. Then we have  
\begin{equation}\label{eq11} 
\delta_n f(x)= \big(\theta(h_nH)f\big)(x)=\big(\theta(-h_n^2\Delta+|x|^2)\widetilde{f}_n\big)(\sqrt{h}_nx).
\end{equation}~

$\bullet$  We first  show that $\|\mathcal{A}_j(u) \|_{\ell^q_{j \geq -1}}  \leq C\|u\|_{\B^{s+\alpha}_{p,q}(\R^d)}$. Firstly, by Proposition~\ref{prop.conti}, the operator  $\delta_j$ is  bounded  in $L^p(\R^d)$, uniformly in $j\geq -1$, hence 
\begin{equation}\label{ineg-Y}
 \big \| \delta_j(\langle x\rangle^s \delta_k u)\big\|_{L^p(\R^d)} \leq  C \big\| \langle x\rangle^s \delta_k u\big\|_{L^p(\R^d)}.
\end{equation}
 Let $n_0 \in \N$ such that $s \leq 2n_0$. Then by the H\"older inequality, \eqref{ineg-Y} gives
 \begin{equation}\label{ineg-Y2}
 \big \| \delta_j(\langle x\rangle^s \delta_k u)\big\|_{L^p(\R^d)} \leq  C \big\|   \delta_k u\big\|^{1-\frac{s}{2n_0}}_{L^p(\R^d)} \big\| \langle x\rangle^{2n_0} \delta_k u\big\|^{\frac{s}{2n_0}}_{L^p(\R^d)}.
 \end{equation}
By Proposition \ref{expension} (applied with $N=1$), it is enough to replace $\delta_k u$ with $\big[Op_{h_k}\big(\theta(|x|^2+|\xi|^2) \big)   \theta_0 \tilde{u}_k  \big]   (\sqrt{h_k}x)$ in the r.hs. of \eqref{ineg-Y2}. Then using that $x \mapsto (h_k+|x|^2)^{n_0}$ is a function of~$x$ only (which does not depend on $\xi$), we get
\begin{multline} 
\big\| \langle x\rangle^{2n_0} \big[Op_{h_k}\big(\theta(|x|^2+|\xi|^2) \big)   \theta_0 \tilde{u}_k  \big]   (\sqrt{h_k}x)\big\|_{L^p(\R^d)} =h_k^{-\frac{d}{2p}}\Big\| \Big[ \langle \frac{x}{\sqrt{h_k}}\rangle^{2n_0} Op_{h_k}\big(\theta(|x|^2+|\xi|^2)  \big) \Big] \theta_0 \tilde{u}_k    \Big\|_{L^p(\R^d)}\\
 = h_k^{-\frac{d}{2p}-n_0} \big\| Op_{h_k}\Big((h_k+|x|^2)^{n_0} \theta(|x|^2+|\xi|^2)   \Big)  \theta_0 \tilde{u}_k   \big\|_{L^p(\R^d)}.\label{multiu}
\end{multline}
Thus by Theorem \ref{1-operator}, using that $(x ,\xi) \longmapsto (h_k+|x|^2)^{n_0} \theta(|x|^2+|\xi|^2) \in S^0$ uniformly in $h_k>0$,
$$ \big\| Op_{h_k}\Big((h_k+|x|^2)^{n_0} \theta(|x|^2+|\xi|^2)\Big)    \theta_0 \tilde{u}_k    \big\|_{L^p(\R^d)} \leq C \big\|    \theta_0 \tilde{u}_k     \big\|_{L^p(\R^d)} \leq h_k^{\frac{d}{2p}}  \|    u   \|_{L^p(\R^d)}.$$
Consider $\delta'_k$ as in \eqref{def-deltap}. Using the previous estimate, \eqref{ineg-Y2} and \eqref{multiu} we get 
$$\big\| \delta_j(\langle x\rangle^s \delta_k u)\big\|_{L^p(\R^d)} \leq  C h_k^{-\frac{s}2}  \big\|    \delta'_ku     \big\|_{L^p(\R^d)} = C 2^{ks}\big\|    \delta'_ku     \big\|_{L^p(\R^d)}.$$

Then 
$$\mathcal{A}_j(u)\leq C   \sum_{k\geq j-3}  2^{-\alpha(k-j)}\big( 2^{k(s+\alpha)} \big\|    \delta'_ku     \big\|_{L^p(\R^d)} \big).$$
Set $c_k=2^{-\alpha k }$ and $d_k(u)= 2^{k(s+\alpha)} \big\|    \delta'_ku     \big\|_{L^p(\R^d)}$, then $\mathcal{A}_j(u) \leq C \big(c_k \star d_k(u)\big)_j$. Observe that since $\alpha >0$ we have $c \in \ell^1$, and $d(u) \in \ell^q$ with $ \|d(u) \|_{\ell^q} \leq C\|u\|_{\B^{s+\alpha}_{p,q}(\R^d)}$. By the Young inequality, 
$$\|\mathcal{A}_j(u) \|_{\ell^q} \leq C \|c \|_{\ell^1} \|d(u) \|_{\ell^q} \leq C\|u\|_{\B^{s+\alpha}_{p,q}(\R^d)}, $$
which was the claim. \medskip

$\bullet$ Assume that $k \leq j-4$. We will show that for all $N \geq 1$, 
\begin{equation}\label{bornepseudo}
\| \delta_j(\langle x\rangle^s \delta_k u)\|_{L^p(\R^d)} \leq C_N 2^{-Nj}   \|  \delta'_k u\|_{L^p(\R^d)} .  
\end{equation}

 The idea of the proof of \eqref{bornepseudo} is to use semi-classical analysis in order to represent $ \delta_j\big(\langle x\rangle^s \delta_k u\big)$ as an oscillatory integral and then to  integrate by parts (application of  the non-stationary lemma) in order to gain arbitrary powers of the semiclassical parameter $h_j=2^{-2j}$. This is in the spirit of the proof of \cite[Lemma 2.6]{BGT}.\medskip
 
By duality and \eqref{eq11} we have
\begin{multline*}
\| \delta_j(\langle x\rangle^s \delta_k u)\|_{L^p(\R^d)} = \sup_{ \|F\|_{L^{p'}(\R^d)} \leq 1}\int_{\R^d} dx \, \langle x\rangle^s   (\delta_{k} u) (\delta_{j} F) \\
  \begin{aligned}
&=  \sup_{ \|F\|_{L^{p'}(\R^d)} \leq 1}\int_{\R^d}dx \,   \langle x\rangle^s   \Big[\big(\theta(-h_k^2\Delta+|x|^2)\widetilde{u}_k\big)(\sqrt{h}_k x)\Big]    \Big[\big(\theta(-h_j^2\Delta+|x|^2)\widetilde{F}_j\big)(\sqrt{h}_j x)\Big] .
  \end{aligned}
\end{multline*}
By Proposition \ref{expension} it is enough to show that for all $N \geq 1$
\begin{eqnarray} 
I_h&:=&\int_{\R^d} dx \,   \langle x\rangle^s   \Big[\big(Op_{h_k}\big(\theta(|x|^2+|\xi|^2)\big) \theta_0\widetilde{u}_k\big) (\sqrt{h_k} x) \Big]   \Big[\big(Op_{h_j}\big(\theta(|x|^2+|\xi|^2)\big) \theta_0\widetilde{F}_j\big)(\sqrt{h_j} x)  \Big] \nonumber \\
& \leq&C_N h^{cN}_j  \|F\|_{L^{p'}(\R^d)} \|   u\|_{L^p(\R^d)}. \label{bouih1}
\end{eqnarray}
By definition \eqref{defi-2},
$$\big(Op_h\big(\theta(|x|^2+|\xi|^2)\big) (\theta_0 G) \big)(\sqrt{h} x)=(2 \pi h )^{-d} \int_{\R^d } d \xi \, e^{i\frac{  x \cdot \xi}{\sqrt{h} }} \theta\big(h |x|^2+|\xi|^2\big) \widehat{\theta_0 G }(\frac{\xi}h) , $$
thus 
\begin{equation}\label{inte1-0}
I_h= (2\pi )^{-2d}( h_k   h_j)^{-d} \int_{(\R^d)^3} dxd \xi d \eta   \, e^{i\frac{\Phi_h(x, \xi, \eta)} {\sqrt{h_j} }} a_h(x, \xi, \eta)\widehat{\theta_0 \tilde{u}_k }(\frac{\eta}{h_k})   \widehat{\theta_0 \tilde{F}_j }(\frac{\xi}{h_j}),
\end{equation}
with $\Phi_h(x, \xi, \eta):= \sqrt{h_j} x \cdot (\frac{\xi}{\sqrt{h_j}}+ \frac{\eta}{\sqrt{h_k}})$ and where
$$a_h(x, \xi, \eta):=  \langle x\rangle^s  \theta\big(h_j |x|^2+|\xi|^2 \big)  \theta\big(h_k |x|^2+|\eta|^2\big) $$
 is a compactly supported function.

Now we claim that on the support of $a_h$ one has 
\begin{equation}\label{nablap}
|\nabla_x \Phi_h |= \sqrt{h_j} \,\big| \frac{\xi}{\sqrt{h_j}}+ \frac{\eta}{\sqrt{h_k}}  \big| \geq c>0,
\end{equation}
for some absolute constant $c>0$. To begin with, using that $\text{Supp} \,\theta  \subset \big\{y \in \R_+ :\; \;  \big(\frac34\big)^2 \leq y \leq \big(\frac83\big)^2  \big\}$, we directly obtain the bounds 
  $ |\xi| ,  |\eta|   \leq \frac{8}3$. Then  the condition  $j-k\geq 4$ implies
  $$\frac{h_j}{h_k} \leq \frac1{16^2}.$$
    Next,  on the support of $a_h$, we have  $h_j |x|^2+|\xi|^2 \geq \big(\frac34\big)^2$ and $h_k |x|^2+|\eta|^2 \leq \big(\frac83\big)^2$, so that $ |x|^2  \leq  h^{-1}_k \big(\frac83\big)^2$ and thus 
\begin{equation}\label{xiinf}
|\xi|^2 \geq (\frac34)^2- {h_j}  |x|^2 \geq (\frac34)^2-\frac{h_j}{h_k} (\frac83)^2  \geq  (\frac34)^2-\frac1{6^2} = \frac{77}{144}\geq \frac12.
\end{equation}
   As a consequence, on the support of $a_h$, we have  
\begin{align*}
 \big| \frac{\xi}{\sqrt{h_j}}+ \frac{\eta}{\sqrt{h_k}} \big| \geq   \frac{1}{\sqrt{h_j}} \Big( |\xi| -\sqrt{\frac{h_j}{h_k}} |\eta|\Big)   \geq  \frac{1}{\sqrt{h_j}} \Big( \frac{\sqrt{2}}{2}-\frac{1}{6}\Big) = \frac{c}{ \sqrt{h_j}}
\end{align*} 
with $c>0$, which is \eqref{nablap}. \medskip

  Observe that  $a_h$ is compactly supported and that on its support we have $|x| \lesssim h^{-\frac12}_k$, and using this fact, we can check for all $N \geq 1$, 
  $$ | (\partial^N_x a_h)(x, \xi, \eta)|    \leq C_N h_k^{-\frac{s}2}   \leq C_N h_k^{-\frac{s}2}.   $$

	\
	
Then, by \eqref{nablap}, we can apply the non stationary phase lemma (see {\it e.g.} \cite[Lemma 3.14]{Zworski}), and  we get that for all $N\geq 1$, 
\begin{align}
 |I_h| &\leq C_N h_j^{\frac{N}2-K_1} \int_{(\R^d)^3}d \xi d \eta   dx \,  | (\partial^N_x a_h)(x, \xi, \eta)||\widehat{\theta_0 \widetilde{u_k} }(\frac{\eta}{h_k})| | \widehat{\theta_0 \tilde{F}_j }(\frac{\xi}{h_j})|  \nonumber \\
&\leq  C_N h_j^{\frac{N}2-K_2}   \big\|  \widehat{\theta_0 \tilde{u}_k }  \big\|_{L^\infty(\R^d)}   \big\|   \widehat{\theta_0 \tilde{F}_j }  \big\|_{L^\infty(\R^d)}, \label{RHS1}
\end{align}
where $K_1, K_2>0$ are some absolute constants coming from the changes of variables and from the computation of the integral. For $G \in L^q(\R^d)$ and all $\xi \in \R^d$, one has obviously
\begin{align} 
| \widehat{\theta_0 G }(\xi) |&\leq \int_{\R^d} dy \, \big|\theta_0(y) G(y)\big| \leq C \|G\|_{L^q(\R^d)}\label{bG},
\end{align} 
and therefore
$$\big\|  \widehat{\theta_0 \tilde{u}_k }  \big\|_{L^\infty(\R^d)}\lesssim \|\tilde{u}_k\|_{L^{p}(\R^d)} \lesssim \|u\|_{L^{p}(\R^d)}, \quad \big\|  \widehat{\theta_0 \tilde{F}_j }  \big\|_{L^\infty(\R^d)} \lesssim \|\tilde{F}_j \|_{L^{p'}(\R^d)} \lesssim \|F\|_{L^{p'}(\R^d)}.$$

\medskip

Going back to \eqref{RHS1}, we deduce the existence of an absolute constant $K>0$ such that for all $N\geq 1$
$$ |I_h| \leq  C_N h_j^{\frac{N}2-K} \|u\|_{L^{p}(\R^d)} \|F\|_{L^{p'}(\R^d)},$$
which implies the bound \eqref{bouih1}.
  \end{proof}
 
  \begin{lemma}\label{lemBesov2} Let $\alpha < 1$ and $1\leq p \leq \infty$. Then for all $k \geq 0$
  $$\|\nabla S_{k} u\|_{L^p(\R^d)} \leq C 2^{k(1-\alpha)}\| u\|_{\B^{\alpha}_{p,\infty}(\R^d)},$$
  and 
  $$\| S_{k} u\|_{L^{p}(\R^d)}\leq C 2^{k(1-\alpha)}\| u\|_{\B^{\alpha}_{p,\infty}(\R^d)}.$$
\end{lemma}

  \begin{proof} 
	
By proceeding as in the proof of Lemma~\ref{lemBesov} (with $\langle x\rangle^s$ replaced with $\nabla=i h_\ell^{-1} Op_{h_\ell}(\xi)$), we can show that 
$$ \|  \nabla \delta_{\ell} u\|_{L^p(\R^d)}  \lesssim 2^{\ell}\big\|    \delta'_{\ell}u     \big\|_{L^p(\R^d)}.$$
Consequently,
\begin{align*}
\|\nabla S_{k} u\|_{L^p(\R^d)}\leq \sum_{\ell \leq k-1} \|  \nabla \delta_{\ell} u\|_{L^p(\R^d)} & \lesssim \sum_{\ell \leq k-1}2^{\ell}\big\|    \delta'_{\ell}u     \big\|_{L^p(\R^d)}\\
&\lesssim \Big(\sum_{\ell \leq k-1}2^{\ell(1-\al)}\Big)\| u\|_{\B^{\alpha}_{p,\infty}(\R^d)}\lesssim 2^{k(1-\alpha)}\| u\|_{\B^{\alpha}_{p,\infty}(\R^d)} ,
\end{align*}
which yields the first assertion. The second assertion is even easier to derive, and we omit the details.
   \end{proof}


  \subsection{Product and paraproduct  estimates}
 
  In the following result, we gather some useful bilinear estimates. For the proof, we refer to \cite{FI}.
 \begin{proposition}\label{Prop-est-para}
 Let $\alpha, \beta \in \R$ and $1 \leq p,p_1,p_2,q \leq \infty$ be such that $\frac1{p_1}+\frac1{p_2}=\frac1{p}$.
 \begin{enumerate}[$(i)$]
 \item If $\alpha+\beta>0$, then the mapping $(f,g) \mapsto f \pe g$ extends to a continuous bilinear map from $\mathcal{B}^{\alpha}_{p_1,q}(\R^d) \times \mathcal{B}^{\alpha}_{p_2,q}(\R^d) $ to $\mathcal{B}^{\alpha+\beta}_{p,q}(\R^d) $.
 \item The mapping $(f,g) \mapsto f \pl g$ extends to a continuous bilinear map from $L^{p_1}(\R^d) \times \mathcal{B}^{\beta}_{p_2,q}(\R^d)$ to $\mathcal{B}^{\beta}_{p,q}(\R^d)$.
 \item If $\alpha <0$, then the mapping $(f,g) \mapsto f \pl g$ extends to a continuous bilinear map from $\mathcal{B}^{\alpha}_{p_1,q}(\R^d)\times \mathcal{B}^{\beta}_{p_2,q}(\R^d)$ to $\mathcal{B}^{\alpha+\beta}_{p,q}(\R^d)$.
  \item If $\alpha <0< \beta$ and $\alpha+\beta>0$, then the mapping $(f,g) \mapsto f  g$ extends to a continuous bilinear map from $\mathcal{B}^{\alpha}_{p_1,q}(\R^d)\times \mathcal{B}^{\beta}_{p_2,q}(\R^d)$ to $\mathcal{B}^{\alpha}_{p,q}(\R^d)$.
    \item If $\alpha >0$, then the mapping $(f,g) \mapsto f  g$ extends to a continuous bilinear map from $\mathcal{B}^{\alpha}_{p_1,q}(\R^d)\times \mathcal{B}^{\alpha}_{p_2,q}(\R^d)$ to $\mathcal{B}^{\alpha}_{p,q}(\R^d)$. Moreover, for $1 \leq p_3, p_4 \leq \infty$ such that $\frac1{p_1}+\frac1{p_2}=\frac1{p_3}+\frac1{p_4}=\frac1{p}$, there exists $C>0$ satisfiyng
    \begin{equation*}
    \| fg \|_{ \mathcal{B}^{\alpha}_{p,q}(\R^d)} \leq C \big( \| f \|_{L^{p_1}(\R^d)}   \| g \|_{ \mathcal{B}^{\alpha}_{p_2,q}(\R^d)}+ \| f \|_{ \mathcal{B}^{\alpha}_{p_3,q}(\R^d)}  \| g \| _{L^{p_4}(\R^d)}       \big).
    \end{equation*} 
 \end{enumerate}
  \end{proposition}


  \subsection{Regularization by heat flow}
 
 Recall that  $K_t(x,y)$  is the heat kernel   of 
${e}^{-t H}$.  It is given by the  following Mehler formula   in dimension~$d\geq 1$
\begin{eqnarray}  \label{mehler}
K_t(x,y)  &=& \sum_{n \geq 0} e^{-t \lambda_n} \varphi_n(x)\varphi_n(y) \nonumber\\
&=& (2\pi\sinh 2t)^{-\frac{d}2} \exp\left(-\frac{\tanh t}{4}\vert x+y\vert^2- \frac{ \vert x-y\vert^2}{4\tanh t}\right).
\end{eqnarray}

 \begin{lemma}\label{lem:actisemi}
Let $d\geq 1$. Let $\alpha, \beta \in \R$ and $1 \leq p, q\leq \infty$.
 \begin{enumerate}[$(i)$]
 \item For all $t\geq 0$
  \begin{equation}\label{heat-Lp}
 \|  e^{-tH } u \|_{L^p(\R^d)} \leq C e^{- d t}  \|  u \|_{L^p(\R^d)}.
        \end{equation}
 \item If $\alpha \geq \beta$, then there exists $C>0$ such that for all $t>0$
 \begin{equation}\label{heat1}
 \| e^{-t H} u \|_{\B^\alpha_{p,q} (\R^d)}\leq C t^{-\frac{\alpha-\beta}2} e^{-\frac{dt}2} \|  u \|_{\B^\beta_{p,q}(\R^d) }.
       \end{equation}
      \item If $0 \leq \beta -\alpha \leq 2$, then there exists $C>0$ such that for all $t>0$
 \begin{equation}\label{heat2}
  \| (1-e^{-t H} )u \|_{\B^\alpha_{p,q} (\R^d)}\leq C t^{-\frac{\alpha-\beta}2}  \|  u \|_{\B^\beta_{p,q}(\R^d) }.
       \end{equation}  
  \end{enumerate}
 \end{lemma}

 \begin{proof}
$(i)$  From \eqref{mehler} we get that for $t \geq 0$, 
 \begin{equation} \label{borno}
 \| K_t (x,\cdot) \|_{L^{1}_y(\R^d)} \lesssim e^{-dt }, \qquad \| K_t (\cdot,y) \|_{L^{1}_x(\R^d)} \lesssim e^{-dt }.
 \end{equation}
Then  from the representation $\dis e^{-tH}u(x)=\int_{\R^d}dy \, K_t(x,y) u(y)$ and \eqref{borno}, we easily get 
 $$ \|  e^{-tH } u \|_{L^\infty(\R^d)} \leq C e^{- d t}  \|  u \|_{L^{\infty}(\R^d)},  \qquad  \|  e^{-tH } u \|_{L^1(\R^d)} \leq C e^{- d t}  \|  u \|_{L^1(\R^d)}, $$
 and by interpolation, we get \eqref{heat-Lp} for all $1\leq p \leq +\infty$. \medskip
 
 $(ii)$ We first show   $\| e^{-t H} u \|_{\B^\alpha_{p,q}(\R^d) }\leq C t^{-\frac{\alpha-\beta}2}   \|  u \|_{\B^\beta_{p,q}(\R^d) }$: 
 let $\theta'  \in \mathcal{C}^{\infty}_0$ 
 be such that $\theta' \theta=\theta$, then $\delta'_k \delta_k = \delta_k$, and we have
 \begin{eqnarray*}
\| e^{-tH} u \|_{\B^{\sigma}_{p,q}(\R^d)}  &=&\big\| 2^{\sigma k} \| \delta_k e^{-tH} u \|_{L^p(\R^d)} \big\|_{\ell^q_{k \geq -1}} \\ 
&=& \big\| \| \delta'_k \left(\frac{H}{2^{2k}} \right)^{-\frac{\sigma}{2}}  H^{\frac{\sigma}{2}}  e^{-tH} \delta_k u  \|_{L^p(\R^d)}  \big\|_{\ell^q_{k \geq -1}} \\
 &\lesssim & \big\| \| H^{\frac{\sigma}{2}}  e^{-tH} \delta_k u \|_{L^p(\R^d)}   \big\|_{\ell^q_{k \geq -1}} \\
 & \lesssim & t^{-\frac{\sigma}{2}} \| u\|_{\B^{0}_{p,q}(\R^d)} 
  \end{eqnarray*} 
where the third inequality follows from Proposition \ref{prop.conti}, and the last inequality has been proved in~\cite[Lemma 3]{dBDF1} for $\sigma \in (0,2].$
 The paper \cite{dBDF1} treats only dimension 2, but the same inequality holds in the general dimension case, too.   
 
 Next we show that there exists   $C>0$ such that 
 $ \|  e^{-tH } u \|_{\B^\beta_{p,q}(\R^d) } \leq C e^{- d t}  \|  u \|_{\B^\beta_{p,q} (\R^d)}$.  Namely, by~\eqref{heat-Lp},  
 \begin{eqnarray*}
\| e^{-tH} u \|_{\B^{\beta}_{p,q}(\R^d)}  &=& \big\| 2^{\beta k} \| \delta_k e^{-tH} u \|_{L^p(\R^d)} \big\|_{\ell^q_{k \geq -1} }
= \big\| 2^{\beta k} \|  e^{-tH} \delta_k u  \|_{L^p(\R^d)}  \big\|_{\ell^q_{k \geq -1}} \\
 &\lesssim & \big\| 2^{\beta k} e^{-d t} \| \delta_k u \|_{L^p(\R^d)}   \|_{\ell^q_{k \geq -1}} 
  \lesssim  e^{-d t} \| u\|_{\B^{\beta}_{p,q}(\R^d)} .
  \end{eqnarray*}   
 Hence, from the previous lines, we deduce
 $$ \| e^{-t H} u \|_{\B^\alpha_{p,q}(\R^d) }=  \| (e^{- \frac{t}2H})( e^{-\frac{t}2H}u) \|_{\B^\alpha_{p,q} (\R^d) }
 \lesssim   t^{-\frac{\alpha-\beta}2}  \|  e^{-\frac{ t}{2} H}u \|_{\B^\beta_{p,q}(\R^d)  } \lesssim    t^{-\frac{\alpha-\beta}2} e^{-\frac{d t}2} \|  u \|_{\B^\beta_{p,q}(\R^d)  }.$$
 
 $(iii)$ In the case $\alpha=\beta$, the inequality (\ref{heat2}) is a direct consequence of \eqref{heat1}. When $\alpha<\beta$, write 
  \begin{equation*}
 (1-e^{-tH}) u = -\int_0^t  ds \,  \frac{d}{ds} \big[e^{-sH} u\big]  = \int_0^t  ds \, He^{-sH}u ,
 \end{equation*}
then take the norm $  \| . \|_{\B^\alpha_{p,q} (\R^d)}$ and use~\eqref{heat1} to conclude.
\end{proof}

 We also have the Schauder estimate.

  \begin{lemma}\label{L-Schauder}
 Assume that $(\partial_t+H)u=v$ with $u(0)=u_0$. Let $\alpha \in \R$ and $1\leq p,q \leq \infty$. Then for all $\eps>0$ and $T>0$, 
 $$\|u\|_{\mathcal{C}^0([0,T]; \B^{\alpha}_{p,q}(\R^d))}  \leq C_{\eps}\|v\|_{\mathcal{C}^0([0,T];\B^{\alpha-2+\eps}_{p,q}(\R^d))} +C   \|  u_0 \|_{\B^\alpha_{p,q} (\R^d)},$$
with  $C_\eps>0$ does not depend on $T>0$.
  \end{lemma}

 \begin{proof}
 We have $\dis u(t)=e^{-t H}u_0+\int_0^t ds \, e^{-(t-s)H}v(s)$. By \eqref{heat1} with $\beta=\alpha$,  $\| e^{-t H} u_0 \|_{\B^\alpha_{p,q} (\R^d)}\leq C   \|  u_0 \|_{\B^\alpha_{p,q}(\R^d) }$. Then, by the Minkowski inequality and \eqref{heat2} with $\beta=\alpha-2+\eps$, for all $0<t<T$
\begin{eqnarray*}
\big\|\int_0^t ds \, e^{-(t-s)H}v(s)\big\|_{\B^\alpha_{p,q} (\R^d)}& \leq &C \int_0^t ds \,  \big\|e^{-(t-s)H}v(s)\big\|_{\B^\alpha_{p,q}(\R^d) }  \\
&\leq & C  \int_0^t ds \, (t-s)^{-1+\frac{\eps}2} e^{-\frac12(t-s)}\|v(s)\|_{\B^{\alpha-2+\eps}_{p,q}(\R^d) } \\
&\leq & C\|v\|_{\mathcal{C}([0,T];\B^{\alpha-2+\eps}_{p,q}(\R^d))}   \int_0^t ds \, s^{-1+\frac{\eps}2} e^{-\frac{s}2} \\
&\leq & C\|v\|_{\mathcal{C}([0,T];\B^{\alpha-2+\eps}_{p,q}(\R^d)) }  
 \end{eqnarray*}
 where the previous constant $C>0$ does not depend on $T>0$.
  \end{proof}


   \subsection{Interaction of spectrally localized functions}

 The next result shows that in some regime the interaction of two spectrally localized functions is neglectable. 
  
    \begin{lemma}\label{lem.SC}
  Let $k\geq 4$, $j\leq k-4$ and $\ell \leq k-4$. Let $p,p_1,p_2 \in [1, \infty]$ be such that $\frac{1}{p}= \frac{1}{p_1}+ \frac{1}{p_2}$.  Then for all $f \in L^{p_1}(\R^d)$, $g \in L^{p_2}(\R^d)$ and all $N \geq 1$
\begin{equation}\label{bound22}
\| \delta_j\big(\delta_\ell f\delta_k g\big)\|_{L^p(\R^d)} \leq C_N 2^{- kN } \| \delta'_{\ell} f\|_{L^{p_1}(\R^d)}\| \delta'_k g\|_{L^{p_2}(\R^d)}.
\end{equation}
As a consequence, 
\begin{equation}\label{bound23}
\| \delta_j\big(\delta_\ell f\delta_k g\big)\|_{L^p(\R^d)} \leq C_N 2^{- kN } \| S_{k}f\|_{L^{p_1}(\R^d)}\| \delta'_k g\|_{L^{p_2}(\R^d)}.
\end{equation}
\end{lemma}

Here is a major difference with the periodic setting (namely, when $H$ is replaced with $\Delta_{\T^d}$ as in~\cite{MW}). In the periodic setting, under the assumptions of Lemma \ref{lem.SC}, by the properties of the periodic convolution, we have $\delta_j\big(\delta_\ell f\delta_k g\big) \equiv 0$. However, in our case,  the exponential decay of these interactions, as described in Lemma \ref{lem.SC}, allows to treat them as perturbations.

\begin{proof} Let us first show how \eqref{bound22} implies \eqref{bound23}. If $\theta'$ is chosen close enough to $\theta$ and such that $\theta' \theta= \theta$, then $\delta'_\ell S_{k}= \delta'x_\ell$ for all $\ell \leq k-4$. Thus, by continuity of $\delta'_{\ell}$ (Proposition \ref{prop.conti}), 
$$\| \delta'_{\ell} f\|_{L^{p_1}(\R^d)}= \| \delta'_{\ell} S_{k}f\|_{L^{p_1}(\R^d)} \leq C \| S_{k}f\|_{L^{p_1}(\R^d)},$$
hence the result. \medskip

We now turn to the proof of \eqref{bound22}. The argument is close to the proof of \eqref{bornepseudo}. In the following we can assume that $f= \delta'_{\ell} f $ and $g=\delta'_k g$.  For $n \in \big\{k, j, \ell\big\}$, set $h_n=2^{-2n}$, and use the representation \eqref{cht}. By duality  we have
\begin{multline*}
  \|\delta_j\big(\delta_\ell f\delta_k g\big)\|_{L^p(\R^d)} = \sup_{ \|F\|_{L^{p'}(\R^d)} \leq 1}\int_{\R^d} dx \, (\delta_{\ell} f)  (\delta_{k} g) (\delta_{j} F) \\
  \begin{aligned}
&=  \sup_{ \|F\|_{L^{p'}(\R^d)} \leq 1}\int_{\R^d}dx \, \Big[\big(\theta(-h_\ell^2\Delta+|x|^2)\widetilde{f}_\ell\big)(\sqrt{h}_\ell x)\Big]  \\
& \hspace{4cm} \Big[\big(\theta(-h_k^2\Delta+|x|^2)\widetilde{g}_k\big)(\sqrt{h}_k x)\Big]  \Big[\big(\theta(-h_j^2\Delta+|x|^2)\widetilde{F}_j\big)(\sqrt{h}_j x)\Big] .
  \end{aligned}
\end{multline*}
By Proposition \ref{expension} it is enough to show that for all $N \geq 1$
\begin{multline*} 
J_h:=\int_{\R^d} dx \, \Big[\big(Op_{h_\ell}\big(\theta(|x|^2+|\xi|^2)\big) \theta_0\widetilde{f}_\ell\big) (\sqrt{h_\ell} x) \Big]   \Big[\big(Op_{h_k}\big(\theta(|x|^2+|\xi|^2)\big) \theta_0\widetilde{g}_k\big)(\sqrt{h_k} x)  \Big]  \\
 \Big[ \big(Op_{h_j}\big(\theta(|x|^2+|\xi|^2)\big) \theta_0\widetilde{F}_j\big)(\sqrt{h_j} x)  \Big] \leq
 \\ \leq C_N h_k^{cN } \|f\|_{L^{p_1}(\R^d)}\|g\|_{L^{p_2}(\R^d)}\|F\|_{L^{p'}(\R^d)}. 
\end{multline*}
By definition \eqref{defi-2},
$$\big(Op_h\big(\theta(|x|^2+|\xi|^2)\big) (\theta_0 G) \big)(\sqrt{h} x)=(2 \pi h )^{-d} \int_{\R^d } d \xi \, e^{i\frac{  x \cdot \xi}{\sqrt{h} }} \theta\big(h |x|^2+|\xi|^2\big) \widehat{\theta_0 G }(\frac{\xi}h) , $$
thus 
\begin{equation}\label{inte1}
J_h= (2\pi )^{-3d}( h_k h_\ell h_j)^{-d} \int_{(\R^d)^4} dxd \xi d \eta d \mu \, e^{i\frac{\Psi_h(x, \xi, \eta, \mu)} {\sqrt{h_k} }} b_h(x, \xi, \eta, \mu)\widehat{\theta_0 \tilde{g}_k }(\frac{\xi}{h_k}) \widehat{\theta_0 \tilde{f}_\ell }(\frac{\eta}{h_\ell}) \widehat{\theta_0 \tilde{F}_j }(\frac{\mu}{h_j}),
\end{equation}
with $\Psi_h(x, \xi, \eta, \mu):= \sqrt{h_k} x \cdot (\frac{\xi}{\sqrt{h_k}}+\frac{\eta}{\sqrt{h_\ell}}+\frac{\mu}{\sqrt{h_j}})$ and where
$$b_h(x, \xi, \eta, \mu):=  \theta\big(h_k |x|^2+|\xi|^2\big) \theta\big(h_\ell |x|^2+|\eta|^2 \big)\theta\big(   h_j |x|^2+|\mu|^2\big) $$
 is a compactly supported function.  \medskip

Now we claim that on the support of $b_h$ one has 
\begin{equation}\label{nablaphase}
|\nabla_x \Psi_h |= \sqrt{h_k} \,\big| \frac{\xi}{\sqrt{h_k}}+\frac{\eta}{\sqrt{h_\ell}}+\frac{\mu}{\sqrt{h_j}}  \big| \geq c>0,
\end{equation}
for some absolute constant $c>0$. To begin with, using that $\text{Supp} \,\theta  \subset \big\{y \in \R_+ :\;   \big(\frac34\big)^2 \leq y \leq \big(\frac83\big)^2  \big\}$, we directly obtain the bounds 
  $ |\xi| ,  |\eta|, |\mu|   \leq \frac{8}3$. Then  the conditions $k-\ell\geq 4$ and $k-j\geq 4$ imply   
  $$\frac{h_k}{h_j} \leq \frac1{16^2}, \qquad   \frac{h_k}{h_\ell} \leq \frac1{16^2}.$$
    Next,  on the support of $b_h$, we have  $h_j |x|^2+|\mu|^2 \leq  \big(\frac83\big)^2$ and $h_k |x|^2+|\xi|^2 \geq \big(\frac34\big)^2$, so that $ |x|^2  \leq  h^{-1}_j \big(\frac83\big)^2$ and thus,  just as in \eqref{xiinf}, we deduce that
  $$|\xi| \geq \frac{\sqrt{2}}{2}.$$
Therefore, on the support of $b_h$, we have 
\begin{align*}
 \big| \frac{\xi}{\sqrt{h_k}}+\frac{\eta}{\sqrt{h_\ell}}+\frac{\mu}{\sqrt{h_j}}  \big| \geq   \frac{1}{\sqrt{h_k}} \Big( |\xi| -\sqrt{\frac{h_k}{h_\ell}} |\eta|-\sqrt{\frac{h_k}{h_j}} |\mu|\Big)   \geq  \frac{1}{\sqrt{h_k}} \Big( \frac{\sqrt{2}}{2}-\frac{1}{3}\Big),
\end{align*}
which corresponds to the bound \eqref{nablaphase}. \medskip

The function $b_h$ is compactly supported and on its support  we have $|x| \lesssim  \min \big( h^{-\frac12}_j,  h^{-\frac12}_\ell \big)$. Consequently, we can check  that all the derivatives in $x$ of $b_h$ are uniformly bounded with respect to $h_k, h_{\ell}$ and~$h_j$.

Then, by \eqref{nablaphase}, we can apply the non stationary phase lemma (see {\it e.g.} \cite[Lemma 3.14]{Zworski}) to the expression  \eqref{inte1}, and  we get that for all $N\geq 1$, 
\begin{align}
 |J_h| &\leq C_N h_k^{\frac{N}2-K_1} \int_{(\R^d)^4}d \xi d \eta d \mu dx \,  | (\partial^N_x b_h)(x, \xi, \eta, \mu)||\widehat{\theta_0 \tilde{g}_k }(\frac{\xi}{h_k})| |\widehat{\theta_0 \tilde{f}_\ell }(\frac{\eta}{h_\ell})|| \widehat{\theta_0 \tilde{F}_j }(\frac{\mu}{h_j})| \nonumber   \\
&\leq  C_N h_k^{\frac{N}2-K_2}   \big\|  \widehat{\theta_0 \tilde{g}_k }  \big\|_{L^\infty(\R^d)}   \big\|  \widehat{\theta_0 \tilde{f}_\ell }  \big\|_{L^\infty(\R^d)}   \big\|  \widehat{\theta_0 \tilde{F}_j }  \big\|_{L^\infty(\R^d)},\label{RHS}
\end{align}
where $K_1, K_2>0$ are some absolute constants coming from the changes of variables and from the computation of the integral. With \eqref{bG}, we can  control each of the terms in the r.h.s of \eqref{RHS}. \medskip

As a conclusion, there exists an absolute constant $K>0$ such that for all $N\geq 1$
$$ |J_h| \leq  C_N h_k^{\frac{N}2-K} \|f\|_{L^{p_1}(\R^d)}\|g\|_{L^{p_2}(\R^d)}\|F\|_{L^{p'}(\R^d)},$$
which implies the bound \eqref{bound22}.
 \end{proof}
 
With the same arguments as in the proof of Lemma \ref{lem.SC}, we get 

  \begin{lemma}\label{LemC14}
  Let $j\geq 4$, $k\leq j-4$ and $\ell \leq j-4$. Let $p,p_1,p_2 \in [1, \infty]$ be such that $\frac{1}{p}= \frac{1}{p_1}+ \frac{1}{p_2}$.  Then for all $f \in L^{p_1}(\R^d)$, $g \in L^{p_2}(\R^d)$ and all $N \geq 1$
\begin{equation}\label{detJ}
\| \delta_j\big(\delta_\ell f\delta_k g\big)\|_{L^p(\R^d)} \leq C_N 2^{- jN } \| \delta'_{\ell} f\|_{L^{p_1}(\R^d)}\| \delta'_k g\|_{L^{p_2}(\R^d)}.
\end{equation}
\end{lemma}

\begin{proof} This proof follows the same lines as the proof of Lemma  \ref{lem.SC}, but here instead of $J_h$, we have to estimate 
\begin{equation*} 
J'_h= (2\pi )^{-3d}( h_k h_\ell h_j)^{-d} \int_{(\R^d)^4} dxd \xi d \eta d \mu \, e^{i\frac{\Xi_h(x, \xi, \eta, \mu)} {\sqrt{h_k} }} b_h(x, \xi, \eta, \mu)\widehat{\theta_0 \tilde{g}_k }(\frac{\xi}{h_k}) \widehat{\theta_0 \tilde{f}_\ell }(\frac{\eta}{h_\ell}) \widehat{\theta_0 \tilde{F}_j }(\frac{\mu}{h_j}),
\end{equation*}
with $\Xi_h(x, \xi, \eta, \mu):= \sqrt{h_j} x \cdot (\frac{\xi}{\sqrt{h_j}}+\frac{\eta}{\sqrt{h_\ell}}+\frac{\mu}{\sqrt{h_j}})$. The details are left here.
\end{proof}


  \subsection{Commutation lemmas}

Assume that $\theta \in \mathcal{C}_0^\infty(\R)$ takes values in $[0,1]$ and 
$$
\text{Supp} \,\theta  \subset \big\{\xi \in \R_+:\; \;  \big(\frac34\big)^2 \leq \xi \leq \big(\frac83\big)^2  \big\}.
$$
Recall that $\delta_k= \theta(\frac{H}{2^{2k}})$.  Here we adapt \cite[Lemma A.10]{MW}.

\begin{lemma}\label{lem-C15}
  For $f,g \in \mathscr{S}(\R^d)$ we define 
$$[\delta_k, f](g)=\delta_k(fg)-f \delta_kg. $$
Let $p,p_1,p_2 \in [1, \infty]$ be such that $\frac{1}{p}= \frac{1}{p_1}+ \frac{1}{p_2}$.  Then there exists $C>0$ such that for every $k \geq 0$, $ \nabla f \in L^{p_1}(\R^d)$ and $g \in L^{p_2}(\R^d)$
\begin{equation}\label{estlem}
 \big\| \big[\delta_k, f\big](g) \|_{L^p(\R^d)} \leq C 2^{-k} \|\nabla f \big\|_{L^{p_1}(\R^d)} \|g \|_{L^{p_2}(\R^d)}. 
\end{equation}
Moreover,  for all $0 \leq \alpha \leq 1$ and $\eps>0$, 
\begin{equation}\label{estlem2}
 \big\| \big[\delta_k, f\big](g) \|_{L^p(\R^d)} \leq C 2^{-\alpha k} \|  f \big\|_{\mathcal{B}^{\alpha+\eps}_{p_1, \infty}(\R^d)} \|g \|_{L^{p_2}(\R^d)}. 
\end{equation}
\end{lemma}
 
   \begin{proof} 
Set $h=2^{-2k}$. Then the estimate \eqref{estlem} can be rewritten as
\begin{equation}\label{cocom}
 \big\| \big[\theta(hH), f\big](g) \|_{L^p(\R^d)} \leq C  h^{\frac12}\|\nabla f \big\|_{L^{p_1}(\R^d)} \|g\|_{L^{p_2}(\R^d)}.
 \end{equation}
 For a function $f$, we define   $\widetilde{f}_h$ by $\dis \widetilde{f}_h(x)=f(\frac{x}{\sqrt{h}})$, and we recall the representation 
$$\big(\theta(hH)f\big)(x)=\big(\theta(-h^2\Delta+|x|^2)\widetilde{f}_h\big)(\sqrt{h}x).$$
Therefore, it is enough to prove that  for all $F,G$
\begin{equation}\label{estim}
 \big\| \big[\theta(-h^2\Delta+|x|^2),F\,\big](G )\big\|_{L^p(\R^d)} \leq C h \|\nabla F \|_{L^{p_1}(\R^d)} \|G \|_{L^{p_2}(\R^d)}.
 \end{equation}
Indeed, assuming \eqref{estim}, we deduce that
\begin{eqnarray*}
  \big\| \big[\theta(hH), f\big](g)  \big\|_{L^p(\R^d)} &=&  \big\| \big[\theta(-h^2\Delta+|x|^2),\widetilde{f}_h\big]\widetilde{g}_h(\sqrt{h} \cdot) \big\|_{L^p(\R^d)}\\
  &=&h^{-\frac{d}{2p}}  \big\| \big[\theta(-h^2\Delta+|x|^2),\widetilde{f}_h\big]\widetilde{g}_h \big\|_{L^p(\R^d)}\\
    &\leq &C h^{1-\frac{d}{2p}}\|\nabla \widetilde{f}_h \|_{L^{p_1}(\R^d)} \|\widetilde{g}_h \|_{L^{p_2}(\R^d)}\\
        &\leq &C h^{\frac12}\|\nabla f \|_{L^{p_1}(\R^d)} \|g\|_{L^{p_2}(\R^d)},
\end{eqnarray*}
which is \eqref{cocom}. 

Next, by Proposition \ref{expension}, the estimate \eqref{estim} will  be implied by the following bound: for all functions $F$, $G$,
\begin{equation}\label{claim10}
\big\| \big[Op_h\big(\theta(|x|^2+|\xi|^2)\big)\theta_0, F\big](G) \big\|_{L^p(\R^d)} \leq C h\|\nabla F \|_{L^{p_1}(\R^d)} \|G \|_{L^{p_2}(\R^d)},
\end{equation}
which we now prove. \medskip

As a preliminary, denote by
\begin{equation}\label{def-gamma}
\Theta(w,x):=(2\pi)^{-d}\int_{\R^d}d \xi \, e^{iw \cdot \xi} \theta\big(|x|^2+|\xi|^2\big), \quad \quad \Gamma(w):=|w|\sup_{x \in \R^d}\big|\Theta(w,x)\big|,
\end{equation}
and let us show that $\Gamma \in L^1(\R^d)$. Since $\theta$ is compactly supported, the supremum in the definition of $\Gamma$ is in fact a maximum. Thus, for all $w \in \R^d$, there exists $x(w) \in \R^d$ such that~$\Gamma(w)=|w|\Theta\big(w, x(w)\big)$. Now observe that the function $\xi \longmapsto \theta(x(w)+\xi^2)$ belongs to $\mathcal{C}_0^\infty(\R^d)$, and so by integrating by parts, we show that the bounded function
$$w \longmapsto \int_{\R^d}  d\xi \, e^{iw \cdot \xi} \theta\big(|x(w)|^2+|\xi|^2\big)$$
has rapid decay, which in turn implies that $\Gamma \in L^1(\R^d)$. \medskip

In the next lines, we adapt the proof of \cite[Lemma 2.97]{BCD}. We have 
\begin{eqnarray*}
 \big[Op_h\big(\theta(|x|^2+|\xi|^2)\big)\theta_0, F\big](G)(x) &=&(2\pi h)^{-d}\int_{\R^d} \int_{\R^d}  dy d\xi  \, e^{i\frac{( x-y) \cdot\xi}h}\theta \big(|x|^2+|\xi|^2\big)\big(F(y)-F(x)\big)\theta_0(y)G(y)\\
 &=&h^{-d}\int_{\R^d} dy \, \Theta(\frac{x-y}h,x)  \big(F(y)-F(x)\big)\theta_0(y)G(y) .
\end{eqnarray*}
 Now write 
 $$F(y)-F(x)=\int_0^1ds \, (y-x) \cdot \nabla F \big(x+s(y-x)\big) $$
then,  with the change of variables $z=x-y$, and with $\Gamma$ defined in \eqref{def-gamma}, we get that
\begin{eqnarray*}
 \big|\big[ Op_h\big(\theta(|x|^2+|\xi|^2)\big)\theta_0, F\big](G)(x)\big|  &\leq   &C h^{-d}\int_0^1 ds \,  \int_{\R^d} dy \, |x-y| |\Theta(\frac{x-y}h,x) |\big | \nabla F \big(x+s(y-x)\big) \big |   \big |G(y) \big | \\
 &\leq    & Ch^{-d+1}\int_0^1ds \, \int_{\R^d} dy \;  \Gamma(\frac{x-y}h)\big | \nabla F \big(x+s(y-x)\big) \big | \big |G(y) \big | \\
 &=    &C h^{-d+1}\int_0^1ds \, \int_{\R^d} dz \, \Gamma(\frac{z}h)\big | \nabla F \big(x-sz\big) \big |   \big |G(x-z) \big | .
 \end{eqnarray*}
Now we take the $L^p(\R^d)$ norm in $x$ of the previous inequality, and by the Minkowski and the H\"older  inequalities  we get
\begin{eqnarray*}
 \big\| \big[ Op_h\big(\theta(|x|^2+|\xi|^2)\big)\theta_0, F \big](G)\big\|_{L^p(\R^d)}  &\leq  & Ch^{-d+1}\int_0^1ds \, \int_{\R^d} dz\,  \Gamma(\frac{z}h)\| \nabla F\|_{L^{p_1}(\R^d)}  \|G  \|_{L^{p_2}(\R^d)} \\
 &\leq  & C h \|\Gamma\|_{L^1(\R^d)}\|\nabla F\|_{L^{p_1}(\R^d)}  \|G \|_{L^{p_2}(\R^d)} ,
\end{eqnarray*}
 which is \eqref{claim10}, and this completes the proof of \eqref{estlem}. \medskip
 
 By using successively \eqref{estlem}, Lemma \ref{lem:inclusion-besov}, and \eqref{p-nabla}, we deduce that 
 \begin{equation*} 
 \big\| \big[\delta_k, f\big](g) \|_{L^p(\R^d)}\leq C 2^{-k} \|\nabla f \big\|_{\cb^{\eps}_{p_1,\infty}(\R^d)} \|g \|_{L^{p_2}(\R^d)} \leq C 2^{- k} \|  f \big\|_{\mathcal{B}^{1+\eps}_{p_1, \infty}(\R^d)} \|g \|_{L^{p_2}(\R^d)},
\end{equation*} 
which is \eqref{estlem2} for $\alpha=1$. By the way, the inequality \eqref{estlem2} holds true for $\alpha=0$, and the general case $0 \leq \alpha \leq 1$ follows by interpolation.
 \end{proof}
 
 Recall that the paraproduct $f \pl g$ is defined in \eqref{def-para} with a slight modification compared to the usual definition.

  \begin{lemma}\label{lem.com} For $f,g \in \mathscr{S}(\R^d)$ we define
   $$\big[\delta_k, \pl\,\big] (f,g)= \delta_k(f \pl g) -f(\delta_k g).$$
 Let $p,p_1, p_2 \in [1, \infty]$ be such that    $\frac1{p}=\frac1{p_1}+\frac1{p_2}$. Let $0<\alpha <1$ and $\beta \in \R$. There exists $C>0$ such that for every $f \in \B^\alpha_{p_1, \infty}(\R^d)$ and $g \in \B^\beta_{p_2, \infty}(\R^d)$
  $$\big\| \big[\delta_k, \pl\,\big] (f,g) \big\|_{L^p(\R^d)} \leq C 2^{-k(\alpha+\beta)}\| f \|_{\B^\alpha_{p_1, \infty}(\R^d)}\| g \|_{\B^\beta_{p_2, \infty}(\R^d)}.$$
\end{lemma}

  \begin{proof} 
	
We adapt the proof of \cite[Lemma A.11]{MW}. \medskip
  
  $\bullet$ We first show that 
  \begin{equation}\label{upper}
 \big\| \delta_k(f \pl g) -f \pl (\delta_k g)\big\|_{L^p(\R^d)} \leq C 2^{-k(\alpha+\beta)}\|  f \|_{\B^{\alpha}_{p_1, \infty}(\R^d)}\| g \|_{\B^\beta_{p_2, \infty}(\R^d)}.
   \end{equation}
  We have 
\begin{eqnarray*}
  \delta_k(f \pl g) -f \pl (\delta_k g)&=& \sum_{j=-1}^{+\infty} \Big(\delta_k \big( (S_{j-3}f) \delta_j g\big)-(S_{j-3}f) \delta_j  \delta_k g \Big) \\
  &=& \sum_{j=-1}^{+\infty} \big[\delta_k,  S_{j-3}f\big] ( \delta_j g) \\
 &=& \mathcal{M}^{\mathbf{1}}_{k}(f,g)+\mathcal{M}^{\mathbf{2}}_{k}(f,g),
        \end{eqnarray*}
        where 
        $$\mathcal{M}^{\mathbf{1}}_{k}(f,g):= \sum_{j = -1}^{ k+3      }  \Big(\delta_k \big( (S_{j-3}f) \delta_j g\big)-(S_{j-3}f) \delta_j  \delta_k g \Big),$$
              $$\mathcal{M}^{\mathbf{2}}_{k}(f,g):= \sum_{j = k+4}^{+\infty}  \Big(\delta_k \big( (S_{j-3}f ) \delta_j g\big)-(S_{j-3}f) \delta_j  \delta_k g \Big).$$
       
              {Contribution of $\mathcal{M}^{\mathbf{1}}_{k}(f,g)$} : this can be treated as in  \cite[Lemma A.11]{MW} (in this part we use~\eqref{estlem}), and we can show that this contribution satisfies the upper bound~\eqref{upper}.\medskip
              
              Contribution of $\mathcal{M}^{\mathbf{2}}_{k}(f,g)$: this can be treated  using the bound \eqref{bound22} which reads here
              $$
              \| \delta_k\big(\delta_\ell f\delta_j g\big)\|_{L^p(\R^d)} \leq C_N 2^{- jN } \| \delta'_{\ell}f\|_{L^{p_1}(\R^d)}\| \delta'_j g\|_{L^{p_2}(\R^d)},
              $$
            for $\ell \leq j-4$ and $k \leq j-4$.  More precisely, using also that $\| \delta'_j g\|_{L^{p_2}(\R^d)} \leq C 2^{-j \beta} \| g \|_{\B^\beta_{p_2, \infty}(\R^d)}$, we get for every $N\geq \max(1,1-\beta)$,
       \begin{eqnarray*}
 \|\mathcal{M}^{\mathbf{2}}_{k}(f,g)\|_{L^p(\R^d)}&\leq &  \sum_{j = k+4}^{+\infty} \sum_{\ell =-1}^{  j-4} \| \delta_k\big(\delta_\ell f\delta_j g\big)\|_{L^p(\R^d)}  \\
  &\leq & C_N \sum_{j = k+4}^{+\infty} \sum_{\ell =-1}^{  j-4}  2^{- j(N+\al) } \| \delta'_{\ell}f\|_{L^{p_1}(\R^d)}\| \delta'_j g\|_{L^{p_2}(\R^d)} \\
  &\leq&  C_N \| f \|_{\B^\alpha_{p_1, \infty}}  \| g \|_{\B^\beta_{p_2, \infty}}\sum_{j  \geq k+4}  2^{-j(\alpha + \beta+ N)} \\
    &\leq&  C_N  2^{-k(\alpha + \beta+ N)} \| f \|_{\B^\alpha_{p_1, \infty}(\R^d)}  \| g \|_{\B^\beta_{p_2, \infty}(\R^d)},
        \end{eqnarray*}      
which gives a contribution of the form \eqref{upper}.\medskip

$\bullet$ It remains to show that 
  $$\|f  \pge (\delta_k g)  \|_{L^p(\R^d)} \leq C 2^{-k(\alpha+\beta)}\| f \|_{\B^\alpha_{p_1, \infty}(\R^d)}\| g \|_{\B^\beta_{p_2, \infty}(\R^d)},$$
but since the proof is similar to \cite[Lemma A.11]{MW}, the details are left here. \medskip

This completes the proof of \eqref{upper}.
   \end{proof}

We are now able to establish the following result, which concerns the commutation between $\pl$ and $\pe$.
   
     \begin{proposition}\label{prop:commutor}
  Let $0 <\alpha<1$, let $\beta, \gamma \in \R$ and $1 \leq p,p_1,p_2,p_3\leq \infty$ be such that 
  $$\beta+\gamma<0, \qquad \alpha+\beta+\gamma >0,   \qquad  \frac1{p_1}+\frac1{p_2}+ \frac1{p_3}=\frac1{p}.$$
   Then  for every $\eps>0$, the mapping
  \begin{equation*} 
  [\pl, \pe] : (f,g,h) \mapsto (f \pl g) \pe h- f(g \pe h)
  \end{equation*}
  extends to a continuous trilinear map from $ \mathcal{B}^{\alpha+\eps}_{p_1, \infty}(\R^d) \times \mathcal{B}^{\beta}_{p_2, \infty}(\R^d) \times \mathcal{B}^{\gamma}_{p_3, \infty}(\R^d) $ to $ \mathcal{B}^{\alpha +\beta+\gamma}_{p, \infty}(\R^d) $.
 \end{proposition}

\begin{proof}
The proof follows overall the arguments developed in \cite[Proposition A.9]{MW}, but additional interactions need to be controlled in our setting. By definition 
\begin{eqnarray*}
(f \pl g) \pe h-f(g \pe h)& =&\sum_{k \sim k'} \big( \delta_k (f \pl g) \big)\delta_{k'}h - f   \sum_{k \sim k'}  (   \delta_{k}g)( \delta_{k'}h ) \\
& =&  \sum_{k \sim k'} \Big ( \big[ \delta_k, \pl \big] (  f, g) \Big)\delta_{k'}h.
\end{eqnarray*}
For the sake of conciseness, in the sequel, we only treat the case $k'=k$, since the other cases $|k'-k| \leq 3$  are similar. Let us split the previous sum as 
\begin{equation*}
 \sum_{k = -1}^{+\infty} \Big ( \big[ \delta_k, \pl \big] (  f, g) \Big)\delta_{k}h= \mathcal{A}^{\mathbf{1}}(f,g,h)+ \mathcal{A}^{\mathbf{2}}(f,g,h)+ \mathcal{A}^{\mathbf{3}}(f,g,h),
\end{equation*}
with 
$$ \mathcal{A}^{\mathbf{1}}(f,g,h):= \sum_{k = -1}^{+\infty} S_{k+4}\bigg( \Big ( \big[ \delta_k, \pl \big] (  f, g) \Big)\delta_{k}h \bigg)   $$
$$ \mathcal{A}^{\mathbf{2}}(f,g,h):= \sum_{k = -1}^{+\infty} \big(1-S_{k+4}\big)\bigg( S_{k+4} \Big ( \big[ \delta_k, \pl \big] (  f, g) \Big)(\delta_{k}h) \bigg)   $$
$$ \mathcal{A}^{\mathbf{3}}(f,g,h):= \sum_{k = -1}^{+\infty} \big(1-S_{k+4}\big)\bigg( \big(1-S_{k+4}\big) \Big ( \big[ \delta_k, \pl \big] (  f, g) \Big)(\delta_{k}h) \bigg) .  $$

$\bullet$ Contribution of $\mathcal{A}^{\mathbf{1}}(f,g,h)$: Since $\alpha+\beta+\gamma>0$, we can apply Lemma \ref{lem-supp} to get 
\begin{equation}\label{prel}
\big\|  \mathcal{A}^{\mathbf{1}}(f,g,h)  \big\|_{\B^{\alpha+\beta+\gamma}_{p, \infty}(\R^d)} \lesssim \sup_{k \geq -1} \Big\{  2^{k(\alpha+\beta+\gamma)} \big\|   \Big ( \big[ \delta_k, \pl \big] (  f, g) \Big)\delta_{k}h    \big\|_{L^p(\R^d)}        \Big\}.
\end{equation}
Then we use that $ \|\delta_k h  \|_{L^{p_3}(\R^d)}   \leq 2^{-k \gamma}  \| h \|_{\B^{\gamma}_{p_3, \infty}(\R^d)} $ 
and by Lemma \ref{lem.com}, we obtain with $\frac{1}{p'_3}=\frac{1}{p}-\frac{1}{p_3}$
\begin{eqnarray*}
\big\|   \Big ( \big[ \delta_k, \pl \big] (  f, g) \Big) \delta_{k}h \big\|_{L^p(\R^d)}       &\leq &\big\|  \big[ \delta_k, \pl \big] (  f, g)     \big\|_{L^{p'_3}(\R^d)}  \big\|   \delta_{k}h \big\|_{L^{p_3}(\R^d)}     \\
&\lesssim &  2^{-k(\alpha+\beta+\gamma)}  {\| f \|_{\B^{\alpha+\eps}_{p_1, \infty}(\R^d)}} \| g \|_{\B^{\beta}_{p_2, \infty}(\R^d)} \| h \|_{\B^{\gamma}_{p_3, \infty}(\R^d)}.
\end{eqnarray*}
As a consequence, from \eqref{prel}, we deduce that 
\begin{equation}\label{Contr-A1}
\big\|  \mathcal{A}^{\mathbf{1}}(f,g,h)  \big\|_{\B^{\alpha+\beta+\gamma}_{p, \infty}(\R^d)} \lesssim   {\| f \|_{\B^{\alpha+\eps}_{p_1, \infty}(\R^d)}} \| g \|_{\B^{\beta}_{p_2, \infty}(\R^d)} \| h \|_{\B^{\gamma}_{p_3, \infty}(\R^d)}.
\end{equation}
\medskip

$\bullet$ Contribution of $\mathcal{A}^{\mathbf{2}}(f,g,h)$: Let $j \geq 4$ (the contributions of the small frequencies $-1 \leq j \leq 3$ are easy to handle). Then  
$$  \delta_{j} \big( \mathcal{A}^{\mathbf{2}}(f,g,h)\big)   = \sum_{k = -1}^{+\infty} \sum_{\ell=k+4}^{+\infty} \delta_j \delta_{\ell}\bigg( S_{k+4} \Big ( \big[ \delta_k, \pl \big] (  f, g) \Big)(\delta_{k}h) \bigg) .    $$
Now recall that $\delta_j \delta_{\ell}=0$ if $|j-\ell | \geq 2$, thus the previous series vanishes if $j \leq k+2$. In the following, we therefore assume that $j \geq k+3$ and we set
$$\dis \widetilde{\delta}_j:=   \sum_{\ell=k+4}^{+\infty} \delta_j \delta_{\ell} =
\begin{cases}
\delta_j\big(\delta_{j-1}+ \delta_j+ \delta_{j+1}\big) &  \text{if} \  j \geq k+5\\
\delta_j\big(  \delta_j+ \delta_{j+1}\big)&  \text{if}\  j = k+4\\
\delta_j \delta_{j+1} &  \text{if}\  j = k+3
\end{cases}.$$
Observe that $\widetilde{\delta}_j$ has the same continuity and support properties as ${\delta}_j$.  Then we get 
$$  \delta_{j} \big( \mathcal{A}^{\mathbf{2}}(f,g,h)\big)   = \sum_{k = -1}^{j-3}\widetilde{\delta}_j\bigg( S_{k+4} \Big ( \big[ \delta_k, \pl \big] (  f, g) \Big)(\delta_{k}h) \bigg) .    $$
We split the previous sum as $ \dis \sum_{k = -1}^{j-3}=  \sum_{k = -1}^{j-8}+  \sum_{k = j-7}^{j-3} $. For the first sum, we use  \eqref{detJ}, which yields  for all $N \geq 1$, that 
\begin{eqnarray*}
\sum_{k = -1}^{j-8}     \big\|     \widetilde{\delta}_j\bigg( S_{k+4} \Big ( \big[ \delta_k, \pl \big] (  f, g) \Big)(\delta_{k}h) \bigg)     \big\|_{L^p(\R^d)}  &\lesssim &2^{-Nj}  \sum_{k = -1}^{j-8}  \sum_{n =-1}^{k+3}  \big\|  \delta'_{n}  \Big ( \big[ \delta_k, \pl \big] (  f, g) \Big)    \big\|_{L^{p'_3}(\R^d)}   \big\|    \delta'_{k}h     \big\|_{L^{p_3}(\R^d)} \\
&\lesssim &2^{-Nj}  \sum_{k = -1}^{j-8} k \,\big\|    \big[ \delta_k, \pl \big] (  f, g)     \big\|_{L^{p'_3}(\R^d)}   \big\|    \delta'_{k}h     \big\|_{L^{p_3}(\R^d)} .
\end{eqnarray*}
Next, using Lemma \ref{lem.com} and choosing for instance $N=1\geq \alpha+\beta+\gamma$, we obtain that 
\begin{multline}\label{C_1}
\sum_{k = -1}^{j-8}     \big\|     \widetilde{\delta}_j\bigg( S_{k+4} \Big ( \big[ \delta_k, \pl \big] (  f, g) \Big)(\delta_{k}h)      \big\|_{L^p(\R^d)}  \lesssim  \\
\begin{aligned}
&\lesssim 2^{-j}   {\| f \|_{\B^{\alpha+\eps}_{p_1, \infty}(\R^d)}}\| g \|_{\B^{\beta}_{p_2, \infty}(\R^d)}    \| h \|_{\B^{\gamma}_{p_3, \infty}(\R^d)}   \sum_{k = -1}^{\infty}  k 2^{-k (\alpha+\beta+\gamma)}       \\
&\lesssim  2^{-j (\alpha+\beta+\gamma)}    {\| f \|_{\B^{\alpha+\eps}_{p_1, \infty}(\R^d)}} \| g \|_{\B^{\beta}_{p_2, \infty}(\R^d)}    \| h \|_{\B^{\gamma}_{p_3, \infty}(\R^d)}.  
\end{aligned}
\end{multline}
For the second sum, we use the H\"older inequality and the continuity (uniform over $k \geq -1$) of the operators~$\delta_k$ and $S_k$ to deduce 
\begin{align}\label{C_2}
\sum_{k = j-7}^{j-3}     \big\|     \widetilde{\delta}_j\bigg( S_{k+4} \Big ( \big[ \delta_k, \pl \big] (  f, g) \Big)(\delta_{k}h)      \big\|_{L^p(\R^d)} & \lesssim  \sum_{k = j-7}^{j-3}   \big\|   \big[ \delta_k, \pl \big] (  f, g)     \big\|_{L^{p'_3}(\R^d)}   \big\|    \delta_{k}h     \big\|_{L^{p_3}(\R^d)}     \nonumber  \\
&\lesssim  {\| f \|_{\B^{\alpha+\eps}_{p_1, \infty}(\R^d)}} \| g \|_{\B^{\beta}_{p_2, \infty}(\R^d)}    \| h \|_{\B^{\gamma}_{p_3, \infty}(\R^d)}  \sum_{k = j-7}^{j-3}    2^{-k (\alpha+\beta+\gamma)}    \qquad \nonumber    \\
&\lesssim   2^{-j (\alpha+\beta+\gamma)} {\| f \|_{\B^{\alpha+\eps}_{p_1, \infty}(\R^d)}} \| g \|_{\B^{\beta}_{p_2, \infty}(\R^d)}    \| h \|_{\B^{\gamma}_{p_3, \infty}(\R^d)}        .  
\end{align}
Putting the estimates \eqref{C_1} and \eqref{C_2} together, we obtain the bound
\begin{eqnarray}\label{Contr-A2}
\big\|  \mathcal{A}^{\mathbf{2}}(f,g,h)  \big\|_{\B^{\alpha+\beta+\gamma}_{p, \infty}(\R^d)}   &\lesssim& \sup_{j \geq -1} \Big\{  2^{j(\alpha+\beta+\gamma)} \big\|  \delta_{j} \big( \mathcal{A}^{\mathbf{2}}(f,g,h)\big)    \big\|_{L^p(\R^d)}        \Big\}\nonumber \\
&\lesssim    & {\| f \|_{\B^{\alpha+\eps}_{p_1, \infty}(\R^d)}}\| g \|_{\B^{\beta}_{p_2, \infty}(\R^d)} \| h \|_{\B^{\gamma}_{p_3, \infty}(\R^d)}.
\end{eqnarray}
\medskip

$\bullet$ Contribution of $\mathcal{A}^{\mathbf{3}}(f,g,h)$: Let $j \geq -1$, then similarly to the previous case
\begin{eqnarray*}
  \delta_{j} \big( \mathcal{A}^{\mathbf{3}}(f,g,h)\big)   &=& \sum_{k = -1}^{j-3}\widetilde{\delta}_j\bigg( \big( 1-S_{k+4} \big)\Big ( \big[ \delta_k, \pl \big] (  f, g) \Big)(\delta_{k}h) \bigg) \\
  &=& \sum_{k = -1}^{j-3} \sum_{m = k+4}^{+\infty}\widetilde{\delta}_j\bigg(  \delta_m \Big ( \big[ \delta_k, \pl \big] (  f, g) \Big)(\delta_{k}h) \bigg),
\end{eqnarray*}
since the contributions for $k \geq j-2$ vanish. We split the  term $ \delta_{j} \big( \mathcal{A}^{\mathbf{3}}(f,g,h)\big)  $ into three parts as follows:
\begin{equation}\label{decomp3}
  \delta_{j} \big( \mathcal{A}^{\mathbf{3}}(f,g,h)\big)  = \mathcal{B}_j^{\mathbf{1}}(f,g,h)+ \mathcal{B}_j^{\mathbf{2}}(f,g,h)+ \mathcal{B}_j^{\mathbf{3}}(f,g,h),
\end{equation}
 with 
 $$ \mathcal{B}_j^{\mathbf{1}}(f,g,h): = \sum_{k = -1}^{j-8} \sum_{m = k+4}^{j-4}\widetilde{\delta}_j\bigg(  \delta_m \Big ( \big[ \delta_k, \pl \big] (  f, g) \Big)(\delta_{k}h) \bigg) $$
  $$ \mathcal{B}_j^{\mathbf{2}}(f,g,h): = \sum_{k = -1}^{j-8} \sum_{m = j-3}^{j+3}\widetilde{\delta}_j\bigg(  \delta_m \Big ( \big[ \delta_k, \pl \big] (  f, g) \Big)(\delta_{k}h) \bigg) $$
   $$ \mathcal{B}_j^{\mathbf{3}}(f,g,h): = \sum_{k = -1}^{j-3} \sum_{m = j+4}^{+\infty}\widetilde{\delta}_j\bigg(  \delta_m \Big ( \big[ \delta_k, \pl \big] (  f, g) \Big)(\delta_{k}h) \bigg) $$
     $$ \mathcal{B}_j^{\mathbf{4}}(f,g,h): = \sum_{k = j-7}^{j-3} \sum_{m = k+4}^{j+3}\widetilde{\delta}_j\bigg(  \delta_m \Big ( \big[ \delta_k, \pl \big] (  f, g) \Big)(\delta_{k}h) \bigg). $$

-- Study of    $ \mathcal{B}_j^{\mathbf{1}}(f,g,h):$ By \eqref{detJ} (with $N=2$) and Lemma \ref{lem.com}, we obtain that 
\begin{eqnarray}
\|\mathcal{B}_j^{\mathbf{1}}(f,g,h)\|_{L^p(\R^d)} & \lesssim & 2^{-2j} \sum_{k = -1}^{j-8} \sum_{m = k+4}^{j-4}  \big\|    \big[ \delta_k, \pl \big] (  f, g)     \big\|_{L^{p'_3}(\R^d)}   \big\|    \delta'_k h     \big\|_{L^{p_3}(\R^d)} \nonumber \\
& \lesssim &j 2^{-2j}  {\| f \|_{\B^{\alpha+\eps}_{p_1, \infty}(\R^d)}} \| g \|_{\B^{\beta}_{p_2, \infty}(\R^d)}  \| h\|_{\B^{\gamma}_{p_3, \infty}(\R^d)}  \sum_{k = -1}^{\infty} 2^{-k(\alpha+\beta+\gamma)}  \nonumber  \\
& \lesssim &   2^{-j (\alpha+\beta+\gamma)}   {\| f \|_{\B^{\alpha+\eps}_{p_1, \infty}(\R^d)}} \| g \|_{\B^{\beta}_{p_2, \infty}(\R^d)}  \| h\|_{\B^{\gamma}_{p_3, \infty}(\R^d)}. \label{BB1}
\end{eqnarray}
\medskip

-- Study of    $ \mathcal{B}_j^{\mathbf{2}}(f,g,h):$ We can check that 
$$    \delta_m \Big ( \big[ \delta_k, \pl \big] (  f, g) \Big)=  \big[  \delta_m\delta_k, \pl \big] (  f, g) -   \big[ \delta_m, f \big] (\delta_k g).$$
In fact, since $|m-k| \geq 2$, the term $ \big[  \delta_m\delta_k, \pl \big] (  f, g) $ vanishes, and we only need to study the contribution of the second term. 

By Lemma \ref{lem-C15}, we get 
\begin{multline}
\sum_{k = -1}^{j-8} \sum_{m = j-3}^{j+3}  \Big\| \widetilde{\delta}_j\bigg(    \Big ( \big[ \delta_m, f \big] (\delta_k g)\Big)(\delta_{k}h) \bigg)  \Big\|_{L^p(\R^d)} \lesssim   \sum_{k = -1}^{j-8} \sum_{m = j-3}^{j+3}  \big\|     \big[ \delta_m, f \big] (\delta_k g)     \big\|_{L^{p'_3}(\R^d)}   \big\|    \delta_k h     \big\|_{L^{p_3}(\R^d)}  \\
  \begin{aligned} \label{BB2-2}
& \lesssim    \sum_{k = -1}^{j-8} \sum_{m = j-3}^{j+3}  2^{-m\alpha}        {\| f \|_{\B^{\alpha+\eps}_{p_1, \infty}(\R^d)}}   \big\|    \delta_k g     \big\|_{L^{p_2}(\R^d)}    \big\|    \delta_k h     \big\|_{L^{p_3}(\R^d)}   \\
& \lesssim   2^{-j\alpha }     {\| f \|_{\B^{\alpha+\eps}_{p_1, \infty}(\R^d)}}\| g \|_{\B^{\beta}_{p_2, \infty}(\R^d)}  \| h\|_{\B^{\gamma}_{p_3, \infty}(\R^d)}  \sum_{k = -1}^{j-8}  2^{-k(\beta+\gamma)}   \\
& \lesssim   2^{-j(\alpha+ \beta+\gamma) }     {\| f \|_{\B^{\alpha+\eps}_{p_1, \infty}(\R^d)}} \| g \|_{\B^{\beta}_{p_2, \infty}(\R^d)}  \| h\|_{\B^{\gamma}_{p_3, \infty}(\R^d)}   ,
\end{aligned}
\end{multline}
where in the last line we used the assumption  $\beta+\gamma<0$. Finally,   \eqref{BB2-2} imply the desired bound for  $ \mathcal{B}_j^{\mathbf{2}}(f,g,h)$.
\medskip

-- Study of    $ \mathcal{B}_j^{\mathbf{3}}(f,g,h):$ By \eqref{bound22} (with $N=1$) and Lemma \ref{lem.com}, we obtain that for all $N \geq 1$
\begin{eqnarray}
\|\mathcal{B}_j^{\mathbf{3}}(f,g,h)\|_{L^p(\R^d)} & \lesssim &  \sum_{k = -1}^{j-3} \sum_{m = j+4}^{+\infty} 2^{-m} \big\|    \big[ \delta_k, \pl \big] (  f, g)     \big\|_{L^{p'_3}(\R^d)}   \big\|    \delta'_k h     \big\|_{L^{p_3}(\R^d)} \nonumber \\
& \lesssim &   {\| f \|_{\B^{\alpha+\eps}_{p_1, \infty}(\R^d)}} \| g \|_{\B^{\beta}_{p_2, \infty}(\R^d)}  \| h\|_{\B^{\gamma}_{p_3, \infty}(\R^d)}\sum_{m = j+4}^{+\infty} 2^{-m} \sum_{k = -1}^{\infty}  2^{-k(\alpha+\beta+\gamma)}  \nonumber  \\
& \lesssim &  2^{-j(\alpha+ \beta+\gamma) }      {\| f \|_{\B^{\alpha+\eps}_{p_1, \infty}(\R^d)}}\| g \|_{\B^{\beta}_{p_2, \infty}(\R^d)}  \| h\|_{\B^{\gamma}_{p_3, \infty}(\R^d)}. \label{BB3}
\end{eqnarray}
\medskip

\medskip

-- Study of    $ \mathcal{B}_j^{\mathbf{4}}(f,g,h):$ For this contribution, we simply observe that 
\begin{align*}
\| \mathcal{B}_j^{\mathbf{4}}(f,g,h) \|_{L^p(\R^d)}&  \leq  \sum_{k = j-7}^{j-3} \sum_{m = j-3}^{j+3} \Big\| \widetilde{\delta}_j\bigg(  \delta_m \Big ( \big[ \delta_k, \pl \big] (  f, g) \Big)(\delta_{k}h) \bigg)  \Big\|_{L^p(\R^d)}\\
&\lesssim   \sum_{k = j-7}^{j-3} \sum_{m = j-3}^{j+3} \big\|  \big[ \delta_k, \pl \big] (  f, g) \big\|_{L^{p'_3}(\R^d)}   \big\|\delta_{k}h \big\|_{L^{p_3}(\R^d)}\\
&\lesssim   {\| f \|_{\B^{\alpha+\eps}_{p_1, \infty}(\R^d)}}\| g \|_{\B^{\beta}_{p_2, \infty}(\R^d)} \| h \|_{\B^{\gamma}_{p_3, \infty}(\R^d)}  \sum_{k = j-7}^{j-3}  2^{-k(\alpha+ \beta+\gamma) } ,
\end{align*}
and so
\begin{equation}\label{BB4}
\|\mathcal{B}_j^{\mathbf{4}}(f,g,h)\|_{L^p(\R^d)} \lesssim   2^{-j(\alpha+ \beta+\gamma) }    {\| f \|_{\B^{\alpha+\eps}_{p_1, \infty}(\R^d)}}\| g \|_{\B^{\beta}_{p_2, \infty}(\R^d)} \| h \|_{\B^{\gamma}_{p_3, \infty}(\R^d)}.
\end{equation}

\smallskip

Gathering the estimates \eqref{BB1}, \eqref{BB2-2},  \eqref{BB3} and \eqref{BB4}, and going back to \eqref{decomp3}, we obtain
\begin{equation}\label{Contr-A3}
\big\|  \mathcal{A}^{\mathbf{3}}(f,g,h)  \big\|_{\B^{\alpha+\beta+\gamma}_{p, \infty}(\R^d)}   \lesssim  {\| f \|_{\B^{\alpha+\eps}_{p_1, \infty}(\R^d)}} \| g \|_{\B^{\beta}_{p_2, \infty}(\R^d)} \| h \|_{\B^{\gamma}_{p_3, \infty}(\R^d)}.
\end{equation}
\medskip

Finally, the bounds \eqref{Contr-A1}, \eqref{Contr-A2} and \eqref{Contr-A3} complete the proof of the proposition.
\end{proof}

 \begin{lemma}\label{lem1.2} Let $\alpha <1$, $\beta \in \R$, $\gamma \geq \alpha+\beta$, and $p,p_1, p_2 \in [1, \infty]$ such that $\frac1p=\frac{1}{p_1}+\frac{1}{p_2}$. For every $t \geq 0$ and $f,g \in \mathscr{S}(\R^d)$, we define 
 $$[e^{-tH}, \pl\,] : (f,g) \longmapsto e^{-tH}(f \pl g) -f \pl (e^{-tH}g).$$
 There exists $C>0$ such that for all $t>0$, and all $f \in \B^\alpha_{p_1, \infty}(\R^d)$, $g \in \B^\beta_{p_2, \infty}(\R^d)$
 \begin{equation}\label{com-1}
 \big \| [e^{-tH}, \pl\,] (f,g) \big\|_{\B^\gamma_{p, \infty}(\R^d)} \leq C t^{\frac{\alpha+\beta-\gamma-\eps}2} e^{-\frac{dt}8}\| f \|_{\B^\alpha_{p_1, \infty}(\R^d)}\| g \|_{\B^\beta_{p_2, \infty}(\R^d)}.
   \end{equation}
\end{lemma}

\begin{proof}

By definition, we have 
 \begin{equation*} 
 [e^{-tH}, \pl\,] (f,g)   = \sum_{k \geq -1}   \mathcal{L}_k(f,g)
    \end{equation*}
with
$$\mathcal{L}_k(f,g):=e^{-tH}\big(({S_{k-3}} f) (\delta_k g)\big) -({S_{k-3}} f )\delta_k(e^{-tH}g)=e^{-tH}\big(({S_{k-3}} f) (\delta_k g)\big) -({S_{k-3}} f )(e^{-tH}\delta_k g).$$
Let us decompose   $  \mathcal{L}_k(f,g) =   \mathcal{L}^{\mathbf{1}}_k(f,g) +   \mathcal{L}^{\mathbf{2}}_k(f,g)  $   with               
 \begin{equation}\label{def-L2}
    \mathcal{L}^{\mathbf{1}}_k(f,g)   := S_{k+4}    \mathcal{L}_k(f,g), \qquad     \mathcal{L}^{\mathbf{2}}_k(f,g)   := \big(1-S_{k+4}  \big)  \mathcal{L}_k(f,g), 
       \end{equation}
and set 
 \begin{equation*} 
 \mathcal{P}^{\mathbf{1}} (f,g)  := \sum_{k \geq -1}   \mathcal{L}^{\mathbf{1}}_k(f,g), \qquad  \mathcal{P}^{\mathbf{2}} (f,g)  := \sum_{k \geq -1}   \mathcal{L}^{\mathbf{2}}_k(f,g),
   \end{equation*}
 so that 
 \begin{equation}\label{C-P}
 [e^{-tH}, \pl\,] (f,g)   = \mathcal{P}^{\mathbf{1}} (f,g) +  \mathcal{P}^{\mathbf{2}} (f,g).
    \end{equation}
     We will see below  that the main contribution in \eqref{C-P} is given by    $ \mathcal{P}^{\mathbf{1}} (f,g) $.
   \medskip  
   
 \underline{Step 1}:    Contribution of    $ \mathcal{P}^{\mathbf{1}} (f,g) $.  \medskip
 
 The  aim of this paragraph is to establish the bound
     \begin{equation}\label{com-P1}
 \big \| \mathcal{P}^{\mathbf{1}} (f,g)  \big\|_{\B^\gamma_{p, \infty}(\R^d)} \leq C t^{\frac{\alpha+\beta-\gamma-\eps}2} e^{-\frac{dt}8}\| f \|_{\B^\alpha_{p_1, \infty}(\R^d)}\| g \|_{\B^\beta_{p_2, \infty}(\R^d)}.
   \end{equation}

Recall that $\dis S_k =\sum_{j \leq k-1}\delta_{j}$ and that $\delta_j \delta_{j'}=0$ when $|j-j'| \geq 2$, thus   $(1-S_{k-3})\delta_k g= \delta_k g$. Therefore, we make the  decomposition 
$$\mathcal{L}^{\mathbf{1}}_k(f,g)=\mathcal{L}^{\mathbf{1},\mathbf{1} }_{k}(f,g)+ \mathcal{L}^{\mathbf{1},\mathbf{2}}_{k}(f,g),$$
 where 	
$$\mathcal{L}^{\mathbf{1},\mathbf{1} }_{k}(f,g):=S_{k+4}   \Big(\big(1-S_{k-3}\big)      e^{-tH}(S_{k-3} f\delta_k g) -(S_{k-3} f) \cdot \big[\big(1-S_{k-3}\big) \big( e^{-tH}\delta_kg\big)\big]\Big),$$
$$\mathcal{L}^{\mathbf{1},\mathbf{2} }_{k}(f,g):=S_{k+4} \Big(e^{-tH}S_{k-3}\big(S_{k-3} f\delta_k g\big) \Big)= e^{-tH}S_{k-3}\big(S_{k-3} f\delta_k g\big) .$$

\medskip

Let $\lambda>0$. Then, by the Cauchy-Schwarz inequality  and the Mehler formula \eqref{mehler}  
\begin{eqnarray}  \label{cs1}
\sum_{
\substack{
n\geq 0\\
\lambda_n \geq \lambda
}}
 e^{-t \lambda_n} |\varphi_n(x)\varphi_n(y)| &\leq & e^{-t\frac{\lambda}2} \big(\sum_{n \geq 0} e^{-t \frac{\lambda_n}2} |\varphi_n(x)|^2\big)^{\frac12} \big(\sum_{n \geq 0} e^{-t \frac{\lambda_n}2} |\varphi_n(y)|^2\big)^{\frac12}\nonumber \\
&\leq &(2\pi\sinh t)^{-\frac{d}2}    e^{-{(\tanh \frac{t}2})\frac{|x|^2+|y|^2}{2}}e^{-t \frac{\lambda}2} .
\end{eqnarray}

$\bullet$ We   show a first bound on $\mathcal{L}^{\mathbf{1}}_k(f,g)$. Namely, we will prove that for all $t\geq 0$
\begin{equation} \label{bound0}
  \|\mathcal{L}^{\mathbf{1}}_k(f,g)\|_{L^p(\R^d)}  \leq  Ct^{\frac12} e^{-dt}\|\nabla S_{{k-3}} f \|_{L^{p_1}(\R^d)} \|\delta_k g \|_{L^{p_2}(\R^d)}. 
  \end{equation}
Recall  that $\mathcal{L}^{\mathbf{1}}_k(f,g)   := S_{k+4}    \mathcal{L}_k(f,g)$, then by Proposition \ref{prop.conti} we get $  \|\mathcal{L}^{\mathbf{1}}_k(f,g)\|_{L^p(\R^d)} \lesssim   \|\mathcal{L}_k(f,g)\|_{L^p(\R^d)}$, and  it is enough to prove the estimate for $\mathcal{L}_k(f,g)$.  We can write
 $$S_{{k-3}} f(y)-S_{{k-3}} f(x)=\int_0^1ds \, (y-x) \cdot \nabla S_{{k-3}} f\big(x+s(y-x)\big) $$
 thus
\begin{eqnarray*} 
\mathcal{L}_k(f,g)(x) 
&=& \int_{\R^d} dy \,  K_t(x,y)\big(S_{{k-3}} f(y)-S_{{k-3}} f(x)\big)\delta_kg(y) \\
&=&\int_0^1 ds \, \int_{\R^d} dy \, K_t(x,y)(y-x) \cdot \nabla S_{{k-3}} f(x+s(y-x)) \delta_kg(y).
\end{eqnarray*}
Therefore, by the Mehler formula \eqref{mehler} and the change of variables $z=x-y$,
\begin{eqnarray*} 
|\mathcal{L}_k(f,g)(x) |
&\leq &   (2\pi\sinh 2t)^{-\frac{d}2}\int_0^1 ds \, \int_{\R^d} dy \, e^{- \frac{ \vert x-y\vert^2}{4\tanh t}}|x-y|\big| \nabla S_{{k-3}} f\big(x+s(y-x)\big)\big||\delta_kg(y)| \\
&=& (2\pi\sinh 2t)^{-\frac{d}2}\int_0^1 ds \,   \int_{\R^d} dz \, e^{- \frac{ \vert z\vert^2}{4\tanh t}}|z|\big| \nabla S_{{k-3}} f(x-sz)\big||\delta_kg(x-z)| .
\end{eqnarray*}
Now observe that $e^{ -\frac{ \vert z\vert^2}{4\tanh t}}|z| \leq  C e^{ -\frac{ \vert z\vert^2}{8\tanh t}}(\tanh t)^{\frac12}$, and so
\begin{equation}\label{inte}
|\mathcal{L}_k(f,g)(x) | \leq C (\sinh 2t)^{-\frac{d}2} (\tanh t)^{\frac12}    \int_0^1 ds \, m_k(s,x) 
\end{equation}
where we have set 
$$m_k(s,x)=   \int_{\R^d} dz \, G_t(z)\big| \nabla S_{{k-3}} f(x-sz)\big||\delta_kg(x-z)| , $$
with $G_t(z) = e^{ -\frac{ \vert z\vert^2}{8\tanh t}}$. Let $p,p_1, p_2 \in [1, \infty]$ such that $\frac1p=\frac{1}{p_1}+\frac{1}{p_2}$, then from the previous inequality, and by the Minkowski and the H\"older  inequalities  we get, for all $0\leq s \leq 1$
\begin{eqnarray*}
 \|m_{k}(s,\cdot)\|_{L^p(\R^d)} &\leq&    \int_{\R^d}dz \, G_t(z) \|\nabla S_{{k-3}} f \|_{L^{p_1}(\R^d)} \|\delta_k g \|_{L^{p_2}(\R^d)} \\
&\leq& \|  G_t\|_{L^1(\R^d)}  \|\nabla S_{{k-3}} f \|_{L^{p_1}(\R^d)} \|\delta_k g \|_{L^{p_2}(\R^d)} \\
&\leq&   C(\tanh t)^{\frac{d}2} \|\nabla S_{{k-3}} f \|_{L^{p_1}(\R^d)} \|\delta_k g \|_{L^{p_2}(\R^d)} ,
\end{eqnarray*}
and finally, thanks to \eqref{inte}
$$  \|\mathcal{L}^{\mathbf{1}}_k(f,g)\|_{L^p(\R^d)}  \lesssim  \|\mathcal{L}_k(f,g)\|_{L^p(\R^d)} \lesssim  (\tanh t)^{\frac{d+1}2} (\sinh 2t)^{-\frac{d}2} \|\nabla S_{{k-3}} f \|_{L^{p_1}(\R^d)} \|\delta_k g \|_{L^{p_2}(\R^d)}.$$
On the one hand, for $0<t \leq 1$ we have $ (\tanh t)^{\frac{d+1}2} (\sinh 2t)^{-\frac{d}2} \leq C t^{\frac12}\leq Ct^{\frac12} e^{-dt}$, and on the other hand for $t \geq 1$ we have $ (\tanh t)^{\frac{d+1}2} (\sinh 2t)^{-\frac{d}2} \leq C e^{-dt}\leq Ct^{\frac12} e^{-dt}$. As a consequence, we get the bound~\eqref{bound0} as desired. \medskip

$\bullet$ Recall that  $\theta' \in \mathcal{C}_0^\infty$ is such that $\theta' \theta= \theta$, and that  we set  $\delta'_k=\theta'(\frac{H}{2^{2k}})$.
Let us show that for all $N\geq 1$, 
\begin{equation}\label{borneh2}
 \|\mathcal{L}^{\mathbf{1}, \mathbf{2}}_{k}(f,g)\|_{L^p(\R^d)} \leq C_N 2^{-2kN} e^{-dt} \|S_{k}f\|_{L^{p_1}(\R^d)}\|\delta'_k g\|_{L^{p_2}(\R^d)}.
 \end{equation}
 Since  $\dis S_k =\sum_{j \leq k-1}\delta_{j}$, we can write
$$ \mathcal{L}^{\mathbf{1}, \mathbf{2}}_{k}(f,g)  = \sum_{j \leq {k-4}}\sum_{\ell  \leq {k-4}}   e^{-tH}\delta_j\big(\delta_\ell f\delta_k g\big).$$
Then, using \eqref{heat-Lp} and \eqref{bound23} we have 
\begin{eqnarray*}
\|\mathcal{L}^{\mathbf{1}, \mathbf{2}}_{k}(f,g)\|_{L^p(\R^d)} &\leq & \sum_{j \leq {k-4}}\sum_{\ell  \leq {k-4}}   \|e^{-tH}\delta_j\big(\delta_\ell f\delta_k g\big)\|_{L^p(\R^d)} \nonumber\\
&\leq & Ce^{-dt}\sum_{j \leq {k-4}}\sum_{\ell  \leq {k-4}}   \|\delta_j\big(\delta_\ell f\delta_k g\big)\|_{L^p(\R^d)}\nonumber \\
&\leq & C_N 2^{- kN } e^{-dt}\| S_{k}f\|_{L^{p_1}(\R^d)}\| \delta'_k g\|_{L^{p_2}(\R^d)}, 
\end{eqnarray*}
which is \eqref{borneh2}.
 \medskip

$\bullet$ Let us show that for all $t>0$
\begin{equation}\label{borneh1}
 \|\mathcal{L}^{\mathbf{1},\mathbf{1}}_{k}(f,g)\|_{L^p(\R^d)} \leq  Ct^{-d-\frac12}e^{-\frac{dt}4} e^{-ct 2^{2k}}\|\nabla S_{k-1} f \|_{L^{p_1}(\R^d)} \|\delta_k g \|_{L^{p_2}(\R^d)}.
 \end{equation}
 Set 
 $$\widetilde{\mathcal{L}}_k(f,g) = {\big(1-S_{k-3}\big)      e^{-tH}(S_{k-3} f\delta_k g) -(S_{k-3} f) \cdot \big[\big(1-S_{k-3}\big) \big( e^{-tH}\delta_kg\big)\big]},$$
so that $\mathcal{L}^{\mathbf{1},\mathbf{1}}_{k}(f,g)= S_{k+4}\widetilde{\mathcal{L}}_k(f,g) $. By Proposition \ref{prop.conti}, it is enough to prove the estimate  \eqref{borneh1} for $\widetilde{\mathcal{L}}_k(f,g)$. We have
\begin{eqnarray*} 
\widetilde{\mathcal{L}}_k(f,g)(x)&=& \sum_{j \geq k-3} \theta\big(\frac{H}{2^{2j}}\big)\Big( e^{-tH}\big((S_{{k-3}}f)\delta_k g \big)(x)-   \big(S_{{k-3}}f(x)\big)(e^{-tH}\delta_k g\big)(x)\Big) \\
&=& \int_{\R^d} dy \, \sum_{j \geq k-3} \theta\big(\frac{H}{2^{2j}}\big)K_t(x,y)\big(S_{{k-3}} f(y)-S_{{k-3}} f(x)\big)\delta_kg(y) \\
&=& \int_{\R^d}dy \,  \sum_{n\geq 0}\sum_{j \geq k-3} \theta\big(\frac{\lambda_n}{2^{2j}}\big)e^{-t \lambda_n} \varphi_n(x)\varphi_n(y) \big(S_{{k-3}} f(y)-S_{{k-3}} f(x)\big)\delta_k g(y) .
\end{eqnarray*}
On the support of $\theta$ we have $\lambda_n \geq c \, 2^{2j}$, and since $j\geq k-3$, we obtain $\lambda_n\geq c \, 2^{2k}$.  Now write 
 $$S_{{k-3}} f(y)-S_{{k-3}} f(x)=\int_0^1 ds \, (y-x) \cdot \nabla S_{{k-3}} f\big(x+s(y-x)\big) $$
then,   by \eqref{cs1} and with the change of variables $z=x-y$
\begin{align*} 
|\widetilde{\mathcal{L}}_k(f,g)(x)|&\leq C(\sinh t)^{-\frac{d}2}e^{-ct 2^{2k}}  \int_{\R^d} dy \, e^{-{(\tanh \frac{t}2})\frac{|x|^2+|y|^2}2}\big|S_{{k-3}} f(y)-S_{{k-3}} f(x)\big||\delta_kg(y)|  \\
&\leq C(\sinh t)^{-\frac{d}2}e^{-ct 2^{2k}} \int_0^1ds \,  \int_{\R^d} dy \, e^{-{(\tanh \frac{t}2})\frac{|x|^2+|y|^2}2}|y-x| |\nabla S_{{k-3}} f\big(x+s(y-x)\big) | |\delta_kg(y)| \\
&\leq C(\sinh t)^{-\frac{d}2}e^{-ct 2^{2k}} \int_0^1 ds \, \int_{\R^d} dz \, e^{-{(\tanh \frac{t}2})\frac{|x|^2+|x-z|^2}2}|z| |\nabla S_{{k-3}} f(x-sz) | |\delta_kg(x-z)| .
\end{align*}
For $c>0$ small enough, $|x|^2+|x-z|^2\geq 4c\,(|x|^2+|z|^2) $ and $e^{-{2c(\tanh \frac{t}2})|z|^2}|z| \leq C(\tanh \frac{t}2)^{-\frac12}e^{-{c(\tanh \frac{t}2})|z|^2}$, which yields 
\begin{eqnarray*} 
|\widetilde{\mathcal{L}}_k(f,g)(x)|&\leq  &C (\sinh t)^{-\frac{d}2}e^{-ct 2^{2k}} \int_0^1 ds \, \int_{\R^d} dz \, e^{-{2c(\tanh \frac{t}2})(|x|^2+|z|^2)}|z| |\nabla S_{{k-3}} f(x-sz) | |\delta_kg(x-z)| \\
&\leq  &C(\tanh \frac{t}2)^{-\frac12}(\sinh t)^{-\frac{d}2}e^{-ct 2^{2k}} \int_0^1ds \, m_{1,k}(s,x),
\end{eqnarray*}
where we have set
$$ m_{1,k}(s,x)=\int_{\R^d} dz \, G_t(z) |\nabla S_{{k-3}} f(x-sz) | |\delta_kg(x-z)| $$
with  $G_t(z)=e^{-{c(\tanh \frac{t}2})|z|^2} $.  We take the $L^p(\R^d)$ norm in $x$ of the previous inequality, and by the Minkowski and the H\"older  inequalities  we get
\begin{eqnarray*} \|m_{1,k}(s,\cdot)\|_{L^p(\R^d)} &\leq& \|  G_t\|_{L^1(\R^d)}  \|\nabla S_{{k-3}} f \|_{L^{p_1}(\R^d)} \|\delta_k g \|_{L^{p_2}(\R^d)} \\
&\leq&   C(\tanh \frac{t}2)^{-\frac{d}2} \|\nabla S_{{k-3}} f \|_{L^{p_1}(\R^d)} \|\delta_k g \|_{L^{p_2}(\R^d)} ,
\end{eqnarray*}
which finally entails
$$ \|\widetilde{\mathcal{L}}_k(f,g)\|_{L^p(\R^d)} \leq C (\tanh {\frac{t}{2}})^{-\frac{d+1}2} (\sinh t)^{-\frac{d}2}e^{-ct 2^{2k}}\|\nabla S_{{k-3}} f \|_{L^{p_1}(\R^d)} \|\delta_k g \|_{L^{p_2}(\R^d)}.$$
For $0<t \leq 1$ we have $ (\tanh \frac{t}2)^{-\frac{d+1}{2}} (\sinh t)^{-\frac{d}{2}} \leq C t^{-d-\frac12} \leq Ct^{-d-\frac12}e^{-\frac{dt}4}$, and for $t \geq 1$ we have  $(\tanh \frac{t}2)^{-\frac{d+1}2} (\sinh t)^{-\frac{d}2} \leq Ce^{-\frac{dt}2} \leq Ct^{-d-\frac12}e^{-\frac{dt}4}$, which implies \eqref{borneh1}. \medskip

$\bullet$ The bounds \eqref{borneh2} and \eqref{borneh1} give
\begin{equation} \label{bb1}
\|\mathcal{L}^{\mathbf{1}}_k(f,g)\|_{L^p(\R^d)} \leq C_N \big( t^{-d-\frac12}         e^{-ct 2^{2k}}+ 2^{-kN}\big) e^{-\frac{dt}4} \big(\|\nabla S_{{k-3}} f \|_{L^{p_1}(\R^d)}+\| S_{k} f \|_{L^{p_1}(\R^d)}\big)  \|\delta'_k g \|_{L^{p_2}(\R^d)}.
\end{equation}
Since $\alpha<1$, we can apply the two estimates of Lemma \ref{lemBesov2}, so that 
\begin{equation*} 
\|\nabla S_{{k-3}} f \|_{L^{p_1}(\R^d)}+\| S_{k} f \|_{L^{p_1}(\R^d)} \leq  C 2^{k(1-\alpha)}\|   f \|_{\B^{\alpha}_{p_1, \infty}(\R^d)}.
\end{equation*}
Then, recalling that the $\B^\alpha_{p,\infty}$ norm does not depend on the choice of the cutoff function $\theta$, we have
$$ \|\delta'_k g \|_{L^{p_2}(\R^d)} \leq C 2^{-\beta k}\|   g\|_{\B^{\beta}_{p_2, \infty}(\R^d)}.$$
Hence from \eqref{bb1} we infer
\begin{eqnarray} \label{bb2}
2^{k \gamma}\|\mathcal{L}^{\mathbf{1}}_k(f,g)\|_{L^p(\R^d)} &\leq &C_N \big( t^{-d-\frac12}         e^{-ct 2^{2k}}+ 2^{-kN}\big) e^{-\frac{dt}4}2^{k (1+\gamma-\alpha -\beta)}  \|   f \|_{\B^{\alpha}_{p_1, \infty}(\R^d)} \|   g\|_{\B^{\beta}_{p_2, \infty}(\R^d)} \nonumber \\
&\leq &C_N 2^kt^{\frac{\eps}2} t^{-\frac{1}2 (\gamma-\alpha -\beta+\eps)}F_1(t,k) e^{-\frac{dt}8} \|   f \|_{\B^{\alpha}_{p_1, \infty}(\R^d)} \|   g\|_{\B^{\beta}_{p_2, \infty}(\R^d)}, 
\end{eqnarray}
with $F_1(t,k) :=   \big( t^{-d-\frac12}        e^{-ct 2^{2k}}+ 2^{-kN}\big) (t2^{2k})^{\frac12 (\gamma-\alpha -\beta)}e^{-\frac{dt}8}$.

Similarly, from \eqref{bound0} we deduce
\begin{eqnarray}  \label{bb3}
 2^{k \gamma} \|\mathcal{L}^{\mathbf{1}}_k(f,g)\|_{L^p(\R^d)}&\leq & Ct^{\frac12} 2^{k (1+\gamma-\alpha -\beta)}e^{-dt}\|   f \|_{\B^{\alpha}_{p_1, \infty}(\R^d)} \|   g\|_{\B^{\beta}_{p_2, \infty}(\R^d)}\nonumber \\
   &\leq  & Ct^{\frac{\eps}2}t^{-\frac{1}2 (\gamma-\alpha -\beta+\eps)} F_2(t,k)e^{-\frac{dt}8} \|   f \|_{\B^{\alpha}_{p_1, \infty}(\R^d)} \|   g\|_{\B^{\beta}_{p_2, \infty}(\R^d)},
  \end{eqnarray}
  with $F_2(t,k) :=  (t2^{2k})^{\frac12 (1+\gamma-\alpha -\beta)}e^{-\frac{dt}8}= (t 2^{2k})^{\frac12} (t2^{2k})^{\frac12 (\gamma-\alpha -\beta)}e^{-\frac{dt}8}$.
  
  Now we interpolate between \eqref{bb2} and \eqref{bb3} to get, for all $0<\eps<1$
  $$ 2^{k \gamma} \|\mathcal{L}^{\mathbf{1}}_k(f,g)\|_{L^p(\R^d)} \leq    C_Nt^{-\frac{1}2 (\gamma-\alpha -\beta+\eps)}  (t2^{2k})^{\frac{\eps}2}F^{\eps}_1(t,k) F^{1-\eps} _2(t,k) e^{-\frac{dt}8}\|   f \|_{\B^{\alpha}_{p_1, \infty}(\R^d)} \|   g\|_{\B^{\beta}_{p_2, \infty}(\R^d)} .   $$
  We now observe that   if $N\geq 1$ is large enough we have  
  \begin{eqnarray*}
 (t2^{2k})^{\frac{\eps}2}F^{\eps}_1(t,k) F^{1-\eps} _2(t,k)  &\lesssim&    \big(  t^{-\eps (d +\frac12)} e^{-c\eps t 2^{2k}}+ 2^{-\eps kN}\big) (t2^{2k})^{\frac12 (1+\gamma-\alpha -\beta)} e^{-\frac{dt}8} \\
 &\lesssim &   t^{-\eps (d +\frac12)}
  \end{eqnarray*}
   uniformly in $t>0$ and $k\geq 0$, which implies that
   \begin{equation}\label{pcom}
     2^{k \gamma} \|\mathcal{L}^{\mathbf{1}}_k(f,g)\|_{L^p(\R^d)} \leq    Ct^{-\frac{1}2 (\gamma-\alpha -\beta)}t^{-\eps(d+1) }  e^{-\frac{dt}8}\|   f \|_{\B^{\alpha}_{p_1, \infty}(\R^d)} \|   g\|_{\B^{\beta}_{p_2, \infty}(\R^d)} .   
      \end{equation}   
    \medskip

Observe now that $S_{k+6} \Big( \mathcal{L}^{\mathbf{1}}_k(f,g)\Big)= \mathcal{L}^{\mathbf{1}}_k(f,g)$, then we can apply Lemma \ref{lem-supp} to deduce that 
\begin{equation*} 
 \big \| \mathcal{P}^{\mathbf{1}} (f,g)  \big\|_{\B^\gamma_{p, \infty}(\R^d)} \lesssim \sup_{k \geq -1}    \big( 2^{k \gamma} \|\mathcal{L}^{\mathbf{1}}_k(f,g)\|_{L^p(\R^d)}   \big),
   \end{equation*}
   which together with \eqref{pcom}, completes the proof of \eqref{com-P1}.
       \medskip
    
 \underline{Step 2}:    Contribution of    $ \mathcal{P}^{\mathbf{2}} (f,g) $.  \medskip
 
Recall the definition \eqref{def-L2} of $ \mathcal{L}^{\mathbf{2}}_k(f,g)$. Then  we have
      \begin{equation*} 
 \big \| \mathcal{P}^{\mathbf{2}} (f,g)  \big\|_{\B^\gamma_{p, \infty}(\R^d)} \leq  \sum_{k \geq -1} \| \mathcal{L}^{\mathbf{2}}_k(f,g) \|_{\B^\gamma_{p, \infty}(\R^d)} \leq  \sum_{k \geq -1}  \sum_{j \geq k+4}  \| \delta_{j} \mathcal{L}_k(f,g) \|_{\B^\gamma_{p, \infty}(\R^d)} .
   \end{equation*}
 Then, from the definition of $ \mathcal{L}_k(f,g)$ we get 
       \begin{equation*} 
 \big \| \mathcal{P}^{\mathbf{2}} (f,g)  \big\|_{\B^\gamma_{p, \infty}(\R^d)} \leq   \mathcal{A}(f,g)+ \mathcal{B}(f,g)
   \end{equation*}
   with
   $$      \mathcal{A}(f,g): =   \sum_{k \geq -1}  \sum_{j \geq k+4}    \sum_{\ell \leq {k-4}}  \big\| e^{-t H }  \delta_{j} \Big((\delta_{\ell} f ) (\delta_k g) \Big) \big\|_{\B^\gamma_{p, \infty}(\R^d)},   $$
   $$      \mathcal{B}(f,g): =  \sum_{k \geq -1}  \sum_{j \geq k+4}    \sum_{\ell \leq {k-4}}   \big\|  \delta_{j} \Big((\delta_{\ell} f ) (\delta_k e^{-t H } g) \Big) \big\|_{\B^\gamma_{p, \infty}(\R^d)} .$$
   We only study the contribution of the term $ \mathcal{A}(f,g)$, the second is treated similarly. \medskip
 
   By \eqref{detJ}, for all $N\geq 1$ we have
   \begin{eqnarray*}
    \big\| e^{-t H }  \delta_{j} \Big((\delta_{\ell} f ) (\delta_k g) \Big) \big\|_{\B^\gamma_{p, \infty}(\R^d)} &\lesssim & e^{-dt  } 2^{j \gamma}\big\|    \delta_{j} \Big((\delta_{\ell} f ) (\delta_k g) \Big) \big\|_{L^p(\R^d)} \\
    &\lesssim&  e^{-dt  }  2^{- jN } \| \delta'_{\ell} f\|_{L^{p_1}(\R^d)}\| \delta'_k g\|_{L^{p_2}(\R^d)}.
       \end{eqnarray*}
   By definition, $\| \delta'_{\ell}  f\|_{L^{p_1}(\R^d)} \lesssim 2^{-\ell \alpha}\|f \|_{\B^\alpha_{p_1, \infty}(\R^d)}$ and $\| \delta'_{k}  g\|_{L^{p_2}(\R^d)} \lesssim 2^{-k \beta} \|g \|_{\B^\beta_{p_2, \infty}(\R^d)}$, thus for ${N\geq 1}$ large enough,
     \begin{eqnarray*}
   \mathcal{A}(f,g)       &\lesssim &     e^{-dt  } \|f \|_{\B^\alpha_{p_1, \infty}(\R^d)} \|g \|_{\B^\beta_{p_2, \infty}(\R^d)} \sum_{k \geq -1} 2^{-k \beta}    \sum_{j \geq k+4}  2^{- jN }   \sum_{\ell \leq {k-4}}   2^{-\ell \alpha} \\
    &\lesssim &     e^{-dt  } \|f \|_{\B^\alpha_{p_1, \infty}(\R^d)} \|g \|_{\B^\beta_{p_2, \infty}(\R^d)} \sum_{k \geq -1} 2^{-(N-|\al|-|\beta|)k}    \\
     &\lesssim &     e^{-dt  } \|f \|_{\B^\alpha_{p_1, \infty}(\R^d)} \|g \|_{\B^\beta_{p_2, \infty}(\R^d)}.
    \end{eqnarray*}
   
   Hence we have obtained
    \begin{equation*} 
 \big \| \mathcal{P}^{\mathbf{2}} (f,g)  \big\|_{\B^\gamma_{p, \infty}(\R^d)}  \lesssim  e^{-dt  }\| f \|_{\B^\alpha_{p_1, \infty}(\R^d)}\| g \|_{\B^\beta_{p_2, \infty}(\R^d)},
   \end{equation*}
which gives an admissible contribution in \eqref{com-1}.  \medskip

This completes the proof of Lemma \ref{lem1.2}.
 \end{proof}

 
 \

\section{Mild Young integral}\label{appendix:young}

The following natural extension of the classical (mild) Lebesgue integral is at the core of our interpretation and control procedure for the auxiliary system \eqref{mild:v}-\eqref{mild:w}.

\begin{proposition}\label{prop:young}
Let $d \geq 1$. Let $T>0$. For $0<\la,\eta<1$, consider a path $ f\in  {\ov \cac}^{1-\la}\big([0,T];  \cb^{-\eta}(\R^d)\big)$.  

\smallskip

Then for all $0\leq s<t\leq T$ and $u\in \cac^{\ga}\big( [0,T];  \cb^{\beta}(\R^d)\big)$ with $\ga>\la$ and $\beta>\eta$, the Riemann sum 
$$S^{(n)}_{s,t}:=\sum_{i=0}^{2^n-1}e^{-(t-t^n_i)H}\big(u_{t^n_i} \cdot (f_{t^n_{i+1}}-f_{t^n_i})\big), \quad \quad t^n_i:=s+\frac{i(t-s)}{2^n},$$
converges in $\cb^{\mu}(\R^d)$ as $n\to \infty$, for every $0\leq \mu <2-\eta-2\la$. 
We naturally denote its limit by
$$\int_s^t  \, e^{-(t-r)H}\big(u_r \cdot d_r f\big)=\ci\big( u \cdot df\big)_{s,t}.$$
Moreover:

\smallskip

\noindent
$(i)$ For every $\varepsilon>0$ small enough, one has
\begin{equation}\label{boun-youn-2}
  \Big\|\ci\big( u \cdot df\big)_{s,t}\Big\|_{\cb_x^{\mu}} \lesssim  |t-s|^{1-\frac{\mu+\eta}{2}-\la-\varepsilon} \big\| u\big\|_{ \cac^{\ga}( [0,T];  \cb^{\beta}_x)} \big\|f\big\|_{{\ov\cac}^{1-\la}([0,T]; \cb^{-\eta}_x)}.
\end{equation}
As a consequence, the element
$$t \mapsto \int_0^t e^{-(t-r)H}\big(u_r \cdot d_r f\big)  $$
is a.s. well defined through Young integration as an element of $ \cac^\ga\big([0,T]; \cb^{\mu}(\R^d)\big)$, for all 
$$0\leq \mu <2-\eta-2\la \;\; \text{ and }\;\;  0\leq \ga <\big(1-\frac{\eta}2-\lambda \big) -\frac{\mu}{2}.$$
\smallskip

\noindent
 $(ii)$   For any regular path $f\in {\ov\cac}^1\big([0,T];  L^\infty(\R^d)\big)$, it holds that  
$$\ci\big( u \cdot df\big)_{s,t} =\int_s^t  \, e^{-(t-r)H}\big(u_r \cdot f'_r\big)\, dr.$$
\end{proposition}

\

\begin{proof} 
For more clarity, we set $P_t=e^{-t H}$,  as well as ${\ov \cac}_T^{\alpha}\cb^{\beta}_x= {\ov \cac}^{\alpha}\big([0,T];\cb^{\beta}(\R^d)\big)$ and $\cac_T^{\alpha}\cb^{\beta}_x= \cac^{\alpha}\big([0,T];\cb^{\beta}(\R^d)\big)$.  We will prove that for  every $\varepsilon>0$ small enough, one has
\begin{equation}\label{boun-youn}
\Big\|\ci\big( u \cdot df\big)_{s,t}-P_{t-s}\big(u_s \cdot (f_t-f_s)\big)\Big\|_{\cb_x^{\mu}} \lesssim |t-s|^{1-\frac{\mu+\eta}{2}-\la-\varepsilon} \big\| u\big\|_{ \cac^{\ga}_T \cb^{\beta}_x} \big\|f\big\|_{{\ov\cac}^{1-\la}_T \cb^{-\eta}_x},
\end{equation}
which will immediately imply \eqref{boun-youn-2}. \medskip

\smallskip

Pick $s,t\in [0,T]$. One can check that
\begin{align*}
&S^{(n+1)}_{s,t}-S^{(n)}_{s,t}=\\
&-\sum_{i=0}^{2^n-1} \bigg[P_{t-t_{2i}}\big(u_{t_{2i}} \cdot (f_{t_{2i+2}}-f_{t_{2i}})\big)-P_{t-t_{2i}}\big(u_{t_{2i}} \cdot (f_{t_{2i+1}}-f_{t_{2i}})\big)-P_{t-t_{2i+1}}\big(u_{t_{2i+1}} \cdot (f_{t_{2i+2}}-f_{t_{2i+1}})\big)\bigg]
\end{align*}
where each $t_k$ in the sum refers in fact to $t^{n+1}_k:=s+\frac{k(t-s)}{2^{n+1}}$. Using this notation, each summand simplifies into
\begin{align*}
&P_{t-t_{2i}}\big(u_{t_{2i}} \cdot (f_{t_{2i+2}}-f_{t_{2i}})\big)-P_{t-t_{2i}}\big(u_{t_{2i}} \cdot (f_{t_{2i+1}}-f_{t_{2i}})\big)-P_{t-t_{2i+1}}\big(u_{t_{2i+1}} \cdot (f_{t_{2i+2}}-f_{t_{2i+1}})\big)\\
&=P_{t-t_{2i}}\big(u_{t_{2i}} \cdot (f_{t_{2i+2}}-f_{t_{2i+1}})\big)-P_{t-t_{2i+1}}\big(u_{t_{2i+1}} \cdot (f_{t_{2i+2}}-f_{t_{2i+1}})\big)\\
&=-\bigg[\big(P_{t-t_{2i+1}} -P_{t-t_{2i}}\big)\big(u_{t_{2i}} \cdot (f_{t_{2i+2}}-f_{t_{2i+1}})\big)+P_{t-t_{2i+1}}\big((u_{t_{2i+1}}-u_{t_{2i}}) \cdot (f_{t_{2i+2}}-f_{t_{2i+1}})\big)\bigg]\\
&=-\bigg[P_{t-t_{2i+1}}\big(\id -P_{t_{2i+1}-t_{2i}}\big)\Big(u_{t_{2i}} \cdot (f_{t_{2i+2}}-f_{t_{2i+1}})\Big)+P_{t-t_{2i+1}}\Big((u_{t_{2i+1}}-u_{t_{2i}}) \cdot (f_{t_{2i+2}}-f_{t_{2i+1}})\Big)\bigg].
\end{align*}

As a result, for any $\nu>\eta$,
\small
\begin{align*}
&\big\|S^{(n+1)}_{s,t}-S^{(n)}_{s,t}\big\|_{\cb_x^{\mu}}\\
&\leq \sum_{i=0}^{2^n-1}\bigg[ \Big\|P_{t-t_{2i+1}}\big(\id -P_{t_{2i+1}-t_{2i}}\big)\Big(u_{t_{2i}} \cdot (f_{t_{2i+2}}-f_{t_{2i+1}})\Big)\Big\|_{\cb_x^{\mu}}+\\
&\hspace{6cm}+\Big\|P_{t-t_{2i+1}}\Big((u_{t_{2i+1}}-u_{t_{2i}}) \cdot (f_{t_{2i+2}}-f_{t_{2i+1}})\Big)\Big\|_{\cb_x^{\mu}}\bigg]\\
&\lesssim \sum_{i=0}^{2^n-1}\bigg[ \frac{1}{|t-t_{2i+1}|^{\frac{\mu+\nu}{2}}}\Big\|\big(\id -P_{t_{2i+1}-t_{2i}}\big)\Big(u_{t_{2i}} \cdot (f_{t_{2i+2}}-f_{t_{2i+1}})\Big)\Big\|_{\cb_x^{-\nu}}\\
&\hspace{6cm}+\frac{1}{|t-t_{2i+1}|^{ \frac{\mu+\eta}{2}}}\Big\|(u_{t_{2i+1}}-u_{t_{2i}}) \cdot (f_{t_{2i+2}}-f_{t_{2i+1}})\Big\|_{\cb_x^{-\eta}}\bigg]\\
&\lesssim \sum_{i=0}^{2^n-1}\bigg[ \frac{1}{|t-t_{2i+1}|^{\frac{\mu+\nu}{2}}}|t_{2i+1}-t_{2i}|^{\frac{\nu-\eta}{2}} \Big\|u_{t_{2i}} \cdot (f_{t_{2i+2}}-f_{t_{2i+1}})\Big\|_{\cb_x^{-\eta}}\\
&\hspace{6cm}+\frac{1}{|t-t_{2i+1}|^{\frac{\mu+\eta}{2}}}\big\|u_{t_{2i+1}}-u_{t_{2i}}\big\|_{\cb_x^\beta} \big\|f_{t_{2i+2}}-f_{t_{2i+1}}\big\|_{\cb_x^{-\eta}}\bigg]\\
&\lesssim  \big\| u\big\|_{\cac_T^\ga \cb^{\beta}_x}  \big\|f\big\|_{{\ov \cac}_T^{1-\la}\cb^{-\eta}_x}\\
&\hspace{1cm}\sum_{i=0}^{2^n-1}\bigg[ \frac{1}{|t-t_{2i+1}|^{\frac{\mu+\nu}{2}}}|t_{2i+1}-t_{2i}|^{\frac{\nu-\eta}{2}}|t_{2i+2}-t_{2i+1}|^{1-\la}+\frac{1}{|t-t_{2i+1}|^{\frac{\mu+\eta}{2}}}|t_{2i+1}-t_{2i}|^{\ga}|t_{2i+2}-t_{2i+1}|^{1-\la}\bigg]\\
&\lesssim  \big\| u\big\|_{\cac_T^\ga \cb^{\beta}_x}  \big\|f\big\|_{{\ov \cac}_T^{1-\la}\cb^{-\eta}_x}\bigg[ \frac{1}{(2^{n+1})^{\frac{\nu-\eta}{2}-\la}} \bigg(\frac{1}{2^{n+1}}\sum_{i=0}^{2^n-1}\frac{1}{|t-t_{2i+1}|^{\frac{\mu+\nu}{2}}}\bigg)+\frac{1}{(2^{n+1})^{\ga-\la}} \bigg(\frac{1}{2^{n+1}}\sum_{i=0}^{2^n-1}\frac{1}{|t-t_{2i+1}|^{\frac{\mu+\eta}{2}}}\bigg)\bigg].
\end{align*}
\normalsize
Let us choose $\nu:=\eta+2\la+2\varepsilon$. Since for any $\kappa>0$ we have 
$$\frac{1}{2^{n+1}}\sum_{i=0}^{2^n-1}\frac{1}{|t-t_{2i+1}|^{1-\ka}} \sim \int_s^t \frac{dr}{|t-r|^{1-\ka}} =c \, |t-s|^{\ka},$$
we can conclude that
$$\big\|S^{(n+1)}_{s,t}-S^{(n)}_{s,t}\big\|_{\cb_x^{\mu}} \lesssim |t-s|^{1-\frac{\mu+\eta}{2}-\la-\varepsilon} \big\| u\big\|_{\cac_T^\ga \cb^{\beta}_x}  \big\|f\big\|_{{\ov \cac}^{1-\lambda}_T\cb^{-\eta}_x} \bigg[ \frac{1}{(2^{n+1})^{\varepsilon}}+\frac{1}{(2^{n+1})^{\ga-\la}}\bigg].$$
Thus $(S^{(n)}_{s,t})$ is a Cauchy sequence in $\cb^{\mu}(\R^d)$ and 
$$\big\|S^{(n)}_{s,t}-S^{(0)}_{s,t}\big\|_{\cb_x^{\mu}} \lesssim |t-s|^{1-\frac{\mu+\eta}{2}-\la-\varepsilon}  \big\| u\big\|_{\cac^\ga ([0,T]; \cb^{\beta}_x)}  \big\|f\big\|_{{\ov \cac}^{1-\la}([0,T]; \cb^{-\eta}_x)},$$
which, by letting $n$ tend to $\infty$, yields the desired bound \eqref{boun-youn}.
\end{proof}
 
\

\section{Some technical estimates}\label{appendix:technical}
 
We hereby gather a few basic ingredients that have proven useful in our analysis.

 \subsection{Elementary estimates}
We start with three elementary estimates on singular integrals.
             \begin{lemma}
       Let $0\leq \gamma< 1 $. Then  
       \begin{equation}\label{elem}
       \int_{[v_1,v_2]^2}\frac{dt_1 dt_2}{|t_2-t_1|^{\gamma}}  \lesssim (v_2-v_1)^{2-\gamma}.
       \end{equation}
       \end{lemma}

       \begin{lemma}\label{Lem-prod}
       Let $0<\alpha, \beta <1 $  such that $\alpha+\beta>1$. Then for all $x,y \in \R$ such that $x \neq y$
       \begin{equation}\label{xayb}
       \int_{-\infty}^{+\infty}   \frac{ds     }{  |s-x|^{\alpha} |s-y|^{\beta}} \lesssim \frac{1}{|x-y|^{\alpha+\beta-1}}.
       \end{equation}
       \end{lemma}
       
       \begin{proof}
       We make the change of variables $t=s-x$ and then $\dis r=\frac{t}{y-x}$, so that 
           \begin{multline*}
       \int_{-\infty}^{+\infty}   \frac{ds     }{  |s-x|^{\alpha} |s-y|^{\beta}} =       \int_{-\infty}^{+\infty}   \frac{dt     }{  |t|^{\alpha} |t-(y-x)|^{\beta}} 
       =  \frac{1}{|x-y|^{\alpha+\beta-1}} \int_{-\infty}^{+\infty}   \frac{dr     }{  |r|^{\alpha} |r-1|^{\beta}}  \lesssim  \frac{1}{|x-y|^{\alpha+\beta-1}}.
       \end{multline*}
       \end{proof}

 \begin{lemma}\label{Lem-prod2}
   Let $t_1\neq t_2$.    Let $0<\gamma,\alpha_1, \alpha_2, \beta_1,\beta_2 <1 $. Assume moreover the following conditions
   \begin{align*}
&   \gamma+\alpha_1+ \alpha_2+ \beta_1+\beta_2>2 \\
&\gamma+\alpha_1+ \alpha_2>1 \\
&\gamma+\beta_1+ \beta_2>1  \\ 
  &\gamma+\alpha_1+ \beta_1<2 \\
    &\gamma+\alpha_2+ \beta_2<2. 
      \end{align*}
 Then  
       \begin{equation}\label{xayb2}
     \int_{0}^{t_1}   \int_{0}^{t_2} \frac{ds_1ds_2     }{  |s_2-s_1|^{\gamma}|t_1-s_1|^{\alpha_1} |t_2-s_1|^{\alpha_2}  |t_1-s_2|^{\beta_1} |t_2-s_2|^{\beta_2}}\lesssim \frac{1}{|t_2-t_1|^{\gamma+\alpha_1+ \alpha_2+ \beta_1+\beta_2-2}}.
       \end{equation}
       \end{lemma}
       
       \begin{proof} In the sequel, we  assume that $t_1<t_2$. Then, observe that the conditions    ${0<\gamma,\alpha_1, \alpha_2, \beta_1,\beta_2 <1}$ ensure that the lhs of \eqref{xayb2} is finite. Then, we make the change of variables $\sigma_1=t_1-s_1$, $\sigma_2=t_2-s_2$ and then $x=\frac{\sigma_1}{t_2-t_1}$, $y=\frac{\sigma_2}{t_2-t_1}$, therefore if   $\delta:= \gamma+\alpha_1+ \alpha_2+ \beta_1+\beta_2-2$
           \begin{multline*}
       \int_{0}^{t_1}   \int_{0}^{t_2} \frac{ds_1ds_2     }{  |s_2-s_1|^{\gamma}|t_1-s_1|^{\alpha_1} |t_2-s_1|^{\alpha_2}  |t_1-s_2|^{\beta_1} |t_2-s_2|^{\beta_2}} \lesssim \\
              \begin{aligned}
      &  \lesssim    \int_{0}^{t_1}   \int_{0}^{t_2} \frac{d\sigma_1d\sigma_2     }{  \big|t_2-t_1-(\sigma_2-\sigma_1)\big|^{\gamma}|\sigma_1|^{\alpha_1} |t_2-t_1+\sigma_1|^{\alpha_2}  |t_2-t_1-\sigma_2|^{\beta_1} |\sigma_2|^{\beta_2}}\\
   &    \lesssim   \frac{1}{(t_2-t_1)^\delta} \int_{0}^{\frac{t_1}{t_2-t_1}}   \int_{0}^{\frac{t_2}{t_2-t_1}} \frac{dxdy     }{  \big|1-(y-x)\big|^{\gamma}|x|^{\alpha_1} |1+x|^{\alpha_2}  |1-y|^{\beta_1} |y|^{\beta_2}} \\ 
      &    \lesssim   \frac{1}{(t_2-t_1)^\delta} \int_{-\infty}^{+\infty} \frac{dy}{ |y|^{\beta_1} |y+1|^{\beta_2}} \int_{0}^{+\infty} \frac{dx     }{  |y-x|^{\gamma}x^{\alpha_1} (1+x)^{\alpha_2} }   .
                          \end{aligned}
       \end{multline*}
       We now have to check that the integral in the rhs is finite. We distinguish different cases:  \medskip
       
       $\bullet$ Let  $|y|\geq 2$. On the one hand we have 
       $$   \int_{0}^{1} \frac{dx     }{  |y-x|^{\gamma}|x|^{\alpha_1} |1+x|^{\alpha_2} } \lesssim \frac{1}{y^{\gamma}}  \int_{0}^{1} \frac{dx     }{x^{\alpha_1} (1+x)^{\alpha_2} } \lesssim \frac{1}{y^{\gamma}}. $$
       On the other hand      
       \begin{eqnarray*} 
        \int_{1}^{+\infty} \frac{dx     }{  |y-x|^{\gamma}x^{\alpha_1} (1+x)^{\alpha_2} } &\lesssim & \int_{1}^{+\infty} \frac{dx     }{  |y-x|^{\gamma}x^{\alpha_1+\alpha_2} } \\
        &\lesssim & \frac1{y^{\gamma +\alpha_1+\alpha_2-1}}\int_{\frac1y}^{+\infty} \frac{dz     }{  |1-z|^{\gamma}z^{\alpha_1+\alpha_2} } .
        \end{eqnarray*} 
        We write $\dis \int_{\frac1y}^{+\infty}= \int_{\frac1y}^{\frac23}+ \int_{\frac23}^{\frac32}+  \int_{\frac32}^{+\infty}$. The two last integrals converge, and for the first one we have $\dis  \int_{\frac1y}^{\frac23}   \frac{dz     }{  |1-z|^{\gamma}z^{\alpha_1+\alpha_2} } \lesssim  \int_{\frac1y}^{\frac23}   \frac{dz     }{   z^{\alpha_1+\alpha_2} }  $. After computation, we get 
    $$ \int_{1}^{+\infty} \frac{dx     }{  |y-x|^{\gamma}x^{\alpha_1} (1+x)^{\alpha_2} } \lesssim           \frac1{y^{\kappa} }   $$    
        where
       $$\kappa:=
\begin{cases}
 \gamma +\alpha_1+\alpha_2-1&  \text{if} \ \alpha_1+\alpha_2<1\\
 \gamma -\eps&  \text{if} \ \alpha_1+\alpha_2=1\\
 \gamma &  \text{if} \ \alpha_1+\alpha_2>1
\end{cases}.$$

       As a consequence
       $$ \int_{|y|\geq 2} \frac{dy}{ |y|^{\beta_1} |y+1|^{\beta_2}} \int_{0}^{+\infty} \frac{dx     }{  |y-x|^{\gamma}x^{\alpha_1} (1+x)^{\alpha_2} } \lesssim  \int_{|y|\geq 2} dy (\frac1{|y|^{\beta_1+\beta_2+\gamma}}+ \frac1{|y|^{\beta_1+\beta_2+\kappa}}) <\infty, $$
       since we assume that $\gamma+ \beta_1+\beta_2>1$ and $\gamma+\alpha_1 +\alpha_2+\beta_1+\beta_2>2$. \medskip
       
       $\bullet$ Let  $|y|\leq 2$.  Firstly, we have 
       $$   \int_{x\geq 3} \frac{dx     }{  |y-x|^{\gamma}|x|^{\alpha_1} |1+x|^{\alpha_2} } \lesssim \int_{x\geq 3} \frac{dx     }{x^{\gamma+\alpha_1+\alpha_2} } \lesssim 1, $$
       since $\gamma+\alpha_1+\alpha_2>1$. Thus 
     $$     \int_{|y|\leq 2} \frac{dy}{ |y|^{\beta_1} |y+1|^{\beta_2}}    \int_{x\geq 3} \frac{dx     }{  |y-x|^{\gamma}|x|^{\alpha_1} |1+x|^{\alpha_2} }  \lesssim 1.$$
       Next, with similar arguments as previously, 
       $$\int_{0}^3 \frac{dx     }{  |y-x|^{\gamma}|x|^{\alpha_1} |1+x|^{\alpha_2} } \lesssim \int_{0}^3 \frac{dx     }{  |y-x|^{\gamma}|x|^{\alpha_1}  }  \lesssim \frac1{y^{\kappa'} } $$
         where
       $$\kappa':=
\begin{cases}
0&  \text{if} \ \gamma+ \alpha_1 <1\\
\eps&  \text{if} \  \gamma+ \alpha_1 =1\\
\gamma+ \alpha_1 -1 &  \text{if} \  \gamma+ \alpha_1 >1
\end{cases}.$$
Therefore
   $$     \int_{|y|\leq 2} \frac{dy}{ |y|^{\beta_1} |y+1|^{\beta_2}}    \int_{0}^3 \frac{dx     }{  |y-x|^{\gamma}|x|^{\alpha_1} |1+x|^{\alpha_2} }  \lesssim   \int_{|y|\leq 2} \frac{dy}{ |y|^{\beta_1+\kappa'} |y+1|^{\beta_2}}.$$
   This latter integral converges under the condition $\gamma+\alpha_1+\beta_1<2$.  \medskip
   
   Notice that the additional condition $\gamma+\alpha_2+\beta_2<2$ comes from the inspection of the case $t_2 <t_1$.
       \end{proof}

 
 \subsection{Estimates in Lebesgue spaces}

For $g : \R^d \times \R^d  \longrightarrow \R$, we denote by $\widetilde{g}$ defined by 
\begin{equation}\label{Ntilde}
\widetilde{g}(x) = \sup_{y \in \R^d}|g(y,y+x)|.
\end{equation}

 \begin{lemma}\label{lem-young}
Let $d \geq 1$. Let $1\leq p \leq 2$. Then,  the following bound holds true
$$\big|\int_{\R^d}\int_{\R^d} dx_1 dx_2 f_1(x_1)f_2(x_2)g(x_1,x_2)\big| \leq \|f_1\|_{L^p(\R^d} \|f_2\|_{L^p(\R^d)} \|\widetilde{g}\|_{L^{\frac{p'}{2}}(\R^d)} .$$
 \end{lemma}
 A particular case is given by the convolution. Namely, if $g(x_1,x_2)=h(x_2-x_1)$, then $\widetilde{g}(x)=|h(x)|$ and the previous estimate is the classical Young estimate
  $$\big|\int_{\R^d}\int_{\R^d} dx_1 dx_2 f_1(x_1)f_2(x_2)h(x_2-x_1)\big| \leq \|f_1\|_{L^p(\R^d)} \|f_2\|_{L^p(\R^d)} \|h\|_{L^{\frac{p'}2}(\R^d)}.$$     
       \begin{proof}
       By the H\"older inequality
       \begin{eqnarray*}
  \Big|     \int_{\R^d}\int_{\R^d} dx_1 dx_2 f_1(x_1)f_2(x_2)g(x_1,x_2) \Big|& \leq &   \|f_1\|_{L^p(\R^d)} \Big(\int dx_1 \Big| \int_{\R^d}  dx_2 f_2(x_2)g(x_1,x_2) \Big|^{p'}\Big)^{\frac{1}{p'}} \\
  & = &   \|f_1\|_{L^p(\R^d)}\Big(  \int dx_1 \Big| \int_{\R^d}  dx_2 f_2(x_2+x_1)g(x_1,x_2+x_1) \Big|^{p'}\Big)^{\frac{1}{p'}} \\
  & \leq  &   \|f_1\|_{L^p(\R^d)}\Big(  \int dx_1 \Big| \int_{\R^d}  dx_2 f_2(x_2+x_1)\widetilde{g}(x_2) \Big|^{p'}\Big)^{\frac{1}{p'}} \\
      & =   &   \|f_1\|_{L^p(\R^d)} \| f_2 \star \check{\widetilde{g}} \|_{L^{p'}(\R^d)} .
       \end{eqnarray*}
    Assume that $1\leq p \leq 2$,  then   we can apply the Young inequality for convolution and get
            \begin{equation*}
  \Big|     \int_{\R^d}\int_{\R^d} dx_1 dx_2 f_1(x_1)f_2(x_2)g(x_1,x_2) \Big|
 \leq    \|f_1\|_{L^p(\R^d)} \|f_2\|_{L^p(\R^d)}  \|\widetilde{g}\|_{L^{\frac{p'}{2}}(\R^d)}  
       \end{equation*}
       which was to prove.
       \end{proof}

The next estimate follows from the Garsia-Rodemich-Rumsey inequality and can be viewed as a form of Sobolev embedding for Hölder functions (see \cite[Corollary A.2, page 575]{FV}).
	
	\begin{proposition}\label{Prop-GRR}
	Consider $\big(E, \| \cdot \|\big)$ a normed space, and let $f \in \mathcal{C}\big(\R_+; E\big)$. Let $q>1$, $\alpha \in (\frac1q,1)$. Then there exists $C=C(\alpha, q)>0$ such that for all $0\leq s<t$
	$$\big\| f(t)-f(s)\big\| \leq C |t-s|^{\alpha-\frac1q}\bigg( \int \int_{[s,t]^2}dudv\frac{ \big\| f(v)-f(u)\big\|^q}{|v-u|^{q \alpha+1}}\bigg)^{\frac1q}.$$
	\end{proposition}
	In particular, recalling the definition \eqref{def-Ceta} of the space $\mathcal{C}^{\eta}\big([T_1,T_2]; \mathcal{B}^\gamma(\R^3)\big)$, we get for all $\eta >0$, $p\geq 1$ and~$\gamma \in \R$
	            \begin{equation}\label{est-GRR}
	       \big\| f\big\|_{\mathcal{C}^{\eta}([T_1,T_2]; \mathcal{B}_x^\gamma)} \leq \big\| f(T_1)\big\|_{ \mathcal{B}_x^\gamma}+C \bigg( \int \int_{[T_1,T_2]^2}dudv\frac{ \big\| f(v)-f(u)\big\|_{\mathcal{B}_x^\gamma}^{2p}}{|v-u|^{2p \eta+2}}\bigg)^{\frac1{2p}}  .
   \end{equation}          	
 	 Similarly, by \eqref{def-CetaBar} we have
	            \begin{equation}\label{est-GRR-bar}
	       \big\| f\big\|_{{\ov \cac}^{\eta}([T_1,T_2]; \mathcal{B}_x^\gamma)} \leq  C \bigg( \int \int_{[T_1,T_2]^2}dudv\frac{ \big\| f(v)-f(u)\big\|_{\mathcal{B}_x^\gamma}^{2p}}{|v-u|^{2p \eta+2}}\bigg)^{\frac1{2p}}  .
   \end{equation}

 \ 

 \

\begin{acknowledgements}
 A. Deya and L. Thomann  were partially supported by the ANR projet "SMOOTH" ANR-22-CE40-0017. 
R. Fukuizumi was supported by JSPS KAKENHI Grant Numbers 19KK0066 and 23H01079. 
\end{acknowledgements}

\


\end{document}